\documentclass[12pt,a4paper]{article}
\input{epsf}

\usepackage{amsmath,amssymb,amsfonts,graphicx,amsthm,mathrsfs}
\usepackage{cite}

\newtheorem{theo}{Theorem}
\newtheorem{lemma}{Lemma}
\newtheorem{prop}{Proposition}
\newtheorem{cor}{Corollary}

\newtheorem{rem}{Remark}
\newtheorem{defi}{Definition}

\DeclareMathOperator\supp{Supp}
\usepackage{enumerate}
\allowdisplaybreaks
\usepackage{hyperref}
\usepackage{stackrel}

\usepackage{array}

\begin{document}
  


 \title{Degenerate Complex Monge-Amp\`ere Equation-Part I}


 \maketitle

\vspace{1cm}

 Alireza Bahraini
\footnote{bahraini@sharif.edu}

Dept. of Mathematical Sciences,

Sharif University of Technology,

P.O.Box 11365-9415, Tehran, Iran.

\vspace{1cm}

\begin{abstract}  
We describe the behavior of the  singularities of  solutions to degenerate complex
Monge-Amp\`ere equations on K\"ahler manifolds. This fundamental question had remained unsolved since the pioneering  paper by S-T Yau \cite{y} on this subject.
\end{abstract} 
\section{Introduction}

Complex Monge-Amp\`ere equation   were studied to resolve a conjecture of Calabi on the existence of canonical metrics over   K\"ahler manifolds. The problem was first treated in a paper by S-T Yau   and T. Aubin in 1978  \cite{y}, \cite{aub}.
Solution to complex Monge Amp\`ere equations has profound consequences in a wide range of areas in mathematics and in  physics. From the study  of K\"ahler-Einstein
 metrics  \cite{gabor} to  the geometry of the space of K\"ahler metrics on a compact manifold  \cite{Mab}.   
From  complex algebraic geometry \cite{dem}  to string theory,
  numerous significant progress has  arisen from the    solutions to complex Monge-Amp\`ere equations.\cite{syz}
  
One of the most general results in this regard has been proved by S. Kolodziej (\cite{kol}), who has shown that on a compact K\"ahler manifold $(M,\omega)$ the Monge-Amp\`ere equation
\[
(\omega +\partial\bar{\partial}\phi)^n =F\omega^n 
\]

admits a continuous  solution for $\phi$ if $F\in L^1(M)$  satisfies $\int F\omega^n =\int_M \omega^n =Vol (M)$.
The case where $F$ is  positive and  $C^k$  or the case where $F$ is $C^k$ with zeros along a smooth divisor had already been 
treated in \cite{y}.
 The  case where $F$ has zeros is known as Degenerate Complex Monge-Amp\`ere  (DCMA) equation. 
One of the most fundamental open problems  with profound  geometric consequences  consists of identifying  the behavior    of the solutions to DCMA equation  near its singular locus. 

In this paper we will   resolve this longstanding open problem and we will prove the following theorem
\begin{theo}
Assume that $(X,\omega)$ is a compact K\"ahler manifold  of complex dimension $n$.
 Let $D\subset X$ be a smooth divisor and $S$ be a holomorphic section of $L:=[D]$ vanishing along $D$. Let $G$ be $C^{k} (X)$ with $k\geq 3$ and $\int_X \exp \{G\} |S|^2 \omega^n =Vol (X)$.
 Then there exists  a function $\phi$ in $C^{k+1, \alpha} (X) $ for $0\leq \alpha <\frac{1}{2}$ such that 
$\omega':=\omega+\partial\bar{\partial}\phi$ defines a K\"ahler metric on $X\setminus D$ and $\omega'|_{D}$ is a K\"ahler metric on $D$ and 
\begin{equation}\label{maineq}
{\omega ' }^ n=|S|^2 \exp\{G\} \omega ^n 
\end{equation}

\end{theo}

To prove the theorem we start by  resolving    in section \ref{sec2}, the  (analytic)  DCMA  equation    locally in a neighborhood of a divisor . In Section \ref{sec3}   we construct global models for K\"ahler 
metrics that arise as solutions to DCMA
equations. Section \ref{sec4} is devoted to the development of a Schauder theorem for Laplacian type operators associated to our degenerate  metric models.
The discussion about second order estimates corresponding to the continuity method is carried out in section \ref{sec5}. Third order estimate is treated in two major steps in section \ref{III}.
Finally we conclude in section \ref{sec7}. 

This is the first part of a series of at least four parts on the study of these singularities and its consequences. In part II we have improved the so called Bogomolov–Miyaoka–Yau inequality for complex surfaces of general type. In part III which is under progress we study the singularities Hermitian Yang Mills equations. 
\\
\\

\textbf{Acknowledgment} The authors interest to this problem has its roots in his PhD thesis  \cite{bah} where he has  studied   a class of singular complex manifolds arising
from singularities of canonical K\"ahler metrics.  This work was  also motivated and proposed  as a conjecture to the  author in 2017 by prof. S-T Yau. The author  is grateful to professor Yau for several helpful discussions. 

 \section{Local Construction of Solutions to Degenerate Monge Amp\`ere Equations with Analytic Coefficients}\label{sec2}

Consider   a complex K\"ahler manifold  $(X,\omega)$ of  complex dimension $n$. Assume that $D\subset X$ is a smooth divisor in $X$ and let  $S$ denote the  holomorphic section of $L:=[D]$ vanishing  along $D$.  
Assuming $\omega$  to be  real analytic, we take a smooth fiberation in a neighborhood of $D$ such that all of its fibers are holomorphic discs and are orthogonal to $D$ with respect 
to the K\"ahler metric $\omega$.  Fixing this fiberation along with  a transversally flat connection on $L$  allow us to talk about global Taylor series expansion in terms of $S$  in this neighborhood. We are then able to decompose  
the K\"ahler form $\omega$ (and more generally any real analytic differential form) in this neighborhood into component induced by  the fiberation and the divisor $D$ in
 the form $\omega=\omega_{S\bar{T}}+ \omega_{T\bar{S}}+ \omega_{T\bar{T}}+\omega_{S\bar{S}}$.  The holomorphic (and antiholomorphic) part of a differential  form  in its
   Taylor series expansion in terms of $S$ and $\bar{S}$  are defined.  (see definition (\ref{holpart})). 
The definitions only depend on the choice of the 
fiberation  and a transversally flat hermitian metric on the line bundle $L$ and  the divisor $D$.

\begin{theo}\label{theoloc}
Let $\omega$ be a real analytic K\"ahler metric on $X$. Let $F:U\rightarrow \mathbb{R}$ be a real  analytic map in a neighborhood $U$ of $D$. Then there exists a unique  solution 
$\Phi: V\rightarrow \mathbb{R}$ to the  degenerate complex  Monge Amp\`ere equation

\begin{equation}\label{degmongeq}
(\omega+\partial\bar{\partial}\Phi ) ^n = F\omega ^n
\end{equation}

satisfying the two properties: 
\\

1. \hspace{0.5cm}  $((\omega+\partial\bar{\partial}\Phi)_{S\bar{T}})_{hol}= (\omega_{S\bar{T}})_{hol}|_{V}$
\\

2.\hspace{0.5cm} $((\omega+\partial\bar{\partial}\Phi)_{T\bar{T}})_{hol}= (\omega_{T\bar{T}})_{hol}|_{V}$
\\

Here $V$ is a small neighborhood of $D$.

  In the case where $F$
is positive outside $D$ which means that $F= e^{\tilde{F}} |S|^k$ for some analytic map $\tilde{F}:U\rightarrow \mathbb{R}$ and for some real  positive integer $k$ the solution $\omega+\partial\bar\partial{\Phi}$ is positive definite in some deleted neighborhood $V_0\setminus D$ of $D$.
\end{theo}

In order to prove the above theorem we  show the existence and uniqueness of a solution locally  in a neighborhood of a point $p\in D$.    It is then clear that    local solutions glue together to 
form a global solution in a neighborhood of $D$. 
\\
\\

The idea of expansion with respect to $S$ was first communicated to the author by Prof. S-T Yau.
It has also been widely used in super-symmetric field theories in physics literature.
 
\subsection{Fibration over $D$ and transversally flat connection}\label{localmodel}
 
\vspace{1cm}

Our aim in  this section is to  construct an analytic  fiberation  on a small neighborhood $U$ of  $D$  as well as a  transversally flat connection over $L$. 
 Let $U$ be an open neighborhood of $D$ in $X$ and assume that $\pi_U : U\rightarrow D$ is a smooth fiberation   such that all the fibers   $\pi_U ^{-1} (\{p\})$, for $p\in D$ are  holomorphic disc  orthogonal to $D$ with respect to the K\"ahler metric $\omega$. 
We also assume that the boundary  of $U$, denoted by $\partial U$, is a smooth  submanifold of codimension one  in  $X$ and for all $p\in D$,  $\partial (\pi_U ^{-1} (\{p\}))$ is diffeomorphic to a circle.  
 We set
\[
N_p:=\pi_U ^{-1} (\{p\}), \hspace{1cm} \text{ for } p\in D
\]

\begin{lemma}
There exists a hermitian metric $||_{\mathfrak{h}}$ on $L|_{U}$ which is flat along all the fibers of $\pi_U$ by which we mean the restriction  $L|_{N_p}$ has zero curvature with respect $||_\mathfrak{h}$ for all $p\in D$.
\end{lemma}
\begin{proof}
 
Let $\mathfrak{g}_p: \bar{N}_p\rightarrow \bar{\mathbb{D}}$ be a diffeomorphism between the closure of $N_p$ and the closure of the unit disc $\mathbb{D}$ in $\mathbb{C}$. We assume that $\mathfrak{g}_p|_{N_p}: N_p \rightarrow \mathbb{D}$      is a  biholomorphism  and  $\mathfrak{g}_p (p )=0$.  The existence of $\mathfrak{g}_p$  is ensured by Riemann mapping theorem and  we know that it is uniquely determined upto a rotation.
 Since the normal bundle of $D$ in $X$ is nontrivial $\mathfrak{g}_p$ can not be globally  defined 
throughout the  neighborhood $U$ as a smooth function of $p$. Nevertheless  the map $\mathfrak{h}$ defined as
\begin{equation}\label{hfrak0}
\mathfrak{h}: U \rightarrow \mathbb{R}, \hspace{1cm} \mathfrak{h}(x)= \log |\mathfrak{g}_p (x)|
\end{equation}
 is  smooth and is  defined without ambiguity   on  $U$.
This is because $\mathfrak{g}_{p}$ is well defined upto a rotation.
If we set $\mathfrak{h}_p :=\mathfrak{h}|_{N_p}$ then since  $\mathfrak{g}_p$ is holomorphic we have
\begin{equation}\label{ddhfrak}
\partial \bar{\partial} \mathfrak{h}_p=0
\end{equation}

which means that $\mathfrak{h}_p$ is harmonic with respect to the induced conformal structure  on $\bar{N}_p$.

Let $S\in H^0 (X,L)$ be the  holomorphic section of the line bundle $L$ vanishing along $D$. We define a hermitian metric $||_{\mathfrak{h}}$ on $L|_{U\setminus D}$
by
\begin{equation}\label{hfrak}
|S(x)|_{\mathfrak{h}}= e^{\mathfrak{h}(x)}, \hspace{1cm} \text{ for } x\in U\setminus D
\end{equation}
  
We want to prove that the above hermitian metric has a smooth extension along the divisor $D$.

Let $U_1$ be an open subset of $X$ such that $U_1 =(\pi_U )^{-1} (U_0) $ where $U_0:=U_1\cap D$. Assume that there exists a smooth map $\mathfrak{g}_1 : U_1 \rightarrow \mathbb{C}$  such that for all $p\in U_1 \cap D$, $\mathfrak{g}_p:=\mathfrak{g}_1|_{N_p}$
is a biholomorphism between $N_p$ and $\mathbb{D}$ with $\mathfrak{g}_p (p)=0$.
 
We define  a diffeomorphism $\mathfrak{g}: U_1 \rightarrow (U_1\cap D)\times \mathbb{D}$ by setting $\mathfrak{g}=\pi_U\times \mathfrak{g}_1$. More precisely   $\mathfrak{g}(x)= (p, \mathfrak{g}_p (x))$ where  $p=\pi_U (x)$.  
If $\sigma :U_1\rightarrow L|_{U_1} $ is a nowhere vanishing holomorphic section of $L|_{U_1}$ then we have
\begin{equation}\label{es}
S|_{U_1}=s \sigma
\end{equation}
where $s:U_1\rightarrow \mathbb{C}$ is a  holomorphic map. In addition since $S$ has a simple zero along $D$ the map
\[
s\circ (\mathfrak{g}_p)^{-1}: \mathbb{D}\rightarrow \mathbb{C} 
\]

for all $p\in D $ is a holomrphic map with a simple zero at the origin. Consequently  for $s \circ \mathfrak{g}^{-1}: (U_{1}\cap D)\times  \mathbb{D}\rightarrow \mathbb{C}$  we have
\begin{equation}\label{s0f}
s \circ \mathfrak{g}^{-1} (x, z)= zg (x,z) \hspace{1cm}  x\in U_1 \cap D \text{ and } z\in \mathbb{D}
\end{equation}

where  $g:  (U_{1}\cap D)\times  \mathbb{D}\rightarrow \mathbb{C}$ is a smooth nowhere vanishing   application  and  for any fixed $x$
the map $z\rightarrow g(x,z)$ is holomorphic.

From the relation \ref{es} we have 
\begin{equation}\label{sog}
S\circ \mathfrak{g}^{-1} = (s \circ \mathfrak{g}^{-1})(\sigma \circ \mathfrak{g}^{-1})
\end{equation}
and from (\ref{hfrak0}) and (\ref{hfrak})  we obtain
\begin{equation}\label{sogvrun}
|S\circ \mathfrak{g}^{-1}(x,z)|_{\mathfrak{h}} =e^{\mathfrak{h}\circ \mathfrak{g}^{-1}(x,z)}=e^{\log |z| }=|z|
\end{equation}
according to (\ref{sog}) we  have
\[
|S\circ \mathfrak{g}^{-1}(x,z)|_{\mathfrak{h}}= |(s \circ \mathfrak{g}^{-1})(x,z)||(\sigma \circ \mathfrak{g}^{-1})(x,z)|_{\mathfrak{h}}
\]
therefore using (\ref{s0f}) and (\ref{sogvrun})  we get to 
\[
|\sigma \circ \mathfrak{g}^{-1} (x,z)|_{\mathfrak{h}}=\frac{|S\circ \mathfrak{g}^{-1}(x,z)|_{\mathfrak{h}}}{ |s \circ \mathfrak{g}^{-1}(x,z)|}=\frac{|z|}{|z||g(x,z)|}=\frac{1}{|g(x,z)|}
\]

as can be seen  $|\sigma_0 \circ \mathfrak{g}^{-1} (x,z)|_{\mathfrak{h}}$ has an smooth extension to $z=0$.
\end{proof}

\subsection{Transversally parallel  basis and decomposition of differential forms  }
The proof of the following lemma follows from flatness of the connection $\mathfrak{h}$ along the fibers of $\pi_U$:

\begin{lemma}\label{partr}
Let $\mathfrak{t}\in \Gamma (D, L|_D)$ be a smooth section of $L|_D$. Then there exists a unique  extension $\hat{\mathfrak{t}}\in \Gamma (U, L|_U)$, obtained
by parallel transport with respect to $||_{\mathfrak{h}}$ satisfying 
\[
\partial^{\nabla} \hat{\mathfrak{t}}|_{N_p}= \bar{\partial } \hat{\mathfrak{t}}|_{N_p}=0 \hspace{1cm} \text{ for all} \hspace{1cm} p\in D
\]

Here $\partial^{\nabla} $ denotes  the $(1,0)$-component of the  covariant derivative with respect to the Chern connection  $\nabla$ associated with the hermitian metric $\mathfrak{h}$. 
\end{lemma}

 \begin{defi}\label{gamhat}
We call the sections $\hat{\mathfrak{t}}$ constructed by lemma \ref{partr}  transversally parallel  sections of $L|_U$
and the space of such sections is denoted by $\hat{\Gamma}(U, L|_U)$. More generally any section $\hat{\mathfrak{s}}\in \Gamma (U, L^{\otimes i}\otimes \bar{L}^{\otimes j}) $ for $i\in \mathbb{Z}$ which is obtained from $\hat{\mathfrak{s}}|_{D}$ by parallel transport along the fibers of $\pi_U$
is called a transversally parallel section of $L^{\otimes i}\otimes \bar{L}^{\otimes j}$ and the space of all transversally parallel sections of $L^{\otimes i}\otimes \bar{L}^{\otimes j}$
is denoted by $\hat{\Gamma} (U, L^{\otimes i}\otimes \bar{L}^{\otimes j}|_U) $.
\end{defi}


Any  analytic map  $\Phi:U\rightarrow \mathbb{C}$ defined in a neighborhood $U$ of $D$ admits a Taylor series expansion like
\begin{equation}\label{Phi}
\Phi= \sum_{i,j=0} ^{\infty} \hat{\mathscr{C}}_{i,\bar{j}} S^i   \bar{S} ^j
\end{equation}

where the coefficients $\hat{\mathscr{C}}_{i.\bar{j}}\in \hat{\Gamma} (U,  \bar{L}^{\otimes (i)} \otimes  L^{\otimes (j)}|_U)$ are transversally parallel sections as  defined in  definition (\ref{gamhat}). Here we identify 
$L^{\otimes (i+j)}\otimes \bar{L}^{\otimes (i+j)}$ with the trivial bundle through the pairing induce by the hermitian structure $||_{\mathfrak{k}}$ on $L|_{U}$.
The existence of the above expansion can  be deduced from the fact that it can be  seen as one in each of the holomorphic fibers $\pi_U ^{-1} (p)$ for $p\in D$ and the fact that $\nabla$ is flat along each of these fibers.

\begin{defi}\label{hols}
$\Phi$ is called S-holomorphic if $\hat{\mathscr{C}} _{i,j} =0$  for all $j>0$. The space of all $S$- holomorphic maps over $U$ is denoted by $Hol_S (U) $  
\end{defi}

Let $U_{0}\subset D$ be an open subset of $D$ and let $U_1=\pi_{U}^{-1} (U_0)$. Consider $n-1$ transversally parallel sections $\hat{\mathfrak{t}}_i\in \hat{\Gamma} (U_0, L|_{U_0})$ for $i=1,...n-1$ such that 
  $\{\partial^{\nabla} \hat{\mathfrak{t}}_1,...,\partial^{\nabla} \hat{\mathfrak{t}}_{n-1}, \partial^{\nabla} S\}$ form a basis for $\Omega^{1,0}(L|_{U_1})$ everywhere over $U_1$. Assume that
\begin{equation}\label{etaa}
\partial^{\nabla}\bar{\partial} \bar{\hat{\mathfrak{t}}}_k= \sum_{i=1} ^{n-1} \eta^{i,\bar{n}} _{\bar{k}}  \partial^{\nabla}  \hat{\mathfrak{t}}_{i} \wedge \bar{\partial }\bar{S} + \sum_{i=1} ^{n-1}  \eta^{n,\bar{i}} _{\bar{k}}    \partial^{\nabla} S \wedge \bar{\partial} \bar{\hat{\mathfrak{t}}}_{i} 
 +\sum _{1\leq i<j\leq n-1}\eta^{i,\bar{j}} _{\bar{k}} \partial^{\nabla} \hat{\mathfrak{t}}_{i} \wedge  \bar{\partial} \bar{\hat{\mathfrak{t}}}_{j} 
\end{equation}
 where $\eta^{i,\bar{j}} _{\bar{k}}\in \Gamma (U_1, \bar{L}|_{U_1})$ for $1\leq i,j \leq n$ and for $1\leq k \leq n-1$  are smooth functions.

We can choose a no where vanishing transversaly parallel section $\sigma\in \Gamma (U_1, L)$ such that $\sigma|_{N_p}$ is holomorphic for all $p\in U_0$.
   We assume that $\partial^{\nabla}$ with respect to  the trivialization obtained by $\sigma$ is represented by

\[
\partial^{\nabla} = \partial + \beta
\]

where $\beta\in A^{1,0} (U_1)$ is a smooth $(1,0)$-form on $U_1$.
 Also since $\sigma|_{N_p}$ is parallel  for all $p\in U_0$ the decomposition of $\beta$ in 
the basis $\{\partial^{\nabla}\hat{\mathfrak{t}}_1,...,\partial^{\nabla}\hat{\mathfrak{t}}_{n-1}, \partial^{\nabla} S\}$  must be of the form
\begin{equation}\label{betsum}
\beta=\sum_{1\leq i \leq n-1} b_i \partial^{\nabla} \hat{\mathfrak{t}}_i 
\end{equation}
Thus if $\hat{\mathfrak{t}}_{k}=\hat{t}_k \sigma$ where $\hat{t}_k  :U_1\rightarrow \mathbb{C}$, then  the fact that $\hat{\mathfrak{t}}_{k}$ is transversally 
parallel  is equivalent to the following two relations
\begin{equation}\label{sigbet}
\bar{\partial}\hat{t}_k|_{N_p} =0, \hspace{0.5cm} (\partial +\beta)\sigma|_{N_p}=-\frac{\partial \hat{t}_k}{\hat{t}_k }\sigma|_{N_p}
\end{equation}

  so we obtain

\[
\begin{split}
\partial^{\nabla}\partial^{\nabla}\bar{\partial}\hat{\mathfrak{t}}_{k}= (\partial + \beta)( \partial + \beta) \big ( \bar{\partial}\hat{t}_k \sigma +\hat{t}_{k}\bar{\partial}\sigma \big )&=
\big [ (\partial+\beta)(\partial +\beta)\bar{\partial}\hat{t}_k\big ] \sigma\\
\quad & +(\partial + \beta)(\bar{\partial}\hat{t}_k )(\partial + \beta)(\sigma)+(\bar{\partial}\hat{t}_k )(\partial + \beta)(\partial + \beta)(\sigma)\\
\quad & =(\beta\wedge\partial \bar{\partial}\hat{t}_k)\sigma\\
\quad & +(\partial + \beta)(\bar{\partial}\hat{t}_k )(\partial + \beta)(\sigma)+(\bar{\partial}\hat{t}_k )(\partial + \beta)(\partial + \beta)(\sigma)
\end{split}
\]

Therefore according to (\ref{betsum}) and based on the fact that $\mathfrak{t}_k$  is holomorphic along all  the fibers of $\pi_U$ we must have
\[
i_{X}i_{\bar{X}}\partial^{\nabla}\partial^{\nabla}\bar{\partial}\hat{\mathfrak{t}}_{k} =0
\]\begin{equation}\label{etaa}
\partial^{\nabla}\bar{\partial} \bar{\hat{\mathfrak{t}}}_k= \sum_{i=1} ^{n-1} \eta^{i,\bar{n}} _{\bar{k}}  \partial^{\nabla}  \hat{\mathfrak{t}}_{i} \wedge \bar{\partial }\bar{S} + \sum_{i=1} ^{n-1}  \eta^{n,\bar{i}} _{\bar{k}}    \partial^{\nabla} S \wedge \bar{\partial} \bar{\hat{\mathfrak{t}}}_{i} 
 +\sum _{1\leq i<j\leq n-1}\eta^{i,\bar{j}} _{\bar{k}} \partial^{\nabla} \hat{\mathfrak{t}}_{i} \wedge  \bar{\partial} \bar{\hat{\mathfrak{t}}}_{j} 
\end{equation}

for any $(1,0)$-vector $X$ tangent to the fibers of $\pi_U$. This means that in the decomposition of  $\partial^{\nabla}\partial^{\nabla}\bar{\partial}\hat{\mathfrak{t}}_{k}$ with respect to the basis  $\{\partial^{\nabla}\hat{\mathfrak{t}}_1,...,\partial^{\nabla}\hat{\mathfrak{t}}_{n-1}, \partial^{\nabla} S\}$  all the terms of the form $\partial^{\nabla} S\wedge \bar{\partial}\bar{S}\wedge \partial^{\nabla}\hat{\mathfrak{t}}_k $, for $k=1,...,n-1$ will be of identically vanishing coefficients over $U_1$.
A similar result holds for   $\bar{\partial}\partial^{\nabla}\bar{\partial}\hat{\mathfrak{t}}_{k}$.  From this observation it follows that  if 

\begin{equation}\label{etahol0}
\eta^{n,\bar{i}} _{\bar{k}} (p)=\sum_{a,b} \eta^{n,\bar{i}} _{\bar{k}, a,\bar{b}} S^a\bar{S}^b 
\end{equation}
denotes the Taylor series expansion of $\eta^{i,\bar{n}} _{\bar{k}} (p)$ with $\hat{\eta}^{n,\bar{i}} _{\bar{k}, a,\bar{b}}\in \hat{\Gamma}(U_1, \bar{L}^{\otimes (b+1)}\otimes L^{\otimes a})$
then we must have
\begin{equation}\label{etahol}
\hat{\eta}^{n,\bar{i}} _{\bar{k}, a,\bar{b}} =0 \hspace{1cm} \text{ for } b>0  
\end{equation}
and similarly if 

\[
\hat{\eta}^{i,\bar{n}} _{\bar{k}} (p)=\sum_{a,b} \hat{\eta}^{i,\bar{n}} _{\bar{k}, a,\bar{b}} S^a\bar{S}^b 
\] 

denotes the Taylor series expansion of $\eta^{i,\bar{n}} _{\bar{k}} (p)$ where $\hat{\eta}^{i,\bar{n}} _{\bar{k}, a,\bar{b}}\in \hat{\Gamma}(U_1, \bar{L}^{\otimes (b+1)}\otimes L^{\otimes a})$ then we get
\begin{equation}\label{etaantihol}
\eta^{i,\bar{n}} _{\bar{k}, a,\bar{b}} =0 \hspace{1cm} \text{ for } a>0 
\end{equation}
 from  the following lemma we can also deduce that 
\[
\eta^{n,\bar{i}} _{\bar{k}, 0,\bar{0}} =0 \hspace{1cm} \text{ and } \hspace{1cm} \eta^{i,\bar{n}} _{\bar{k}, 0,\bar{0}} =0
\]

for all $1\leq k,i\leq n-1$.

\begin{lemma}\label{etzer}
$\eta^{n,\bar{i}} _{\bar{k}}|_D=\eta^{n,\bar{k}} _{\bar{i}}|_D\equiv 0$ for all $p\in D$ and for all $1\leq i,k\leq n-1$.
\end{lemma}
\begin{proof} We fix a point $q\in U_0$ and we choose a trivialization of $L$ in a neighborhood of $q$ such that $d^{\nabla}=d$ holds at $q$. Therefore we can apply the following identity at $q$:
\[
\partial^\nabla\bar{\partial} \hat{\mathfrak{t}}_i  (X,Y)= d^\nabla \bar{\partial} \hat{\mathfrak{t}}_i  (X,Y)= X. \bar{\partial} \hat{\mathfrak{t}}_i (Y)-Y.\bar{\partial} \hat{\mathfrak{t}}_i (X)-\bar{\partial} \hat{\mathfrak{t}}_i ([X,Y])
\]

We can assume that $Y(x)\in T'' _x  N_p$  and $X\in T' _p D$ for all $p\in D$ and for all $x\in N_p$. Since $\hat{\mathfrak{t}}_i \in \hat{\Gamma} (U_1, L|_{U_1})$
 we have $\bar{\partial} \hat{\mathfrak{t}}_i (Y)\equiv 0$ and since $X$ is a $(1,0)$- vector it follows that $\bar{\partial} \hat{\mathfrak{t}}_i (X)\equiv 0$.
In addition since $N_p \perp D$ for all $p\in D$ we can choose the vector fields $X$ and $Y$ in such a way that $[X,Y](q)= 0 $ at a fixed point $q\in D$.
To see this we fix $Y_q \in  T'' _q  N_q$ and we extend it  by parallel transport with respect to the Levi-Civita connection associated to the initial K\"ahler metric $\omega$ along all directions in a small neighborhood of $q$ in $D$. Since $N_p\perp D$ for all $p\in D$ we deduce that the vector filed $Y$
is tangent to all the discs $N_p$ along $D$ and thus it can be extended to a vector field which is everywhere  tangent to $N_p$ . For the vector field $X$ we can 
also assume that it is  parallel along the fiber $N_q$.   Then we use the fact that $ [X,Y](q)=(\nabla_X Y ) (q)- (\nabla_Y X) (q) =0$. The vanishing of $\partial^\nabla\bar{\partial} \hat{\mathfrak{t}}_i  (X,Y)$
 is equivalent to the assertion of the  lemma.

\end{proof}





Similar argument can show that if we set
\[
\bar{\partial}\partial ^{\nabla} \hat{\mathfrak{t}}_k = \sum_{i=1} ^{n-1} \theta^{i,\bar{n}} _{k} \partial^{\nabla}  \hat{\mathfrak{t}}_i\wedge \bar{\partial}\bar{S}+\sum_{i=1} ^{n-1}  \theta^{n,\bar{i}} _{k} \partial^{\nabla}S  \wedge \bar{\partial}\bar{\hat{\mathfrak{t}}}_i+ \sum_{i,j=1} ^{n-1}  \theta^{i,\bar{j}} _{k} \partial^{\nabla}\hat{\mathfrak{t}}_i  \wedge \bar{\partial}\bar{\hat{\mathfrak{t}}}_j
\]
 and if $ \theta^{i,\bar{n}} _{k}$ and $ \theta^{i,\bar{n}} _{k} $ are expanded as
\[
\theta^{i,\bar{n}} _{k}=\sum_{a,b }\theta^{i,\bar{n}} _{k,a,\bar{b}} S^a\bar{S}^b \text{ and } \theta^{n,\bar{i}} _{k}=\sum_{a,b }\theta^{n,\bar{i}} _{k,a,\bar{b}} S^a\bar{S}^b
\]
 then we have
\begin{equation}\label{tetahol}
\theta^{i,\bar{n}} _{k,a,\bar{b}}=0 \hspace{0.5cm} \text{ for   } a>0 \hspace{0.5cm}\text{ and }\hspace{0.5cm} \theta^{n,\bar{i}} _{k,a,\bar{b}}=0 \hspace{0.5cm} \text{ for   } b>0 
\end{equation}

and

\begin{equation}\label{tetahol1}
\theta^{i,\bar{n}} _{k,0,\bar{0}}=0 \hspace{0.5cm} \text{and}\hspace{0.5cm}  \theta^{i,\bar{n}} _{k,0,\bar{0}}=0
\end{equation}

for $1\leq k,i \leq n-1$.

From here onwaard, for the sake of simplicity, we use the notation $\partial$ instead of $\partial ^{\nabla}$.  Any analytic differential form $\alpha\in \Omega^{1,1} (U)$ admits an expansion as follows,
\begin{equation}\label{albast}
\begin{split}
\alpha&=\sum_{i=1} ^{n-1} \alpha^{i,\bar{n}}  \partial\hat{\mathfrak{t}}_{i} \wedge \bar{\partial }\bar{S} + \sum_{i=1} ^{n-1}  \alpha^{n,\bar{}i}     \partial S \wedge \bar{\partial} \bar{\hat{\mathfrak{t}}}_{i} \\
\quad & \sum _{1\leq i<j\leq n-1}\alpha ^{i,\bar{j}}  \partial \hat{\mathfrak{t}}_{i} \wedge  \bar{\partial} \bar{\hat{\mathfrak{t}}}_{j} +\alpha^{n,\bar{n}} \partial S\wedge \bar{\partial }\bar{S}+  
\end{split}
\end{equation}

where  $\alpha^ {i,\bar{j}}:U\rightarrow \mathbb{C}$ for $i,j=1,...,n$ are all analytic maps.  

We define $ \alpha_{T\bar{T}} $, $\alpha_{S\bar{T}}$, $\alpha_{T\bar{S}}$ and $\alpha_{S\bar{S}}$ by
\[
\alpha_{T\bar{T}} =  \sum _{1\leq i<j\leq n-1}\alpha^{i,\bar{j}}  \partial \hat{\mathfrak{t}}_{i} \wedge  \bar{\partial} \bar{\hat{\mathfrak{t}}}_{j} 
\]

\[
\alpha_{S\bar{T}}=  \sum_{i=1} ^{n-1}  \alpha^{n,\bar{i}}     \partial S \wedge \bar{\partial} \bar{\hat{\mathfrak{t}}}_{i}
\]

\[
\alpha_{T\bar{S}}=\sum_{i=1} ^{n-1} \alpha^{i,\bar{n}}  \partial \hat{\mathfrak{t}}_{i} \wedge \bar{\partial }\bar{S} 
\]
\[
\alpha_{S\bar{S}}=\alpha^{n,\bar{n}} \partial S\wedge \bar{\partial }\bar{S}
\]


Given a map $f:U_0 \rightarrow \mathbb{C}$  then the function $\hat{f}:U_1\rightarrow \mathbb{C}$ is defined by
$\hat{f}:=f\circ \pi_U$. We set
\[
\partial \hat{f}= \sum_{1\leq i\leq n-1} f^i \partial \hat{\mathfrak{t}}_i +f^n\partial S 
\]
where $f^i \in \Gamma (U_1, \bar{L}|_{U_1})$ for $1\leq i\leq n$.
 Since   $\pi_U ^{-1} (p)$ for all $p\in D$ is a holomorphic curve and since $\hat{f}$ is constant along all the fibers $\pi_U ^{-1} (\{p\})$ for $p\in D$ therefore we have 
\[
f^n =0
\]  

 From this observation it is not difficult to see that the differential forms $\alpha_{T\bar{T}}$, $\alpha_{S\bar{T}}$, $\alpha_{T\bar{S}}$  are  well-defined independent of the choice of $\hat{\mathfrak{t}}_1,...,\hat{\mathfrak{t}}_{n-1}$. 

\begin{lemma}\label{partrafo}
For any  $(p,q)$-form $\alpha\in A^{p,q} ((L^{\otimes i }\otimes \bar{L}^{\otimes j})|_{U_0} )$ defined on an open subset $U_0\subset D$, there exists a unique $(p,q)$-form $\hat{\alpha}\in A^{p,q} ((L^{\otimes i }\otimes \bar{L}^{\otimes j})|_{U_1})$  where $U_1=\pi_U ^{-1} (U_0)$ satisfying the following equations
\begin{equation}\label{ixixbar}
i_X \hat{\alpha}= i_{\bar{X}}\hat{\alpha}= i_{\bar{X}} \bar{\partial } \hat{\alpha}=i_X \partial ^{\nabla} \hat{\alpha}=0
\end{equation}

for any vector field $X\in T' (U_1)$ which is everywhere tangent to the fiberation obtained by $\pi_U$
\end{lemma}

\begin{proof}
 We prove the above lemma  for the case $p=q=1$ and it can be obviously stated for arbitrary $p$ and $q$.   Assume that  $\beta\in A^{p,q}((L^{\otimes i }\otimes \bar{L}^{\otimes j})|_{U_1}) $ satisfies the relation (\ref{ixixbar}). Since $i_X \beta= i_{\bar{X}}\beta =0 $  we have
\[
\beta= \sum_{1\leq i,j\leq n-1} \beta^{i,\bar{j}}\partial^{\nabla} \hat{\mathfrak{t}}_i\wedge \bar{\partial}\bar{\hat{\mathfrak{t}}}_{j}
\]
therefore   $\bar{\partial}\beta$ is given by
\begin{equation}\label{betaexp}
\bar{\partial}\beta=\sum_{1\leq i,j\leq n-1}\bar{\partial} (\beta^{i,\bar{j}}) \partial^{\nabla} \hat{\mathfrak{t}}_i\wedge \bar{\partial}\bar{\hat{\mathfrak{t}}}_{j}+\sum_{1\leq i,j\leq n-1} \beta^{i,\bar{j}}\bar{\partial}\partial^{\nabla} \hat{\mathfrak{t}}_i\wedge \bar{\partial}\bar{\hat{\mathfrak{t}}}_{j}
\end{equation}
 We consider the expansion of the coefficients $\beta^{i,\bar{j}}$, for $1\leq i,j\leq n-1$ in the form
\[
\beta^{i,\bar{j}}=\sum_{a,b} \beta^{i,\bar{j}}_{a,\bar{b}} S^a\bar{S}^b 
\]
Since $\beta|_{D\cap U_0}= \alpha$  we have
\[
\beta^{i,\bar{j}}_{0,\bar{0}}= \alpha^{i,\bar{j}}  \hspace{1cm} \text{ for } 1\leq i,j \leq n-1
\]
 Here $\alpha^{i,\bar{j}} $ comes from the  expansion of $\alpha=\sum_{1\leq i,j \leq n-1} \alpha^{i,\bar{j}} \partial^{\nabla} \hat{\mathfrak{t}}_i\wedge \bar{\partial}\bar{\hat{\mathfrak{t}}}_{j}$ .
Due to relation (\ref{tetahol})  if $ i_X d^{\nabla} \hat{\alpha}=i_{\bar{X}} d^{\nabla}\hat{\alpha}=0$ holds for   $X\in T' (U_1)$, which is everywhere tangent to the fibers of  $\pi_U$, then according to the  relation (\ref{betaexp}) we must have
\[
\beta^{i,\bar{j}}_{a,\bar{b}}=0 \hspace{1cm} \text{ for } a,b >0
\]
Also from $i_{\bar{X}}\bar{\partial} \beta =i_{X}\partial^{\nabla}\beta=0$  by using the relations (\ref{tetahol}) and (\ref{tetahol1}) we can inductively determine all the coefficients $\beta^{i,\bar{j}}_{a,\bar{0}}$ and $\beta^{i,\bar{j}}_{0,\bar{b}}$ for $a,b>0$. 
\end{proof}

\begin{defi}
The space of differential forms obtained as in lemma (\ref{partrafo}) is called the space of transversally parallel differential forms and is denoted by $\hat{A} _T^{p,q} (L|_{U_1})$ we also define 
\[ 
\begin{split}
\hat{A}^{p,q} ((L^{\otimes i }\otimes \bar{L}^{\otimes j})|_{U_1})=& \{\hat{f}_1\hat{\alpha}_1 + \hat{f}_2\hat{\alpha}_2\wedge \partial S +\hat{f}_3\hat{\alpha}_3\wedge\bar{\partial}\bar{S} +\hat{f}_4\hat{\alpha}_4\wedge  \partial S \wedge \bar{\partial}\bar{S} \text{ such that } \\
\quad & \hat{\alpha}_1 \in \hat{A}^{p,q} _T ((L^{\otimes i_1 }\otimes \bar{L}^{\otimes j_1})|_{U_1}), \hat{\alpha}_2 \in \hat{A}^{p-1,q} _T ((L^{\otimes i_2 }\otimes \bar{L}^{\otimes j_2})|_{U_1}),\\
\quad & \hat{\alpha}_3 \in \hat{A}^{p,q-1} _T ((L^{\otimes i_3 }\otimes \bar{L}^{\otimes j_3})|_{U_1}),\hat{\alpha}_4 \in \hat{A}^{p-1,q-1} _T ((L^{\otimes i_4 }\otimes \bar{L}^{\otimes j_4})|_{U_1}) \\
\quad & \hat{f}_a\in \hat{\Gamma} {(L^{\otimes {i-i_a} }\otimes \bar{L}^{\otimes {j-j_a}})|_{U_1}}, i_a\leq i , j_a \leq j ,a=1,2,3,4 \}
\end{split}
\]

\end{defi}
Any $(p,q)$-form $\alpha\in A^{p,q} ((L^{\otimes i }\otimes \bar{L}^{\otimes j})|_{U_1})$ admits an expansion
\begin{equation}\label{alfexp}
\alpha =\sum \hat{\alpha}_{a,b} S^a \bar{S}^b
\end{equation}
where $ \hat{\alpha}_{a,b}\in \hat{A}^{p,q} ((L^{\otimes i }\otimes \bar{L}^{\otimes j})|_{U_1})$.

\begin{lemma}\label{lemhol}
If $\alpha_i\in A^{0,1} (\bar{L}|_{U_0})$, is given by  $\alpha_i=\bar{\partial}\mathfrak{t}_i$ then $\hat{\alpha}_i$ has the form 
\[
\hat{\alpha}_i= \sum_j \alpha_i ^j \bar{\partial}\mathfrak{t}_j
\]

where
\[
\alpha_i ^j \in Hol_S (U)
\]
 (see definition (\ref{hols}))
Conversely if  we decompose $\bar{\partial}\mathfrak{t}_i$ in terms of $\hat{\alpha}_1,...,\hat{\alpha}_{n-1} $  as follows
\[
\bar{\partial}\mathfrak{t}_i=\sum t^j _i\hat{\alpha}_j
\]
then we have
\[
 t^j _i\in Hol_S (U)
\]

\end{lemma}
\subsection{Proof of the Existence and Uniqueness } 
\begin{defi}\label{holpart}
By the holomorphic part of $\alpha$ we mean the summation of those terms in (\ref{alfexp}) which do not contain $\bar{S}$
\[
(\alpha)_{hol}= \sum_{a=0} ^{\infty} \hat{\alpha}_{a,0} S^a
\]
similarly we can define the antiholomorphic part of $\alpha$.
\[
(\alpha)_{antihol}= \sum_{b=0} ^{\infty} \hat{\alpha}_{0,b} \bar{S}^b
\]
\end{defi}
Let $\hat{\mathfrak{f}}\in \hat{\Gamma} (U_1, \mathbb{C})$ be a smooth map. So

\begin{equation}\label{ddbarf}
\partial\bar{\partial} \hat{\mathfrak{f}}= \sum_{i=1} ^{n-1} \gamma^{i,\bar{n}}  \partial^{\nabla}  \hat{\mathfrak{t}}_{i} \wedge \bar{\partial }\bar{S} + \sum_{i=1} ^{n-1}  \gamma^{n,\bar{i}}     \partial^{\nabla} S \wedge \bar{\partial} \bar{\hat{\mathfrak{t}}}_{i} 
 +\sum _{1\leq i<j\leq n-1}\gamma^{i,\bar{j}}  \partial^{\nabla} \hat{\mathfrak{t}}_{i} \wedge  \bar{\partial} \bar{\hat{\mathfrak{t}}}_{j} 
\end{equation}
 where $\gamma ^{i,\bar{j}}\in \Gamma (U_1, \mathbb{C})$ for $1\leq i,j \leq n$.  

\begin{lemma}\label{lemfive}
If we expand $\gamma^{i,\bar{j}}$ for $1\leq i,j \leq n$   
\[
\gamma^{i,\bar{j}} =\sum_{a,b\geq 0}   \hat{\gamma}^{i,\bar{j}} _{ a,\bar{b}}S^a \bar{S}^b
\]
where $  \hat{\gamma}^{i,\bar{j}} _{ a,\bar{b}}\in\hat{\Gamma} (U_1, \bar{L}^{\otimes a}\otimes L^{\otimes b}|_{U_1})$,
then we have
\[
 \hat{\gamma}^{i,\bar{n}}_{a,\bar{b}}=0 \hspace{0.5cm} \text{ for } a>0 \hspace{0.5cm}   \text{ and }\hspace{0.5cm}   \hat{\gamma}^{n,\bar{j}}_{a,\bar{b}}=0 \hspace{0.5cm} \text{ for } b>0
\]

\[
 \hat{\gamma}^{i,\bar{j}} _{ a,\bar{b}}=0 \hspace{0.5cm} \text{for }\hspace{0.5cm}   a,b>0\hspace{0.5cm} \text{ and for }\hspace{0.5cm} 1\leq i,j\leq n-1
\]
 
 This is equivalent to say that  the expansion  of 
\begin{equation}\label{ddbf1}
(\partial \bar{\partial} \hat{\mathfrak{f}})_{S\bar{T}}=((\partial \bar{\partial} \hat{\mathfrak{f}})_{S\bar{T}})_{hol}
\end{equation}
and  
\begin{equation}\label{ddbf2}
(\partial \bar{\partial} \hat{\mathfrak{f}})_{T\bar{S}}=((\partial \bar{\partial} \hat{\mathfrak{f}})_{T\bar{S}})_{antihol}
\end{equation}
    Also non-zero terms in the expansion of  $(\partial \bar{\partial} \hat{\mathfrak{f}})_{T\bar{T}}$ can only occur in its holomorphic and anti holomorphic parts. 
\end{lemma}

\begin{proof} To prove this we first note that since $\partial\bar{\partial}\hat{\mathfrak{f}}$ is both $\partial $ and $\bar{\partial}$-closed we have
\[
 \hat{\gamma}^{i,\bar{n}}_{a,\bar{b}}=0 \hspace{0.5cm} \text{ for } a>0 \hspace{0.5cm}  \text{ and }  \hspace{0.5cm} \hat{\gamma}^{n,\bar{j}}_{a,\bar{b}}=0 \hspace{0.5cm} \text{ for } b>0
\]
This is because otherwise $\partial\bar{\partial}\hat{\mathfrak{f}}$ would not be $\partial $ or $\bar{\partial }$-closed due to existence of non-zero terms of the  form $\partial^{\nabla} \hat{\mathfrak{t}}_i\wedge \partial S \wedge \bar{\partial }\bar{S}$ or $\bar{\partial} \bar{\hat{\mathfrak{t}}}_j\wedge \partial S \wedge \bar{\partial }\bar{S}$. From this observation we can also  deduce that  if 
there exists non-zero coefficients of the form  $\hat{\gamma}^{i,\bar{j}} _{ a,\bar{b}}$ for  $a,b>0$  and for $1\leq i,j\leq n-1$
 then we would get to a contradiction with $\partial$ or $\bar{\partial}$-closedness of $\partial\bar{\partial}\hat{\mathfrak{f}}$. The proof of (\ref{ddbf1}) and (\ref{ddbf2}) follow from lemma (\ref{lemhol})

\end{proof}

Assume  now that $\partial\bar{\partial }\Phi =\omega|_{U_1}$ where $\Phi:U_1\rightarrow \mathbb{R}$ is defined by (\ref{Phi}) and $\omega$ is the K\"ahler metric on $X$. We set
 \begin{equation}\label{phibarexp}
\Phi=\sum_{l\geq 0} \Phi_{\bar{l}} \bar{S}^l
\end{equation}
 where

\begin{equation}\label{filbar}
\Phi_{\bar{l}} := \sum _{k\geq 0} \hat{\mathscr{C}}_{k\bar{l}}S^k
\end{equation}
 and $\hat{\mathscr{C}}_{k\bar{l}}\in \hat{\Gamma} (L^{\otimes l}\otimes \bar{L}^{\otimes k})$
so we have
\[
\bar{\partial }\Phi_{\bar{l}}= \sum_{k\geq 0} \bar{\partial } \hat{\mathscr{C}}_{k,\bar{l}}  S^k  
\]

then
\[
\partial \bar{\partial} \Phi=\sum_{l\geq 0}  \partial \bar{\partial}\Phi_{\bar{l}}\bar{S}^l +\sum_{l\geq 1} l \partial\Phi_{\bar{l}}\bar{S}^{l-1}\bar{\partial}\bar{S}
\]
This shows that
\begin{equation}\label{ddbarfist}
(\partial \bar{\partial} \Phi)_{S\bar{T}}=\sum_{l\geq 0}  (\partial \bar{\partial}\Phi_{\bar{l}} )_{S\bar{T}}\bar{S}^l 
\end{equation}

Let
\[
\bar{\partial } \hat{\mathscr{C}}_{k,\bar{l}}=\sum_{1\leq i\leq n-1}  \mathscr{C}_{k,\bar{l}} ^{\bar{i}} \bar{\partial } \bar{\hat{\mathfrak{t}}}_i
\]
with   $\mathscr{C}_{k,\bar{l}} ^{\bar{i}} \in \Gamma (\bar{L}^{\otimes (k)}\otimes L^{\otimes (l+1)}) $ for $1\leq i \leq n-1$.
We claim that the Taylor series expansion of $\mathscr{C}_{k,\bar{l}} ^{\bar{i}} $ has the form $
  \mathscr{C}_{k,\bar{l}} ^{\bar{i}} =\sum_{l\geq 0} \hat{  \mathscr{C}}^{\bar{i}} _{j,k,\bar{l}} S^{j}
$ where $\hat{  \mathscr{C}}^{\bar{i}} _{j,k,\bar{l}}\in \hat{\Gamma} (\bar{L}^{\otimes (k+j)}\otimes L^{\otimes (l+1)})$ since if there exists non zero terms containing $\bar{S}$ then $\bar{\partial}\bar{\partial }\hat{\mathscr{C}}_{k,\bar{l}}\neq 0$.
 Therefore $\bar{\partial }\Phi_{\bar{l}}$ can be expanded under the form,


\begin{equation}\label{Phizer}
\bar{\partial }\Phi_{\bar{l}} =\sum_{a\geq 0} (\sum _{\substack{k+j=a\\ 1\leq i\leq n-1}} \hat{  \mathscr{C}}^{\bar{i}} _{j,k,\bar{l}}  \bar{\partial}\bar{\hat{\mathfrak{t}}}_i) S^{a}
\end{equation}


using (\ref{etaa}) we obtain

\begin{equation}\label{parPhiST}
\begin{split}
\partial\bar{\partial }\Phi_{\bar{l}} &=(\sum_{a\geq 1} \sum _{\substack{k+j=a\\ 1\leq i\leq n-1}} a\hat{  \mathscr{C}}^{\bar{i}} _{j,k,\bar{l}} \partial S\wedge   \bar{\partial}\bar{\hat{\mathfrak{t}}}_i ) S^{a-1}+\sum_{a\geq 0} (\sum _{\substack{k+j=a\\ 1\leq i\leq n-1}} \partial{\hat{  \mathscr{C}}^{\bar{i}} _{j,k,\bar{l}} } \wedge \bar{\partial}\bar{\hat{\mathfrak{t}}}_i) S^{a}\\
\quad & +\sum_{a\geq 0}\sum_{\substack{k+j=a\\ 1\leq i\leq n-1}}\hat{  \mathscr{C}}^{\bar{i}} _{j,k,\bar{l}}   S^{a}\eta^{p,\bar{n}} _{\bar{i}}  \partial\hat{\mathfrak{t}}_{p} \wedge \bar{\partial }\bar{S}+\sum_{a\geq 0}\sum_{\substack{k+j=a\\ 1\leq i\leq n-1}}\hat{  \mathscr{C}}^{i} _{j,k,\bar{l}}   S^{a}\eta^{n,\bar{p}}_ {\bar{i}} \partial S\wedge \bar{\partial}\bar{\hat{\mathfrak{t}}}_{p}\\
\quad &+\sum_{a\geq 0}\sum_{\substack{k+j=a\\ 1\leq i\leq n-1}}\hat{  \mathscr{C}}^{i} _{j,k,\bar{l}}   S^{a}\eta^{p,\bar{q}} _{\bar{i}} \partial \hat{\mathfrak{t}}_p\wedge \bar{\partial}\bar{\hat{\mathfrak{t}}}_{q}
\end{split}
\end{equation}

Hence
\begin{equation}\label{Phizer1}
(\partial\bar{\partial }\Phi_{\bar{l}} )_{S\bar{T}}=\sum_{1\leq p \leq n-1}\sum_{a\geq 0} (\sum _{\substack{k+j=a+1\\ 1\leq i\leq n-1}} (a+1)\hat{  \mathscr{C}}^{\bar{p}} _{j,k,\bar{l}}   +\sum_{\substack{k+j=a\\ 1\leq i\leq n-1}}\hat{  \mathscr{C}}^{\bar{i}} _{j,k,\bar{l}}  \eta^{n.\bar{p}} _{\bar{i}}  )S^{a} \partial S\wedge \bar{\partial}\bar{\hat{\mathfrak{t}}}_{p}
\end{equation}
 if $\omega$  is given by
\[
\omega =\sum \hat{\omega}_{a,\bar{b}} S^a \bar{S}^b
\]
where $\hat{\omega}_{a,\bar{b}}\in \hat{\Gamma }(U_1, L^{\otimes a}\otimes \bar{L}^{\otimes b}|_{U_1})$ then  by (\ref{ddbarfist}) we have $\omega_{S\bar{T}}=\sum_b (\partial\bar{\partial}\Phi_{\bar{b}})_{S\bar{T}} \bar{S}^b$ and  by applying (\ref{Phizer1})
 
\begin{equation}\label{alfst1}
(\hat{\omega}_{a,\bar{b}})_{S\bar{T}}=\sum_{1\leq p\leq n-1}\big ( \sum _{\substack{k+j=a+1\\ 1\leq i\leq n-1}} (a+1)\hat{  \mathscr{C}}^{p} _{j,k,\bar{b}}   +\sum_{\substack{k+j=a\\ 1\leq i\leq n-1}}\hat{  \mathscr{C}}^{i} _{j,k,\bar{b}}  \eta_{n.p} ^i  \big )  \partial S\wedge \bar{\partial}\bar{\hat{\mathfrak{t}}}_{p}
\end{equation}

 In particular we obtain

 \begin{equation}\label{alfst1}
(\hat{\omega}_{a,\bar{0}})_{S\bar{T}}=\big ( \sum _{\substack{k+j=a+1\\ 1\leq i\leq n-1}} (a+1)\hat{  \mathscr{C}}^{p} _{j,k,\bar{0}}   +\sum_{\substack{k+j=a\\ 1\leq i\leq n-1}}\hat{  \mathscr{C}}^{i} _{j,k,\bar{0}}  \eta_{n.p} ^i  \big )  \partial S\wedge \bar{\partial}\bar{\hat{\mathfrak{t}}}_{p}
\end{equation}

\begin{lemma}\label{foot}
Let $\Phi:U_1\rightarrow \mathbb{R}$, defined by the relation (\ref{Phi}), be an arbitrary potential for the K\"ahler metric $\omega$,
  \[
 \omega|_{U_1}=\partial \bar{\partial}\Phi  
\]
Then $(\bar{\partial }\Phi_{\bar{0}})^+ :=\bar{\partial }\Phi_{\bar{0}}-\sum _i \hat{\mathscr{C}}^{i} _{0,0,\bar{0}} \bar{\partial}\bar{\hat{\mathfrak{t}}}_i$ 
is independent of the choice of the potential $\Phi$ and is uniquely determined in terms of $(\omega_{S\bar{T}})_{hol}$ (see  relation (\ref{Phizer})). If in addition we set

 \[
\Phi_{\mathcal{I}}:=\sum_{k=0 \text{ or } l=0} \hat{\mathscr{C}}_{k,\bar{l}}S^k \bar{S}^l
\]
then $\partial\bar{\partial}\Phi_{\mathcal{I}}$  is independent of the choice of $\Phi$ and is uniquely determined in terms of 
  $(\omega_{S\bar{T}})_{hol}$ and $(\omega_{T\bar{T}})_{hol}$.  Moreover  $(\partial\bar{\partial}\hat{\mathscr{C}}_{0,\bar{0}})_{S\bar{T}}$   is thoroughly characterized  by $(\omega _{S\bar{T}})_{hol}$
 and $(\partial\bar{\partial}\hat{\mathscr{C}}_{0,\bar{0}})_{T\bar{S}}$ is thoroughly characterized in terms of $(\omega_{T\bar{S}})_{hol}$. Also $(\partial\bar{\partial}\hat{\mathscr{C}}_{0,\bar{0}})_{T\bar{T}}$  is 
characterized in terms of $(\omega_{T\bar{T}})_{hol}$ and $(\omega_{T\bar{T}})_{antihol}$.
\end{lemma}
\begin{proof}
According to the definition \ref{holpart})

\begin{equation}\label{khash1}
(\omega_{S\bar{T}})_{hol}=\sum_a( \hat{\omega}_{a,\bar{0}}) _{S\bar{T}}S^a
\end{equation}

From the equation  $\partial \bar{\partial}\Phi = \omega$  and relation (\ref{ddbarfist}) we get  
\begin{equation}\label{khash2}
(\omega_{S\bar{T}})_{hol}=(\partial\bar{\partial}\Phi_{\bar{0}})_{S\bar{T}})_{hol}
\end{equation}

and from  (\ref{Phizer1})  we have

\begin{equation}\label{doalf}
 (\partial\bar{\partial}\Phi_{\bar{0}})_{S\bar{T}}=\sum_a\big ( \sum _{\substack{k+j=a+1\\ 1\leq i\leq n-1}} (a+1)\hat{  \mathscr{C}}^{\bar{p}} _{j,k,\bar{0}}   +\sum_{\substack{k+j=a \\1\leq i\leq n-1}}\hat{  \mathscr{C}}^{\bar{i}} _{j,k,\bar{0}}  \eta^{n.\bar{p}} _{\bar{i}}  \big ) S^a \partial S\wedge \bar{\partial}\bar{\hat{\mathfrak{t}}}_{p} 
\end{equation}

According to the equation (\ref{Phizer1}) and  lemma  (\ref{lemhol}),  
  we see that 
\begin{equation}\label{fi=fihol}
 (\omega_{S\bar{T}})_{hol}=((\partial\bar{\partial}\Phi_{\bar{0}})_{S\bar{T}})_{hol}= (\partial\bar{\partial}\Phi_{\bar{0}})_{S\bar{T}}
\end{equation}
and  from lemma (\ref{etzer}) and  relations (\ref{doalf}) and (\ref{fi=fihol}) , it turns out that the value of the coefficients
  $ \sum _{\substack{k+j=a+1\\ 1\leq i\leq n-1}}\hat{  \mathscr{C}}^{\bar{p}} _{j,k,\bar{0}}|_D$ for $1\leq p \leq n-1$ and for $\{k,j\in \mathbb{N}\cup \{0\}|k+j\geq 1\}$ can be  inductively  determined by  $(\omega_{S\bar{T}})_{hol}$.   Therefore according to (\ref{Phizer})
 $(\bar{\partial }\Phi_0)^+ :=\bar{\partial }\Phi_0-\sum _i \hat{\mathscr{C}}^{i} _{0,0,\bar{0}} \bar{\partial}\bar{\hat{\mathfrak{t}}}_i $ is uniquely determined in terms of  $(\omega_{S\bar{T}})_{hol}$. 

We also know from (\ref{phibarexp}) that $(\omega_{T\bar{T}})_{hol}=((\partial\bar{\partial}\Phi_{\bar{0}})_{T\bar{T}})_{hol} $. 
 Let
 \[
\Phi_{\mathcal{I}}:=\sum_{k=0 \text{ or } l=0} \hat{\mathscr{C}}_{k,\bar{l}}S^k \bar{S}^l
\]
Following the above discussion we know that $\bar{\partial} (\sum_{k=1} ^{+\infty} \hat{\mathscr{C}}_{k,\bar{0}} S^k )$ is determined in terms of $(\omega_{S\bar{T}})_{hol}$. Similarly $\partial (\sum_{l=1} ^{+\infty} \hat{\mathscr{C}}_{0,\bar{l}} \bar{S}^l )$  is determined in terms of $(\omega_{T\bar{S}})_{antihol}$. Therefore 
\[
\partial\bar{\partial} (\sum_{k=1} ^{+\infty} \hat{\mathscr{C}}_{k,\bar{0}} S^k ) \hspace{0.5cm} \text{ and } \hspace{0.5cm} \bar{\partial}\partial (\sum_{l=1} ^{+\infty} \hat{\mathscr{C}}_{0,\bar{l}} \bar{S}^l )
\]

are known if $(\omega_{S\bar{T}})_{hol}+(\omega_{T\bar{S}})_{antihol}$ is fixed. In addition  according to (\ref{doalf}) and (\ref{fi=fihol}) we have
\[
(\omega_{S\bar{T}})_{hol}=(\partial\bar{\partial}\Phi_{\bar{0}})_{S\bar{T}}= (\partial\bar{\partial}\hat{\mathscr{C}}_{0,\bar{0}})_{S\bar{T}}+(\partial\bar{\partial}(\sum_{k=1} ^{+\infty} \hat{\mathscr{C}}_{k,\bar{0}} S^k ))_{S\bar{T}}
\]

Hence we can also characterize $(\partial\bar{\partial}\hat{\mathscr{C}}_{0,\bar{0}})_{S\bar{T}}$   in terms of $(\omega_{S\bar{T}})_{hol}$. Similarly $(\partial\bar{\partial}\hat{\mathscr{C}}_{0,\bar{0}})_{T\bar{S}}$ is determined in terms of $(\omega_{S\bar{T}})_{antihol}$.

Since we also know from (\ref{phibarexp}) that $(\omega_{T\bar{T}})_{hol}=((\partial\bar{\partial}\Phi_{\bar{0}})_{T\bar{T}})_{hol} $ therefore $((\partial\bar{\partial}\hat{\mathscr{C}}_{0,\bar{0}})_{T\bar{T}})_{hol}$ is also
 uniquely determined if we fix $(\omega_{T\bar{T}})_{hol}$ and $(\omega_{S\bar{T}})_{hol}$. This is because $\partial ( \bar{\partial} \Phi_{\bar{0}})^{+}$ is determined in terms of $(\omega_{S\bar{T}})_{hol}$ 
as proved  above and moreover we have 
$((\partial\bar{\partial}\hat{\mathscr{C}}_{0,\bar{0}})_{T\bar{T}})_{hol}=((\partial\bar{\partial}\Phi_{\bar{0}})_{T\bar{T}})_{hol} -((\partial ( \bar{\partial} \Phi_{\bar{0}})^{+})_{T\bar{T}})_{hol}$.
Similarly  $((\partial\bar{\partial}\hat{\mathscr{C}}_{0,\bar{0}})_{T\bar{T}})_{antihol}$ is uniquely determined in terms of $(\omega_{T\bar{T}})_{antihol}$. To complete the proof we  note that according to lemma (\ref{lemfive})
$(\partial\bar{\partial}\hat{\mathscr{C}}_{0,\bar{0}})_{T\bar{T}}$ contains only holomorphic and antiholomorphic parts and $\Phi$ is real valued.
\end{proof}

\begin{lemma}\label{lemhasht}
Any map $\Phi:U\rightarrow \mathbb{C}$ which is  real analytic in  the neighborhood $U$ of $D$ and which is analytic on each of the fibers of $\pi_U$  admits a Taylor series expansion of the form
\begin{equation}\label{Phi1}
\Phi= \sum_{i,j=0} ^{\infty} \hat{\mathscr{C}}_{i,\bar{j}} S^i   \bar{S} ^j
\end{equation}

Moreover if 
\begin{equation}
\Phi=\sum B_{a,\bar{b}}(w_1,...,w_{n-1}) z^a \bar{z}^b
\end{equation}

is a Taylor series expansion in a local holomorphic coordinates system $(w_1,...,w_{n-1},z)$  in which $D=\{z=0\}$, then  

\begin{equation}
\bar{\partial} B_{a,0}=\frac{\partial ^a}{\partial z ^a} \bigg [ \bar{\partial} (\sum_{k=1} ^{+\infty} \hat{\mathscr{C}}_{k,\bar{0}} S^k ) \bigg ]\bigg\vert_{D}
\end{equation}
 for $a\geq 1$. In other words $\bar{\partial} B_{a,0}$ only depends on  $\bar{\partial} (\sum_{k=1} ^{+\infty} \hat{\mathscr{C}}_{k,\bar{0}} S^k ) $ and hence on  $(\omega_{S\bar{T}})_{hol}$.
\end{lemma}

\begin{proof}
Let $\hat{\sigma}_0\in \hat{\Gamma} (U_1 , L|_{U_1})$ be a nowhere vanishing transversally parallel   real analytic section of $L|_{U_1}$  satisfying  $\|\sigma_0\|_{\mathfrak{h}}=1$. We assume that there exists a   holomorphic coordinates system $(w_1,...,w_{n-1},z)$  over $U_1$ such that $D\cap U_1 = \{z=0\}$. We also  assume that $\sigma_0$  is a no-where vanishing   holomorphic section of $L|_{U_1}$.

Let
\[
S= \xi \sigma_0
\] 
where $\xi:U_1 \rightarrow \mathbb{C}$ is a holomorphic map vanishing along $D\cap U_1$ so $\frac{\xi}{z}$ admits a holomorphic extension through $D\cap U_1$. Consider the map $\Phi:U_1\rightarrow \mathbb{C}$ defined by a Taylor series expansion like (\ref{Phi}). 
  Then in  the above coordinates the coefficients  $\hat{\mathscr{C}}_{k,\bar{l}}\in \hat{\Gamma}(L^{\otimes k}\otimes \bar{L}^{l}|_{U_1})$ for  $k,l\in \mathbb{N}\cup \{0\}$ can be represented as
\[
\hat{\mathscr{C}}_{k,\bar{l}}=\hat{C}_{k,\bar{l}} \hat{\sigma}_0 ^{l}\otimes \bar{\hat{\sigma}}_{0} ^l
\]

where $\hat{C}_{k,\bar{l}}=c_{k,\bar{l}}\circ \pi_U$ for some real analytic function $c_{k,\bar{l}}:U_1 \cap D\rightarrow \mathbb{C}$. Thus  if we assume that $\hat{\sigma}_0 =e^{\tau} \sigma_0$  for a smooth map $\tau$ then 
\begin{equation}\label{firep}
\Phi= \sum  \mathcal{C}_{k,\bar{l}}\xi^k \bar{\xi}^l
\end{equation}

where $\mathcal{C}_{k,\bar{l}}=|e^{k\tau+l\bar{\tau}}|^2 \hat{C}_{k,\bar{l}}$. It is thus clear that $\mathcal{C}_{k,\bar{l}}$ is holomorphic and can be determined by its restriction to $D\cap U_1$ as well.
Since $S$ is a holomorphic section of $L$ with simple zero along $D$, we can take the section $\sigma_0$ in such a way that in local coordinates $(w_1,...,w_{n-1},z)$, the function $\xi$   admits the following expansion:
\begin{equation}\label{xis}
\xi=\sum_{p\geq 1, q\geq 0}\xi_{p,\bar{q}}(w_1,...,w_{n-1}) z^p \bar{z}^q
\end{equation}

We also consider the corresponding Taylor series expansion of    $\mathcal{C}_{k,\bar{l}}$ in this coordinates:
\[
\mathcal{C}_{k,\bar{l}}= \sum \mathcal{C}_{k,\bar{l}}^{i, \bar{j}}(w_1,...,w_{n-1})z^i \bar{z}^j
\]
thus we get
\[
\Phi=\sum B_{a,\bar{b}}(w_1,...,w_{n-1}) z^a \bar{z}^b= \sum_{a,b}\big ( \sum_{\substack{i+\sum_{e=1} ^k  p_e + \sum_{f=1} ^l q_{k+f} =a   
\\ j+\sum_{l=1} ^k q_e + \sum _{f=1} ^l p_{k+f}=b}} \mathcal{C}_{k,\bar{l}} ^{i,\bar{j}} \xi_{p_1\bar{q}_1}...\xi_{p_k\bar{q}_k} \bar{\xi}_{p_{k+1}\bar{q}_{k+1}}...\bar{\xi}_{p_{k+l}\bar{q}_{k+l}} \big ) z^a \bar{z}^b
\]
 
where $B_{a,\bar{b}}$ is defined by the second equality above.
Now we claim that any analytic map $\Phi: U_1 \rightarrow \mathbb{C}$ admits  a representation of the form (\ref{Phi}) in a neighborhood of $D$. This is clearly equivalent to obtain an expansion as  (\ref{firep}) in local holomorphic  coordinates.
To see this we first note that $B_{0,\bar{0}} = \mathcal{C}_{0,\bar{0}} ^{0,\bar{0}}=\mathcal{C}_{0,\bar{0}}|_{D}=c_{0,\bar{0}}$. Moreover  since $\mathcal{C}_{0,\bar{0}}= c_{0,\bar{0}}\circ \pi_U$,
$\mathcal{C}_{0,\bar{0}}$ is uniquely determined in terms of $\mathcal{C}_{0,\bar{0}} ^{0,\bar{0}}$. 
Therefore the  coefficient of $z^a \bar{z}^b$  consists of: 
\[
B_{a,\bar{b}}=\mathcal{C}^{0,\bar{0}} _{a,\bar{b}}(\xi_{1,0})^a(\bar{\xi}_{1,0})^b+\sum_{k+l<a+b} \sum_{\substack{i+\sum_{e=1} ^k  p_e + \sum_{f=1} ^l q_{k+f} =a   
\\ j+\sum_{l=1} ^k q_e + \sum _{f=1} ^l p_{k+f}=b}} \mathcal{C}_{k,\bar{l}} ^{i,\bar{j}} \xi_{p_1\bar{q}_1}...\xi_{p_k\bar{q}_k} \bar{\xi}_{p_{k+1}\bar{q}_{k+1}}...\bar{\xi}_{p_{k+l}\bar{q}_{k+l}}
\] 
Thus by induction it is always possible to determine $\mathcal{C}^{0,\bar{0}} _{k,\bar{l}}$  in terms of $B_{a,\bar{b}}$'s from which $\mathcal{C}_{k,\bar{l}}$ is uniquely determined.  It can also be seen that the coefficient  $ \mathcal{C}_{a,\bar{b}}$ depends only on $\{B_{k,\bar{l}}\}_{k+l\leq a+b}$.
We thus obtain  an expansion like (\ref{firep})
with mentioned properties   for $\Phi$.


Also from   (\ref{xis})   we know that   the terms $B_{a,0} z^a$ for $a\geq 1$ is known if we know $\sum_{k=1} ^{+\infty} \hat{\mathscr{C}}_{k,\bar{0}} S^k $ due to the following

\begin{equation}
\frac{\partial ^a}{\partial z ^a} \bigg [ \bar{\partial} (\sum_{k=1} ^{+\infty} \hat{\mathscr{C}}_{k,\bar{0}} S^k ) \bigg ]\bigg\vert_{D\cap U_1}=\bar{\partial} B_{a,0}
\end{equation}
\end{proof}
\subsection{Proof of theorem \ref{theoloc}}

In order to prove the above lemma we first choose  an arbitrary potential $\Phi: U_1\rightarrow \mathbb{R}$ for $\omega$  
\[
\partial\bar{\partial}\Phi = \omega|_{U_1}
\]

  We assume that $\Phi$ has an expansion of the form $\Phi=\sum_{k,l\geq 0}\hat{\mathscr{C}}_{k,\bar{l}}S^k \bar{S}^l$. We define $\omega'_1$ by
\begin{equation}\label{omgp1}
\omega'_1 = \partial\bar{\partial}\Phi' 
\end{equation}
where 
\[
\Phi' = \sum_{k=0 \text{  or } l=0} \hat{\mathscr{C}}_{k,\bar{l}}S^k \bar{S}^l +\Phi''
\]
and $\Phi''$ has the form
\[
\Phi''= \sum_{k>0 \text{  and } l>0} \hat{\mathscr{B}}_{k,\bar{l}}S^k \bar{S}^l 
\]
We then prove that  the  map $\Phi'' $ with the above expansion can be found in such a way that   $\omega'_1 $ defined by relation (\ref{omgp1}) satisfies the DCMA equation (\ref{degmongeq}).  According to lemma (\ref{foot})  we know that $\partial\bar{\partial} (\sum _{k=0 \text{ or } l=0} \hat{\mathscr{C}}_{k,\bar{l}} S^k \bar{S} ^{l} )$  is uniquely determined in terms of $((\omega)_{S\bar{T}})_{hol}$ and $(\omega_{T\bar{T}})_{hol}$.
 Since we have

\[
\omega'_1 = \partial\bar{\partial} \Phi' = \partial\bar{\partial} ( (\sum _{k=0 \text{ or } l=0} \hat{\mathscr{C}}_{k,\bar{l}} S^k \bar{S} ^{l} ))+ \partial\bar{\partial}\Phi''
\]
 the  above lemma is proved  if  we show that   $\partial\bar{\partial}\Phi''$ does not depend on the choice of the potential $\Phi$.
  
To this end we  prove that  the coefficients  $\hat{\mathscr{B}}_{k,\bar{l}}$ for $k,l >0$
are uniquely determined in terms of $\bar{\partial} (\sum _{k=1} ^{\infty} \hat{\mathscr{C}}_{k,\bar{0}} S^k)=(\bar{\partial} \Phi_{\bar{0}})^+$ and $\omega|_D$ and we apply lemma (\ref{foot}).  
Consider a holomorphic coordinates system $(w_1,...,w_{n-1},z)$ in a neighborhood  $U_{1}$  as described in the proof of lemma (\ref{lemhasht})

 If we assume that
\[
\Phi' = \sum B_{i,\bar{j}}z^i \bar{z}^j
\]

 then by lemma (\ref{lemhasht}) we know that   $B_{k,\bar{0}}$ for $k\geq 1$ is determined in terms of   $(\omega_{S\bar{T}})_{hol}$.

 It is also clear that $(\partial\bar{\partial} B_{0,\bar{0}}|_{D})=(\omega |_{D}) (p)$.  In the coordinates $(w_1,...,w_{n-1},z)$  we have
\[
\partial\bar{\partial} \Phi=  \sum_{1\leq i,j\leq n-1} g'_{i,\bar{j}}  dw_i \wedge d\bar{w}_{j}+ \sum_{1\leq i\leq n-1} g'_{i\bar{n}}dw_i \wedge d\bar{z} +\sum_{1\leq i\leq n-1} g'_{n\bar{j}}dz\wedge d\bar{w}_j  +g'_{n\bar{n}}dz\wedge d\bar{z} 
\]
 therefore

\begin{equation}\label{g'nn}
g'_{n\bar{n}}=\sum B_{i\bar{j}}z^{i-1}\bar{z}^{j-1}
\end{equation}
\begin{equation}\label{g'jn}
g'_{j\bar{n}}=\sum l\frac{\partial B_{kl}}{\partial w_j}z^k\bar{z}^{l-1}
\end{equation}

\begin{equation}\label{g'ni}
g'_{n\bar{i}}= \sum k\frac{\partial B_{kl}}{\partial \bar{w}_i}z^{k-1}\bar{z}^l
\end{equation}

\begin{equation}\label{g'ij}
g'_{i\bar{j}}=\sum \frac{\partial^2 B_{kl}}{\partial w_i\partial\bar{w}_j}  z^{k}\bar{z}^l
\end{equation}

From (\ref{g'ni}) and (\ref{g'jn}) we obtain
\begin{equation}\label{g'nig'}
g'_{n\bar{i}}g'_{j\bar{n}}=\sum_{\substack{r+p=k\\s+q=l}}(r+1)(q+1)\frac{\partial B_{r+1 s}}{\partial \bar{w}_i}\frac{\partial B _{p q+1}}{\partial w_j} 
\end{equation}





 we define   the matrix $G=[g_{i\bar{j}}]_{n\times n}$    as follows

\[
G= \left [\begin{array}{ccc|c}
& & & g'_{1\bar{n}}\\
&G_0&& \vdots\\
& & &g'_{n-1\bar{n}}\\ \hline
g'_{n\bar{1}}&\hdots &g'_{n\overline{n-1}}& g'_{n\bar{n}}
\end{array}\right ]
\]
Thus

\begin{equation}\label{ann0}
\det G = g'_{n\bar{n}}\det [G_0]+  \sum_{i=1} ^{n-1} (-1)^{i+n}  g'_{n\bar{i}}\sum_{j=1} ^{n-1} (-1)^{j+n-1}g'_{j\bar{n}}(-1)^{i+j}M^{j, i}
\end{equation}

where $M^{j,i}$ is the $(j ,i )$ cofactor of $G_0$. Therefore we get to
\begin{equation}\label{matma}
g'_{n\bar{n}}\det [G_0]=\det G-  \sum_{i=1} ^{n-1} (-1)^{i+n}  g'_{n\bar{i}}\sum_{j=1} ^{n-1} (-1)^{j+n-1}g'_{j\bar{n}}(-1)^{i+j}M^{j, i}
\end{equation}

If $\Phi'$ satisfies the Monge Amp\`ere equation (\ref{degmongeq}) then $\det G$ corresponds to  the term $F_1 (\omega|_{V_1})^n$  in the right hand side of (\ref{degmongeq})  written in local coordinates. We consider the corresponding Taylor series expansion   $\det G= \sum_{k,l=0} ^{\infty} (\det G)_{k,\bar{l}} z^k \bar{z}^{l}$   of $\det G$ in terms of $(z, \bar{z})$ where the coefficients $(\det G )_{k,\bar{l}}$ depend only on $w_1,...,w_{n-1}$. Substituting   (\ref{g'nig'}) into  (\ref{matma}) leads to

\begin{equation}\label{bkl}
\begin{split}
B_{k+1. \bar{l+1}}=&\frac{(\det G )_{k, \bar{l}}}{(k+1)(l+1)\det G_0 } - \sum_{\substack{r+p=k\\s+q=l}} \frac{(r+1)(q+1)}{(k+1)(l+1)}\langle \bar{\partial} B_{r+1,\bar{s}} , \partial B_{p,\bar{q+1}}\rangle\\
\end{split}
\end{equation}
where the inner product $\langle, \rangle $ is the one induced by the restriction  $\omega'|_{D}=\omega|_D$.
 
 Applying the operator $\partial$ on both sides of (\ref{bkl}) we obtain

\[
\begin{split}
\partial B_{k+1.\bar{ l+1}}=&\partial (\frac{(\det G )_{k,\bar{l}}}{(k+1)(l+1)\det G_0 } )- \sum_{\substack{r+p=k\\s+q=l}} \frac{(r+1)(q+1)}{(k+1)(l+1)}\langle \partial\bar{\partial} B_{r+1,\bar{s}} , \partial B_{p,\bar{q+1}}\rangle\\
\end{split}
\]

Thus we can inductively determine all the coefficients $B_{k+1,\bar{l+1}}$ for $k,l\geq 0$  in terms of $\bar{\partial}B_{k,\bar{0}}$'s and  $\partial B_{0,\bar{l}}$'s  for $k,l\geq 1$ as well as  $\det G_0$.
If we can prove the convergence of the Taylor series whose coefficients are inductively determined according to (\ref{bkl}) then the  theorem will follow from   lemmas (\ref{foot}) and (\ref{lemhasht}).

In order to prove the convergence of the series $\sum_{k,l} B_{k,\bar{l}} z^k \bar{z}^l$ we assume by induction that  the series $\sum_{\substack{k\geq 0\\ l\leq n}}(D_{i_1},...,D_{i_{j_1}}\bar{D}_{i_{j_1+1}},...,\bar{D}_{i_j}B_{k,\bar{l}}) z^k \bar{z}^l$ is convergent for any choice of  the operators $D_{i_1},...,D_{i_j}\in \{\frac{\partial }{\partial w_1},...,\frac{\partial }{\partial w_{n-1}}\}$ and for all $l\leq m$. Thus there exists $R$ such that 
\begin{equation}\label{indh}
\limsup_{k}\big (|D_{i_1},...,\bar{D}_{i_j}B_{k,\bar{l}}|^{\frac{1}{k+l}} \big ) < R \hspace{0.5cm} \text{ for } \hspace{0.5cm}   k\geq 0  \text{ and } l\leq m
\end{equation}
and since $\det G$ is analytic we can choose $R$ in such a way that
\[
|(\det G)_{k,\bar{l}}|<R^{k+l} ,\hspace{0.5cm} \text{ for all } k,l \geq 0
\]

We want to prove that

\begin{equation}\label{indh2}
\limsup_{k}|D_{i_1},...,D_{i_j}B_{k,\bar{m+1}}|^{\frac{1}{k+m+1}}< R \hspace{0.5cm}  
\end{equation}
for all $k\geq 0 $ for all $j\geq 0$ and for all the choices of the operators $D_{i_1},...,D_{i_j}$.
 We set 
\begin{equation}\label{bkna}
A_{k, l}:=\max\{ | B_{k ,\bar{l}}|, |\bar{\partial}  B_{k ,\bar{l}}| \}
\end{equation}
We prove (\ref{indh2}) for $j=0$ the higher order case follows similarly. From (\ref{bkl}) we have
 
\[
\begin{split}
A_{k+1, \bar{m+1}} &\leq \frac{R^{k+m}}{(k+1)(m+1)\det G_0}+ \sum_{\substack{r+p=k\\s+q=m\\q<m}} \frac{(r+1)(q+1)}{(k+1)(m+1)}|\langle \bar{\partial} B_{r+1,\bar{s}} , \partial B_{p,\bar{q+1}}\rangle|\\
\quad &+  \sum_{\substack{r+p=k}}\frac{(r+1)}{(k+1)}|\langle \bar{\partial} B_{r+1,\bar{0}} , \partial B_{p,\bar{m+1}}\rangle|\\
\quad & \leq   \frac{R^{k+m}}{(k+1)(m+1)\det G_0}+  \frac{(k+2)m}{4}R^{k+m+2}+\sum_{\substack{0\leq p\leq k \\ p+r=k}}  \frac{(r+1)}{(k+1)}A_{p,m+1}R^{r+1}
\end{split}
\]

where in the last line we are applying the induction  hypothesis (\ref{indh}) and  the relation (\ref{bkna}) and we deduce

 \begin{equation}\label{ak1ie}
A_{k+1, m+1}\leq  \frac{R^{k+m}}{(k+1)(m+1)\det G_0}+\frac{(k+2)m}{4}R^{k+m+2}+\sum_{p=0} ^k\frac{k-p+1}{k+1} A_{p,m+1}R^{k-p+1},\hspace{1cm} k\geq 0
\end{equation}

  We define $\alpha_{p,q}$ by
\begin{equation}\label{aalf}
A_{p,q}= \alpha_{p,q}R^{p+q}
\end{equation}
 Then from the   inequality (\ref{ak1ie}) we have
\begin{equation}\label{ak2ie}
\alpha_{k+1, m+1}R^{k+m+2}\leq \frac{R^{k+m}}{(k+1)(m+1)\det G_0}+[\frac{(k+2)m}{4} +  \sum_{p=0} ^k\frac{k-p+1}{k+1} \alpha_{p,m+1} ]R^{k+m+2}
\end{equation}

We define the sequence $\{\beta_{k,m+1}\}$ as follows

  \begin{equation}\label{betalfa}
\beta_{k,m+1}:=\sum_{p=0} ^k \alpha_{p,m}
\end{equation}

Hence 

\[
\sum_{p=0} ^k \beta_{p,m+1}=\sum_{p=0} ^k  (k-p+1)\alpha_{p,m+1}
\]
 If $c$ is an upper bound for  $\frac{1}{(k+1)(m+1)\det G_0 \times R^2}$ then from  (\ref{ak2ie}) we conclude that  
\begin{equation}\label{bet-bet}
\beta_{k+1,m+1}-\beta_{k,m+1}\leq c+\frac{(k+2)m}{4}+\frac{\sum_{p=0} ^k \beta_{p,m+1}}{k+1}
\end{equation}
If  we  define the sequence  $\{\gamma_{k,m+1}\}$ by

\begin{equation}\label{gamabeta}
\gamma_{k,m+1}:=\sum_{p=0} ^k \beta_{p, m+1}
\end{equation}
 
then from (\ref{bet-bet}) we deduce that

\begin{equation}\label{gamkn}
\gamma_{k+1,m+1}-(2+\frac{1}{k+1})\gamma_{k,m+1}+\gamma_{k-1,m+1}\leq  c+\frac{(k+2)m}{4}
\end{equation}

It  can be easily verified that the sequence $\{\delta_{k,m+1}\}_{k\geq 0}$ defined by
\begin{equation}\label{deltasarih}
\delta_{k,m+1}=-(c+\frac{3m}{4})(k+1)- \frac{m}{4} (k+1)^2
\end{equation}

satisfies
\[
\delta_{k+1,m+1}-(2+\frac{1}{k+1})\delta_{k,m+1}+\delta_{k-1,m+1}=c+\frac{(k+2)m}{4}
\]

Thus if we define the sequence $\{\eta_{k,m+1}\}_{k\geq 0}$ by 
\begin{equation}\label{etagamadelta}
\eta_{k,m+1}=\gamma_{k,m+1}-\delta_{k,m+1}
\end{equation}
 
then the inequality  (\ref{gamkn}) will be equivalent to 
\[
\eta_{k+1,m+1}-(2+\frac{1}{k+1})\eta_{k,m+1}+\eta_{k-1,m+1}\leq 0
\]

Now we take a  fixed  $k_0$ and we define  the sequence $\{u^{k_0} _j\}_{j\geq k_0-1}$ by
\[
u^ {k_0}_{j+1}-(2+\frac{1}{k_0+1})u^{k_0} _{j}+u^{k_0} _{j-1}=0, \hspace{1cm} u^{k_0}_{k_0-1}=\eta_{k_0-1, m+1},  u^{k_0}_{k_0}=\eta_{k_0, m+1}
\]

 then it is not difficult to see that

\begin{equation}\label{tetatakhmin}
\eta_{i,m+1} \leq u^{k_0} _i \hspace{1cm} \text{for } i\geq k_0-1
\end{equation}

This follows from the fact that if  $b_k-a_k \geq b_{k-1}-a_{k-1}\geq 0$ then 
\[
((2+\frac{1}{k+1})b_{k}-b_{k-1})-((2+\frac{1}{k+1})a_{k}-a_{k-1})\geq b_k-a_k
\]

On the other hand we know that

\begin{equation}\label{usarih}
u^{k_0} _j = a(1+\frac{1+\sqrt{4k_0+5}}{2(k_0+1)})^j + b(1+\frac{1-\sqrt{4k_0+5}}{2(k_0+1)})^j
\end{equation}

where $a$ and $b$ are determined in terms of $u^{k_0} _{k_0-1}$ and $u^{k_0} _{k_0}$. From  (\ref{deltasarih})  ,   (\ref{etagamadelta}) and (\ref{tetatakhmin}) we have

\begin{equation}\label{gamanam}
0\leq \gamma_{k,m+1}\leq u^{k_0} _{k} +|\delta _{k, m+1}|
\end{equation}

  From (\ref{betalfa}) and(\ref{gamabeta}) we see that
\begin{equation}\label{alfagama}
\alpha_{k,m+1}= \gamma_{k, m+1}-2\gamma_{k-1,m+1}+\gamma_{k-2, m+1}
\end{equation}
It then follows from (\ref{usarih}), (\ref{gamanam})and  (\ref{alfagama}) that
\begin{equation}\label{lsalf}
\limsup_{k\rightarrow \infty} (\alpha_{k,m+1})^{\frac{1}{k+m+2}}\leq 1+\frac{c_0}{\sqrt{k_0+1}}
\end{equation}

for  some  $c_0$ independent of $k$, $m$ and $k_0$. Since $k_0$ is arbitrary therefore from (\ref{aalf}) and (\ref{lsalf}) we conclude that
\[
\limsup_{k} (B_{k,m+1})^{\frac{1}{k+m+1}}<R
\]

If $\tilde{F}>0$ outside $D$ and $\tilde{F}$ vanishes along $D$ since it is assumed to be analytic we have $\tilde{F} = |S|^{2k}e^F$ for an analytic map $F:U\rightarrow \mathbb{R}$. We prove the lemma for the case $k=1$ the argument is similar for $k>1$. Following lemma (\ref{rela}) in the appendix \ref{app5} we can take a holomorphic coordinates system in the neighborhood of a point $p\in D$  in such a way that $p=(0,...,0)$ in this coordinates system and  the corresponding  potential $\Phi' _1$ of $\omega' _1$  in this neighborhood satisfies $\bar{\partial} B_{1,\bar{0}}|_D (p)= \partial B_{0,\bar{1}}|_{D} (p)= \bar{\partial} B_{1,\bar{1}}|_D (p)= \partial B_{1,\bar{1}}|_{D} (p)=0$.  By the same lemma in this coordinates we have $g'_{w_i \bar{z}}=O(|z|^3)$ and $g'_{z\bar{z}}=O(|z|^2)$
for $i=1,...,n-1$, along the fiber $w_1=...=w_{n-1}=0$. Therefore the  positivity of the matrix $G$ in this coordinates for small values of $|z|$
is obvious.
\\

 \textbf{Remark} It is well-known that in each K\"ahler class there exists an analytic metric hence we can apply the above result for 
constructing a degenerate K\"ahler metric as in the following section on a general K\"ahler manifold.

\section{Construction of Global degenerate K\"ahler metrics}\label{sec3}

In this section we construct globally defined K\"ahler metrics $\omega'$ in each K\"ahler class $[\omega]\in H^{1,1} (X)$ resolving a DCMA equation   of the form 
\[
{\omega ' }^ n = |S|^2 e^F \omega^n
\]

such that $\omega'|_D$ defines  a  K\"ahler metric over $D$. 
 Here   $F$ is a smooth map $F: X\rightarrow \mathbb{R}$ which can be determined in terms of $\omega'$.

   The idea is to glue a local degenerate  metric constructed by theorem (\ref{theoloc}) in a sufficiently small neighborhood of $D$,  to $\omega$ outside  another sufficiently small neighborhood of $D$.


 
\begin{lemma}\label{lem9}
For any $0<m<2$
there exists  a function $h:[0,+\infty)\rightarrow \mathbb{R }$  and positive numbers $0<\epsilon<\epsilon '$ such that $h(x)=x^2$, for $0\leq x<\epsilon $, $h(x) = 0$ for $x>\epsilon '$ and $x\rightarrow x^2 -h(x)$ is a convex function.  Moreover if $u$ is defined by $u=\frac{x^2 -h}{x^2}$ then $u$ is an increasing  function    satisfying $x^2 u'' +4xu' +mu\geq 0$ and it is strictly positive on $[\epsilon, \epsilon']$.
\end{lemma}

\begin{proof} As can be seen by a simple computation the required inequality is equivalent to
\[
(x^2 u)''\geq (2-m)u
\]

Substituting $u$ in terms of $h$ we get to 

\[
(x^2 -h)''\geq (2-m) (1-\frac{h}{x^2})
\]
 
or equivalently

\begin{equation}\label{hm}
m+\frac{(2-m)h}{x^2}\geq h''
\end{equation}

Moreover $u$ is increasing iff $(1-\frac{h}{x^2})'\geq 0$ which simplifies to
\begin{equation}\label{hinc}
2h>xh'
\end{equation}

Therefore  if we can find $\tilde{h}: \mathbb{R}\rightarrow \mathbb{R}$ for which  the following five properties hold 

1) 
\begin{equation}\label{p1}
\tilde{h}(x) = 2x  \text{ for } x<\epsilon
\end{equation}

2) 
\begin{equation}\label{p2}
 \tilde{h}(x) = 0 \text{ for }  x>\epsilon'
\end{equation}

3) 
\begin{equation}\label{p3}
|\tilde{h}'(x)|\leq m +\frac{(2-m)\int_ 0 ^x \tilde{h}}{x^2}
\end{equation}

4) 
\begin{equation}\label{p4}
\int_0 ^{\epsilon'} \tilde{h}(x) dx =0 
\end{equation}
 
5) 
\begin{equation}\label{p5}
2\int_0 ^x \tilde{h}>x\tilde{h}
\end{equation}


Then (\ref{hm}) and (\ref{hinc}) will be established for  the map $h:=\int_0 ^x \tilde{h}$ and   lemma \ref{lem9} is proved.

To prove the existence of $\tilde{h}$ we note that  inequality (\ref{p3}) in property (3) is automatically satisfied for $x\leq \epsilon$ due to the fact that by property 1 we have  $\tilde{h}(x) = 2x$.
 
We first extend the map $\tilde{h} (x)=2x$ initially defined on the interval $[0,\epsilon ]$ to a map defined on a  larger interval $[0, x_0 ]$ for some $x_0>\epsilon $
  in such a way that  it becomes concave and increasing  on $[\epsilon, x_0]$ and  satisfies the inequality of \ref{p3}. Moreover we require that the following relation holds: 
\[
\lim_{t\rightarrow 0^-} \frac{\tilde{h}(x_0-t) -\tilde{h}(x_0)}{t}=0
\]
Then we extend $\tilde{h}$ over the whole interval $[0, 2x_0]$ by setting
\[
 \tilde{h}(x)=\tilde{h}(2x_0 -x) \hspace{0.5cm}\text{for}  \hspace{0.5cm} x_0\leq x\leq 2x_0
\]
 Our construction will be so that the strict inequality 

\begin{equation}\label{p3strict}
|\tilde{h}'(x)|< m +\frac{(2-m)\int_ 0 ^x \tilde{h}}{x^2}
\end{equation}
holds for $0\leq x\leq 2x_0$.

In order to define $\tilde{h}$ for $x>2x_0$ we first take a small $\epsilon_0$ and we extend $\tilde{h}$ on $[x_0 , x_0 + \epsilon _0]$ such that the inequality $(\ref{p3})$ still holds on this interval and at the same time the derivative $|\tilde{h}' (2x_0+ \epsilon_0)|$ becomes smaller than $m$. The map  $\tilde{h}'$ is continued  beyond $2x_0 +\epsilon $ so that the ineqauality $|\tilde{h}' (x)|<m$ keeps holding true,   $\int_0 ^x \tilde{h} \geq 0$, for all $x$ and, $\int_0 ^x \tilde{h}= 0$ if  $x\geq \epsilon '$ for some $\epsilon '$.

     The inequality (\ref{p5}) is automatically true for $x>2x_0$  since its right hand side is negative. Therefore all we need to do is to establish  the two inequalities (\ref{p3}) and (\ref{p5})  over the interval $[0,2x_0]$. In order to

  To this end we set

\begin{equation}\label{tilh}
\tilde{h}(x) = 2x+ \frac{ab}{(x-\epsilon)^3} e^{-\frac{b}{(x-\epsilon)^2}}1_{x\geq \epsilon}
\end{equation}

where $a<0<b$ are the parameters to be determined and $1_{x\geq \epsilon}$  denotes the characteristic function of 
the
 interval $[\epsilon, +\infty)\subset\mathbb{R}$.

We first  prove that  $a,b$ and $\epsilon$  can be chosen in such a way that $\tilde{h}$   on the interval  $x\in [\epsilon , x_0]$  satisfies the properties 1-5, where  $x_0$ is the point at which the maximum of $\tilde{h}$ occures; $\tilde{h}'(x_0)=0$. The derivative $\tilde{h}$ is given by

\[
\tilde{h}'(x)=2+[-\frac{3ab}{(x-\epsilon)^4} +\frac{2ab^2}{(x-\epsilon)^6}]e^{-\frac{b}{(x-\epsilon)^2}}
\]

thus $\tilde{h}'(x_0)=0$ iff
\[
e^{-\frac{b}{(x_0-\epsilon)^2}}=-\frac{2(x_0-\epsilon)^6}{2ab^2-3ab (x_0-\epsilon)^2}
\]

or

\begin{equation}\label{x0}
e^{-\frac{b}{(x_0-\epsilon)^2}}=(-\frac{2}{ab})\frac{(x_0-\epsilon)^6}{2b-3 (x_0-\epsilon)^2}
\end{equation}

The map $x\rightarrow (-\frac{2}{ab})\frac{(x-\epsilon)^6}{2b-3 (x-\epsilon)^2}$ on the right hand side  of (\ref{x0}) is positive increasing for $\epsilon<x <\epsilon+\sqrt{\frac{2b}{3}}$ and it is negative for 
$x >\epsilon+\sqrt{\frac{2b}{3}}$. Also the function $x\rightarrow e^{-\frac{b}{(x-\epsilon)^2}}$ on the left hand side  of (\ref{x0})
is increasing  for $x\geq \epsilon$ and takes values in the interval $(0,1)$.

The inequality
\begin{equation}\label{ineq3}
|\tilde{h}'(x)|\leq m +\frac{(2-m)\int_{0} ^x \tilde{h}}{x^2}
\end{equation}

holds if and only if

\[
\frac{2ab^2-3ab (x-\epsilon)^2}{(x-\epsilon)^6}e^{-\frac{b}{(x-\epsilon)^2}}\leq  \frac{(2-m)a}{2x^2}e^{-\frac{b}{(x-\epsilon)^2}}
\]

or

\begin{equation}\label{nms}
(2-m)(x-\epsilon)^6\leq 4b^2 x^2 -6b (x-\epsilon)^2 x^2
\end{equation}

which is equivalent to 

\begin{equation}\label{nms2}
\begin{split}
(2-m)(x-\epsilon)^6 &\leq  (4b^2- 6b(x-\epsilon )^2)  \bigg[(x-\epsilon)^2+2\epsilon (x-\epsilon)+\epsilon ^2\bigg ]
\\
\quad & = -6b(x-\epsilon)^4-12b\epsilon (x-\epsilon)^3+[-6b\epsilon ^2+4b^2](x-\epsilon)^2+8\epsilon b^2 (x-\epsilon)+4b^2\epsilon ^2
\end{split}
\end{equation}

We claim that it is 
possible to  determine the  parameters $a$ and $b$ and $\epsilon$ in such a way that the equation (\ref{x0}) admits a solution  for $x_0$ and  moreover  the   inequality (\ref{nms2}) is valid for 
$\epsilon\leq x\leq x_0$.

To prove this claim we fix $b$ and we determine $a$ such that (\ref{x0}) admits a solution denoted by $x_0$. We then prove that for $\epsilon$ large enough  the inequality (\ref{nms2}) holds on the interval $[\epsilon, x_0]$
 
 if we set $y:=x-\epsilon$ and we regroup the inequality  (\ref{nms2}) in terms of $\epsilon$ we obtain
\begin{equation}\label{nms22}
[4b^2-6by^2]\epsilon ^2+[8b^2y-12by^3]\epsilon+[-(2-m)y^6 +4b^2y^2-6by^4]\geq 0
\end{equation}
 Similarly  the equation  (\ref{x0})  can be stated   as
\begin{equation}\label{x0y}
e^{-\frac{b}{y^2}}=-\frac{2}{ab}\frac{y^6}{2b-3y^2}
\end{equation}

We note that the  function  $y\rightarrow e^{-\frac{b}{y^2}}$ on  left hand side of the equation (\ref{x0y}) takes its values  in the interval $[0,e^{-3/2}]$ when $y\in [0, \sqrt{\frac{2b}{3}}]$. Also the function
$y\rightarrow (-\frac{2}{ab})\frac{2y^6}{2b-3 y^2}$ on the right hand side for $y$ in the interval 
$[0, \sqrt{\frac{2b}{3}}]$ starts from $0$ and tends towards $+\infty$ as $y$ approaches $ \sqrt{\frac{2b}{3}}$. Thus for any fixed $b$
 if $|a|$  is large enough the equation (\ref{x0y}) (and hence (\ref{x0})) will admit a solution  $y_0$  in the interval $(0,\sqrt{\frac{2b}{3}})$ 
\[
y_0\in (0, \sqrt{\frac{2b}{3}})
\]

This is equivalent to say that for any fixed $b>0$ there exists $a<0$ such that the equation (\ref{x0}) admits a solution  
\[
x_0\in (\epsilon, \epsilon+\sqrt{\frac{2b}{3}})
\]

 since $0<y_0< \sqrt{\frac{2b}{3}}$, for $0\leq y\leq y_0$ we have $4b^2-6by^2\geq 4b^2-6by_0^2>0$ thus for a fixed $b$ and for $0\leq y\leq y_0$
the inequality (\ref{nms22}) will be established if $\epsilon $ is chosen to be large enough.


As mentioned before e extend  the function $\tilde{h}$ defined on the interval $[0, x_0]$ by (\ref{tilh}) according to 
\[
\tilde{h} (x)= \tilde{h}(2x_0-x) \hspace{1cm} \text{for }  x_0\leq x\leq 2x_0 
\]
 
and we can assume that for $x>2x_0$ we have $h(x_0)\leq 0$. Therefor  the inequality of property 5 is established for $x\geq 2x_0$ and we need to prove it for $0\leq x\leq 2x_0$.

\[
\int_0 ^x \tilde{h} = x^2+\frac{1}{2} ae^{-\frac{b}{(x-\epsilon)^2}}
\]

 For the inequality $2\int_0 ^x \tilde{h}\geq x\tilde{h}$  of property 5, it is obviously true if $0\leq x \leq \epsilon$.  If $\epsilon< x\leq x_0$  we have

\[
\int_0 ^x \tilde{h} = x^2+\frac{1}{2} ae^{-\frac{b}{(x-\epsilon)^2}}
\]
and the inequality
 
\[
2x^2+ ae^{-\frac{b}{(x-\epsilon)^2}}\geq 2x^2 + \frac{abx}{(x-\epsilon )^3}e^{-\frac{b}{(x-\epsilon)^2}}
\]
 
holds if 
 
\[
(x-\epsilon)^3\leq bx =b (x-\epsilon)+ b\epsilon
\]

By the assumption $\epsilon< x\leq x_0$ the left hand side of the above equation depends only on $b$ and the right hand side grows to infinity as $\epsilon \rightarrow \infty$.
  
For $x_0\leq x \leq 2x_0$ we use the relation $h(x)=h(2x_0-x)$ to derive:

\[
\begin{split}
\int_0 ^x \tilde{h}&= 2\int_0 ^{x_0}\tilde{h}-\int_0^{2x_0-x}\tilde{h}\\
\quad & =  2x_0 ^2+ ae^{-\frac{b}{(x_0-\epsilon)^2}}- (2x_0-x)^2-\frac{1}{2} ae^{-\frac{b}{(2x_0-x-\epsilon)^2}} 1_{2x_0-x>\epsilon}
\end{split}
\]

and by definition of (\ref{tilh}) we have

\[
x\tilde{h}(2x_0-x)= 2x(2x_0-x)+\frac{abx}{(2x_0-x-\epsilon)^3}e^{-\frac{b}{(2x_0-x-\epsilon)^2}}1_{2x_0-x>\epsilon}
\]

hence the inequality (\ref{p5}) of property 5 is equivalent to 
\[
4x_0 ^2+ 2ae^{-\frac{b}{(x_0-\epsilon)^2}}- 2(2x_0-x)^2- ae^{-\frac{b}{(2x_0-x-\epsilon)^2}} 1_{2x_0-x>\epsilon}
\geq  2x(2x_0-x)+\frac{abx}{(2x_0-x-\epsilon)^3}e^{-\frac{b}{(2x_0-x-\epsilon)^2}}1_{2x_0-x>\epsilon}
\]

or 

\begin{equation}\label{nequ1}
 4x_0 (x-x_0)\geq - 2ae^{-\frac{b}{(x_0-\epsilon)^2}} +\big [ ae^{-\frac{b}{(2x_0-x-\epsilon)^2}} +
  \frac{abx}{(2x_0-x-\epsilon)^3}e^{-\frac{b}{(2x_0-x-\epsilon)^2}}\big ]1_{2x_0-x>\epsilon}
\end{equation}

Fixing a positive constant $\delta >0 $ such that $\delta< x_0 -\epsilon$ and assuming that $x_0<x<x_0+\delta$ then one can see that
\begin{equation}\label{2x0x}
2x_0-x-\epsilon >x_0-\delta-\epsilon >0
\end{equation}

and thus the right hand side of (\ref{nequ1}) is upper estimated as follows
\[
\begin{split}
- 2&ae^{-\frac{b}{(x_0-\epsilon)^2}} +\big [ ae^{-\frac{b}{(2x_0-x-\epsilon)^2}} +
  \frac{abx}{(2x_0-x-\epsilon)^3}e^{-\frac{b}{(2x_0-x-\epsilon)^2}}\big ]1_{2x_0-x>\epsilon}\\
\quad & \leq  \frac{(2x_0-x-\epsilon)^3
\big [ - 2ae^{-\frac{b}{(x_0-\epsilon)^2}} +ae^{-\frac{b}{(x_0-\delta -\epsilon)^2}} \big ]+
  abx_0e^{-\frac{b}{(x_0-\delta-\epsilon)^2}} }{(2x_0-x-\epsilon)^3}
\end{split}
\]

We also know that $x_0 <\epsilon +\sqrt{\frac{2b}{3}}$ thus we have
\[
e^{-\frac{b}{(x_0-\epsilon)^2}}\leq e^{-3/2} \leq 1
\]
and
\[
e_0:=e^{-\frac{b}{(\sqrt{\frac{2b}{3}}-\delta)^2}}\leq e^{-\frac{b}{(x_0-\delta -\epsilon)^2}}\leq 1
\]

\begin{equation}\label{avdo}
\begin{split}
 &\frac{(2x_0-x-\epsilon)^3
\big [ - 2ae^{-\frac{b}{(x_0-\epsilon)^2}} +ae^{-\frac{b}{(x_0-\delta -\epsilon)^2}} \big ]+
  abx_0e^{-\frac{b}{(x_0-\delta-\epsilon)^2}} }{(2x_0-x-\epsilon)^3}\\
\quad & \leq \frac{(2x_0-x-\epsilon)^3
\big [ - 2a +a e_0 \big ]+
  abx_0e_0 }{(2x_0-x-\epsilon)^3}
\end{split}
\end{equation}
utilizing  $x_0-\epsilon <\sqrt{\frac{2b}{3}}$, $x_0-x<0$ and (\ref{2x0x}) yields
\begin{equation}\label{av1}
0< 2x_0 -x-\epsilon<\sqrt{\frac{2b}{3}}
\end{equation}

and 

\begin{equation}\label{do2}
bx_0 >b\epsilon
\end{equation}

from (\ref{av1}) and (\ref{do2}) we conclude that for $\epsilon $ large enough the right hand side of (\ref{avdo}) becomes negative.
Therefore the inequality (\ref{nequ1}) is established for $x_0<x<x_0+\delta$ if $\epsilon$ is taken to be large enough.

If $x_0+\delta< x< 2x_0$, then   we have $4x_0 (x-x_0)> 4\epsilon \delta$ and  thus the left hand side of (\ref{nequ1}) tends to $+\infty
 $ as $\epsilon \rightarrow +\infty$ while the right hand side remains bounded. Hence the inequality (\ref{nequ1}) is again established for large values of $\epsilon$.
This completes the proof of lemma (\ref{lem9})
\end{proof}

We now set 
\begin{equation}\label{alph1}
\alpha:= 1-u
\end{equation}

 and 

\begin{equation}\label{alph}
\alpha_{\lambda } (x):= \alpha(\lambda x)
\end{equation}

where  $u$ is  defined by  lemma (\ref{lem9}) and  $\lambda\in \mathbb{R}$ is a real parameter. 
 We  assume that
\begin{equation}\label{tilal}
 \tilde{\alpha}_{\lambda} (r):=\alpha[\lambda (ar+ O(r^2))]
\end{equation}
 where $a>0$ is a constant, and $ar +O(r^2)$ is a first order Taylor series expansion of a smooth function which depends only on the  variable $r$ in polar coordinates system $(r,\theta)$.

\begin{prop}\label{222}
Let $f,g:\mathbb{R}^2 \rightarrow \mathbb{R}$ be two strictly subharmonic functions such that $f(0,0)=g(0,0)$ and $Df(0,0)=Dg(0,0)$.  Let also
\begin{equation}\label{hessi}
 Hess (g)(0,0)>Hess (f)(0,0)\geq 0
\end{equation}

and 

\begin{equation}\label{hessi}
Hess(g-f) (0,0)>m_0>0
\end{equation}
 for some positive constant $m_0$. Then for large values of $\lambda$ the map $T_{\lambda}:\mathbb{R}^2 \rightarrow \mathbb{R}$
defined by 
\[
 T_{\lambda}:=\tilde{\alpha}_{\lambda} f+ (1-\tilde{\alpha}_{\lambda}) g
\]

  is a  strictly subharmonic map which glues $f$ and  $g$  together in a small neighborhood of  the origin $(0,0)$.
\end{prop}

\begin{proof} We have
 
\begin{equation}\label{delt}
\begin{split}
\Delta T_{\lambda} &= \Delta \tilde{\alpha}_{\lambda} (f-g) + 2(\tilde{\alpha}_{\lambda})_r (f_r - g_r )+ \tilde{\alpha}_{\lambda} \Delta f + (1- \tilde{\alpha}_{\lambda})\Delta g\\
\quad &= ((\tilde{\alpha}_{\lambda})_{rr}+ \frac{1}{r} (\tilde{\alpha}_{\lambda})_r ) (f-g)+ 2(\tilde{\alpha}_{\lambda})_r (f_r- g_r) +\tilde{\alpha}_{\lambda} \Delta f + (1- \tilde{\alpha}_{\lambda})\Delta g\\
\quad &= ((\tilde{\alpha}_{\lambda})_{rr}+ \frac{1}{r} (\tilde{\alpha}_{\lambda})_r ) (f-g)+ 2(\tilde{\alpha}_{\lambda})_r (f_r- g_r) + (1- \tilde{\alpha}_{\lambda})(\Delta g-\Delta f) +\Delta f
\end{split}
\end{equation}

and from (\ref{tilal}) it follows that 
\[
\frac{d}{dr}(\tilde{\alpha} _{\lambda} (r) )  = \lambda (a+O(r) )\alpha_r [\lambda (ar+ O(r^2))]
\]
and
\[
\frac{d^2}{dr^2}(\tilde{\alpha} _{\lambda} (r) )=\lambda O(1)   \alpha_r [\lambda (ar+ O(r^2))]+ \lambda^2 (a+O(r) )^2\alpha_{rr} [\lambda (ar+ O(r^2))]
\]

hence if we assume that $g-f= c(\theta)r^2+O(r^3)$ then since from (\ref{hessi}) it is known  that $c(\theta)>0$,  one can deduce

\[
\begin{split}
\Delta T_{\lambda}&= - (a+O(r))^2 \alpha_{rr}(\lambda (ar+ O(r^2))) )[c(\theta)(\lambda  r)^2 +\lambda^2 O(r^3)] \\
\quad & - O(1) \alpha_r (\lambda (ar+ O(r^2)))[c(\theta)\lambda r^2 +\lambda O(r^3)]\\
\quad &-(a+O(r))\alpha_{r}(\lambda (ar+ O(r^2)) )[c(\theta) (\lambda r )+\lambda O(r^2) ] \\
\quad  &-4 (a+O(r))\alpha _r (\lambda (ar+ O(r^2))) [c(\theta)(\lambda r) +\lambda O(r^2)]+ \\ 
\quad & \bigg (1-\alpha_{\lambda} \big ( \lambda (ar+ O(r^2)) \big ) \bigg ) (4c(\theta )+c''(\theta )+O(r))+\Delta f\\
\quad &=\bigg [-  (a+O(r))^2\alpha_{rr}(\lambda (ar+ O(r^2))) (c(\theta)(\lambda r)^2  )-(a+O(r))\alpha_{r}(\lambda (ar+ O(r^2)))(c(\theta) (\lambda r))\\
\quad &  -4(a+O(r)) \alpha _r (\lambda (ar+ O(r^2))) c(\theta)(\lambda r )\\
\quad & +\bigg (1-\alpha_{\lambda} \big ( \lambda (ar+ O(r^2)) \big ) \bigg )(4c(\theta )+c''(\theta )+O(r))+\Delta f\bigg ]\\
\quad &- (a+O(r))^2 \alpha_{rr}(\lambda (ar+ O(r^2))) )[\lambda^2 O(r^3)]  - O(1) \alpha_r (\lambda (ar+ O(r^2))) [\lambda O(r^2)] \\
\quad & -(a+O(r))\alpha_{r}(\lambda (ar+ O(r^2)) )[ \lambda O(r^2) ] +-4 (a+O(r))\alpha _r (\lambda (ar+ O(r^2))) [\lambda O(r^2)] \\
\end{split}
\]

 Substituting $R:=\lambda r  $ and $u(r) :=1-\alpha (r)$ where $u$ is introduced  by relation (\ref{alph1})  yields

\begin{equation}\label{pert}
\begin{split}
\Delta T_{\lambda}&=a^2 c(\theta)u''(aR+O(\frac{R}{\lambda}))R^2+5au' (aR+O(\frac{R}{\lambda}))c(\theta) R+(4c+c'')u(aR+O(\frac{R}{\lambda}))+\Delta f+O(\frac{R}{\lambda})\\
\quad & =c(\theta)[u''(aR+O(\frac{R}{\lambda})) (aR)^2+ 4u' (aR+O(\frac{R}{\lambda}))(aR)] +u'(aR+O(\frac{R}{\lambda}))c(\theta)(aR)\\
\quad & + (4c+c'')u(aR+O(\frac{R}{\lambda}))+\Delta f+O(\frac{R}{\lambda}))
\end{split}
\end{equation}

 The hypothesis $\Delta (g-f) (0.0)>m_0$ (\ref{hessi}) implies that  $c''+ 4c>  m_0$.  Hence  $m<2$  can be chosen such that it satisfies
\begin{equation}\label{c''}
c''+4c> (2-m)c
\end{equation}

 This is becase  we can find    $m$   such that $\frac{m_0}{2-m}>c$.  We also know
that  $c(\theta)=\frac{\partial ^2 (g-f)}{\partial r^2} (0,0)>0$ which follows by  positivity of $Hess (g-f)$. Since $u$ is increasing (lemma \ref{lem9}) we have 
\begin{equation}\label{u'}
u'(R)c(\theta)R\geq 0 
\end{equation}
from relations (\ref{pert}), (\ref{c''}) and (\ref{u'}) and lemma (\ref{lem9}) one can  conclude that $\Delta T_{\lambda}>0$  for large values of $\lambda$ on 
$\epsilon  \leq R\leq \epsilon '$.

 \end{proof}

 Let $\beta_{\lambda}:X\rightarrow \mathbb{R}$ be defined by 
\[
\beta_{\lambda} (x)=\alpha_{\lambda} (|S|)
\]
 where $S$ is as before the holomorphic section of the line bundle  $L=[D]$ with simple zero along 
$D$. 
Let $\Phi : U\rightarrow \mathbb{R}$ be  a real analytic map defined in a neighborhood $U$ of $D$ such that 
$\omega':=(\omega+\partial\bar{\partial}\Phi )|_{U}$  defines a degenerate K\"ahler metric as described in theorem (\ref{theoloc}).

\begin{prop}\label{pgl1}
For large enough $\lambda$ the $(1,1)$-form $\omega_{\lambda}$ given by
\[
\omega_{\lambda}:=  \omega+\partial\bar{\partial}(\beta_{\lambda} \Phi )
\]
 defnes  a K\"ahler metric  on $X\setminus D$ gluing together the  degenerate metric $(\omega+\partial\bar{\partial}\Phi )|_{U_1}$ in a neighborhood $U_1$ of $D$ and $\omega|_{X\setminus U_2}$ outside a neighborhood $U_2$ of $D$ with $U_1\subset U_2$. 
\end{prop}
\begin{proof}
Let $U_p$ be a neighborhood of a point $p\in D$
 and assume that $\Phi_0:U_p\rightarrow \mathbb{R}$ is a potential for $\omega|_{U_{p}}$:
\begin{equation}\label{fi0}
\omega|_{U_p}= \partial\bar{\partial}\Phi_0
\end{equation}

We assume that on $U_p$ there exists a potential $\Phi_{1}:U_p\rightarrow \mathbb{R}$ for $\omega'$ 

\begin{equation}\label{fi0}
\omega'|_{U_p}= \partial\bar{\partial}\Phi_1
\end{equation}

According to theorem (\ref{theoloc})  $\Phi_1$ can be chosen  in such a way that
\begin{equation}\label{fizfik}
(\Phi_0)_{hol}= (\Phi_1)_{hol} \hspace{0.5cm} \text{ and }\hspace{0.5cm} (\Phi_0)_{antihol}= (\Phi_1)_{antihol}
\end{equation}
from  theorem  (\ref{degmongeq}) we also know  that
\begin{equation}\label{fizfik1}
(\omega_{T\bar{T}})_{hol}=(\omega'_{T\bar{T}})_{hol}\hspace{0.5cm} \text{and} \hspace{0.5cm} (\omega_{S\bar{T}})_{hol}= (\omega' _{S\bar{T}})_{hol}
\end{equation}

 Consider the map $\Phi_{\lambda}:U_{p}\rightarrow \mathbb{R}$ defined by 
\[
\Phi_{\lambda}:=\beta_{\lambda}\Phi_1+ (1-\beta_{\lambda} )\Phi_0 :U_{p}\rightarrow \mathbb{R}
\]

As we have seen in  theorem (\ref{theoloc}) $\Phi=\Phi_1-\Phi_0$   does not depend on the choice of $\Phi_0$ and $\Phi_1$. Therefore
from the obvious identity $\beta_{\lambda}\Phi_1+ (1-\beta_{\lambda} )\Phi_0= \Phi_0+ \beta_{\lambda}\Phi$  we obtain
\[
\omega_{\lambda}=\partial\bar{\partial}\Phi_{\lambda} = \omega+ \partial\bar{\partial}(\beta_{\lambda}\Phi)
\]

and this  shows that $\partial\bar{\partial}\Phi_{\lambda} $ is well-defined on the neigborhood $U$ of $D$. Now according to (\ref{fizfik})
we have $(\Phi_{\lambda})_{hol}=(\Phi_0)_{hol}= (\Phi_1)_{hol}$
 and $(\Phi_{\lambda})_{antihol}=(\Phi_0)_{antihol}= (\Phi_1)_{antihol}$. Again from theorem (\ref{degmongeq}) it follows that
\[
((\omega_{\lambda})_{S\bar{T}})_{hol}= ((\omega)_{S\bar{T}})_{hol} \hspace{0.5cm} \text{ and }\hspace{0.5cm}((\omega_{\lambda})_{T\bar{T}})_{hol}= ((\omega)_{T\bar{T}})_{hol}
\]

Since 
\begin{equation}\label{fi1122}
\Phi_1-\Phi_0 = O(|S|^2)
\end{equation}

 one can show that

\begin{equation}\label{tomed1}
\begin{split}
(\partial\bar{\partial} \Phi_{\lambda})_{T\bar{T}}&= ((\omega)_{T\bar{T}})_{hol}+ ((\omega)_{T\bar{T}})_{antihol}+O(|S|^2)\\
\quad &+ (\partial \beta_{\lambda})_T\wedge (\bar{\partial} (\Phi_{1}-\Phi_0))_{\bar{T}}\\
\quad & -(\bar{\partial} \beta_{\lambda})_{\bar{T}}\wedge (\partial (\Phi_{1}-\Phi_0))_{T}\\
\quad & +(\partial\bar{\partial} \beta_{\lambda})_{T\bar{T}}(\Phi_1-\Phi_0)
\end{split}
\end{equation}

 and 

\begin{equation}\label{tomed2}
\begin{split}
(\partial\bar{\partial} \Phi_{\lambda})_{S\bar{T}}&= ((\omega)_{S\bar{T}})_{hol}+ ((\omega)_{S\bar{T}})_{antihol}+O(|S|^2)\\
\quad &+ (\partial \beta_{\lambda})_S\wedge (\bar{\partial} (\Phi_{1}-\Phi_0))_{\bar{T}}\\
\quad & -(\bar{\partial} \beta_{\lambda})_{\bar{T}}\wedge (\partial (\Phi_{1}-\Phi_0))_{S}\\
\quad & + (\partial\bar{\partial} \beta_{\lambda})_{S\bar{T}}(\Phi_1-\Phi_0)
\end{split}
\end{equation}

From the  definition $\beta_{\lambda} = \alpha (\lambda |S|)$  we obtain 

\begin{equation}\label{betlam0}
\partial \beta_{\lambda}= \lambda \alpha' (\lambda |S|)\partial |S|
\end{equation}

\begin{equation}\label{betlam}
\partial\bar{\partial} \beta_{\lambda}= \lambda^2   \alpha'' (\lambda |S|)\partial |S|\wedge \bar{\partial} |S|+ \lambda \alpha' (\lambda |S|)\partial\bar{\partial}|S|
\end{equation}

Since 

\begin{equation}\label{pares}
\bar{\partial}|S|=\bar{\partial} \sqrt{|S|^2}= \frac{\bar{\partial} |S|^2}{2\sqrt{|S|^2}}=\frac{S\bar{\partial} \bar{S}}{2|S|}
\end{equation}

 and 

\[
\partial \frac{1}{|S|}=\partial [ (|S|^2)]^{-1/2}= -\frac{1}{2}\partial (|S|^2) (|S|^2)^{-3/2}=-\frac{1}{2}\frac{\bar{S}\partial S}{|S|^{3}}
\]

 we have

\begin{equation}\label{betlam2}
\partial\bar{\partial}|S|=\frac{\partial S\wedge \bar{\partial} \bar{S}}{2|S|}-\frac{\partial S\wedge \bar{\partial}\bar{S}}{4|S|}
\end{equation}

from (\ref{betlam0}),  (\ref{betlam}) and (\ref{betlam2}) it follows that 
\begin{equation}\label{betlatt}
(\partial\beta_{\lambda})_T =0\hspace{1cm} (\partial\bar{\partial} \beta_{\lambda})_{T\bar{T}}= (\partial\bar{\partial} \beta_{\lambda})_{S\bar{T}}=0
\end{equation}

and from (\ref{fi1122}) (\ref{pares}) and (\ref{betlam0}) we get

\begin{equation}\label{betlatt2}
(\partial\beta_{\lambda})_{S}  (\bar{\partial} (\Phi_{1}-\Phi_0))_{\bar{T}}=\lambda O( |S|^2)
\end{equation}

Therefore from (\ref{tomed1}),  (\ref{tomed2}), (\ref{betlatt}) and  (\ref{betlatt2})   one concludes that
\begin{equation}\label{filamtt}
(\partial\bar{\partial} \Phi_{\lambda})_{T\bar{T}}= ((\omega)_{T\bar{T}})_{hol}+ ((\omega)_{T\bar{T}})_{antihol}+O(|S|^2)
\end{equation}

and
 
\begin{equation}\label{filamst}
(\partial\bar{\partial} \Phi_{\lambda})_{S\bar{T}}= ((\omega)_{S\bar{T}})_{hol}+ ((\omega)_{S\bar{T}})_{antihol}+O(|S|^2)+\lambda O(|S|^2)
\end{equation}

According to  (\ref{filamtt}) and  theorem (\ref{theoloc})  we find 
\[
(\partial\bar{\partial} \Phi_{\lambda})_{T\bar{T}}= \omega'_{T\bar{T}}+ O(|S|^2)
\]

and the relation  (\ref{filamst}) and  theorem (\ref{theoloc}) imply 

\[
(\partial\bar{\partial} \Phi_{\lambda})_{S\bar{T}}= \omega'_{S\bar{T}}+ O(|S|^2)+\lambda O(|S|^2)
\]
If we set $R:=\lambda |S|$,  from proposition (\ref{222}) we know that  in the limit $\lambda\rightarrow \infty$
\begin{equation}\label{filamss}
(\partial \bar{\partial}\Phi_{\lambda})_{S\bar{S}}\geq \omega' _{S\bar{S}}\hspace{1cm} \text{ over } \epsilon<R<\epsilon '
\end{equation}

Also the $(1,1)$-form 

\begin{equation}\label{omegatild}
\tilde{\omega}'=\omega'_{S\bar{S}}+ ((\omega ')_{S\bar{T}})_{hol}+ ((\omega')_{S\bar{T}})_{antihol}+((\omega')_{T\bar{T}})_{hol}+ ((\omega')_{T\bar{T}})_{antihol}
\end{equation}

is positive definite for $|S|$ small enough. This is because $\omega'$ is positive definite and $g'_{i\bar{j}}-\tilde{g}'_{i\bar{j}}=O(|S|^2)$ for $i\neq n$
 or $j\neq n$, and $g'_{n\bar{n}}= \tilde{g}'_{n\bar{n}}$ where $\tilde{g}'$ denotes the matrix associated with  $\tilde{\omega}'$ defined by (\ref{omegatild}) in canonical coordinates introduced in the appendix (\ref{app5}).
(One can use an inequality of the form $\sum_{i=1} ^{n-1}|z|^3 x_i x_n\leq x_n ^2  |z|^2+ \sum_{i=1} ^{n-1} |z| |x_i|^2$ to prove this)

Similarly   according to (\ref{filamtt}) (\ref{filamst}) (\ref{filamss})      if $\lambda\rightarrow \infty$  
 it can be shown that  $\partial \bar{\partial}\Phi_{\lambda}>0$ for $\epsilon<R<\epsilon'$.

  \end{proof}

\section{Schauder Estimates}\label{sec4}

Shcauder estimate as an essential tool in the proof of regularity of the solutions through continuity method has been developed  for cone metrics by S.Donaldson (\cite{d})  in order to prove the existence of K\"ahler -Einstein metrics for Fano manifolds. More  precisely given a smooth  anticanonical divisor $D$ in a complex Fano manifold $X$,  in order to study the existence of smooth K\"ahler-Einstein metrics on $X$, 
 Donaldson first studies K\"ahler metrics with cone singularities of cone angle $2\pi \beta$ transverse to $D$, where $0<\beta <1$ , and then he takes the limit when $\beta$ tends to $1$.  
Meanwhile the Schauder estimate provided by  (\cite{d}) is not appropriate in our case. First of all the degenerate K\"ahler metric we are considering correspond to cone metrics with $\beta>1$. In addition the type of third order estimates we have derived in section (\ref{III}) must be taken into account for exploring the Shauder  type inequality that is expected.

Here we  apply the method in (\cite{d})  to formulate  and prove the Shauder estimate for the case where the angle $\beta$ is bigger than $1$.

Let $g'_0$ denote the degenerate metric on $\mathbb{C}^n=\{(w_1,...,w_{n-1},z)\}\sim \mathbb{R}^{m}$, associated with  the degenerate K\"ahler form 
\begin{equation}\label{gflat}
\omega' _0=|z|^2dz\wedge d\bar{z}+ dw_1 \wedge d\bar{w}_1 +...+\wedge dw_{n-1} \wedge d\bar{w}_{n-1}
\end{equation}

where $m=2n$.
Consider the application 
\begin{equation}\label{pii}
\pi_2:\mathbb{C}^n \setminus  \mathbb{C}^{n-1} \times \{0\} \rightarrow \mathbb{C}^n \setminus  \mathbb{C}^{n-1}\times \{0\}  
\end{equation}

defined as
\[
\pi_2 (w_1,...,w_{n-1},z) \rightarrow (w_1,...,w_{n-1}, z|z|)
\]

If $z=\rho e^{i\theta}$ we have $z|z|= \rho ^2 e^{i\theta}$.  
  Then by setting    $r=\rho^2$ we write the  metric $g_0:=\pi_{2*} g'_0$ as 

\[
g_0=dr^2 + \beta^2 r^2 d\theta ^2+ \sum ds_i^2
\]

with $\beta=2$. Thus the  associated Laplace operator has the form

 \[
\Delta_{g_0}= 4 (\frac{\partial ^2 }{\partial r ^2} +\frac{1}{r}\frac{\partial }{\partial r} + \frac{1}{4 r^2 } \frac{\partial ^2}{\partial \theta ^2})+\Delta_w
\]

where $\Delta_w$ denotes the standard Euclidean Laplacian in $w$.
 

 Let
$H$ be the completion of $C^{\infty} _c$
under the Dirichlet norm $\|\nabla f\|_{L^2}$.  As is put forward  in (\cite{d}) for $q=2m/(m+2)$ and any $\psi\in L^q$ the linear map
\[
\begin{split}
&T_{\psi}  : H\rightarrow \mathbb{R},\\
\quad & T_{\psi}(f)= \int f\psi
\end{split}
\]

is bounded with respect to the $H$ norm, so  by Riesz representation theorem there is 
 a unique linear map $G:L^q \rightarrow H$  
 such that
\[
\int f\psi =\int (\nabla f , \nabla G\psi)_{g_0}
\]

 The following proposition is proved  in \cite{d}
\begin{prop} \cite{d}
There is a locally-integrable kernel function $G(x, y)$ such that

\[
G\psi (x) =\int G(x,y)\psi (y) dy
\]

for $\psi \in C^\infty _c$. The function $G(x, y)$ is smooth away from the diagonal and points
$x, y \in D$ where $D=\{0\}\times \mathbb{R}^{2n-2}$.

\end{prop}

\subsection{Laplacian around singular points of the metric}\label{lpls}
Consider a degenerate K\"ahler manifold $(X,D,\omega')$ as in section 3.1 and assume that $(w_1,...,w_{n-1},z)$ is a holomorphic coordinates system  near a point $p\in D$ where in this coordinates system we have $p=(0,...,0)$ and  $D$ is described by $\{z=0\}$.  Assume  that $\{ \frac{\partial}{\partial w_1},...,\frac{\partial}{\partial w_{n-1}}  \}$ forms an orthogonal frame w.r.t. $\omega' $ at $0$ and moreover the metric $j_D^* \omega '$ in this coordinates  is written like $\sum (\delta_{ij} +O(|w|^2))dw_i \wedge d \bar{w_j}$ where $j_D:D\rightarrow X$ denotes the inclusion map.
\\
In this coordinates system  the metric $\omega '$ is described as
\[
\begin{split}
\omega' = &(|z|^2+ O(3)) dz\wedge d\bar{z} + \sum (\delta_{ij} +  b_{ijn}z+ c_{ijn}\bar{z} +O(2)) dw_i \wedge d\bar{w}_j\\
\quad & + \sum O(2) dw_{i} \wedge  d\bar{z} +\sum  O(2)  dz\wedge d\bar{w}_{i}
\end{split}
\]
We then perform a  holomorphic change of coordinates
\[
w_i \rightarrow w_i +  \sum_{1\leq i,j\leq n-1} b_{ijn} w_j z
\]
 to eliminate the first order terms $b_{ijn} z+ c_{ijn}\bar{z}$ in the above expansion of  $\omega'$. In fact  under this change of coordinates $dw_i$, for $i=1,...,n-1$ turns into

\[
dw_i \rightarrow dw_i + \sum_j b_{ijn}w_j dz + \sum_m  b_{imn} z dw_m
\]

so $dw_k\wedge d\bar{w}_l$ for $1\leq k,l\leq n-1$ transforms as 
 
\[
\begin{split}
dw_k \wedge d\bar{w}_l\rightarrow & dw_k \wedge d\bar{w}_l +\sum b_{kjn}w_j dz \wedge d\bar{w}_l
 + \overline{(\sum b_{ljn} w_j )} dw_k \wedge d\bar{z}\\
& + \sum_m \overline{( b_{lmn}z)} dw_k \wedge d\bar{w}_m+ \sum_m (  b_{kmn}z)dw_m  \wedge d\bar{w}_l+O(2)
\end{split}
\]
 and we obaine
\[
\begin{split}
\sum(1+ a_{kk}z +\bar{a}_{kk} \bar{z} )dw_k\wedge d\bar{w}_k &\rightarrow \sum(1+ a_{kk}z +\bar{a}_{kk} \bar{z} +  
b_{kkn}z+ \bar{b}_{kkn} \bar{z})dw_k\wedge d\bar{w}_k \\
\quad & +\sum O(|w|) dz\wedge d\bar{w_k} +\sum O(|w|)dw_k \wedge d\bar{z}\\
\quad &  +\sum_{m\neq k} O(|z|)dw_k\wedge d\bar{w}_m+\sum_{m\neq k} O(|z|) dw_m\wedge d\bar{w}_k+O(2)
\end{split}
\]

thus  if we set $b_{kk1}=-a_{kk}$ we can prove the following lemma:
\begin{lemma}\label{eslo}
In an appropriate holomorphic  coordinates system  we can assume that

\[
\begin{matrix}
g'_{z\bar{z}}= |z|^2 +O(3) & g'_{w_{\bar{i}}w_i}=1+O(2) & g'_{w_i w_{\bar{j}}}= O(|z|) &  g'_{z,w_{\bar{i}} }= O(|w|)
\end{matrix}
\] 
\end{lemma}

Now  with the above  local model as before we compute the highest order terms of the laplacian of the degenerate metric around some point $(0,0)$ on $D$. 

\[
\Delta_{g'} \phi= |\det ({g'_{i\bar{j}}) }|^{-1 } \frac{\partial}{\partial z_i} (g'^{i\bar{j}}|\det(g'_{i\bar{j}})| \frac{\partial \phi}{\partial z_{\bar{j}}} ) 
\]

\[
\begin{split}
\Delta_{g'}& =\frac{1}{|\det(g')|} (\frac{\partial}{\partial z}[ |\det (g')|(g'^{z\bar{z}} \frac{\partial }{\partial \bar{z}} +\sum g^{z\bar{w_i}} \frac{\partial}{\partial \bar{w_i}})] +\sum \frac{\partial}{\partial w_i}[|\det(g')| (g'^{w_i\bar{z}} \frac{\partial}{\partial \bar{z}} +g'^{w_i\bar{w_i}} \frac{\partial}{\partial \bar{w_i}})])\\
\end{split}
\]

From lemma (\ref{eslo}) it follows that 
\[
\det (g') g'^{z\bar{z}} =\det (g'_{w_i \bar{w_j}}) =1+O(2), \hspace{1cm} \det(g')g'^{z\bar{w_i}}=O(2)
\]

Hence we obtain

\[
\frac{\partial}{\partial z}\det(g')g'^{z\bar{z}}=\frac{\partial}{\partial z}\det(g')g'^{z\bar{w_i}}=O(1)
\]

\[
\frac{\partial}{\partial w_i}\det(g')g'^{z\bar{w_i}}=O(1)
\]
So 
\begin{equation}\label{lpl}
|z|^2 \Delta_{g'} = |z|^2 \Delta_w + \Delta _z + O(1)
\end{equation}
 
Now we consider the operator  $\frac{1}{|z|^2} \Delta _z + \Delta _w$. Let $(R, \theta)$ denote the polar coordinates in $z$ space. Then we can write
\[
\frac{1}{|z|^2} \Delta _z = \frac{1}{R^2} (\frac{\partial ^2}{\partial R^2} +\frac{1}{R} \frac{\partial}{\partial R}+ \frac{1}{R^2}\frac{\partial ^2}{\partial \theta ^2})
\]
\\

 after a change of coordinates $R=\sqrt{r}$ the oparator $\frac{1}{|z|^2} \Delta _z$  turnes into 
\[
\frac{1}{|z|^2} \Delta _z = 4 (\frac{\partial ^2 }{\partial r^2} +\frac{1}{r}\frac{\partial }{\partial r} + \frac{1}{4 r^2 } \frac{\partial ^2}{\partial \theta ^2})
\]

the term in the parenthesis is nothing but the laplace operator associated to the metric $dr^2 +\beta ^2 r^2 d\theta$ with $\beta =2$.
\\

\textbf{Schauder theory for  conical metrics with $\beta >1$}.
 As we mentioned earlier Shauder theory for conical metrics of the type  (\ref{metric})
 has been developed by Donaldson in \cite{d} with emphasises on the case where  $\beta <1$.  
We would like to treat here the case where $\beta >1$  which corresponds to metrics with branched
 type singularities.

As before let  $D= \mathbb{R}^{m-2}\times\{0\}\subset \mathbb{R}^m = \mathbb{R}^{m-2}\times\mathbb{R}^2 $. Consider polar coordinates $r$, $\theta$ on $\mathbb{R}^2$ and standard coordinates $(s_1,...,s_{m-2})$ on $\mathbb{R}^{m-2}$. Let $\pi_2 :\mathbb{R}^m\rightarrow \mathbb{R}^m$ be defined by 
\[
\pi _2:\mathbb{R}^2\setminus \{0\}\times \mathbb{R}^{m-2} \rightarrow \mathbb{R}^2\setminus \{0\}\times \mathbb{R}^{m-2}
\]
\begin{equation}\label{p2}
\pi_2 (s_1,...,s_{m-2},R\cos \theta, R\sin \theta  )=( s_1,...,s_{m-2}, r\cos \theta, r\sin \theta )
\end{equation}

where $r=R^2$,
 then the push forward $(\pi_2 ) _* g'$ of the degenerate metric $g'$ under $\pi_2$ is a singular metric  of the type

\begin{equation}\label{metric}
g=d\rho ^2 +\beta ^2 r^2 d\theta ^2 +\sum ds_i ^2
\end{equation}

where here  $\beta>1$.   The  Green function associated to  the operator 
$\Delta = \Delta _{\beta}+ \Delta _{\mathbb{R} ^{m-2}} $ is calculated in (\cite{d})   under the form :

\begin{equation}\label{green}
G(r,\theta , s; r',\theta' , s')=\sum G_k (r,r', \mathfrak{R})\cos k(\theta -\theta ' )
\end{equation}

 where $\mathfrak{R}=|s-s'|$. Two representations for $G_{k} (r,r ',\mathfrak{R})=\frac{1}{2\pi ^m} \mathfrak{R}^{2-m/2} g_k$ are considered:
\\

1)

\begin{equation}\label{10}
g_k=2\int_0 ^{\infty} \lambda ^{m/2 -2} K_{m/2-2} (\mathfrak{R}\lambda) J_{\nu} (r\lambda ) J_{\nu}(r'\lambda)d\lambda
\end{equation}

for which the corresponding expansion $G=(\frac{2}{r'})^{m-3} \sum a_{j,k} (\frac{2\mathfrak{R}}{r'}) (\frac{r}{2r'})^{\nu+2j} \cos k(\theta-\theta')$ is shown to be absolutely convergent for $\mathfrak{R}>0$ and $r,r'<\mathfrak{R}/2$.
\\

2)

\begin{equation}\label{11}
g_k= 2\int_0 ^{\infty} \lambda^{m/2-2} J_{m/2 -2}(\mathfrak{R}\lambda) K_{\nu} (r'\lambda) I_{\nu} (r\lambda)d\lambda
\end{equation}

where  the corresponding  series
\begin{equation}\label{green2}
G=G(r,\theta ,s ; r',\theta ' , s' )=\sum_{k\geq 0} b_{j,j',k} (\mathfrak{R}) r^{\nu + 2j} (r') ^{\nu + 2j'} cos k(\theta -\theta ') 
\end{equation}
is absolutely convergent for $r<r'/2$ and for any $\mathfrak{R}\geq 0$.
\\

( See (\cite{d}) for the definition of $a_{j,k}$ and $b_{j,j',k}$). We  set
\[
\tilde{G}(x,y):=(\pi_2\times \pi_2 ) ^* G(x,y)=G\circ (\pi_2\times \pi_2)(x,y)
\]
where $\pi_2$ is defined by (\ref{p2})
and we define
\\

\begin{equation}\label{rtg1}
\tilde{\partial}_i :=
\begin{cases}
R\frac{\partial}{\partial s_i}& \text{ for } i\leq m-2\\
\frac{\partial}{\partial R} & \text{ for } i=m-1\\
\frac{1}{R}\frac{\partial}{\partial \theta}&\text{ for } i=m
\end{cases}
\end{equation}

For a real positive $\gamma$ we also set
\begin{equation}\label{rtg2}
\tilde{\partial}_{i,\gamma}=\begin{cases}\frac{1}{R^{\gamma}}\tilde{\partial}_i,&  i=1,...,m-2\\
\tilde{\partial}_i & i=m-1,m
\end{cases}
\end{equation}


Define $H'$ as the completion of $C^{\infty}_c$ with respect to the  Dirichlet norm $\|\nabla_{g'} f\|_{g'}$  where 
$\nabla _{g'}$ denotes the gradient  of $f$ with respect to the degenerate metric $g'$  and  $\|\|_{g'}$ denotes the $L^2$ norm with respect to $g'$.  We have in fact
\begin{equation}\label{hh}
H'=\pi_2 ^* H
\end{equation}

where $H$ is the Hilbert space defined for the cone metric $g$ for $\beta=2$ as in (\cite{d}). This is because $\pi_2 ^* C_{c} ^{\infty}\subset C^{\infty} _c$.  
 

 For  the space of $L^q$
functions with respect to the weighted volume form  $dvol_{g'}$, denoted by  $L'^q$,   we also have:
\begin{equation}\label{ll}
L'^q=\pi_2 ^* L^q
\end{equation}

(ordinary $L^q$). 
Thus if for  $\rho\in L'^q$ we consider  the  linear map $H'\rightarrow \mathbb{R}$
\begin{equation}\label{hl}
f\rightarrow \int f\rho\, d vol_{g'} 
\end{equation}
then due to   (\ref{ll}) and (\ref{hh}) and the identity $\int f\rho\, d vol_{g'}=\int\underline{f}\underline{\rho}\,  d vol_g$
where $\underline{f}={\pi_2} _* f\in H$ and $\underline{\rho}={\pi_2 }_* \rho\in L^q$   we find that (\ref{hl}) is a bounded map.  Therefore the Green function $\tilde{G}:L ^{'q}\rightarrow H'$ can be defined by the relation
\begin{equation}\label{frogp}
\int f\rho dvol_{g'}=\int\langle \nabla_{g'} f, \nabla_{g'} \tilde{G}\rho \rangle_{g'},\hspace{1cm} \text{ for all }  \rho\in C^{\infty} _c
\end{equation}

 It also follows that $\phi:=\tilde{G} \rho$ is a weak solution of the equation $\Delta_{g'}\phi =\rho$, where "weak sense" is  with respect to the weighted norms and  weighted Sobolev spaces. ($\int \phi\Delta_{g'} \eta\, dvol_{g'} =\int\rho\eta\,dvol_{g'}$ for any test fucntion $\eta$.)

The kernel function $\tilde{G}(x,y)=\pi_2 ^* G(x,y)$  satisfies also the following relation
\[
\tilde{G}\rho=\int\tilde{G}(x,y)\rho (y) dvol_{g'} (y) \text{ for all }  \rho\in C^{\infty} _c
\]

To see this we note that  the relation
\begin{equation}\label{groul}
G\underline{\rho}=\int G(\underline{x},\underline{y})\underline{\rho} (\underline{y}) dvol_{g} (\underline{y})
\end{equation}

holds for $\underline{\rho}\in (\pi_2)_* (C^{\infty}_c )  $ with the same arguments as in  (\cite{d}).   (We note that $\nabla_g  (\sqrt{r})\in L^2 $ and the relation (\ref{groul}) holds also for $\rho$ with compact support and smooth as a function of $\sqrt{r}$).

Consider the transformation $a_{\lambda}$ defined by
\begin{equation}\label{alambda}
(w_1,...,w_{n-1},z)\xrightarrow{a_{\lambda}} (\lambda w_1,...,\lambda w_{n-1}, \sqrt{\lambda}z)
\end{equation}

If

\[
|z|^2\Delta_{g'}\phi=|z|^2\rho
\]

then we have

\begin{equation}\label{lnd1}
(\Delta_{z}+|z|^2\Delta_{w})\phi( \lambda w, \sqrt{\lambda}z)=\lambda \left[ (\Delta _z \phi )+|\sqrt{\lambda}z|^2 (\Delta_{w} \phi) \right] ( \lambda w,\sqrt{\lambda}z)=\lambda \left( |\sqrt{\lambda}z|^2\rho ( \lambda w,\sqrt{\lambda} z) \right)
\end{equation}


For two given  real positive numbers  $\gamma_1, \gamma_2\in \mathbb{R}^+$ we set
\[
D_{\gamma_1,\gamma_2}=\tilde{\partial}_{i,\gamma_1}\tilde{\partial}_{j,\gamma_2}
\]

where  $\tilde{\partial}_{i,\gamma}$ is defined by (\ref{rtg2}) and we define,
\begin{equation}\label{defitp}
T'_{\gamma_1, \gamma_2}=D_{\gamma_1,\gamma_2}\circ\tilde{G}
\end{equation}

where $D_{\gamma_1,\gamma_2}=\tilde{\partial}_{i,\gamma_1}\tilde{\partial}_{j,\gamma_2}$ is the derivation with respect to the variable $x$ in $\tilde{G}(x,y)$.
 
 Likewise for $f:\mathbb{R}^m\rightarrow \mathbb{R}$,   the map $f_{\lambda}$   is defined by,

\begin{equation}\label{dilation}
f_{\lambda}:=f\circ a_{\lambda^{-1}}  
\end{equation}

where $a_{\lambda}$ is given by  (\ref{alambda}).
 Then the above relation (\ref{lnd1}) can be restated in the form 
\[
\Delta_{g'} (\phi_{\lambda})=\lambda^{-1}( |z|^2 \rho)_{\lambda}=\lambda^{-2}|z|^2 (\rho)_{\lambda}
\]

Consequently we deduce that,

\begin{equation}\label{giro}
(\tilde{G}\rho)_{\lambda}=\lambda^{-2}\tilde{G}(\rho_{\lambda})
\end{equation}

From the relations

\begin{equation}\label{dtg1}
R^{-\gamma} \frac{\partial }{\partial R} f_{\lambda}= \frac{R^{-\gamma}}{\sqrt{\lambda}}(\frac{\partial f}{\partial R})_{\lambda}=\lambda^{-\frac{\gamma}{2}-\frac{1}{2}} (R^{-\gamma} \frac{\partial f}{\partial R})_{\lambda}
\end{equation}

\begin{equation}\label{dtg2}
R^{-\gamma+1} \frac{\partial }{\partial s_i} f_{\lambda} = \frac{R^{-\gamma+1}}{\lambda} (\frac{\partial f}{\partial s_i})_{\lambda}=\lambda^{-\frac{1}{2}-\frac{\gamma}{2}} (R^{-\gamma+1}\frac{\partial f}{\partial s_i})_{\lambda}
\end{equation}

\begin{equation}\label{dtg3}
R^{-\gamma-1}\frac{\partial}{\partial \theta}f_{\lambda}=R^{-\gamma-1} (\frac{\partial f}{\partial \theta})_{\lambda}=\lambda^{-\frac{1}{2}-\frac{\gamma}{2}} (R^{-\gamma-1}\frac{\partial f}{\partial \theta})_{\lambda}
\end{equation}
 
we can conclude that  
 
\begin{equation}
\tilde{\partial}_{i,\gamma} (f_{\lambda})=\lambda^{-\frac{1}{2}-\frac{\gamma}{2}} (\tilde{\partial}_{i, \gamma} f)_{\lambda}
\end{equation}
 
Therefore by setting $f:=\tilde{G}\rho$ we obtain
\begin{equation}\label{dtg4}
\tilde{\partial}_{i,\gamma}( (\tilde{G}\rho)_{\lambda})= \lambda^{-\frac{1}{2}-\frac{\gamma}{2}}(\tilde{\partial}_{i,\gamma} \tilde{G} \rho)_{\lambda}
\end{equation}

also from (\ref{giro})
\begin{equation}\label{dtg5}
\tilde{\partial}_{i,\gamma}( (\tilde{G}\rho)_{\lambda})=\lambda^{-2} \tilde{\partial}_{i,\gamma}\tilde{G}(\rho_{\lambda})
\end{equation}

thus we get 

\begin{equation}\label{dgro1}
(\tilde{\partial}_{i,\gamma} \tilde{G} \rho)_{\lambda}= \lambda^{-2+\frac{1}{2}+\frac{\gamma}{2}} \tilde{\partial}_{i,\gamma}\tilde{G}(\rho_{\lambda})
\end{equation}
 Applying (\ref{dtg4}) and (\ref{dtg3}) we see that
\[
\begin{split}
T'_{\gamma_1,\gamma_2} (\rho_{\lambda})&=\tilde{\partial}_{i,\gamma_1}\tilde{\partial}_{j,\gamma_2} \tilde{G} (\rho_{\lambda})=  \lambda^{2-\frac{1}{2}-\frac{\gamma_2}{2}}\tilde{\partial}_{i,\gamma_1}(\tilde{\partial}_{j,\gamma_2} \tilde{G} (\rho))_{\lambda}\\
\quad & 
=\lambda^{1-\frac{\gamma_1+\gamma_2}{2}}(\tilde{\partial}_{i,\gamma_1}\tilde{\partial}_{j,\gamma_2} \tilde{G} (\rho))_{\lambda}
\end{split}
\]

So

\begin{equation}\label{homg1}
(T' _{\gamma_1,\gamma_2}\rho)_{\lambda}=\lambda^{\frac{\gamma_1+\gamma_2}{2}-1} T'_{\gamma_1,\gamma_2} (\rho_{\lambda})
\end{equation}

 if 
\begin{equation}\label{pidef}
\pi:\mathbb{R}^2\times \mathbb{R}^{m-2}\rightarrow \mathbb{R}^2
\end{equation}

denotes the projection map, then from the definition of $T'_{\gamma_1, \gamma_2}$ in  (\ref{defitp})

\begin{equation}\label{dtg7}
\begin{split}
T' _{\gamma_1, \gamma_2}\rho_{\lambda}&=\int \tilde{\partial}_{i,\gamma_1} \tilde{\partial}_{i,\gamma_2}\tilde{G}(x,y)\rho_{\lambda} (y) |\pi (y)|^2 dy\\
\quad & =\lambda^n\int \tilde{\partial}_{i,\gamma_1} \tilde{\partial}_{i,\gamma_2}\tilde{G}(x,a_{\lambda} (y))\rho(y)  |\pi (y)|^2dy\\
\end{split}
\end{equation}
 
and 

\begin{equation}\label{dtg8}
(T' _{\gamma_1, \gamma_2} \rho )_{\lambda}=\int  \tilde{\partial}_{i,\gamma_1} \tilde{\partial}_{i,\gamma_2}\tilde{G} (a_{\lambda^{-1}}(x), y)\rho (y) |\pi (y)|^2dy
\end{equation}
thus  from (\ref{homg1}) (\ref{dtg7}) and (\ref{dtg8}) we obtain
\begin{equation}\label{hamgeni}
 \tilde{\partial}_{i,\gamma_1} \tilde{\partial}_{i,\gamma_2}\tilde{G} (a_{\lambda}(  x), a_{\lambda} (y))=\lambda^{-n+1-\frac{\gamma_1+\gamma_2}{2}} \tilde{\partial}_{i,\gamma_1} \tilde{\partial}_{i,\gamma_2}\tilde{G} (x ,  y)
\end{equation}

The  properties I-IV in the lemma (\ref{schauderpr1}) below,   follow by the same line of arguments as in   proposition 3 in (\cite{d}).
We recall that the derivatives  $\tilde{\partial}$  and
the derivative $\nabla$ in part IV  are taken with respect to the first variable in $\tilde{G}$. Also note that $\pi_{2*}\frac{\partial}{\partial R}=2\sqrt{r}\frac{\partial}{\partial r} $ and in polyhomogeneous expansion (\ref{green2}) $r$ and $r'$  appear with powers of the form
$\frac{1}{2}+2j$ for $j\in \mathbb{N}\cup\{0\}$. Therefore in  the corresponding expansion of $\tilde{G}$ only integer powers of $R$ can occur. If we assume that $\gamma_1, \gamma_2\in \{0,\frac{1}{2}, 1\}$ then  in the polyhomogenous expansion of 
$ \tilde{\partial}_{i,\gamma_1}\tilde{\partial}_{j,\gamma_2}\tilde{G}$  powers of the form $R^{\frac{1}{2}}$  occur only  if  $i\leq m-2$ and $\gamma_1=\frac{1}{2}$ and $j\in\{m-1, m\}$, or  if $j\leq m-2$ and $\gamma_2=\frac{1}{2}$ and $i\in\{m-1, m\}$. This proves the second part of the following lemma with the same argument as in (\cite{d}).

\begin{lemma}\label{schauderpr1}
 
For $ \gamma_1, \gamma_2\in \{0,\frac{1}{2}, 1\}$
\\

I)

\[
\hspace{1cm} |\tilde{\partial}_{i,\gamma_1}\tilde{\partial}_{j,\gamma_2}\tilde{G} (0, \zeta)|\leq \kappa \hspace{1cm}  \text{ for } \|\zeta\|_{g'}=1 
\]
\\

\hspace{3.5cm}

 II) \hspace{1cm}  If $\|\zeta \|_{g'}=1$ then  If $\|\zeta \|_{g'}=1$ then 

\[
| \tilde{\partial}_{i,\gamma_1}\tilde{\partial}_{j,\gamma_2}\tilde{G}(w_1,\zeta)- \tilde{\partial}_{i,\gamma_2}\tilde{\partial}_{j,\gamma_1}\tilde{G}(w_2,\zeta)   |\leq \kappa_2 \|w_1-w_2\|_{g'} ^{1/2}
\]
 for any $w_1$, $w_2$ with $\|w_i\|_{g'} \leq 1/2$
\\

III)
If $\|\zeta \|_{g'}=1$ and $\|\pi (\zeta )\|_{g'}\geq  1/2$ then  
\[
|\tilde{\partial}_{i,\gamma_1}\tilde{\partial}_{j,\gamma_2} G(w,\zeta ) |\leq \kappa_3 \|\zeta-w\|_{g'} ^{-n+1-\frac{\gamma_1+\gamma_2}{2}}
\]
for $\|w\|_{g'}\leq 5$. ($\pi$ is defined  by (\ref{pidef}).
\\

IV) If $\|\zeta \|_{g'}=1$ and $\|\pi(\zeta )\|_{g'} \geq 1/2$ then 
\[
|\nabla  \tilde{\partial}_{i,\gamma_1}\tilde{\partial}_{j,\gamma_2}\tilde{ G}(w,\zeta )|\leq \kappa _3 \|\zeta-w \|_{g'}^{-n+\frac{\gamma_1+\gamma_2}{2}}
\]

for $\|w\|_{g'} \leq 5$.
\\
\end{lemma}

 \begin{lemma}\label{schauderpr}

i) For     $B_1= \{\|(y_1,..,y_{n-2},x_1,x_2)\|_{g'}<1\}$, if $n+k+c-1>0$, we have
\[
I_1=\int_{B_1} ((x_1^2+x_2^2 )^2+y_1^2+...+y_{n-2} ^2)^{k/2} (x_1^2 +x_2 ^2)^c     dx_1dx_2dy_1...dy_{n-2}<+\infty
\]
\\

ii) For  $B_2= \{\|(y_1,..,y_{n-2},x_1,x_2)\|_{g'}>1\}$ if $n+k+c-1<0$ we have
\[
I_1=\int_{B_2} ((x_1^2+x_2^2 )^2+y_1^2+...+y_{n-2} ^2)^{k/2}  (x_1^2 +x_2 ^2)^c     dx_1dx_2dy_1...dy_{n-2}<+\infty
\]

\end{lemma}

\emph{Proof:}  Taking polar coordinates  for the variables $(x_1,x_2)$ with radial component denoted by $r_1$ and hyper spherical coordinates for the variables $(y_1,...,y_{n-2})$  with radial components  denoted by $r_2$ we can write

\[
\begin{split}
I_1=V_0 \int_{r_1^4+r_2^2\leq 1} (r_1^4+r_2 ^2)^{k/2}r_1 ^{2c+1}dr_1 r_2^{n-3}dr_2
\end{split}
\]
where $V_0$  is a constant which depends on $n$.

by setting $r_1 ^2 =u$  and $r_2=v$ we get to

\[
I_1=\frac{V_0}{2} \int_{r_1^4+r_2^2\leq 1}  (u^2+v^2)^{k/2}u^c v^{n-3}du dv
\]

Then if we take $u=r\cos \theta$ and $v=r\sin \theta$ we obtain

\[
I_1=\frac{V_0}{2}\int_{r_1^4+r_2^2\leq 1}  r^{n+k+c-2} (\cos \theta)^c (\sin\theta )^{n-3} drd\theta
\]

which is convergent if $n+k+c-1>0$.

Part (ii) follows similarly.
$\square$
\\


From now we assume that $T'$
$T'_{\gamma_1, \gamma_2}=\tilde{\partial}_{i,\gamma_1}\tilde{\partial}_{j,\gamma_2}$
has one of the following forms: 

\begin{enumerate}
\item  $T'=\tilde{\partial}_{i,\gamma_1}\tilde{\partial}_{j,\gamma_2} \tilde{G}$  for  $1\leq i,j\leq m-2$
\item  $T'=\tilde{\partial}_{i,\gamma_1}\tilde{\partial}_{j,\gamma_2} \tilde{G}$  for $j=m-1, m-2$ and $1\leq i \leq m-2$
\item  $T'= \frac{\partial ^2 }{\partial z\partial {\bar{z}}} \tilde{G}$
\end{enumerate}

\begin{theo}\label{main3}
For $0<\alpha <1/2$ and for any $0\leq \gamma_1, \gamma_2\leq 1$ there exists a constant $C$ which depends on $\alpha, m, \gamma_1 $ and $\gamma_2$ such that for all the functions $\rho$ such that  $|\pi|^b\rho\in C^{\infty} _c (\mathbb{R}^m)$ we have
\[
[T'  \rho]_{\alpha}\leq C[(|\pi|^b \rho )]_{\alpha} 
\]
where $b=2-(\gamma_1+\gamma_2)$. We recall that $\pi: \mathbb{R}^{m}\rightarrow \mathbb{R}^2$ is defined as
$\pi (w_1,...,w_{n-1},z)=z$ written in holomorphic coordinates $(w_1,...,w_{n-1},z)$
on $\mathbb{R}^m =\mathbb{C}^n$.

\end{theo}




  \begin{proof}   The equality
\[
\frac{|(T'\rho)_{\lambda}(x_1)- (T'\rho)_{\lambda}(x_2)|}{\|a_{\lambda^{-1}}(x_1)-a_{\lambda^{-1}}(x_2)\|_{g'}^{\alpha}}=\lambda^{\frac{\gamma_1+\gamma_2}{2}+\alpha-1} \frac{|T'(\rho_{\lambda})(x_1)-T'(\rho_{\lambda})(x_2)|}{\|x_1-x_2\|^{\alpha}}
\]
 shows that

\begin{equation}\label{tprimro}
[T'\rho_{\lambda}]_{\alpha}= \lambda^{\frac{b}{2}-\alpha}[T'\rho]_{\alpha}
\end{equation}

where  $b=2-\gamma_1-\gamma_2$.
Since we also have

\begin{equation}\label{landalf}
 [|z|^b \rho_{\lambda}]_{\alpha}=\lambda^{\frac{b}{2}-\alpha} [|z|^b \rho]_{\alpha}
\end{equation}

Using the above inequalities and linearity of $T'$ 
the inequality $[T'\rho]_{\alpha}\leq C [|z|^b\rho]_{\alpha}$ is proved  for $\{\rho| |z|^b\rho\in C^{\alpha}\}$   if we can show that there exists a constant $C$ for which 
\[
|T'\rho (x_1)-T' \rho (x_2)|\leq C
\]

holds for any  $\rho\in C^{\infty}_c$  satisfying $[|z|^b\rho]_{\alpha}=1$ and for any  $x_1,x_2\in \mathbb{R}^m$
such that $\|x_1-x_2\|_{g'}=1$.
 With similar argument as in (\cite{d}) we can assume that $x_1=0$ and $\|x_2\|_{g' } =1$

We consider a  smooth function $\psi$ supported in the unit ball 
(with respect to $g'$)  such that $\Delta_{g'}\psi =1$ on the $\delta$-ball for some fixed 
$\delta>0$. The existence of $\psi$ follows from the theory of Grushin and Vishik for degenerate elliptic operators.  We then set $\chi=\Delta_{g'} \psi$ so $\chi$ has compact support,
is equal  to 1 on the $\delta$-ball and $T' \chi =D_{\gamma_1, \gamma_2}\psi$. Taking
\begin{equation}\label{landaa}
\lambda=\max \{\rho (x_2)^{\frac{1}{\alpha}} , \delta^{-1}\}
\end{equation}

and
\[
\sigma_0=\rho (x_2) \chi_{\lambda}
\]

\[
T'\chi_{\lambda}=\lambda^{b/2}(T' \chi )_{\lambda}=\lambda^{b/2}(D_{\gamma_1, \gamma_2}\psi)_{\lambda}
\]

According to the definition of $T'$ we have $ (T' \chi)_{\lambda}=(R^{b}\frac{\partial ^2 }{\partial s_i \partial s_j} \psi)_{\lambda}$
, $(R^{b/2} \frac{\partial ^2 }{\partial s_i \partial R} \psi)_{\lambda}$ or $(\frac{\partial ^2}{\partial z\partial {\bar{z}}} \psi)_{\lambda} $.
Also according to the equation $\chi = \Delta_{g'} \psi$  we have  
$\frac{\partial ^2}{\partial z\partial {\bar{z}}} \psi=|z|^2\chi -|z|^2 \sum_{i=1} ^{n-1} \frac{\partial ^2 }{\partial w_i \partial \bar{w}_i}\psi$.
Therefore it follows that  
\[
[T'\chi_{\lambda}]_{\alpha}=\lambda^{-\alpha} [\tilde{\psi}]_{\alpha}
\] 

for a  $C^{\alpha}$  map   $\tilde{\psi}$  which only depends on $\psi$. We also have 

\[
[|z|^b \sigma_0]_{\alpha}=\rho(x_2) [|z|^b \chi_{\lambda}]_{\alpha}\leq C\rho(x_2) ( [|z|^b]_{\alpha} +[\chi]_{\alpha})\leq \lambda^{-\alpha}\rho(x_2) C'
\]

for some constant $C'$  which only depends on $\chi$.

 The above relations  ensure that  for $\lambda $ large enough we have  
\begin{equation}\label{150}
[T'\sigma_0]_{\alpha}\leq [T' \chi]_{\alpha} 
\end{equation}.

 Similar properties hold for $\sigma_1:= (\rho (0)-\rho (x_2))\chi$.
Thus if we define $\tilde{\rho}:=\rho-\sigma_0-\sigma_1$ then 
\[
\tilde{\rho}(0)= \rho (0) - \rho(x_2)\chi_{\lambda} (0)- (\rho(0)- \rho(x_2))\chi(0)=0
\]
since we have $\chi(0)=\chi_{\lambda}(0) =1$. We also have 
\[
\tilde{\rho}(x_2)= \rho (x_2) - \rho(x_2)\chi_{\lambda} (x_2)- (\rho(0)- \rho(x_2))\chi(x_2)=0
\]
since  according to (\ref{landaa})  we have $\chi_{\lambda} (x_2) =1$ and $\chi(x_2) =0$. 
Hence $\tilde{\rho}$ vanishes at 
$0$ and $x_2$.
 From  (\ref{landaa})  we can  also deduce that
\begin{equation}\label{152}
[|z|^b \sigma_0]_{\alpha}=\rho(x_2) [|z|^b \chi_{\lambda}]_{\alpha}\leq C\rho(x_2) ( [|z|^b]_{\alpha} +[\chi]_{\alpha})\leq \lambda^{-\alpha}\rho(x_2) C'\leq C'
\end{equation}
Here   $C$ and $C'$ are constants   which only depend on $\chi$.
 From (\ref{150}) and (\ref{152}) weconclude that

\[
[|z|^b \tilde{\rho} ]_{\alpha}\leq [|z|^b\rho ]_{\alpha}+2[|z|^b\chi]_{\alpha}
\]
\[
|T'\rho (x_2) -T'\rho (0)|\leq |T'\tilde{\rho}(x_2) -T'\tilde{\rho} (0)|+2[T'\chi]_{\alpha}
\]
Consequently in the proof of theorem (\ref{main3}) it is possible to assume  $\rho$ vanishes at $x_2$ and $0$.
\\ 
We consider two cases A and B corresponding to $d\leq 2$ and  $d>2$, respectively,
where $d=\min \{\ \|\pi (x_1)\|_{g'}, \|\pi (x_2)\|_{g'} \}$.
\\

\textbf{Case A}
If $x' _1$ and $x'_ 2$ denote the projections of respectively $x_1$ and $x_2$ on $D$  in the case A we have
$
\|x_1-x'_1\|_{g'}, \|x'_1 -x'_2\|_{g'},\|x_2-x'_2\|_{g'}\leq 3
$. Thus as in (\cite{d}) we are reduced to one of the following two subcases
\\

\textbf{Sub-case A1} $x_1=0$ ,$\|x_2\|_{g'}=1$, $x_2\in D$
\\

\textbf{Sub-case A2} $x_1=0$, $\|x_2\|_{g'} =1$, $x'_2=0$
\\ 

In order to estimate 
\[
\int  D_{\gamma_1,\gamma_2}\tilde{G} (0,y) \rho (y) \pi(y)^2 dy  -  \int D_{\gamma_1, \gamma_2}\tilde{G}(x_2,y) \rho (y) \pi(y)^2 dy
\]

we consider the contribution from the  two regions  $\|y\|_{g'}\geq 2$ and $\|y\|_{g'}\leq 2$ separately.
Then we  have
\[
\begin{split}
|\int D_{\gamma_1,\gamma_2}\tilde{G} (0,y) \rho (y) |\pi(y)|^2 dy - \int D\tilde{G}(x_2,y) \rho (y) |\pi(y)|^2 dy |\leq I'_1+I'_2
\end{split}
\]
where
\[
I'_1=|\int_{\|y\|_{g'}>2} D_{\gamma_1,\gamma_2}\tilde{G} (0,y) \rho (y) |\pi(y)|^2 dy - \int_{\|y\|_{g'}>2} D_{\gamma_1,\gamma_2}\tilde{G}(x_2,y) \rho (y) |\pi(y)|^2 dy |
\]

and

\[
I'_2= |\int_{\|y\|_{g'}<2} D_{\gamma_1,\gamma_2}\tilde{G} (0,y) \rho (y) |\pi(y)|^2 dy - \int_{\|y\|_{g'}<2} D_{\gamma_1,\gamma_2}\tilde{G}(x_2,y) \rho (y) |\pi(y)|^2 dy |
\]

In the case A by the same argument as in (\cite{d}) we can assume that if $\|y\|_{g'}\geq 2$
then    from  relation (\ref{hamgeni})   for $D_{\gamma_1, \gamma_2}\tilde{G}$  
\[
|D_{\gamma_1,\gamma_2}\tilde{G} (x_2,y) -D_{\gamma_1,\gamma_2}\tilde{G} (0,y)|=\|y\|^{-n+1-\frac{\gamma_1+\gamma_2}{2}} _{g'} (|D_{\gamma_1,\gamma_2}\tilde{G} (\frac{x_2}{\|y\|_{g'}},\frac{y}{\|y\|_{g'}}) -D_{\gamma_1,\gamma_2}\tilde{G} (0,\frac{y}{\|y\|_{g'}})|)
\]
hence  by properties (II)  of lemma (\ref{schauderpr1}) we get
\begin{equation}\label{dg1}
|D_{\gamma_1,\gamma_2}\tilde{G} (x_2,y) -D_{\gamma_1,\gamma_2}\tilde{G} (0,y)|\leq \kappa_2 \|y\|_{g'} ^{-(n-1)-\frac{1}{2}-\frac{\gamma_1+\gamma_2}{2}}
\end{equation}

Now if we assume that  $\rho\in C^{\infty} _c$ has the property $\rho(0)=\rho (x_2)=0$ and  assuming that
$
[\rho |\pi|^b]_{\alpha}<C_0
$ we get to 
\[
|\pi(y)|^b |\rho (y)|\leq C_0 \|y\|_{g'} ^{\alpha}
\]

Therefore from the assumption that $|\pi|^{b}\rho\in C^{\alpha}$ we have 
\[
|\pi |^b |\rho||D_{\gamma_1,\gamma_2}\tilde{G} (x_2,y) -D_{\gamma_1,\gamma_2}\tilde{G} (0,y)||\pi|^{2-b} \leq C_1 \|y\|_{g'} ^{-n+\frac{1}{2}-\frac{\gamma_1 + \gamma_2}{2}+ \alpha}|\pi|^{2-b}
\]
for some constant $C_1$.
Now we can  apply 
 part (ii) of lemma (\ref{schauderpr})to conclude that 
 $I'_1$ is bounded if  $\alpha<\frac{1}{2}$. (Here we  set $k=-n+\frac{1}{2}-\frac{\gamma_1+\gamma_2}{2}$ and  $c=\frac{2-b}{2}=\frac{\gamma_1+\gamma_2}{2}$).

Also for  boundedness of $I'_2$  one approximates the following two integrals separately
\[
|\int_{\|y\|_{g'}<2} D_{\gamma_1,\gamma_2}\tilde{G} (0,y) \rho (y) |\pi(y)|^2 dy|  \hspace{2cm} |\int_{\|y\|_{g'}<2} D_{\gamma_1,\gamma_2}\tilde{G}(x_2,y) \rho (y) |\pi(y)|^2 dy | 
\]


 We  consider the two subcases A1 and A2. In the subcase A1  from  part (I) of lemma (\ref{schauderpr1})  

\[
|\int_{\|y\|_{g'}<2}D_{\gamma_1,\gamma_2}\tilde{G} (0,y) \rho (y) |\pi(y)|^2 dy|\leq \kappa ' \int_{\|y\|_{g'} \leq 2} \|y\|_{g'} ^{\alpha-n+1-\frac{\gamma_1+\gamma_2}{2}}|\pi (y)|^{{\gamma_1+\gamma_2}}dy
\]
 
which is bounded  by first part of lemma (\ref{schauderpr}) since $\alpha>0$. The constant $\kappa'$ depends on  $C^{\alpha}$-norm of
 $|\pi (y)|^b\rho(y)$ and the constant $\kappa$ in lemma (\ref{schauderpr1}).

To estimate 
\[
 |\int_{\|y\|_{g'}<2}D_{\gamma_1,\gamma_2}\tilde{G}(x_2,y) \rho (y) |\pi(y)|^2 dy | 
\]

we again consider the two subcases A1 and A2. The subcase A1 is similar to  the above argument  and the subcase A2 follows from third part of lemma 
(\ref{schauderpr1}) and the first part of lemma (\ref{schauderpr}). Case B can also be treated similarly.\end{proof}

 If $\|\zeta \|_{g'}=1$ then  If $\|\zeta \|_{g'}=1$ then  for $\gamma=1$  similar to  part II of lemma (\ref{schauderpr1})  and relations  (\ref{dgro1}) and (\ref{hamgeni})  we can derive

\[
|  \tilde{\partial}_{j,\gamma}\tilde{G}(w_1,\zeta)- \tilde{\partial}_{j,\gamma}\tilde{G}(w_2,\zeta)   |\leq \kappa ' _2\|w_1-w_2\|_{g'} ^{1/2}
\]

\begin{equation}\label{dgro46}
 \tilde{\partial}_{i,\gamma}\tilde{G}(\rho_{\lambda})=\lambda (\tilde{\partial}_{i,\gamma} \tilde{G} \rho)_{\lambda}
\end{equation}

\begin{equation}\label{hamgeni46}
 \tilde{\partial}_{i,\gamma}\tilde{G} (a_{\lambda}(  x), a_{\lambda} (y))=\lambda^{-n+1} \tilde{\partial}_{i,\gamma} \tilde{G} (x ,  y)
\end{equation}

The by repeating the same argument as in the proof of (\ref{main3}) and taking into account the fact $C^{,\alpha}$ functions can be approximated by smooth functions in the $C^{,\underline{\alpha}}$-norm for any $\underline{\alpha}<\alpha$ as explained in \cite{d}  we can also conclude that  

\begin{cor}\label{corhom}
If $|\pi|^b \rho\in C^{\alpha} _c (\mathbb{R}^m)$    for $\gamma=1$  and for any $0<\alpha<\underline{\alpha} <\frac{1}{2}$ we have
\[
\bigg [ \tilde{G} (\rho)\bigg ]_{\underline{\alpha}} \leq  C \bigg [ |\pi|^b \rho\bigg ]_{\alpha}
\]

\[
\bigg[\tilde{\partial}_{i,\gamma}\tilde{G} (\rho) \bigg]_{\underline{\alpha}}
 \leq C\bigg[ |\pi|^b \rho\bigg]_{\alpha}
\]

we also have 
\[
\bigg[\tilde{\partial}_{j,\gamma_2}\tilde{\partial}_{i,\gamma_1}\tilde{G} (\rho) \bigg]_{\alpha}
 \leq C\bigg[ |\pi|^b \rho\bigg]_{\alpha}
\]

 for some constant $C$ which depends on $m, \gamma_1$ and $\gamma_2$.
\end{cor}



\subsection{Local theory }
 

 Consider a smooth cut-off function $\chi$ such that $supp (\chi)\subset B'_2$
and $\chi^{-1} (1)=B'_2$ where $B_2\subset B_1$ are two $g'$ balls in $\mathbb{R}^m$. We  have
\[
\Delta_{g'}(\chi \phi)=\phi \Delta_{g'}(\chi)+ \chi \Delta_{g'}\phi+ \nabla_{g'}\phi.\nabla_{g'}\chi
\]   
It is clear that 
\[
\chi\phi=\tilde{G} \bigg ( \phi \Delta_{g'}(\chi)+ \chi \Delta_{g'}\phi \bigg )+\tilde{G}\bigg (\nabla_{g'}\phi.\nabla_{g'}\chi  \bigg ) 
\]



If $|\pi|^{b}\Delta_{g'}\phi$, $\phi$ and $d\phi$ are assumed to belong to $C^{\alpha}(B'_2)$, for some $0\leq b \leq 2$ then we have
\[
\begin{split}
[\tilde{\partial}_{i,\gamma_1}\tilde{\partial}_{j,\gamma_2}   (\chi\phi)]_{\alpha} &=\bigg [\tilde{\partial}_{i,\gamma_1}\tilde{\partial}_{j,\gamma_2}  \tilde{G} \bigg (\chi\Delta_{g'} \phi + \phi\Delta_{g'} \chi +\nabla_{g'}\chi.\nabla_{g'}\phi \bigg )\bigg ]_{\alpha}\\
\quad & = \bigg  [\tilde{\partial}_{i,\gamma_1}\tilde{\partial}_{j,\gamma_2}  \tilde{G} \bigg (\chi\Delta_{g'} \phi + \phi\Delta_{g'} \chi \bigg )\bigg ]_{\alpha}+\bigg  [\tilde{\partial}_{i,\gamma_1}\tilde{\partial}_{j,\gamma_2}  \tilde{G} \bigg ( \nabla_{g'}\chi.\nabla_{g'}\phi\bigg )\bigg ]_{\alpha}\\
\end{split}
\]
According to theorem  (\ref{main3}) we have
\[
\bigg [\tilde{\partial}_{i,\gamma_1}\tilde{\partial}_{j,\gamma_2}  \tilde{G} \bigg (\chi\Delta_{g'} \phi + \phi\Delta_{g'} \chi \bigg )\bigg ]_{\alpha, B'}\leq  C\bigg   [|\pi|^{b}\bigg (\chi\Delta_{g'} \phi + \phi\Delta_{g'} \chi\bigg  )\bigg  ]_{\alpha, B}
\]
 and
\[
 [\tilde{\partial}_{i,\gamma_1}\tilde{\partial}_{j,\gamma_2}  \tilde{G} ( \nabla_{g'}\chi.\nabla_{g'}\phi)]_{\alpha, B'}\leq C [|\pi|^b( \nabla_{g'}\chi.\nabla_{g'}\phi)]_{\alpha, B}
\]

In fact $\chi$ can be chosen in such a way that $\frac{1}{|\pi|^2}\chi_z\in C^{\infty}$ thus 
\[
 [|\pi|^b( \nabla_{g'}\chi.\nabla_{g'}\phi)]_{\alpha}\leq C'( [d\phi]_{\alpha, B'_2}
\]
 
here $\nabla$ denotes the ordinary derivative of $\phi$.

for some constant $C'$ depending on $\chi$ and  $C$  and $b$. hence we have proved the follwonig theorem

 \begin{theo}\label{the2}
Let $B'_2\subset B'_1$ be two small relatively compact  $g'$-balls  in some open set $U\subset\mathbb{R}^n$.  Assume that $\phi\in C^{,\alpha} (U)$ is such that $\Delta_{g'} \phi$ is defined pointwise outside $D$ and such that  $ |\pi|^b \Delta_{g'} \phi, \phi $ and $d\phi$ have bounded $C^{\alpha}$-norm over $B'_2$ for some $0\leq b\leq 2$.
Then for any $0\leq  \gamma_1, \gamma_2\leq 1$ with $\gamma_1 + \gamma_2=2-b$, 
and for all  $i,j$ , $\tilde{\partial}_{i,\gamma_1}\tilde{\partial}_{j, \gamma_2} \phi\in C^{,\alpha}$ and   there exists a constant $C$   depending on  $B'_1$  $B'_2$, $\gamma_1$  and $\gamma_2 $  such that


\[
[\tilde{\partial}_{i, \gamma_1}\tilde{\partial}_{j, \gamma_2}\phi]_{\alpha;B'_{1}}\leq C\big ( [\phi]_{\alpha;B'_{2}}+[d \phi]_{\alpha, B'_2}+ [|\pi |^b\Delta_{g'}\phi]_{\alpha; B'_{2}}\big )
\]

\begin{rem}
By taking $B' _1$ and $B' _2$ small enough the term  $[d \phi]_{\alpha, B'_2}$ can be ignored in the above relation.
\end{rem}


\end{theo}



\subsection{Global case}
Define the opeator $\tilde{\partial}_{\gamma}$ by
\[
\begin{split}
\tilde{\partial}_{\gamma}&=\frac{\partial}{\partial z} dz +\sum_{i=1} ^{n-1} \frac{z}{|z|^\gamma}\frac{\partial}{\partial w_i} dw_i\\
\quad & = \tilde{\partial}_{z}dz+\sum_{i=1} ^{n-1} \tilde{\partial}_{w_i, \gamma}dw_i= \sum_{i=1} ^n  \tilde{\partial}_{z_i,\gamma}dz_i
\end{split}
\]
where in the second line we are defining the partial derivative operators $\tilde{\partial}_z$ and $\tilde{\partial} _{w_i}$, $1\leq i\leq n-1$.
We also use the new notation $(z_1,...,z_n)$ for the coordinates $ (w_1,...,w_{n-1},z)$ to simplify summation symbol.

Let  $0\leq b\leq 2$  and $0\leq \gamma_1, \gamma_2\leq 1$ be such that $\gamma_1+\gamma_2 = 2-b$, and  conider an operator
of the form 

\begin{equation}\label{coor0}
|\pi|^2\Delta ' _{\eta}:=|\pi|^2\Delta_{g'}+\epsilon\eta.i\partial\bar{\partial}
\end{equation}

where $\eta$ is a  section of the bundle of $(1,1)$ forms and where  the hermitian  product  $\eta.i\partial\bar{\partial}$ is the standard (nondegenerate) hermitian product on $\mathbb{C}^n$.

Assume that around each point $p\in D$ there exists a holomorphic coordinates system  $(w_1,...,w_{n-1},z)$ such that  $p=(0,...,0)$ and $D=\{z=0\}$   and such that the operator $|\pi|^2\Delta ' _{\eta}$ in this coordinates system  is represented in the form
\begin{equation}\label{coor1}
|\pi|^2\Delta ' _{\eta}:=|\pi|^2\Delta_{g'}+\epsilon\tilde{\eta}.i\tilde{\partial}_{\gamma_1}\bar{\tilde{\partial}}_{\gamma_2}
\end{equation}

 where $\tilde{\eta}$ is a section of the  bundle of $(1,1)$ forms which has a  $C^{\alpha}$-norm along the $z$-axis $w_1=...=w_{n-1}=0$ smaller that a constant $\tilde{C}$ independent of the point $p$. 
  If  we set

\[
B=\tilde{\eta}.i\tilde{\partial}_{\gamma_1}\tilde{\partial}_{\gamma_2}
\]



then the   Neumann series expansion  
\[
H=\tilde{G}\sum_{n=0} ^{\infty} \epsilon^n (B\tilde{G})^n 
\]
defines  a right inverse for $|\pi|^2\Delta' _{\eta}$.
Consider the equation 
\[
\Delta'_{\eta} \phi =\rho
\]
where $\phi$ and $\rho$ are both of compact support and 
with $|\pi|^b \rho\in C^{\alpha}$. Then we have $H\rho=\phi$.
In the above mentioned coordinates system  along the $z$-axis  we have
\[
B=\epsilon\tilde{\eta}.i\tilde{\partial}_{\gamma_1}\bar{\tilde{\partial}}_{\gamma_2}=\epsilon\sum \tilde{a}^{i\bar{j}}\tilde{\partial}_{i,\gamma_1}\bar{\tilde{\partial}}_{j, \gamma_2}
\]
 where $\tilde{a}^{i\bar{j}}$'s   have $C^{\alpha}$-norm along the $z$ axis, smaller than $\tilde{C}$.
 We also assume thatwe have

\begin{equation}\label{rel159}
\frac{\tilde{a}^{i\bar{j}}}{|z|^{1/2}} \in C^{\alpha}, \hspace{0.5cm} \text{ for } i=n, j\neq n  \text{ or } i\neq n,  j=n  \text{ along the } z\text{ -axis} 
\end{equation}
If  we assume that the support of $\eta$ lies in an open  relatively compact neighborhood   $U_p$ of the origin  then we claim that for $\epsilon $ small enough  and for any  $\rho\in C^{\infty} _c (U_p)$ the following inequality holds
\begin{equation}\label{rondh}
[\tilde{\partial}_{i, \gamma_1} \bar{\tilde{\partial}}_{j,\gamma_2} H \rho ]_{\alpha, N_p}\leq C_0\big( [|\pi|^b\rho]_{\alpha,N_p}+[\rho]_{0, N_p}\big ) 
\end{equation}

where
\[
N_p=U_p\cap \{w_1=...=w_{n-1}=0\}
\]

and  where  $C_0$ depends on $[\tilde{a}^{i\bar{j}}]_{\alpha,N_p}, [\tilde{a}^{i\bar{j}}]_{0,N_p}$ and $U_p$.
 
To prove this  we observe that

\[
\begin{split}
[B\tilde{G}\rho]_{\alpha ;N_p}= &\epsilon  [\sum_{i,j} \tilde{a}^{i\bar{j}}\tilde{\partial}_{i, \gamma_1} \bar{\tilde{\partial}}_{j,\gamma_2}\tilde{ G}\rho ]_{\alpha; N_p}=\epsilon\sum_{i,j} \bigg \{[\tilde{a}^{i\bar{j}}]_{0;N_p}[\tilde{\partial}_{i,\gamma_1} \bar{\tilde{\partial}}_{j,\gamma_2}\tilde{ G}\rho]_{\alpha;N_p}+[\tilde{\partial}_{i,\gamma_1} \bar{\tilde{\partial}}_{j,\gamma_2}\tilde{ G}\rho]_{0;N_p}[\tilde{a}^{i\bar{j}}]_{\alpha;N_p} \bigg \}\\
\quad & \leq \epsilon \sum_{i,j} \bigg \{C[\tilde{a}^{i\bar{j}}]_{0;N_p}[|\pi|^b\rho]_{\alpha;N_p}+C'[\rho]_{0;N_p}[\tilde{a}^{i\bar{j}}]_{\alpha;N_p}\bigg  \}\\
\quad  & \leq   \epsilon C'' ([|\pi|^b \rho]_{\alpha;N_p}+ [\rho]_{0;N_p})
\end{split}
\]

where $C''=\max\{C\sum_{i,j}[\tilde{a}^{i\bar{j}}]_{0;N_p}, C' \sum_{i,j} [\tilde{a}^{i\bar{j}}]_{\alpha;N_p}\}$ and 
for the inequality  $[\tilde{\partial}_{i,\gamma_1} \bar{\tilde{\partial}}_{j,\gamma_2}\tilde{G} \rho]_{0;N_p}\leq C  [\rho]_{0;N_p}$  in the second line we are utilizing corollary (\ref{corhom}) . Now we do induction and  we assume that  $[(B\tilde{G})^{n-1}\rho]_{\alpha, N_p}\leq C''_{n-1}\epsilon ^{n-1}([|\pi |^b\rho]_{\alpha, N_p}+[\rho]_{0,N_p})$ for some constant $C''_{n-1}$. Then we have,

\[ 
\begin{split}
[(B\tilde{G})^n\rho]_{\alpha, N_p}&\leq C''_{n-1}\epsilon ^{n-1}([|\pi |^bB\tilde{G}\rho]_{\alpha,N_p}+[B\tilde{G}\rho]_{0,N_p})\\
\quad &\leq  C''_{n-1}\epsilon ^{n-1}([|\pi |^b]_{0}[B\tilde{G}\rho]_{\alpha,N_p}+ [|\pi |^b]_{\alpha}[B\tilde{G}\rho]_{0,N_p}+[B\tilde{G}\rho]_{0,N_p})\\
\quad &  = C''_{n-1}\epsilon ^{n-1}([|\pi |^b]_{0}[B\tilde{G}\rho]_{\alpha,N_p}+ ([|\pi |^b]_{\alpha,N_p}+1)[B\tilde{G}\rho]_{0,N_p})\\
\quad & \leq C''_n \epsilon ^n ([|\pi |^b \rho]_{\alpha,N_p}+[\rho]_{0,N_p})
\end{split}
\]

This shows that 
\[
[(B\tilde{G})^n\rho]_{\alpha;N_p}\leq \epsilon ^n \kappa ^n ([|\pi |^b\rho]_{\alpha;N_p}+[\rho]_{0;N_p})
\]

where $\kappa$ depends on $[|\pi |^b]_{\alpha;N_p},[|\pi|^b]_{0;N_p}$, $[\tilde{a}^{ij}]_{0;N_p},  [\tilde{a}^{ij}]_{\alpha;N_p}$.
Thus if $\epsilon<\frac{1}{\kappa}$   (\ref{rondh}) is satisfied. 
Since these quantities are assumed to be uniformly bounded on $U_p$ we get to 

\begin{equation}\label{rondh2}
[\tilde{\partial}_{i, \gamma_1} \bar{\tilde{\partial}}_{j,\gamma_2} H \rho ]_{\alpha, U_p}\leq C_0([|\pi|^b\rho]_{\alpha,U_p}+[\rho]_{0, U_p}) 
\end{equation}

Now suppose  $\eta$ is not  of compact support and  assume that the equation
\[
\Delta' _{\eta}\phi =\rho
\]
holds  point-wise  in $\{z\neq 0\}$ for $\phi \in C^{1+\alpha}$ and  
\begin{equation}\label{piroalfa}
|\pi |^b\rho \in C^{\alpha} (U_p).
\end{equation}
 
   We take a smooth cut-off function $\chi$ with support in some $g'$-ball $ B_1\subset U_p$.

Then we can write
\begin{equation}\label{ppp}
\begin{split}
|\pi|^2\Delta' _{\eta} (\chi \phi)= (|\pi|^2\Delta' _{\eta} \chi)\phi+\chi (|\pi|^2\Delta' _{\eta} \phi)+D^1\chi.D^2 \phi
\end{split}
\end{equation}

where $D^1$ and $D^2$ are first order operator.  We set
 
\begin{equation}\label{rtild1}
\tilde{\rho}:=  (\Delta' _{\eta} \chi)\phi+\chi (\Delta' _{\eta} \phi)+\frac{1}{|\pi |^2}D^1\chi.D^2\phi 
\end{equation}

to get  the equation $ |\pi |^2 \Delta' _{\eta} (\chi \phi) =\tilde{\rho}$. We also take another smooth cut-off function $\tilde{\chi}$ such that 
\[
B_1=Supp\, \chi \subset \tilde{\chi}^{-1} (1) \subset B_2:= supp \, \tilde{\chi}\subset U_p
\]
 if we set
\[
|\pi |^2\tilde{\Delta}_{\eta}:=|\pi|^2\Delta_{g'}+\epsilon \tilde{\chi}\eta. i\tilde{\partial}_{\gamma_1} \bar{\tilde{\partial}}_{\gamma_2}
\]

then we still have
\[
|\pi|^2\tilde{\Delta}_{\eta}(\chi \phi)=|\pi|^2\tilde{\rho}
\]


Since $\chi\phi$ is of compact support the equation

\[
\tilde{G}\Delta_{g'} (\chi\phi)=\chi\phi
\]

 holds.
So if we define

\[
\tilde{B}=\epsilon\tilde{\chi}\eta.i\tilde{\partial}_{\gamma_1}\bar{\tilde{\partial}}_{\gamma_2}=\epsilon\sum \tilde{a}^{i\bar{j}}\tilde{\partial}_{i,\gamma_1}\bar{\tilde{\partial}}_{j, \gamma_2}
\]

then we get 

\[
(I+ \epsilon \tilde{B} \tilde{G})\Delta_{g'} (\chi\phi) = \tilde{\Delta} _{\eta} (\chi \phi )=\tilde{\rho}
\]

So 

\[
\Delta_{g'} (\chi\phi) =(I+ \epsilon \tilde{B} \tilde{G})^{-1} \tilde{\rho}
\]

and again using the fact that $\chi\phi $ has compact support  we get

\[
\chi\phi= \tilde{G}(I+ \epsilon \tilde{B} \tilde{G})^{-1} \tilde{\rho}=\tilde{H}\tilde{\rho}
\]

where $\tilde{H}:=\tilde{G}(I+ \epsilon \tilde{B} \tilde{G})^{-1}$.
We now prove that an  appropriate choice  of $\chi$ ensures that $|\pi|^b\tilde{\rho}\in C^{\alpha}$.
First we observe that $\chi$ can be chosen  in such a way that $\Delta' _{\eta}\chi$ to be $C^{\alpha}$.

We claim that by an appropriate choice of $\chi$, the term $|\pi|^{b-2}D^1\chi.D^2\phi=|\pi|^{-\gamma_1-\gamma_2}D^1\chi.D^2\phi$ becomes $C^{\alpha}$ along $z$ axis.  This suffices to conclude that  $\tilde{\rho}$  is $C^{\alpha}$ over $U_p$.
Because by change of coordinates   we can deduce the same result for different points $q\in U_p\cap D$.
  \\
In fact  $|\pi|^{-\gamma_1}\tilde{\partial}_{i,\gamma_1}\phi$   and $|\pi|^{-\gamma_2}\tilde{\partial}_{i,\gamma_2}\phi$
 for $i=1,...,n-1$, belong to $C^{\alpha}$ since $\phi$ does so. Also for the terms like $|\pi|^{-\gamma_1}\tilde{\partial}_{n,\gamma_1}\phi |\pi|^{-\gamma_2}\tilde{\partial}_{j,\gamma_2}\chi$ we can apply the assumption  (\ref{rel159}) to conclude that 
$|\pi|^{b-2}D^1\chi.D^2\phi$ is $C^{\alpha}$ along $z$-axis and therfore $|\pi|^b \tilde{\rho} \in C^{\alpha}$.
Now from  (\ref{rondh2}) for $\epsilon $ small enough we  get

\begin{equation}\label{aha1}
[ \tilde{\partial}_{i,\gamma_1}\tilde{\partial}_{j, \gamma_2}\phi]_{\alpha; B_1 }\leq [\tilde{\partial}_{i, \gamma_1}\tilde{\partial}_{j,\gamma_2} \tilde{H}\tilde{\rho}]_{\alpha; B_2}\leq C_0 ([|\pi|^b\tilde{\rho}]_{\alpha; B_2} +[\tilde{\rho}]_{0; B_2} ) 
\end{equation}


where the appropriate range of $\epsilon$ depends on $[\tilde{\chi}\tilde{a}^{ij}]_{\alpha}$,  $[\tilde{\chi}\tilde{a}^{ij}]_{0}$ $[|\pi|^2]_{\alpha; U_p}$ and $[|\pi|^2]_{0; U_p}$.

From (\ref{aha1}) 
 and defining relation  (\ref{rtild1}) and  (\ref{rondh2}) we obtain

\begin{equation}\label{aha3}
[ \tilde{\partial}_{i,\gamma_1}\tilde{\partial}_{j,\gamma_2}\phi]_{\alpha; B_1 } \leq C_0 ([|\pi|^b\Delta' _{\eta} \phi]_{\alpha; B_2} +[|\pi|^b\Delta' _{\eta} \phi]_{0; B_2} +[\tilde{\partial} \phi]_{\alpha, B_2}+[\phi]_{0; B_2} ) 
\end{equation}

where $[\tilde{\partial} \phi]_{\alpha, B_2} =\max_{i} [\tilde{\partial_{i} } \phi]_{\alpha, B_2}$  and $\tilde{\partial}_i$ is defined in (\ref{rtg1}) 



\begin{lemma}\label{lem121}
Assume that $\eta (p) =0$ and it satisfies (\ref{rel159})  and  $\eta\in C^{\alpha}(U_p) $. Then for any two open balls $B_1\subset B_2$ around $p$  which are    relatively compact  in $U_p$ there exists a constant $C_0$  such  that for any $\phi\in C^{\alpha} (U_p)$ all of whose second partial derivatives exist and such that  $|\pi|^2\Delta '_{\eta} \phi\in C^{\alpha} (U_p)$ the inequalities (\ref{aha5})   holds.


\end{lemma}

\begin{proof} To prove  this  lemma  we first  make a change of coordinates like 
\[
(u_1,...,u_{n-1}, v)=(\frac{1}{\lambda} w_1,...,\frac{1}{\lambda} w_{n-1}, \frac{1}{\sqrt{\lambda}} z)
\]
and we set 
\[
b^{i\bar{j}}:=\tilde{a}^{i\bar{j}}(\lambda u_1,...,\lambda u_{n-1}, \sqrt{\lambda} v),\hspace{1cm} \psi :=\phi (\lambda u_1,...,\lambda u_{n-1}, \sqrt{\lambda} v ) 
\]
\[
\sigma:= \rho (\lambda u_1,...,\lambda u_{n-1}, \sqrt{\lambda} v ), \hspace{1cm} \theta :=\frac{1}{\sqrt{\lambda}}\pi (\lambda u_1,...,\lambda u_{n-1}, \sqrt{\lambda} v )
\]
 Then in new  coordinates the equation $\Delta'_{\eta} \phi =\rho$ is given by 
\begin{equation}\label{newcdn}
(|\theta|^2\Delta_g'+\lambda^{\alpha}\sum \frac{b^{i\bar{j}}}{\lambda ^{\alpha}} \tilde{\partial}_{i}\bar{\tilde{\partial}}_{j} ) \psi =\lambda ^2  |\theta|^2\sigma 
\end{equation}
Here  $b^{i\bar{j}} (u_1,...,u_{n-1},v) = a^{i\bar{j}} (\lambda u_1,...,\lambda u_{n-1},\sqrt{\lambda}v) $ and  to avoid the complexity of notation we apply the $\Delta_{g'}$ and $\tilde{\partial_i}$  for new coordinates  $(u_1,...,u_{n-1}, v)$ as well.

We note that if we set
\[
c^{i\bar{j}} _{\lambda}(u_1,...,u_{n-1}, v) :=\frac{a^{i\bar{j}}(\lambda u_1,...,\lambda u_{n-1}, \sqrt{\lambda} v) }{\lambda ^{\alpha}}
\]

then for any neighbourhood $U$ of $p$ there exists a constant $M$ independent of $\lambda $ such that
\[
[c^{i\bar{j}} _{\lambda} ]_{0}, [c^{i\bar{j}} _{\lambda} ]_{\alpha}\leq M
\]

The boundedness of $[c^{i\bar{j}} _{\lambda} ]_{0}$ is a consequence of the hypothesis that $\tilde{a}^{ij}(0)=0$.  It is also clear that $[c^{i\bar{j}} _{\lambda} ]_{\alpha}= [a^{i\bar{j}}]_{\alpha}$.
This means that in (\ref{newcdn}) the coefficient $\lambda^{\alpha}$ can play the role of $\epsilon$ thus for $\lambda$
small enough we can derive an  inequalities  similar to (\ref{aha3}) 

\begin{equation}\label{aha7}
[ \tilde{\partial}_{i,\gamma_1}\tilde{\partial}_{j,\gamma_2}\psi]_{\alpha; \tilde{B}_1 } \leq C_0 ([|\theta|^b\Delta' _{\eta} \psi]_{\alpha; \tilde{B}_2} +[|\theta|^b\Delta' _{\eta} \psi]_{0; \tilde{B}_2}+[\tilde{\partial} \psi]_{\alpha; \tilde{B}_2} +[\psi]_{0; \tilde{B}_2} ) 
\end{equation}

where $\tilde{B}_1$ and $\tilde{B}_2$ are appropriate  $g'$-balls  in coordinates $(u_1,...,u_{n-1},v)$.

Returning to initial $(w_1,...,w_{n-1}, z)$ coordinates we get to

\begin{equation}\label{aha5}
[ \tilde{\partial}_{i, \gamma_1}\tilde{\partial}_{j,\gamma_2}\phi]_{\alpha; B_1 } \leq C_1 ([|\pi|^b\Delta' _{\eta} \phi]_{\alpha; B_2} +[|\pi|^b\Delta' _{\eta} \phi]_{0; B_2} +[\tilde{\partial}\phi]_{\alpha;B_2}+[\phi]_{0; B_2} ) 
\end{equation}

for appropriate constant $C_1$.
\end{proof}




 
 \begin{theo}\label{schth}
Consider an operator  $\Delta'$ defined in an open neighborhood of $D$. Assume that there exists $\gamma_1=\gamma_2=\frac{1}{2}$ and for any point $p\in D$ there exists a coordinates system $(w_1,..,w_{n-1},z)$  in a neighborhood $U_p$ of $p$ such that $p=(0,...,0)$ and $D\cap U_p = \{z=0\}$ and such that 

\[
|\pi |^2\Delta' = |\pi |^2\Delta_{g'}+ \tilde{\eta}.i\tilde{\partial}_{\gamma_1}\bar{\tilde{\partial}}_{\gamma_1}
\]
 where $\tilde{\eta}$ is a section of the  bundle of $(1,1)$ forms which has a  $C^{\alpha}$-norm along the axis $w_1=...=w_{n-1}=0$ smaller that a constant $\tilde{C}$ independent of the point $p$. We also assume that  $\tilde{\eta}(p) =0$ and relation (\ref{rel159})  is satisfied in this coordinates..   Let $\phi\in C^{1+\alpha} (U)$ and assume that  the equation
\[
\Delta_{\eta}\phi' = \rho
\]
holds for  a function $\rho$ satisfying $|\pi |^b \rho \in C^{\alpha}$  where  $b=2-\gamma_1-\gamma_2 $. Let
$|\pi|^{-\gamma_1-\gamma_2}\tilde{\partial}_{i,\gamma_1}\phi$   and $|\pi|^{-\gamma_1-\gamma_2}\tilde{\partial}_{i,\gamma_2}\phi$
 for $i=1,...,n-1$, belong to $C^{\alpha}$.

Then for  relatively compact balls $B_1\subset B_2$ centered at $p$ in  $U_p$ we have
\[
[\tilde{\partial}_{i,\gamma_1}\tilde{\partial}_{j,\gamma_2 }\phi]_{\alpha;B_1}\leq C\big ( [\phi]_{\alpha;B_2}+ [\tilde{\partial} \phi]_{\alpha; B_2}+[|\pi|^b\Delta_{\eta}\phi]_{\alpha; B_2}\big )
\]


\end{theo}



\section{Upto Second Order Estimates for Monge-Amp\`ere equation}\label{sec5}

Let $(X, \omega_{reg})$ be a K\"ahler manifold  and let $D$ be a smooth divisor on $X$. Assume that $S$ is a holomorphic section of $L:=[D]$ with simple zero along $D$. Let $|.|$ denote a hermitian metric on $L$.
  We consider  the following Monge-Amp\`ere equation:  
\begin{equation}\label{maso}
(\omega_{reg} +\partial\bar\partial \phi) ^n = |S|^2e^G \omega  ^n
\end{equation}

where $G$ is assumed to be in $C^{3} (X)$.  We also assume that $\phi\in C^5 (X)$ and we set

\[
\omega':=\omega_{reg}+ \partial\bar{\partial} \phi
\]
 
The metric  $\omega'$  is assumed to degenerate transversally and is nondegenerate when restricted to $T D$.

 

If $g_{reg}$ and $g'$ denote  the  metrics associated  respectively to the  K\"ahler form $\omega_{reg}$ and  degenerate K\"ahler form $\omega'$  we want to prove the following second order estimations on $\phi$:

\begin{prop}\label{secondorder}
 There exists a constant $C_{\ref{secondorder1}}$  depending on $G$, $g_{reg}$  such that
\begin{equation}\label{secondorder1}
0\leq (n+\Delta_{g_{reg}} \phi)(x)\leq C_{\ref{secondorder1}}
\end{equation}
for all $x\in X$, 
where $\Delta _{g_{reg}}$ denotes the Laplacian associated to the degenerate metric $g_{reg}$.  


 Then there are positive constants
$C_1$, $C_2$, $C_3$, and $C_4$, depending on $G$, $g_{reg}$ and $U_p$   such
that $\sup _X |\phi|\leq C_1$, $sup_X |\nabla_{g_{reg}} \phi| \leq C_2$, $0\leq C_3\leq 1+\phi_{i\bar{i}}\leq C_4$ for all $i$. Where $\nabla_{g_{reg}}$ denotes the covariant derivative with respect to the degenerate metric $g_{reg}$.
\end{prop}

\begin{proof}  We first note that the $C^0$ estimation for $\phi$ can be obtained by Moser iteration method exactly as in non-degenerate case. For first and second order estimate we start  by  the following inequality  that can be proved by the same way as  in \cite{y} (relation (2.22)) on $X\setminus D$ 

Let $\tilde{\chi}: \mathbb{R}^{2n}\rightarrow [0,1]$ be a cut-off function such that $Supp (\tilde{\chi})\subset B_{2} (0)$ and $\tilde{\chi}^{-1} (\{1\})= B_{1}(0)$, where $B_{r} (0)$ denotes the ball of radius $r$ centered at the origin. We assume that $\tilde{\chi}= e^{u}$ for some smooth function $u:\mathbb{R}^{2n}\rightarrow \mathbb{R}$.

Consider a point $p\in D$  and a neighborhood $U_p $ over which there exists a coordinates system $(w_1,...,w_{n-1},z)$
 such that $p=(0,...,0)$ and $D\cap U_p = \{z=0\}$. Let $\epsilon >0$ be such that   $\bar{B}_{2\epsilon} (0)\subset U_p$. We assume that $\chi_{\epsilon}:X\rightarrow [0,1]$ is a cut-off function which in this coordinates is defined by
\begin{equation}\label{chilambd}
 \chi_{\epsilon} (w_1,...,w_{n-1},z)= e^{\epsilon^2 (u(\frac{w_1}{\epsilon},..., \frac{w_{n-1}}{\epsilon}, \frac{z}{\epsilon})) }. 
\end{equation}
then we have $supp( \chi_{\epsilon} )\subset B_{2\epsilon} (0)$ and $\chi_{\epsilon } ^{-1} (\{1\})= B_{\epsilon} (0)$  moreover since we have
\begin{equation}\label{karandar}
(\Delta_{g'} \chi_{\epsilon}-\frac{| \nabla_{g'} \chi _{\epsilon}|_{g'}^2}{\chi_{\epsilon}})(w_1,...,w_{n-1},z)=\sum g'^{\alpha\bar{\alpha}} (u_{\alpha\bar{\alpha}}e^{\epsilon^2 u}) (\epsilon^{-1} w_1,...,\epsilon^{-1} w_{n-1}, \epsilon^{-1} z)
\end{equation}
  we can choose $\tilde{\chi}$ such that $(\Delta_{g'} \chi_{\epsilon}-\frac{| \nabla_{g'} \chi _{\epsilon}|_{g'}^2}{\chi_{\epsilon}})<C$  for some $C$ independent of $\epsilon$.

According to relation (\ref{f}) in the appedix (\ref{app6}) we have

\begin{equation}\label{ffchiep}
\begin{split}
\Delta_{g'} (\chi_{\epsilon} \exp \{-C\phi\} (m+\Delta_{g_{reg}} \phi))\geq &\chi_{\epsilon} \exp\{-C\phi\} \big \{  \Delta_{reg} (G )\\
\quad & - C\chi_{\epsilon} m^2 \inf _{i\neq l} R_{i\bar{i} l\bar{l}} \\
\quad &exp \{-C\phi\} (-Cm\chi_{\epsilon} +\Delta_{g'} \chi_{\lambda}-\frac{| \nabla_{g'} \chi _{\epsilon}|_{g'}^2}{\chi_{\epsilon}}) (m+\Delta_{reg} \phi ) \\
\quad  &+  \chi_{\epsilon} (C+\inf_{i\neq l} R_{i\bar{i}l\bar{l}})\exp \{\frac{-G}{m-1}\}  (m+\Delta_{reg} \phi) ^{1+1/(m-1)}\big \}\\
\end{split}
\end{equation}

where $R_{i\bar{i}j\bar{j}}$ denotes the curvature tensor associated to $g_{reg}$ and the constant $C$ is chosen in such a way that 
$C+\inf_{i\neq l} R_{i\bar{i} l\bar{l}}>0$. 
\\

 Assume that $p\in X$ is a point at which $\exp (-C\phi )[   (m+\Delta_{reg} \phi )]$  attains its maximum.  If $p\in  X\setminus D$  then  using the inequality (\ref{f}) we obtain

\begin{equation}\label{f}
\begin{split}
\Delta_{g'} (\chi \exp \{-C\phi\} (m+\Delta_{g_{reg}} \phi))\geq & \exp\{-C\phi\} \bigg ( \chi  \Delta_{reg} (G )- C\chi m^2 |S|^2 \inf _{i\neq l} R_{i\bar{i} l\bar{l}}\bigg ) \\
\quad &\exp \{-C\phi\} (-Cm\chi +\Delta_{g'} \chi-\frac{| \nabla_{g'} \chi|_{g'}^2}{\chi}) (m+\Delta_{reg} \phi ) \\
&\quad +\exp \{-C\phi\}   \chi (C+\inf_{i\neq l} R_{i\bar{i}l\bar{l}})\exp \{\frac{-G}{m-1}\}  (m+\Delta_{reg} \phi) ^{1+1/(m-1)}\\
\end{split}
\end{equation}

\begin{equation}\label{f1}
\begin{split}
0\geq  \chi_{\epsilon} \Delta_{reg} G - &Cm^2   \inf _{i\neq l} R_{i\bar{i} l\bar{l}}+\frac{( -C  m\chi_{\epsilon}+\frac{1}{\chi_{\lambda}}(\Delta_{g'}\chi_{\epsilon} -\frac{|\nabla_{g'} \chi_{\epsilon}|^2 _{g'}}{\chi_{\epsilon}} )}{\chi_{\epsilon} ^{\frac{m-1}{m}}} (\chi_{\epsilon} ^{\frac{m-1}{m}})(m+\Delta_{reg} \phi ) \\
&\quad + (C+\inf_{i\neq l} R_{i\bar{i}l\bar{l}})\exp \{\frac{-G}{m-1}\}\big ( \chi_{\epsilon} ^{\frac{m-1}{m}}  (m+\Delta_{reg} \phi)\big ) ^{1+1/(m-1)}
\end{split}
\end{equation}
where the right hand side is evaluated at $p$. Due to (\ref{karandar}) $\frac{( -C  m\chi_{\epsilon}+\frac{1}{\chi_{\epsilon}}(\Delta_{g'}\chi_{\epsilon} -\frac{|\nabla_{g'} \chi_{\epsilon}|^2 _{g'}}{\chi_{\epsilon}} )}{\chi_{\epsilon} ^{\frac{m-1}{m}}}$  has an upper bound independent of $\epsilon$. Therefore by the same argument as  in (\cite{y})  we can deduce that
\begin{equation}\label{cemd1}
\chi_{\epsilon} ^{\frac{m-1}{m}}  (m+\Delta_{reg} \phi)\leq C_{\ref{cemd1}}
\end{equation}
 for some constant $C_{\ref{cemd1}}$ which only depends on $G$ and $g_{reg}$ and does not depend on the  open set $U_p$. In particular we obtain

\begin{equation}\label{nambep}
 (m+\Delta_{reg} \phi)|_{B_{\epsilon}}\leq C_{\ref{cemd1}}
\end{equation}


If $p\in D$    we consider a canonical coordinates system $(w_1,...,w_{n-1},z)$  in a neighborhood $U_p$ of $p$ in such a way that 
$p=(0,...,0)$ in this coordinates and 
\[
g'_{w_i\bar{z}}=g'_{z\bar{w}_i}=O(|z|^3)\hspace{1cm} g'_{z\bar{z}}=|z|^2+O(|z|^3), \hspace{1cm} \text{ for } i=1,..., n-1
\]
 see appendix (\ref{app5}).
Thus we have $g'^{z\bar{z}}=\frac{1}{|z|^2}+O(\frac{1}{|z|})$, $g'^{z\bar{w}_i}=O(|z|)$.
 Then the Laplacian $\Delta_{g'}$ has the following form along $z$-axis $w_1=...=w_{n-1}=0$: 
\begin{equation}\label{del1}
\Delta_{g'} =(\frac{1}{|z|^2}+O(\frac{1}{|z|}))\frac{\partial ^2 }{\partial z\partial \bar{z}} + \sum_{i,j=1} ^{n-1} (\delta_{ij}+ A_{ij})\frac{\partial^2}{\partial w_i\partial\bar{w}_j}+\sum_{i=1} ^{n-1} A_{i\bar{n}}\frac{\partial^2}{\partial w_i\partial\bar{z}}+\sum_{i=1} ^{n-1} A_{n\bar{i}}\frac{\partial^2}{\partial z\partial \bar{w}_i}
\end{equation}
\begin{equation}\label{del2}
A_{ij}= O(|z|),\hspace{1cm}  \text{for } i,j=1,....,n-1
\end{equation}
\begin{equation}\label{del3}
A_{i\bar{n}}=A_{n\bar{i}}=O(|z|),\hspace{1cm} \text{ for } i=1,...,n-1
\end{equation}
 If we set $f=\chi_{\epsilon} e^{-C\phi}(m+\Delta_{reg} \phi)$, since the maximum of $f$ has occurred at $p=(0,...,0)$, then $\frac{\partial ^2 f}{\partial z\partial \bar{z}}$ can not be non negative at all the points of any   neighborhood of $p$  on the  $z$ axis $w_1=...=w_{n-1}=0$. Since otherwise $f$  would be sub-harmonic and it cannot attain a maximum at the interior point $p$. Thus we can find a sequence $\{q_i\}_{i\in \mathbb{N}}$
on the $z$ axis such that $q_i\rightarrow p$ as $i\rightarrow \infty$ and  $\frac{\partial ^2 f}{\partial z\partial \bar{z}} (q_i)<0$.
 Since $p$   is also a maximum for $f|_{D\cap U_p}$  we have $\sum_{i=1} ^{n-1} \frac{\partial^2 f}{\partial w_i\partial\bar{w}_i} (p)\leq 0$.  From these observations and relations (\ref{del1})(\ref{del2}) and (\ref{del3}) we find that 
\begin{equation}\label{limsup}
 \limsup _{i\rightarrow \infty} \Delta_{g'} f(q_i)\leq 0
\end{equation}

 Hence we have $\Delta_{g'} f\leq 0$ and we can  deduce the inequality (\ref{nambep}).

Using Schauder estimate  the proposition (\ref{secondorder}) is proved.

\end{proof}




\section{Third Order Estimates}\label{III}

Let $(X,D,\omega_{deg})$ be a degenerate K\"ahler manifold in the sense described in  section 3.1. Let $\omega_{reg}$ be an ordinary K\"ahler metric in the same cohomology class as $\omega_{deg}$ and     assume that $\phi:X\rightarrow \mathbb{R}$ solves the degenerate Monge-Amp\`ere equation

\begin{equation}\label{ma3rd}
(\omega_{reg}+\partial\bar\partial \phi )^n =e^{G}|S|^2 \omega_{reg} ^n
\end{equation}

  where   $G\in C^3 (X)$.    We set 
\[
\omega':=\omega_{reg}+\partial\bar{\partial}\phi
\]
and we assume that
in local holomorphic coordinates 
\[
\omega'=g'_{i\bar{j}}dz^i\otimes d\bar{z}^j
\]


As in (\cite{y}) we define
\begin{equation}\label{say}
\Psi:=\sum g'^{i\bar{r}} g'^{j\bar{s}} g'^{k\bar{t}}\phi_{i\bar{j}k}\phi_{\bar{r}s\bar{t}}
\end{equation}

where $g'^{-1}=(g'^{i\bar{j}})$ is the inverse of the matrix  $g'=(g'_{i\bar{j}})$.

Since $g'$  degenerates along
$D$, the map   $\Psi$ is only defined on $X\setminus D$. 
In fact $\Psi$ has the following representation

\[
\|\nabla_{g_{reg}} \circ \bar{\partial} \circ \partial \phi \|^2 _{g'}
\]

where $\nabla_{g_{reg}}$ is the covariant derivative induced by $g_{reg}$  such that  $\nabla_{g_{reg}} \circ \partial \circ \bar{\partial} \phi  $ defines   a section of 
$T^{'*}X  \otimes T^{''*} X\otimes T^{'*} X$ and $\|\|_{g'}$ denotes the  norm induced on this bundle by $g'$. (Here  $TX=T'X\oplus T'' (X)$ is the decomposition into holomorphic and antiholomorphic tangent bundles).

In   subsection  \ref{esti61} below  we will prove   there exists  a constant $\mathscr{C}$ depending only on $\omega_{reg}$ and $G$   in the right hand side of the  Monge Amp\`ere equation (\ref{ma3rd}) such that
$\Psi\leq \frac{\mathscr{C}}{|S|^4}$.
\\

 In   subsection \ref{esti62} we then define $\Psi_1$    as follows
\[
\Psi_1:=\| \nabla_{reg} \circ \partial \circ  \partial_D \phi \|_{g'} ^2 =\sum_{\substack{1\leq i,r\leq n-1 \\ 1\leq j,k,s,t\leq n} } g'^{i\bar{r}}g'^{\bar{j}s}g'^{k\bar{t}}\phi_{i\bar{j}k}\phi_{\bar{r}s\bar{t}}
\]
(see (\ref{sayek}) and we will prove that $\Psi_1\leq \mathscr{C}_1$ for  some constant $\mathscr{C}_1$ which  depends only on $\omega_{reg}$ and $G$.

\subsection{Estimation on $\Psi$}\label{esti61}

We consider a holomorphic coordinate system $(\tilde{w}_1,...,\tilde{w}_{n-1},\tilde{z})$ on an open set $U_p$ containing  $p\in D$ such that $D=\{\tilde{z}=0\}$. We  equippe this neighborhood with a moving frame $V_1,...,V_n$ as described in the appendix (\ref{app7}).
In order to do computation and make approximations on $U_p$ we take  a point $x$  and in a neighborhood of this point we consider an other holomorphic 
coordinates system $(z_1,...,z_n)$ such that $\frac{\partial}{\partial z_i}$ coincide with $V_i $ at  $x$ for $i=1,...,n$.
By making a correction on $w_{reg}$  as discussed in the appendix (\ref{app4}) we can assume that
  %
the equalities  

\begin{equation}\label{eqlits}
 g'_{i\bar{j}}=\delta_{ij}g'_{i\bar{j}} \hspace{0.5cm} \text{ and  }\hspace{0.5cm} \frac{\partial g_{reg, i\bar{j}}}{\partial z^k}= \frac{\partial g_{reg, i\bar{j}}}{\partial \bar{z}^l}=0
\end{equation} 
hold at $x$.  It can be seen that all the computation regarding the third order estimates in  \cite{y} holds true in this coordinates. 

  We then define


\[
A_{ijk\alpha  1}:=    (g'^{i\bar{i}})^{1/2}  (g'^{j\bar{j}})^{1/2}   (g'^{k\bar{k}})^{1/2} ( (\phi_{i\bar{j}k\bar{\alpha}}-\sum_p \phi_{i\bar{p}k}\phi_{p\bar{j}\bar{\alpha}}   g'^{p\bar{p}} )
\]

and 

\[
A_{ijk\alpha 2}:=   (g'^{i\bar{i}})^{1/2}   (g'^{j\bar{j}})^{1/2}   (g'^{k\bar{k}})^{1/2}(\phi_{\bar{i}j\bar{k}\bar{\alpha}}- \sum_p (\phi_{p\bar{i}\bar{\alpha}}  \phi_{\bar{p}j\bar{k}} + \phi_{p\bar{i}\bar{k}} \phi_{\bar{p}j\bar{\alpha}})   g'^{p\bar{p}}) ) 
\]

According to (A.9) in page (406) of \cite{y} we know that

\begin{theo}\cite{y}\label{yau406}
$\Delta_{g'}\Psi =\sum _{ijk\alpha s} g'^{\alpha\bar{\alpha}} |A_{ijk\alpha  s}|^2+R$ for some remainder term $R$.
\end{theo}

We note that the first term on the right hand side $\sum _{ijk\alpha s}g'^{\alpha\bar{\alpha}} |A_{ijk\alpha  s}|^2$ and thus the remainder $R$ are both  independent of the choice of particular coordinates system.
\\

In the ordinary non-degenerate case $R$ can be approximated as
\[
|R|\leq C_1+ C_2\sqrt{\Psi} +C_3\Psi
\]

for some constants $C_1,C_2$ and $C_3$. But in the degenerate case we consider an open neighborhood  $U_p$ of some point $p\in D $   such that   in this neighborhood the coefficients $C_1,C_2$ and $C_3$  can only be replaced by terms of the form $\frac{C_i}{|S|^{a_i}}$ for $i=1,,3$  and  we want  to identify their corresponding orders of  singularities $a_i$'s.

In order to do this we first recall the relation (A.1) in page (403) of (\cite{y}):

\begin{equation}\label{403a1}
\begin{split}
\phi_{i\bar{j}k\bar{\beta}\alpha}=\phi_{i\bar{\beta}\alpha \bar{j}k}+&\big( \sum_{p} \phi_{i\bar{p}} R^{\bar{p}} _{\bar{\beta} \bar{j} \alpha}-\sum_p \phi_{p\bar{\beta}} R^p_{i\alpha\bar{j}} \big)_k\\
\quad & +\big( \sum_{p} \phi_{i\bar{p}} R^{\bar{p}} _{\bar{j}\bar{\beta}  k}-\sum_p \phi_{p\bar{j}} R^p_{ik\bar{\beta}} \big)_{\alpha}
\end{split}
\end{equation}

Here $R^{j} _{ik\bar{l}}$ denotes the curvature tensor associated to the metric $g_{reg}$ which in the above mentioned coordinates  is given by
\begin{equation}\label{curyau1}
R^{j} _{ik\bar{l}}=-\sum_p g_{reg} ^{j\bar{p}} \frac{\partial ^2 g_{reg, i\bar{p}} }{\partial z^k \partial \bar{z}^l}+\sum_{p.q}g_{reg} ^{p\bar{q}}\frac{\partial g_{reg, p\bar{j}}}{\partial \bar{z}^l} 
\end{equation}





Thus we have
\begin{equation}\label{5tho}
\begin{split}
\phi_{i\bar{j}k\bar{\alpha}\alpha}=\phi_{i\bar{\alpha}\alpha \bar{j}k}+& \sum_{p} \phi_{i\bar{p}k} R^{\bar{p}} _{\bar{\alpha} \bar{j} \alpha}-\sum_p \phi_{p\bar{\alpha}k} R^p_{i\alpha\bar{j}} \\
\quad & + \sum_{p} \phi_{i\bar{p}\alpha} R^{\bar{p}} _{\bar{j}\bar{\alpha}  k}-\sum_p \phi_{p\bar{j}\alpha} R^p_{ik\bar{\alpha}}\\
\quad & +   \sum_{p} \phi_{i\bar{p}} (R^{\bar{p}} _{\bar{\alpha} \bar{j} \alpha})_k -\sum_p \phi_{p\bar{\alpha}} (R^p_{i\alpha\bar{j}})_k  + \sum_{p} \phi_{i\bar{p}} (R^{\bar{p}} _{\bar{j}\bar{\alpha}  k})_{\alpha}-\sum_p \phi_{p\bar{j}} (R^p_{ik\bar{\alpha}})_{\alpha}
\end{split}
\end{equation}

based on relation (A6-A8) in page (405-407)  of (\cite{y}) the remainder  $R$ in theorem (\ref{yau406}) contains the terms like
\begin{equation}\label{t1y}
\begin{split}
&T_1=g'^{\alpha\bar{\alpha}}g'^{i\bar{i}}g'^{j\bar{j}} g'^{k\bar{k}}\big[(\sum_{p} \phi_{i\bar{p}k} R^{\bar{p}} _{\bar{\alpha} \bar{j} \alpha}-\sum_p \phi_{p\bar{\alpha}k} R^p_{i\alpha\bar{j}}) \phi_{\bar{i}j\bar{k}}\\
\quad & + (\sum_{p} \phi_{i\bar{p}\alpha} R^{\bar{p}} _{\bar{j}\bar{\alpha}  k}-\sum_p \phi_{p\bar{j}\alpha} R^p_{ik\bar{\alpha}})\phi_{\bar{i}j\bar{k}} \big ]
\end{split}
\end{equation}

\begin{equation}\label{t11y}
\begin{split}
&T'_1=g'^{\alpha\bar{\alpha}}g'^{i\bar{i}}g'^{j\bar{j}} g'^{k\bar{k}}\big[ (\sum_{p} \phi_{i\bar{p}} (R^{\bar{p}} _{\bar{\alpha} \bar{j} \alpha})_k -\sum_p \phi_{p\bar{\alpha}} (R^p_{i\alpha\bar{j}})_k)\phi_{\bar{i}j\bar{k}}  \\
\quad & +  (\sum_{p} \phi_{i\bar{p}} (R^{\bar{p}} _{\bar{j}\bar{\alpha}  k})_{\alpha}-\sum_p \phi_{p\bar{j}} (R^p_{ik\bar{\alpha}})_{\alpha})\phi_{\bar{i}j\bar{k}} \big ]
\end{split}
\end{equation}

\begin{equation}\label{t23y}
T_2=\sum g'^{i\bar{i}}g'^{j\bar{j}}g'^{q\bar{q}}g'^{k\bar{k}} F_{j\bar{q}}|\phi_{i\bar{j}k}|^2, \hspace{0.3cm} T_3=\sum  g'^{i\bar{i}}g'^{j\bar{j}}g'^{k\bar{k}}\phi_{\bar{i}j\bar{k}} F_{i\bar{j}k}
\end{equation}

Since we have assumed  that  
$(\frac{\partial}{\partial z_1}(x),...,\frac{\partial}{\partial z_n}(x))=(V_1 (x),...,V_n(x))$ then  according to the relation (\ref{m'm'1})
in the appendix (\ref{app7})    we have
\[
 g'^{n\bar{n}}\leq \frac{\mathscr{M}' _1}{|S|^2} 
\]
 where $\mathscr{M}' _1$ depends on $G$ and $g_{reg}$. Also

\[
\begin{split}
| g'^{\alpha\bar{\alpha}}g'^{i\bar{i}}g'^{j\bar{j}} g'^{k\bar{k}} \phi_{i\bar{p}k} \phi_{\bar{i}j\bar{k}}R^{\bar{p}} _{\bar{\alpha} \bar{j} \alpha} |\leq & C'_{pj\alpha} \sqrt{g'^{j\bar{j}}} g'^{\alpha\bar{\alpha}}(\sqrt{g'^{i\bar{i}}}\sqrt{g'^{k\bar{k}}}|\phi_{i\bar{p}k}|)(\sqrt{g'^{i\bar{i}}}\sqrt{g'^{j\bar{j}}}\sqrt{g'^{k\bar{k}}} |\phi_{i\bar{j}\bar{k}}|) \\
\quad \leq &\frac{C_{pj\alpha}}{|S|^3}\Psi \\
\end{split}
\]

where $C'_{pj\alpha}$ is an upper bound for $|R^{\bar{p}} _{\bar{\alpha} \bar{j} \alpha}|$ and $C_{pj\alpha}$ is such that
\[
C'_{pj\alpha}\sqrt{g'_{p\bar{p}}} \sqrt{g'^{j\bar{j}}} g'^{\alpha\bar{\alpha}}\leq \frac{C_{pj\alpha}}{|S|^3}
\]

By repeating the same argument for all the other terms  of (\ref{t1y}) we find that
\begin{equation}\label{t1ue}
|T_1|\leq \frac{C_{\ref{t1ue}}}{|S|^3}\Psi
\end{equation}

where $C_{\ref{t1ue}}$ is a constant which only depends on $g_{reg}$ and $G$ on the right side of (\ref{ma3rd}).

  By similar arguments for $T'_1, T_2$ and $T_3$ we can prove that



 
\begin{prop}\label{remainder}
There exists constants $A_1, A_2$ and $A_3$  such that 
\[
|R|\leq \frac{A_1\Psi +A_2 \sqrt{\Psi} +A_3}{|S|^3}
\]

where $R$ denotes the remainder term in theorem (\ref{yau406}) and the constant $A_1, A_2$  and $A_3$ depend only on   $g_{reg}$ and $G$ on the right side of (\ref{ma3rd}).
\end{prop}

 We recall that  according to appendix (\ref{app7}) we also have
\begin{equation}\label{rel193}
(g'_{i\bar{i}})^{1/2}=\|V_i\|_{g'}\geq \mathscr{M} '\hspace{0.5cm} \text{ for } i=1,...,n-1
\end{equation}

Now we set 

\[
\begin{split}
II_{ijk\alpha}= g'^{i\bar{i}}  g'^{j\bar{j}} g'^{k\bar{k}} & \bigg [\phi_{\bar{i} j\bar{k}}\big (\phi_{i\bar{j}k\bar{\alpha}}-\sum_p \phi_{i\bar{p}k}\phi_{p\bar{j}\bar{\alpha}} g'^{p\bar{p}}\big )\\
\quad & + \phi_{i\bar{j}k} \big (\phi_{\bar{i}j\bar{k}\bar{\alpha}}- \sum_p (\phi_{p\bar{i}\bar{\alpha}}  \phi_{\bar{p}j\bar{k}} + \phi_{p\bar{i}\bar{k}} \phi_{\bar{p}j\bar{\alpha}}\big ) g'^{p\bar{p}} ) ) \bigg ]
\end{split}
\]

and
\[
\begin{split}
I_{ijk\alpha}= \sum_{p,q} \bigg [&-2 g'^{i\bar{i}}  g'^{q\bar{q}}  g'^{j\bar{j}}  g'^{k\bar{k}}  \phi_{q\bar{i}\bar{\alpha}} \phi_{i\bar{j}k} \phi_{\bar{q}j\bar{k}}\\
\quad & -  g'^{i\bar{i}}  g'^{j\bar{j}}  g'^{q\bar{q}}  g'^{k\bar{k}}   \phi_{j\bar{q}\bar{\alpha}}\phi_{i\bar{j}k}\phi_{\bar{i}q\bar{k}}\\
\quad & + g'^{i\bar{i}}  g'^{j\bar{j}}  g'^{k\bar{k}}  g'^{p\bar{p}}  \phi_{\bar{i} j\bar{k}}\phi_{i\bar{p}k}\phi_{p\bar{j}\bar{\alpha}}  \\
\quad & + g'^{i\bar{i}}  g'^{j\bar{j}}  g'^{k\bar{k}}  g'^{p\bar{p}}\phi_{i\bar{j}k}   (\phi_{p\bar{i}\bar{\alpha}}  \phi_{\bar{p}j\bar{k}} + \phi_{p\bar{i}\bar{k}} \phi_{\bar{p}j\bar{\alpha}}) \bigg ] \\
\end{split}
\]

Since in the last term we can interchange partial derivatives in the last line to get $ \phi_{p\bar{i}\bar{k}} \phi_{\bar{p}j\bar{\alpha}}= \phi_{\bar{i}p\bar{k}} \phi_{j\bar{p}\bar{\alpha}}$ hence the second line eliminates the last terms in the last line.
So  it is easy to see that 

\begin{lemma}\label{isimp}

\begin{equation}\label{lemijk}
\begin{split}
I_{ijk\alpha}=&
  \sum_p g'^{i\bar{i}}  g'^{j\bar{j}}  g'^{k\bar{k}}  g'^{p\bar{p}}\phi_{\bar{i} j\bar{k}}\phi_{i\bar{p}k}\phi_{p\bar{j}\bar{\alpha}} \\
\quad & -\sum_q  g'^{i\bar{i}}  g'^{q\bar{q}}  g'^{j\bar{j}}  g'^{k\bar{k}}  \phi_{q\bar{i}\bar{\alpha}} \phi_{i\bar{j}k} \phi_{\bar{q}j\bar{k}}
\end{split}
\end{equation}

\end{lemma}

According to (\cite{y}) page 404 we   have
\begin{lemma}[\cite{y}]\label{si12}
\[
\Psi_{\bar{\alpha}}=\sum_{ijk}  (I_{ijk\bar{\alpha}}+II_{ijk\bar{\alpha}})
\]
\end{lemma}

We  also set

\begin{equation}\label{12ab}
II_{ijk\alpha}=\sum_{s } A_{ijk\alpha s} B_{ijks}
\end{equation}

where

\[
A_{ijk\alpha p 1}=  (g'^{i\bar{i}} )^{1/2} (g'^{j\bar{j}} )^{1/2} (g'^{k\bar{k}} )^{1/2} ( (\phi_{i\bar{j}k\bar{\alpha}}-\sum_p \phi_{i\bar{p}k}\phi_{p\bar{j}\bar{\alpha}}g'^{p\bar{p}} )
\]

\[
A_{ijk\alpha p 2}=    (g'^{i\bar{i}} )^{1/2} (g'^{j\bar{j}} )^{1/2} (g'^{k\bar{k}} )^{1/2} (\phi_{\bar{i}j\bar{k}\bar{\alpha}}- \sum_p (\phi_{p\bar{i}\bar{\alpha}}  \phi_{\bar{p}j\bar{k}} + \phi_{p\bar{i}\bar{k}} \phi_{\bar{p}j\bar{\alpha}}) g'^{p\bar{p}} ) ) 
\]

\[
B_{ijk1} =    (g'^{i\bar{i}} )^{1/2} (g'^{j\bar{j}} )^{1/2} (g'^{k\bar{k}} )^{1/2} \phi_{\bar{i} j\bar{k}}
\]

 and 

\[
B_{ijk2}=   (g'^{i\bar{i}} )^{1/2} (g'^{j\bar{j}} )^{1/2} (g'^{k\bar{k}} )^{1/2} \phi_{i\bar{j}k}
\]

\begin{lemma}\label{ineqabsi}
\[
  |\sum_{ijks} A_{ijk\alpha s} B_{ijks} |^2\leq 2  \Psi \sum_{ijks} |A_{ijk\alpha   s } |^2
\]
\end{lemma}

\begin{proof}
It suffices to apply  Cauchy-Schwarz 
\[
\begin{split}
  |\sum_{ijks} A_{ijk\alpha s} B_{ijks} |^2 &\leq     (\sum_{ijks} |B_{ijks}| ^2)(\sum_{ijks}| A_{ijk\alpha  s}| ^2 ) =  2 \Psi  (\sum_{ijks}| A_{ijk\alpha  s}| ^2 ) 
\end{split}
\]
\end{proof}


  In the formula (\ref{lemijk}) for $I_{ijk\alpha}$ given in  lemma (\ref{isimp}) if $i=j=k=p=q$  we get
\[
\phi_{\bar{i}j\bar{k}}\phi_{i\bar{p}k}\phi_{p\bar{j}\bar{\alpha}}=\phi_{q\bar{i}\bar{\alpha}}\phi_{i\bar{j}k}\phi_{\bar{q}j\bar{k}}
\]
  In particular at least  one of the terms $\phi_{\bar{i}j\bar{k}}$ and  $\phi_{i\bar{p}k}$  multiplied in the first sum of formula (\ref{lemijk}) must be different from $\phi_{nnn}$ and the same is true for the tow terms $\phi_{i\bar{j}k}$  and $\phi_{\bar{q}j\bar{k}}$  in the second summation.


\begin{lemma}\label{lem260}
Over the  open set $U_p$ we have

\begin{equation}\label{lem262}
 (g'^{\alpha\bar{\alpha}} )^{1/2}I_{ijk\alpha}\leq C_{\ref{lem262}}\Psi \times \bigg [\sum_{p,q,l} g'^{p\bar{p}} g'^{q\bar{q}}  \phi_{p\bar{q}\bar{l}}\phi_{q\bar{p}l}\bigg ]^{1/2}
\end{equation}
where $C_{\ref{lem262}}=\frac{\tilde{C_{\ref{lem262}}}}{\mathscr{M}'}$and $\tilde{C_{\ref{lem262}}}$ depends on $n$. The constant $\mathscr{M}'$ is defined in the appendix (\ref{app7}).

We also have 

\begin{equation}\label{lem263}
(g'^{n\bar{n}})^{1/2}I_{ijkn}\leq C_{\ref{lem263}}\Psi^{1/2} \times\bigg [ \sum_{p,q,l} g'^{p\bar{p}} g'^{q\bar{q}}  (\phi)_{p\bar{q}\bar{l}}(\phi)_{q\bar{p}l}\bigg ]
\end{equation}.

where $C_{\ref{lem263}}=\frac{\tilde{C_{\ref{lem263}}}}{(\mathscr{M}')^2}$ and $\tilde{C_{\ref{lem263}}}$ only depends on $n$.
\end{lemma}

\begin{proof}


 For $a,b,c\in \{1,\bar{1},...,n,\bar{n}\}$ we set
\begin{equation}\label{siabc1}
\psi_{abc}:=(g'^{a\bar{a}})^{1/2}(g'^{b\bar{b}})^{1/2}(g'^{c\bar{c}})^{1/2}  \phi_{abc}
\end{equation}

 Obviously we have
\begin{equation}\label{siabc2}
|\psi_{abc}|\leq \Psi^{1/2}
\end{equation}

Then by lemma (\ref{lemijk})

\begin{equation}\label{128}
(g'^{\alpha\bar{\alpha}})^{1/2}I_{ijk\alpha}=\sum_{p}\psi_{\bar{i}j\bar{k}}\psi_{i\bar{p}k}\psi_{p\bar{j}\alpha}-\sum_{q}\psi_{q\bar{i}\bar{\alpha}}\psi_{i\bar{j}k}\psi_{\bar{q}j\bar{k}}
\end{equation}

Now if $a,b,c$ are not identical we can find  one of the letters $a,b,c$ for instance $a$ such that $a\neq n$. Therefore according to the relation (\ref{rel193})   we have $g'^{a\bar{a}} \leq \frac{1}{(\mathscr{M}')^2}$ from which it follows   that 
\begin{equation}\label{136}
\begin{split}
|\psi_{abc}|\leq  \frac{1}{\mathscr{M}'} (g'^{b\bar{b}})^{1/2}(g'^{c\bar{c}})^{1/2}  |\phi_{abc}| & \leq  \frac{1}{\mathscr{M}'}\sum_{p,q,l} (g'^{p\bar{p}})^{1/2}(g'^{q\bar{q}})^{1/2}  |\phi_{p\bar{q}\bar{l}}|\\
\quad &\leq \frac{\tilde{C}_{\ref{136}}}{\mathscr{M}'}\bigg [ \sum_{p,q,l}  (g'^{p\bar{p}})(g'^{q\bar{q}}) \phi_{p\bar{q}\bar{l}}\phi_{q\bar{p}l}\bigg ] ^{1/2}
\end{split}
\end{equation}

 where  $\tilde{C}_{\ref{136}}$ depends  on $n$  and the second line  results  from Cauchy-Schwarz inequality.

Since we know that at least one of the letters $\{i,j,k,p\}$ and at least one of the letters $\{i,j,k,q\}$ are different from $n$   thus from (\ref{128}),(\ref{136}) and the inequality (\ref{siabc2}) we obtain
\[
(g'^{\alpha\bar{\alpha}})^{1/2}I_{ijk\alpha}\leq C_{\ref{lem262}}\Psi \times\bigg [ \sum_{p,q,l} g'^{p\bar{p}} g'^{q\bar{q}}  (\phi)_{p\bar{q}\bar{l}}(\phi)_{q\bar{p}l}\bigg ]^{1/2}
\].
 
where $C_{\ref{lem262}}=\frac{\tilde{C_{\ref{lem262}}}}{\mathscr{M}'}$and $\tilde{C_{\ref{lem262}}}$ depends  on $n$.

Also if $\alpha=n$ then at least one of the letters $\{i,j,k,p\}$ and at least one of the letters $\{i,j,k,q\}$ are different from $n$.
Since all these letters are repeated twice we can derive the following inequality

\[
(g'^{n\bar{n}})^{1/2}I_{ijkn}\leq C_{\ref{lem263}}\Psi^{1/2} \times\bigg [ \sum_{p,q,l} g'^{p\bar{p}} g'^{q\bar{q}}  (\phi)_{p\bar{q}\bar{l}}(\phi)_{q\bar{p}l}\bigg ]
\].

where $C_{\ref{lem263}}=\frac{\tilde{C_{\ref{lem263}}}}{(\mathscr{M}')^2}$ and $\tilde{C_{\ref{lem263}}}$ only depends on $n$.
 \end{proof}


In order to prove the theorem (\ref{thirdorder}) we need the following computation
\[
\begin{split}
\Delta_{g'} (\log(\eta \Psi+\epsilon))=& \frac{(\Delta_{g'}\eta) \Psi +\eta (\Delta_{g'} \Psi)}{\eta \Psi +\epsilon}+\sum g'^{\alpha\bar{\beta}}\frac{\eta_{\bar{\beta}} \Psi_{\alpha} +\eta_{\alpha} \Psi_{\bar{\beta}}}{(\eta\Psi+\epsilon)}\\
\quad &\sum - g'^{\alpha\bar{\beta}}\frac{\eta_{\bar{\beta}} \eta_{\alpha} \Psi^2+\eta_{\bar{\beta}} \eta \Psi\Psi_{\alpha}+\eta\eta_{\alpha}\Psi\Psi_{\bar{\beta}} +\eta^2\Psi_{\alpha}\Psi_{\bar{\beta}} }{(\eta\Psi+\epsilon)^2}\\
\quad &= \frac{(\Delta_{g'}\eta) \Psi +\eta (\Delta_{g'} \Psi)}{\eta \Psi +\epsilon}+\sum g'^{\alpha\bar{\beta}}\frac{\eta_{\bar{\beta}} \Psi_{\alpha} +\eta_{\alpha} \Psi_{\bar{\beta}}}{(\eta\Psi+\epsilon)}\\
\quad &\sum - g'^{\alpha\bar{\beta}}\frac{ (\eta \Psi_{\alpha}+\eta_{\alpha}\Psi)(\eta \Psi_{\bar{\beta} }+\eta_{\bar{\beta}} \Psi) }{(\eta\Psi+\epsilon)^2}\\ 
\end{split}
\]

where  $\epsilon $ is a positive constant which will be determined later  and 
\[
\eta=\chi|S|^4
\]

Here we  fix a cut-off function 
\begin{equation}\label{chitild}
\tilde{\chi} :\mathbb{C}^n \rightarrow [0,1]
\end{equation}
such that $\supp (\tilde{\chi})\subset B_1 (0)$, where $B_r(0)$ denotes  the ball of radius $r$ centered at $0$  in $\mathbb{C}^n$ and such that $\tilde{\chi}^{-1} (\{1\}) =B_{1/2}(0) $. We  assume that the coordinates system $(\tilde{w}_1,...,\tilde{w}_{n-1}\tilde{z})$ over  $U_p$ is such that   it contains a ball  $B_{\delta}(0))$.  Then  we set 
\begin{equation}   \label{chitild}
\chi (\tilde{w}_1,...,\tilde{w}_{n-1},\tilde{z})=\tilde{\chi} (\frac{\tilde{w}_1}{\delta},...,\frac{\tilde{w}_{n-1}}{\delta} , \frac{\tilde{z}}{\delta}) 
\end{equation}

We assume that $\delta$ is so small that the coordinates system $(w_1,...,w_{n-1}, z)$ compatible with the relation (\ref{eqlits})
in which our computation is carried out can be taken in such a way that the $z$  axis passing through $x$ coincide with the $z$-axis with respect to a coordinates system over $\supp \chi$ which is  still denoted by $(w_1,...,w_{n-1},z)$  and which satisfies $D\cap \supp \chi =\{z=0\} $. The value of $\delta $ depends on the angle $\theta_n$ defined in the appendix (\ref{app7}) hence on $\mathscr{M}'$. Also the coordinates sytem above can be obtained by a linear transformation on the initial  coordinates on $U_p$.  This hypothesis  on $\delta$ is used  in the proof of lemma (\ref{lem20nn}) to  approximate $\frac{|\eta_z|^2}{\eta|\eta_{z\bar{z}}|}$.
  We have

 \[
\Delta_{g'} ( \log( \eta\Psi +\epsilon))=\mathcal{E}_1+\mathcal{E}_2
\]

where 
\[
\begin{split}
&\mathcal{E}_1  = \frac{\sum_{\alpha} g'^{\alpha\bar{\alpha}}\bigg [\sum_{ijks}\eta |A_{ijk\alpha s}|^2 +2\Re    \big (\sum_{ijk}\eta_{\alpha}I_{ijk\alpha} + \sum_{ijks}\eta_{\alpha} A_{ijk\alpha  s}B_{ijks}\big )\bigg] }{\eta\Psi  +\epsilon}\\
\quad  & -\frac{\sum_{\alpha} g'^{\alpha\bar{\alpha}}\bigg |\big ( \sum_{ijks} \eta A_{ijk\alpha s} B_{ijks} +\sum_{ijk}\eta I_{ijk\alpha}+\sum_{ijk}\eta_{\bar{\alpha}}|\phi _{i\bar{j}k}|^2 g'^{i\bar{i}}g'^{j\bar{j}}g'^{k\bar{k}} \big )\bigg|^2}{(\eta\Psi +\epsilon)^2} \\
\end{split}
\]

and

\begin{equation}\label{mce2}
\mathcal{E}_2 = \frac{(\Delta_{g'} \eta) \big (\sum |\phi _{i\bar{j}k}|^2 g'^{i\bar{i}}g'^{j\bar{j}}g'^{k\bar{k}}\big ) }{\eta\Psi +\epsilon}+\frac{\eta R}{\eta\Psi+\epsilon}
\end{equation}

Here $R$ is the remainder defined in theorem (\ref{yau406})  and  we are using lemma  (\ref{si12}), relation (\ref{12ab}) and theorem  (\ref{yau406}).

 Applying  lemma (\ref{ineqabsi}) and  Cauchy Schwarts inequality we obtain

\begin{equation}\label{no2}
\begin{split}
&\mathcal{E}_1  \\
\quad & \geq \frac{  \sum_{\alpha}  g'^{\alpha\bar{\alpha}} \bigg [\sum_{ijks}  \eta| A_{ijk\alpha s}|^2 -\sqrt{2} |\eta_{\alpha}|\Psi^{1/2}  ( \sum_{ijks} |A_{ijk\alpha s}|^2 )^{1/2}}{(\eta\Psi +\epsilon)^2}\\
\quad &\hspace{6cm} - \frac{\sum_{ijk}2 |\eta_{\bar{\alpha}} ||I_{ijk\alpha} |   \bigg ] \times (\eta\Psi+\epsilon) }{(\eta\Psi +\epsilon)^2}\\
\quad & -\frac{ \sum_{\alpha}  g'^{\alpha\bar{\alpha}}C_{\ref{no2}}\bigg [ 2\eta ^2 \sum_{ijks}B_{ijks} ^2\sum_{ijks} |A_{ijk\alpha s}|^2 +\sum_{ijk} \eta ^2  |I_{ijk\alpha}|^2 +|\eta _{\bar{\alpha}}|^2 |\Psi |^2 \bigg] }{(\eta \Psi +\epsilon) ^2} \\
\quad & = \frac{  \sum_{\alpha}\bigg [ g'^{\alpha\bar{\alpha}}  \eta\big [\epsilon- (2C_{\ref{no2}}-1)\eta \Psi \big ]  \sum_{ijks} |A_{ijk\alpha s}|^2 -\sqrt{2}(\epsilon+\eta\Psi)|\eta_{\alpha}|\Psi ^{1/2} (\sum_{ijks} | A_{ijk\alpha s}|^2 )^{1/2} \bigg ]}{(\eta\Psi+\epsilon)^2}\\
\quad &-\frac{\sum_{\alpha} g'^{\alpha\bar{\alpha}}\bigg [C_{\ref{no2}}\eta^2\sum_{ijk} |I _{ijk\alpha}|^2 + 2(\epsilon+\eta \Psi)\sum_{ijk}|\eta_{\bar{\alpha}}| |I_{ijk\alpha}| +C_{\ref{no2}}|\eta_{\bar{\alpha}}|^2 \Psi^2 \bigg]  }{(\eta\Psi +\epsilon)^2} 
\end{split}
\end{equation}

 The constant $C_{\ref{no2}}$ depends on the number of the terms.
We represent  the last expression as a quadratic function,
\[
T=\sum_{\alpha}  g'^{\alpha\bar{\alpha}} [a_{\alpha}x_{\alpha}^2+b_{\alpha}x_{\alpha}+c_{\alpha}]
\]

where
\begin{equation}\label{aslfa}
x_{\alpha}=(\sum_{ijks} |A_{ijk\alpha s}|^2)^{1/2}
\end{equation}

\begin{equation}\label{aalfa}
a_{\alpha}=\frac{\eta[\epsilon-(2C_{\ref{no2}}-1)\eta\Psi]}{(\eta\Psi+\epsilon)^2}
\end{equation}

\begin{equation}\label{balfa}
b_{\alpha}=\frac{-\sqrt{2}|\eta_{\alpha}| \Psi^{1/2}}{(\eta\Psi+\epsilon)}
\end{equation}

\begin{equation}\label{calfa}
c_{\alpha}=\frac{-C_{\ref{no2}}\eta^2 \bigg [ \sum_{ijk}|I_{ijk\alpha}|^2 \bigg ]+2(\epsilon+\eta\Psi) |\eta_{\bar{\alpha}} |\bigg [ \sum_{ijk}|I_{ijk\alpha}|\bigg ]+C_{\ref{no2}}|\eta_{\bar{\alpha}}|^2\Psi ^2}{(\eta\Psi+\epsilon)^2}
\end{equation}

We have thus proved 

\begin{lemma}\label{belakhare}
\[
\mathcal{E}_1 \geq T
\]
\end{lemma}

Minimizing $T$  leads to 

\begin{equation}\label{tavval}
\mathcal{E}_1 \geq \sum_{\alpha}  g'^{\alpha\bar{\alpha}}(-\frac{b_{\alpha} ^2}{4a_{\alpha}} +c_{\alpha}) 
\end{equation}

According to  (\ref{aalfa}) and (\ref{balfa}) 

\begin{equation}\label{balaal}
\frac{b_{\alpha}^2}{4a_{\alpha}}= \frac{2|\eta_{\alpha}|^2\Psi}{4\eta[\epsilon-(2C_{\ref{no2}}-1)\eta\Psi]}
\end{equation}

We  define $\eta_{\tau}$ by

\begin{equation}
\eta_{\tau}= \chi^{\tau} |S|^4
\end{equation}

where $\tau$ is  a small constant that will be determined. We then  set
\begin{equation}\label{epskap}
\epsilon = \kappa \| \eta_{1-\tau} \Psi\|_{\infty}
\end{equation}

where  $\kappa$ is a large constant that depends on the open set $U_p$ and  will be discussed.

 \begin{lemma}\label{cor4}

\begin{equation}\label{epsi11}
\begin{split}
\mathcal{E}_1\geq &\frac{C'_{\ref{epsi11}}\chi^{\tau}}{\kappa^2}\bigg [ \sum_{p,q,l} g'^{p\bar{p}}g'^{q\bar{q}} (\phi)_{p\bar{q}\bar{l}}(\phi)_{q\bar{p}l}\bigg ]\\
\quad &\frac{C''_{\ref{epsi11}}\chi ^{\tau}}{\kappa}  \bigg [ \sum_{p,q,l} g'^{p\bar{p}}g'^{q\bar{q}} (\phi)_{p\bar{q}\bar{l}}(\phi)_{q\bar{p}l}\bigg ]^{1/2} \\
\quad &+C_{\ref{no2}}\sum_{\alpha}\frac{ g'^{\alpha\bar{\alpha}}|\eta_{\bar{\alpha}}|^2\Psi^2}{(\eta\Psi+\epsilon)^2}
 - \sum_{\alpha}g'^{\alpha\bar{\alpha}}\frac{|\eta_{\alpha}|^2\Psi}{2\eta [\epsilon-(2C_{\ref{no2}}-1)\eta\Psi]}
\end{split}
\end{equation}


for some constants $C'_{\ref{epsi11}}$ and $C'_{\ref{epsi11}}$    which  depend on $g_{reg}$ and the function $G$ on the r.h.s. of (\ref{ma3rd}), $\delta $, and $\mathscr{M}'$.  
Here $\delta$  is determined in relation (\ref{chitild}) and $\mathscr{M}'$ in the appendix (\ref{app7})
\ 

\end{lemma}

 \begin{proof}

To prove the above lemma   we need to estimate  $\sum_{\alpha} g'^{\alpha\bar{\alpha}}|c_{\alpha}|$ on the right hand side of (\ref{tavval}). By using lemma (\ref{lem260}) we can derive upper estimate for $\sum_{\alpha}  g'^{\alpha\bar{\alpha}}|c_{\alpha}|$ where $c_{\alpha}$ is defined by (\ref{calfa}):

\begin{equation}\label{calfa}
c_{\alpha}=\frac{-C_{\ref{no2}}\eta^2 \bigg [\sum_{ijk}|I_{ijk\alpha}|^2 \bigg ]+2(\epsilon+\eta\Psi) |\eta_{\bar{\alpha}} |\bigg [ \sum_{ijk}|I_{ijk\alpha}|\bigg ]+C_{\ref{no2}}|\eta_{\bar{\alpha}}|^2\Psi ^2}{(\eta\Psi+\epsilon)^2}
\end{equation}

\begin{equation}\label{tdoyom}
\begin{split}
\sum_{\alpha} g'^{\alpha\bar{\alpha}}&|c_{\alpha}|\leq  \frac{C_{\ref{tdoyom}}\eta^2 \Psi^2\times \bigg ( \sum_{p,q,l} g'^{p\bar{p}} g'^{q\bar{q}}  \phi_{p\bar{q}\bar{l}}\phi_{q\bar{p}l}\bigg  ) }{(\eta\Psi+\epsilon)^2}\\
\quad&+ \frac{2C_{\ref{tdoyom}}'(\epsilon+\eta\Psi) \bigg (\sum_{\alpha\neq n}g'^{\alpha\bar{\alpha}} |\eta_{\bar{\alpha}}|^2\bigg )^{1/2}\Psi\times \bigg (\sum_{p,q,l} g'^{p\bar{p}}g'^{q\bar{q}} \phi_{p\bar{q}\bar{l}}\phi_{q\bar{p}l} \bigg ) ^{1/2} }{(\eta\Psi+\epsilon)^2}\\
\quad &   \frac{2C_{\ref{tdoyom}}''(\epsilon+\eta\Psi) \bigg ( g'^{n\bar{n}} |\eta_{\bar{n}}|^2\bigg )^{1/2}(\Psi )^{1/2}\times \bigg (\sum_{p,q,l} g'^{p\bar{p}}  g'^{p\bar{p}} \phi_{p\bar{q}\bar{l}}\phi_{q\bar{p}l} \bigg )  }{(\eta\Psi+\epsilon)^2}\\
\quad &+ \sum_{\alpha}C_{\ref{no2}} \frac{ g'^{\alpha\bar{\alpha}}|\eta_{\bar{\alpha}}|^2\Psi^2}{(\eta\Psi+\epsilon)^2} \bigg ]\\
\quad
\end{split}
\end{equation}

Where $C_{\ref{tdoyom}}=\frac{\tilde{C}_{\ref{tdoyom}}}{(\mathscr{M}')^2}$ ,   $C'_{\ref{tdoyom}}=\frac{\tilde{C}'_{\ref{tdoyom}}}{\mathscr{M}'}$ and  $C''_{\ref{tdoyom}}=\frac{\tilde{C}''_{\ref{tdoyom}}}{(\mathscr{M}')^2}$  where $\tilde{C}_{\ref{tdoyom}}$,  $\tilde{C}'_{\ref{tdoyom}}$  and $\tilde{C}''_{\ref{tdoyom}}$  depend on $n$. Here for the first two lines we are applying the inequality (\ref{lem262}) and for the third line we use inequality $\ref{lem263}$.

Hence to prove the above lemma it suffices to verify the following two inequalities

\begin{equation}\label{kappa1}
\frac{\eta ^2 \Psi ^2}{(\eta \Psi + \epsilon )^2}\leq \frac{\chi^{2\tau} (\|\eta_{1-\tau} \Psi\|_{\infty})^2}{ \kappa^2 (\|\eta_{1-\tau} \Psi\|_{\infty})^2}\leq \frac{\chi ^{2\tau}}{(\kappa )^2}
\end{equation}



and 
 
\begin{equation}\label{hend}
\begin{split}
( g'^{n\bar{n}} &|\eta_{z}|^2)^{1/2} \Psi^{1/2}=[\eta_{1-\tau}\Psi]^{1/2}\chi^{\frac{\tau-1}{2}}|S|^{-2}\bigg [ g'^{n\bar{n}} \bigg |\chi_{z}|S|^4 +\chi (|S|^4)_{z}\bigg |^2\bigg ]^{1/2}\\
\quad &\leq [\eta_{1-\tau}\Psi]^{1/2}\chi^{\frac{\tau-1}{2}}|S|^{-2}\times
\bigg  [\mathscr{M}'_1\frac{ \bigg |\chi_{z}|S|^4 +\chi (|S|^4)_{z}\bigg |^2}{|S|^2} \bigg ]^{1/2}\\
\quad &\leq  4[\eta_{1-\tau}\Psi]^{1/2}\chi^{\frac{1+\tau}{2}}\times \sqrt{\mathscr{M}'_1}(|S|_z) +  [\eta_{1-\tau}\Psi]^{1/2}\chi^{\frac{1+\tau}{2}}\times \sqrt{\mathscr{M}'_1}|\chi|_z| |S|  \\ 
 \quad & \leq C_{\ref{hend} } [\eta_{1-\tau}\Psi]^{1/2}\chi^{\frac{1+\tau}{2}}
\end{split}
\end{equation}

Here we are using the inequality $g'^{n\bar{n}} |S|^2\leq \mathscr{M}' _1$ deduced in lemma (\ref{m'm'1}) in appendix (\ref{app7}).
In $\chi_z$ a  term $\frac{1}{\delta}$  is obtained   which is controlled by  $|S|$. Hence the constant $C_{\ref{hend} }$ only depends on $G$ and $g_{reg}$. 

\begin{equation}\label{hend2}
\begin{split}
(\sum_{\alpha\neq n} g'^{\alpha\bar{\alpha}} &|\eta_{\alpha}|^2)^{1/2} \Psi=[\eta_{1-\tau}\Psi]\chi^{\tau}|S|^{-4}\bigg [\sum_{\alpha\neq n} g'^{\alpha\bar{\alpha}} |[\chi_{\alpha}|S|^4 +\chi (|S|^4)_{\alpha}]|^2\bigg ]^{1/2}\\
\quad &\leq [\eta_{1-\tau}\Psi]\chi^{\tau}|S|^{-4}\times \bigg  [\frac{1}{\mathscr{M}'}\sum_{\alpha\neq n} \bigg | [\chi_{\alpha}|S|^4 +\chi (|S|^4)_{\alpha}]\bigg |^2\bigg ]^{1/2}\\
\quad &\leq   [\eta_{1-\tau}\Psi]\chi^{\tau}|S|^{-4}\times \bigg [ \frac{1}{\sqrt{\mathscr{M}'}}\sum_{\alpha\neq n} \bigg |\chi_{\alpha}|S|^4 +\chi (|S|^4)_{\alpha} \bigg|  \bigg ] \\
\quad & \leq \bigg (\frac{C_{\ref{hend}}}{\delta\sqrt{\mathscr{M}'}}) \bigg )[\eta_{1-\tau}\Psi](\chi)^{\tau}
\end{split}
\end{equation}

 Here we apply the inequality $g'^{\alpha\bar{\alpha}}\leq \mathscr{M}' $fo $\alpha\neq n$  deduced from  lemma (\ref{m'm'1}) in appendix (\ref{app7}). The constant $C_{\ref{hend}}$ epends on the hermitian metric over the line bundle $L=[D]$.

    From the inequalities (\ref{kappa1}), (\ref{hend})  and (\ref{hend2})  it follows that 

\begin{equation}\label{epsi11}
\begin{split}
\mathcal{E}_1\geq & -\frac{C_{\ref{epsi11}}}{(\mathscr{M}' )^2}\big ( \frac{\chi^{2\tau}}{\kappa ^2}+\frac{\chi^{\frac{1+\tau}{2}}}{\kappa(\|\eta_{1-\tau}\Psi\|_{\infty})^{1/2}} \big ) \bigg [ \sum_{p,q,l} g'^{p\bar{p}}g'^{q\bar{q}} (\phi)_{p\bar{q}\bar{l}}(\phi)_{q\bar{p}l}\bigg ]\\
\quad &-\frac{C_{\ref{epsi11}}}{\kappa} \frac{\chi^{\tau}}{\delta\sqrt{\mathscr{M}'}}  \bigg [ \sum_{p,q,l} g'^{p\bar{p}}g'^{q\bar{q}} (\phi)_{p\bar{q}\bar{l}}(\phi)_{q\bar{p}l}\bigg ]^{1/2} \\
\quad &+C_{\ref{no2}}\sum_{\alpha}\frac{ g'^{\alpha\bar{\alpha}}|\eta_{\bar{\alpha}}|^2\Psi^2}{(\eta\Psi+\epsilon)^2}
 - \sum_{\alpha}g'^{\alpha\bar{\alpha}}\frac{|\eta_{\alpha}|^2\Psi}{2\eta [\epsilon-(2C_{\ref{no2}}-1)\eta\Psi]}
\end{split}
\end{equation}

for some constants $C_{\ref{epsi11}}$    which  only depend on $g_{reg}$ and the function $G$ on the r.h.s. of (\ref{ma3rd}).  

If $\tau<\frac{1}{2}$ we have $\chi^{\tau}\geq \chi^{2\tau}\geq \chi^{\frac{1+\tau}{2}}$ so

\[
\begin{split}
-\frac{C_{\ref{epsi11}}}{(\mathscr{M}' )^2}\big ( \frac{\chi^{2\tau}}{\kappa ^2}+\frac{\chi^{\frac{1+\tau}{2}}}{\kappa(\|\eta_{1-\tau}\Psi\|_{\infty})^{1/2}} \big )&\geq  -\frac{C_{\ref{epsi11}}\chi^{\tau}}{(\mathscr{M}' )^2}\big ( \frac{1}{\kappa ^2}+\frac{1}{\kappa(\|\eta_{1-\tau}\Psi\|_{\infty})^{1/2}} \big )\\
\quad & =\frac{C'_{\ref{epsi11}}\chi^{\tau}}{\kappa^2}
\end{split}
\]

So if we set 
\begin{equation}\label{c'220}
C'_{\ref{epsi11}}= -\frac{C_{\ref{epsi11}}}{(\mathscr{M}' )^2}\big ( 1+\frac{\kappa}{(\|\eta_{1-\tau}\Psi\|_{\infty})^{1/2}} \big )
\end{equation}

and

\begin{equation}\label{c''220}
C''_{\ref{epsi11}}=-\frac{C_{\ref{epsi11}}}{\delta\sqrt{\mathscr{M}'}}
\end{equation}

So the proof of lemma is complete.

\end{proof}

 We want also to show that the last (negative term ) in (\ref{epsi11})   can be controlled  by  the term  
\[
\frac{(\Delta_{g'} \eta) (\sum |\phi _{i\bar{j}k}|^2 g'^{i\bar{i}}g'^{j\bar{j}}g'^{k\bar{k}}) }{\eta\Psi +\epsilon} 
\]
in  $\mathcal{E}_2$ defined by (\ref{mce2}).

\begin{lemma}\label{lem19n}
 We have
\begin{equation}\label{lem20nn}
\frac{(\Delta_{g'} \eta) (\sum |\phi _{i\bar{j}k}|^2 g'^{i\bar{i}}g'^{j\bar{j}}g'^{k\bar{k}}) }{\eta\Psi +\epsilon}  - \sum_{\alpha}g'^{\alpha\bar{\alpha}}\frac{|\eta_{\alpha}|^2\Psi}{2\eta [\epsilon-(2C_{\ref{no2}}-1)\eta\Psi]}\geq \frac{C_{\ref{lem20nn}} \chi^{\tau} }{\kappa|S|^4}+ \frac{\tilde{C}_{\ref{lem20nn}}}{\kappa }\chi^{\tau/2}
\end{equation}

 where $C_{\ref{lem20nn}}>0$  is a positive  constant and $C_{\ref{lem20nn}}$ and $\tilde{C}_{\ref{lem20nn}}$  depend on $\mathscr{M}'$, $G$ , $\delta$ and $g_{reg}$.
\end{lemma}

\begin{proof}
Here we work in the coordinates system described below the relation (\ref{chitild}) over $\supp \chi$.
In this  coordinates system  we can represent $S$ in the form  $S= azdw_1\wedge ...\wedge dw_{n-1}\wedge dz$ for some constant $|a|=O(\sin \theta_n (p))$ where $\theta_n (p)$ is the angel with respect to $g_{reg}$ between $V_n (p)$ and $D$.
 Thus  if we asssume that $\chi = e^u$  then we have $\eta=\chi |S|^4=a^4|z|^4 e^{4h+u}$ where $e^h$  represents the  hermitian metric on $K_X$ in local coordinates.
 
Derivating with respect to $V_{\alpha}$  for $\alpha=1,...,n-1$ at  $x$ in the coordinates compatible with (\ref{eqlits}) we obtain  $|\eta_{\alpha}|^2= |4h_{\alpha}+ u_{\alpha} |^2\eta ^2$
and    $\eta_{\alpha\bar{\alpha}}= (4h_{\alpha\bar{\alpha}} +u_{\alpha\bar{\alpha}})\eta+ |4h_{\alpha}+u_{\alpha}|^2 \eta$. 
Here by $f_{\alpha\bar{\alpha}}$ for a function map we mean $V_{\bar{\alpha}}. V_{\alpha}.f$

\[
\begin{split}
-\frac{g'^{\alpha\bar{\alpha}}|\eta_{\alpha}|^2\Psi}{2\eta [\epsilon-(2C_{\ref{no2}}-1)\eta\Psi]} &+\frac{g'^{\alpha\bar{\alpha}}\eta_{\alpha\bar{\alpha}} \Psi}{\eta\Psi+\epsilon}=\\
\quad & g'^{\alpha\bar{\alpha}} |4h_{\alpha}+u_{\alpha}|^2   (-\frac{\eta \Psi}{2[\epsilon-(2C_{\ref{no2}}-1)\eta\Psi]} +\frac{\eta\Psi}{\epsilon + \eta \Psi})+\\
\quad & +  (g'^{\alpha\bar{\alpha}}(4h_{\alpha\bar{\alpha}}+ u_{\alpha\bar{\alpha}})\frac{\eta \Psi}{\eta \Psi + \epsilon}\\
\quad & \geq 
 g'^{\alpha\bar{\alpha}} |4h_{\alpha}+u_{\alpha}|^2  (\chi)^{\tau}  (-\frac{\eta_{1-\tau} \Psi}{2[\epsilon-(2C_{\ref{no2}}-1)\eta\Psi]} +\frac{\eta_{1-\tau}\Psi}{\epsilon + \eta \Psi})+\\
\quad & -  g'^{\alpha\bar{\alpha}} |4h_{\alpha\bar{\alpha}}+ u_{\alpha\bar{\alpha}}|(\chi)^{\tau}\frac{\eta_{1-\tau} \Psi}{\eta \Psi + \epsilon}\\
\end{split}
\]

Since the bump function $\chi$ can be constructed such that 
$|u_{\alpha}|^{a}(\chi)^{\tau/2}$ and  $|u_{\alpha\bar{\alpha}}|^b (\chi)^{\tau/2}$ for any $a$ and $b$ have an upper bound which only depend on $\delta$.  thus since
$\epsilon = \kappa \|\eta_{1-\tau}\Psi\|_{\infty}$ we obtain

\begin{equation}
-\frac{g'^{\alpha\bar{\alpha}}|\eta_{\alpha}|^2\Psi}{2\eta [\epsilon-(2C_{\ref{no2}}-1)\eta\Psi]} +\frac{g'^{\alpha\bar{\alpha}}\eta_{\alpha\bar{\alpha}} \Psi}{\eta\Psi+\epsilon}\geq \frac{C_{\ref{lem20nn}}\chi^{\tau/2}}{\delta^2\kappa\mathscr{M}'}
\end{equation}
 where $C_{\ref{lem20nn}}$ depends on  $g_{reg}$.

For $\alpha=n$ we first show   for $\kappa $ large enough and for $|z|<\delta r_0$    the following  inequality holds where $r_0$ only depend on $g_{reg}$:

\begin{equation}\label{epsi2}
\frac{2}{3}\times \frac{g'^{n\bar{n}}\eta_{z\bar{z}} \Psi}{\eta\Psi+\epsilon}>\frac{g'^{n\bar{n}}|\eta_{z}|^2\Psi}{2\eta [\epsilon- (2C_{\ref{no2}}-1)\eta\Psi]}
\end{equation}

Equivalently we must have

\begin{equation}\label{34om}
\frac{[\epsilon- (2C_{\ref{no2}}-1)\eta\Psi]}{\eta\Psi+\epsilon}> \frac{3|\eta_z|^2}{4\eta\eta_{z\bar{z}}}
\end{equation}

If we assume that $\chi= e^u$ using the definition $\eta= \chi |z|^4 e^{4h}$ we  have
\[
\begin{split}
\eta_z &= e^{4h}( 2z\bar{z}^2\chi + 4h_z |z|^4 \chi+ |z|^4 \chi_{z})\\
\quad & = e^{4h}\chi (2z\bar{z}^2 +4h_z |z|^4 + |z|^4 u_z)\\
\end{split}
\]

 and

\begin{equation}\label{etazbarz}
\begin{split}
\eta_{z\bar{z}}=& e^{4h}4h_{\bar{z}} ( 2z\bar{z}^2\chi + 4h_z |z|^4 \chi+ |z|^4 \chi_{z})\\
\quad & +e^{4h} (4z\bar{z} \chi+ 2z \bar{z}^2 \chi_{\bar{z}} +4h_{z\bar{z}} |z|^4 \chi+ 8h_z z^2\bar{z} \chi + 4h_z |z|^4 \chi_{\bar{z}}+ \\
\quad &  2\bar{z}z^2 \chi_{z} + |z|^4 \chi_{z\bar{z}})
\end{split}
\end{equation}
The lowest order term in the above expression for $\eta_z$ and $\eta_{z\bar{z}}$ respectively are  $2e^{4h}\chi z\bar{z}^2$ and 
$4e^{4h}\chi z\bar{z}  $.  We have   $\lim_{z\rightarrow 0 } \frac{3|\eta_z |^2}{4\eta\eta_{z\bar{z}}}= \frac{3}{4}$.  More precisely we have

\[
\frac{|\eta_z|^2}{|\eta\eta_{z\bar{z}}|}=\frac{e^{8h}\bigg [ 4\chi^2 |z|^6+ |z|^6\chi(2\frac{\bar{z}}{\delta} \tilde{\chi}_{z}+2\frac{z}{\delta}\tilde{\chi}_{\bar{z}} )+|z|^6 \frac{|z|^2}{\delta ^2} |\tilde{\chi}_z|^2\bigg ] +f_1}{e^{8h} \bigg [ 4\chi^2 |z|^6+|z|^6\chi(2\frac{\bar{z}}{\delta} \tilde{\chi}_{z}+2\frac{z}{\delta}\tilde{\chi}_{\bar{z}} )+\chi\tilde{\chi}_{z\bar{z}}|z|^6 \frac{|z|^2}{\delta^2}\bigg ] +f_2 }
\]
 
where
\[
f_1 = e^{8h}\bigg[ 8\chi^2  |z|^6 (\bar{z}h_{\bar{z}})+8\chi^2  |z|^6 (zh_{z})+4h_z \chi\chi_{\bar{z}}|z|^8 + 4h_{\bar{z}} \chi\chi_{z}|z|^8+16 |h_z|^2\chi^2 |z|^8\bigg]
\]

\[
f_2=e^{8h}\bigg [ 8\chi^2 h_{\bar{z}}\bar{z}|z|^6+8 \chi^2 h_{z}z|z|^6+ 16|h_z|^2 |z|^8 \chi^2+ 4 h_{\bar{z}} \chi_z \chi|z|^8 +    4 h_{z} \chi_{\bar{z}} \chi|z|^8 +4h_{z\bar{z}} |z|^8 \chi^2\bigg ]
\]

Therefore it can be seen that for any $r_1>0$ there exists $r_0>0$ such that if $\frac{|z|}{\delta}<r_0 $ then $|\frac{|\eta_z|^2}{|\eta\eta_{z\bar{z}}|}-1|<r_1$ and the value of $r_0$ only depends on $r_1$ and hermtian metric on $L$. If $r_1$ is so small  that $\frac{3(1+r_1)}{4} <1$ then  the right hand side of (\ref{34om}) will become smaller than 1.

   On the other hand we have  $\lim_{\kappa\rightarrow \infty} \frac{\epsilon - (2C_{\ref{no2}}-1)\eta \Psi}{\eta \Psi + \epsilon} =1$, from which we conclude that for $\kappa $ large enough and for $|z|<\delta r_0$    the inequality (\ref{34om}) will hold.
In this case we can also deduce from (\ref{epsi2}) that

\begin{equation}\label{paka}
\frac{g'^{n\bar{n}}\eta_{z\bar{z}} \Psi}{\eta\Psi+\epsilon}-\frac{g'^{n\bar{n}}|\eta_{z}|^2\Psi}{2\eta [\epsilon- (2C_{\ref{no2}}-1)\eta\Psi]}\geq  \frac{1}{3}\frac{g'^{n\bar{n}}\eta_{z\bar{z}} \Psi}{\eta\Psi+\epsilon}
\end{equation}

It can be seen that the dominant term in $\frac{1}{3}\frac{g'^{n\bar{n}}\eta_{z\bar{z}} \Psi}{\eta\Psi+\epsilon}$ is $\frac{1}{|z|^4}$, therefore 

\begin{equation}\label{paka2}
\frac{g'^{n\bar{n}}\eta_{z\bar{z}} \Psi}{\eta\Psi+\epsilon}-\frac{g'^{n\bar{n}}|\eta_{z}|^2\Psi}{2\eta [\epsilon- (2C_{\ref{no2}}-1)\eta\Psi]}\geq   \frac{C_{\ref{paka2}}\chi ^{\tau}}{\kappa|S|^4}
\end{equation}

for $|S|<r$ where d $C_{\ref{paka2}}>0$ is a positive constant and $r$ and $C_{\ref{paka2}}$  depend on $g_{reg}$, $G$, $\delta$ and $\mathscr{M}'$.
 We can also choose $r_0$ such that $\eta_{z\bar{z}}>0$ for $|z|<\delta r_0$ hence the positivity

 \begin{equation}\label{paka}
\frac{g'^{n\bar{n}}\eta_{z\bar{z}} \Psi}{\eta\Psi+\epsilon}-\frac{g'^{n\bar{n}}|\eta_{z}|^2\Psi}{2\eta [\epsilon- (2C_{\ref{no2}}-1)\eta\Psi]}\geq 0
\end{equation}

Now if $|z|>r_0 \delta$ we can derive the following estimation

\begin{equation}\label{grand}
\begin{split}
\frac{g'^{n\bar{n}}\eta_{z\bar{z}} \Psi}{\eta\Psi+\epsilon}&-\frac{g'^{n\bar{n}}|\eta_{z}|^2\Psi}{2\eta [\epsilon- (2C_{\ref{no2}}-1)\eta\Psi]}\geq\\
\quad &  \frac{\eta_{z\bar{z}} \Psi e^{-G} (\mathscr{M}')^{2n-2}}{|S|^2(\eta\Psi+\epsilon)}-\frac{\mathscr{M}' _1|\eta_{z}|^2\Psi}{2 |S|^2\eta [\epsilon- (2C_{\ref{no2}}-1)\eta\Psi]}\\
\quad & \geq \frac{C_{\ref{grand}}}{\kappa \delta^4}\chi^{\tau}
\end{split}
\end{equation}

where $C_{\ref{grand}}$  constant and depends on $G$  and $g_{reg}$ and $\mathscr{M}' $.

 \end{proof}

 We  also have 

\begin{equation}\label{re205}
C_{\ref{no2}}\sum_{\alpha}\frac{ g'^{\alpha\bar{\alpha}}|\eta_{\bar{\alpha}}|^2\Psi^2}{(\eta\Psi+\epsilon)^2}\geq 0
\end{equation}

and  since $\Delta_{g'}  (\log (\eta\Psi+\epsilon))=\mathcal{E}_1+\mathcal{E}_2$ from lemma (\ref{cor4}) and (\ref{lem19n}) and relations (\ref{re205})  we obtain,
\begin{equation}\label{epsi1}
\begin{split}
\Delta_{g'}  (\log (\eta\Psi+\epsilon))\geq &   -\frac{C_{\ref{epsi11}}}{(\mathscr{M}' )^2}\big ( \frac{\chi^{2\tau}}{\kappa ^2}+\frac{\chi^{\frac{1+\tau}{2}}}{\kappa(\|\eta_{1-\tau}\Psi\|_{\infty})^{1/2}} \big ) \bigg [ \sum_{p,q,l} g'^{p\bar{p}}g'^{q\bar{q}} (\phi)_{p\bar{q}\bar{l}}(\phi)_{q\bar{p}l}\bigg ]\\
\quad &-\frac{C_{\ref{epsi11}}}{\kappa} \frac{\chi^{\tau}}{\delta\sqrt{\mathscr{M}'}}  \bigg [ \sum_{p,q,l} g'^{p\bar{p}}g'^{q\bar{q}} (\phi)_{p\bar{q}\bar{l}}(\phi)_{q\bar{p}l}\bigg ]^{1/2} + \frac{C_{\ref{lem20nn}} \chi^{\tau} }{\kappa|S|^4}+ \frac{\tilde{C}_{\ref{lem20nn}}}{\kappa }\chi^{\tau/2}+\frac{\eta R }{\eta \Psi+\epsilon}\\
\end{split}
\end{equation}



\begin{cor}\label{corsup}   

\begin{equation}\label{corsup1}
\begin{split}
\Delta_{g'}  (\log (\eta\Psi+\epsilon))\geq &  \frac{C'_{\ref{epsi11}}\chi^{\tau}}{\kappa^2}  \bigg [ \sum_{p,q,l} g'^{p\bar{p}}g'^{q\bar{q}} (\phi)_{p\bar{q}\bar{l}}(\phi)_{q\bar{p}l}\bigg ]\\
\quad &\frac{C''_{\ref{epsi11}}\chi ^{\tau}}{\kappa}   \bigg [ \sum_{p,q,l} g'^{p\bar{p}}g'^{q\bar{q}} (\phi)_{p\bar{q}\bar{l}}(\phi)_{q\bar{p}l}\bigg ]^{1/2} \\
\quad &   \frac{C_{\ref{lem20nn}} \chi^{\tau} }{\kappa|S|^4}+ \frac{C_{\ref{corsup1}}}{\kappa }\chi^{\tau/2}
\end{split}
\end{equation}

 where $C_{\ref{lem20nn}} $ is a positive constant and all the constants  $C'_{\ref{epsi11}}$,  $C''_{\ref{epsi11}}$  $C_{\ref{lem20nn}}$ and $C_{\ref{corsup1}}$  depend on $G$ and $g_{reg}$, $\mathscr{M}'$ and $\delta$. (Here the term  $ \frac{C_{\ref{corsup1}}}{\kappa }\chi^{\tau/2}$ is obtained from  $\frac{\eta R }{\eta \Psi+\epsilon}$ and $\frac{\tilde{C}_{\ref{lem20nn}}}{\kappa }\chi^{\tau/2}$ in the relation (\ref{lem20nn})
\end{cor}




According to (2.10) in \cite{y} we have
\begin{equation}\label{yau10}
\begin{split}
\Delta_{g'} (\Delta_{reg} (\phi) ) &\geq \Delta_{reg} F + (\sum_{p,q,l} g'^{p\bar{p}}g'^{q\bar{q}} |\phi_{p\bar{q}\bar{l}} |^2+
(\inf_{i\neq l} R_{i\bar{i}l\bar{l}})[\sum_{i,l} \frac{g'_{i\bar{i}}}{g'_{l\bar{l}}}-n^2]\\
\end{split}
\end{equation}


\begin{equation}\label{eq233}
\begin{split}
\Delta_{g'} (\chi^{\tau} \Delta_{reg} \phi)&= \chi^{\tau} \Delta_{g'} (\Delta_{reg} \phi)+ (\Delta_{g'} (\chi)^{\tau})(\Delta_{reg}\phi)+ g'^{i\bar{i}}(\chi ^{\tau} )_i  (\Delta_{reg}\phi)_{\bar{i}}\\
\quad & =  \chi^{\tau} \Delta_{g'} (\Delta_{reg} \phi)+ (\Delta_{g'} (\chi^{\tau}))(\Delta_{reg}\phi)+ \sum_{i,\alpha}g'^{i\bar{i}}(\chi^{\tau})_i  (\phi)_{\alpha\bar{\alpha}\bar{i}}\\
\quad & \geq \chi^{\tau} \bigg (\Delta_{reg} F + \sum_{p,q,l} g'^{p\bar{p}} g'^{q\bar{q}} |\phi_{p\bar{q}\bar{l}} |^2+
 \frac{1}{2} \sum_{i,l}\big (\inf_{i\neq l}  R_{i\bar{i}l\bar{l}}\big ) \big [ \frac{g'_{i\bar{i}}}{g'_{l\bar{l}}}\big ] \bigg )\\
\quad & +  (\Delta_{g'} (\chi^{\tau}))(\Delta_{reg}\phi)\\
\quad & +  \sum_{i,\alpha}g'^{i\bar{i}}(\chi ^{\tau})_i  (\phi)_{\alpha\bar{\alpha}\bar{i}}
\end{split}
\end{equation}

From lemma  (\ref{m'm'1}) in the appendix (\ref{app7}) and  proposition (\ref{secondorder})  we know that
\begin{equation}\label{rel234}
\frac{g'_{\alpha\bar{\alpha}}}{g'_{z\bar{z}}}\leq\frac{C_{\ref{secondorder1}} \mathscr{M}' _1}{|S|^2} 
 \end{equation}
  and  for $i\neq z$. 

\begin{equation}\label{rel235}
\frac{g'_{\alpha\bar{\alpha}}}{g'_{i\bar{i}}}\leq \frac{C_{\ref{secondorder1}}}{ (\mathscr{M}')^2}
 \end{equation}

 Therefore  we have

 \begin{equation}
\begin{split}
\sum_{i,l} \bigg [ \frac{g'_{i\bar{i}}}{g'_{l\bar{l}}} \bigg ]&= \sum_{i,l\neq z}  \bigg [\frac{g'_{i\bar{i}}}{g'_{l\bar{l}}}\bigg ]+\sum_{i\neq n} \bigg [\frac{g'_{i\bar{i}}}{g'_{z\bar{z}}}\bigg ]\\ 
\quad & \leq  (n-1)^2 \frac{C_{\ref{secondorder1}}}{\mathscr{M}' } +  (n-1) \frac{C_{\ref{secondorder1}}\mathscr{M}' _1}{|S|^2}\\
\quad & 
\end{split}
\end{equation}

 So 

\begin{equation}\label{riill}
|(\inf_{i\neq l} R_{i\bar{i}l\bar{l}})[\sum_{i,l} \frac{g'_{i\bar{i}}}{g'_{l\bar{l}}}-n^2]|\leq C_{\ref{riill}} (\frac{ 1}{|S|^2}+\frac{1}{(\mathscr{M}')^2})
\end{equation}

for a constant    $C_{\ref{riill}}$ which  only depends on $G$ and $g_{reg}$. In addition  $| \Delta_{g'} (\chi ^{\tau}) |= \sum g'^{\alpha\bar{\alpha}} [\chi^{\tau}]_{\alpha\bar{\alpha}}$.
If we assume that $\chi = e^{u}$ for some smooth map $u$ then we obtain
\[
\frac{(\chi ^{\tau})_{z\bar{z}}}{\chi^{\tau}}= \frac{\tau}{2}u_{z\bar{z}}+\frac{\tau}{4}|u_{z}|^2
\]

so according to relation (\ref{ttkmin}) in the appendix (\ref{app6}) we have
\begin{equation}\label{chizz}
g'^{z\bar{z}}(\chi^{\tau}) _{z\bar{z}}\leq   \mathscr{M}' _1\chi^{\tau}  \frac{ \tau u_{z\bar{z}}+\tau  |u_{z}|^2}{ |S|^2}\leq \frac{C_{\ref{chizz}}\tau \chi^{\tau/2}}{\delta^2  |S|^2}
\end{equation}

 Similarly  for $\alpha\neq z$ 
\begin{equation}\label{eq219}
g'^{\alpha\bar{\alpha}}\chi_{\alpha\bar{\alpha}}\leq  \frac{C_{\ref{eq219}} \tau\chi^{\tau/2}}{(\mathscr{M}')^2\delta ^2}
\end{equation}

where the  constants $C_{\ref{eq219}}$ and $C_{\ref{chizz}}$  depend on $g_{reg}$ and $G$.

  From  (\ref{chizz}) and  (\ref{eq219}) we conclude that

\begin{equation}\label{delgc}
| \Delta_{g'} (\chi ^{\tau}) |\leq  C_{\ref{delgc}}\big ( \frac{1} {\delta ^2 |S|^2}+ \frac{1}{(\mathscr{M}')^2\delta ^2}\big )\tau \chi^{\tau /2}
\end{equation}

for some constant $C_{\ref{delgc}}$ which depends  $g_{reg}$ and $G$. Therefore,
\begin{lemma}
\begin{equation}\label{c'''10}
\begin{split}
\chi^{\tau}(\inf_{i\neq l} R_{i\bar{i}l\bar{l}})[\sum_{i,l} \frac{g'_{i\bar{i}}}{g'_{l\bar{l}}}-n^2])+ &(\Delta_{g'} (\chi)^{\tau})(\Delta_{reg}\phi)+\chi^{\tau} \Delta_{reg} F \geq   C_{\ref{c'''10}}\chi^{\tau/2}\big ( \frac{\tau} {\delta ^2 |S|^2}+ \frac{\tau}{(\mathscr{M}')^2\delta ^2}\big )+\\
\quad & C_{\ref{c'''10}}\chi^{\tau/2}\big ( \frac{1} { |S|^2}+ \frac{1}{(\mathscr{M}')^2}\big )
\end{split}
\end{equation}

where $C_{\ref{c'''10}}< -\max\{C_{\ref{delgc}}, C_{\ref{riill}}\}$ only depends on $G$ and $g_{reg}$. Also the lower bound for $\chi^{\tau} \Delta_{reg} F$ is integrated in this constant.  .

  \end{lemma}

 By applying (\ref{rel234})  and (\ref{rel235}) we can deduce that

\begin{equation}\label{b220}
\begin{split}
| \sum_{i,\alpha}g'^{i\bar{i}}(\chi ^{\tau}) _i  (\phi)_{\alpha\bar{\alpha}\bar{i}}|&=  \sum_{i,\alpha}(g'^{i\bar{i}})^{1/2}(\chi ^{\tau})_i    (g'^{i\bar{i}})^{1/2}(\phi)_{\alpha\bar{\alpha}\bar{i}}\\
\quad & \leq  B_{\ref{b220}}(\frac{1}{|S|}+\frac{1}{\mathscr{M}'})\tau \chi^{\tau/2} (\sum_{p,q,l} g'^{p\bar{p}}g'^{q\bar{q}} |\phi_{p\bar{q}\bar{l}} |^2)^{1/2}
\end{split}
\end{equation}

for some constant $B_{\ref{b220}}$ which depends on $g_{reg}$ and $G$.

From lemma (\ref{c'''10}), relations  (\ref{b220}) and (\ref{eq233})  one can deduce that

\begin{lemma}

\begin{equation}\label{eq209}
\begin{split}
\Delta_{g'} (\chi^{\tau} \Delta_{reg} \phi) \geq & \chi^{\tau} \sum_{p,q,l}g'^{p\bar{p}}g'^{q\bar{q}} |\phi_{p\bar{q}\bar{l}} |^2-\\
\quad  & \frac{C''_{\ref{eq209}}}{|S|}(\chi^{\tau}\sum_{p,q,l} g'^{p\bar{p}}g'^{q\bar{q}}|\phi_{p\bar{q}\bar{l}} |^2))^{1/2} \\
\quad & +\frac{C'''_{\ref{eq209}} \chi^{\tau/2}}{|S|^2}
\end{split}
\end{equation}

where 
\[
C''_{\ref{eq209}}=- \tau B_{\ref{b220}}(1+\frac{|S|}{\mathscr{M}'})
\], 

and  
\[
C'''_{\ref{eq209}}=C_{\ref{c'''10}}\big ( \frac{\tau} {\delta ^2 }+ \frac{\tau |S|^2}{(\mathscr{M}')^2\delta ^2}\big )+ C_{\ref{c'''10}}\chi^{\tau/2}\big ( 1+ \frac{|S|^2}{(\mathscr{M}')^2}\big )
\]

 and $B_{\ref{b220}}$ and  $C_{\ref{c'''10}}$ only depend on $G$ and $g_{reg}$.
\end{lemma}

From corollary (\ref{corsup}) the inequality (\ref{eq209})   we conclude that if we take 

\begin{equation}\label{AA}
A=A_0\frac{C''_{\ref{epsi11}}}{\kappa^2}  
\end{equation}

for some $A_0>1$, then

\begin{equation}\label{mainineqq2}
\begin{split}
\Delta_{g'} [\log (\eta\Psi+\epsilon)+A\chi^{\tau}\Delta_{reg} (\phi)  ]\geq &(A_0-1)\frac{\chi^{\tau} C'_{\ref{epsi11}}}{\kappa ^2}\sum_{p,q,l} g'^{p\bar{p}}g'^{q\bar{q}} |\phi)_{p\bar{q}\bar{l}} |^2\\
\quad &+A\big (  \frac{C''_{\ref{eq209}}}{|S|}+\frac{C''_{\ref{epsi11}}}{A\kappa}  \big )   (\chi^{\tau}\sum_{p,q,l} g'^{p\bar{p}}g'^{q\bar{q}} |\phi_{p\bar{q}\bar{l}} |^2 )^{1/2}\\
\quad & + \frac{C'''_{\ref{eq209}}\chi^{\tau/2}}{|S|^2}\\ 
\quad & +\frac{C_{\ref{lem20nn}} \chi^{\tau} }{\kappa|S|^4}
\end{split}
\end{equation}

Here since the last terms $  \frac{C_{\ref{corsup1}}}{\kappa }\chi^{\tau/2}$ in  (\ref{corsup1}) is bounded  they can be absorbed in the
 constant of the  unbounded terms of the type $\frac{1}{|S|^2} $ in the last line of the above inequality.  Also since $C_{\ref{lem20nn}}$ is positive the
 term $\frac{C_{\ref{lem20nn}} \chi^{\tau} }{\kappa|S|^4}$ can be ignored.  Therefore we have proved: 
    
\begin{lemma}\label{lem255}

\begin{equation}\label{mainineq2}
\begin{split}
\Delta_{g'} [\log (\eta\Psi+\epsilon)+A\Delta_{reg} (\chi^{\tau}\phi)  ]\geq & C'  \chi^{\tau}\sum_{p,q,l} g'^{p\bar{p}}g'^{q\bar{q}} |\phi_{p\bar{q}\bar{l}} |^2\\
\quad &+A\frac{C''}{|S|}  (\chi^{\tau}\sum_{p,q,l} g'^{p\bar{p}}g'^{q\bar{q}}|\phi_{p\bar{q}\bar{l}} |^2 )^{1/2}\\
\quad & +\frac{C'''}{|S|^2}
\end{split}
\end{equation}

where 
\begin{equation}\label{c'c'}
C' =\frac{(A_0-1)C'_{\ref{epsi11}}}{\kappa ^2}
\end{equation}

\begin{equation}\label{c''c''}
C'' :=  C''_{\ref{eq209}}+\frac{C''_{\ref{epsi11}}|S|}{A\kappa} 
\end{equation}

and

\begin{equation}\label{c'''}
C'''=AC'''_{\ref{eq209}}
\end{equation}

\end{lemma}

Now consider the point $p_0\in X$ where the maximum of  $[\log (\eta\Psi+\epsilon)+A\Delta_{reg} (\chi^{\tau}\phi)  ]$ occurs. Since the maximum can occur on $D$ we need to multiply both sides of \ref{mainineq2} by $|S|^2$ and repeat the argument as in the proof of the proposition (\ref{secondorder}). Then from the inequality (\ref{mainineq2})   we find that


\begin{equation}\label{c'c''}
\begin{split}
\bigg [ C' \bigg (\chi ^{\tau}|S|^2\sum_{p,q,l} g'^{p\bar{p}}g'^{q\bar{q}}& |\phi_{p\bar{q}\bar{l}} |^2\bigg ) + AC'' \bigg (|S|^2 \chi^{\tau} \sum_{p,q,l} g'^{p\bar{p}}g'^{q\bar{q}} |\phi_{p\bar{q}\bar{l}} |^2 \bigg )^{1/2}\\
\quad &   + C'''  \bigg ] (p_0)\leq 0 
\end{split}
\end{equation}

from the above relation and the definition of $C'$, $C''$ and $C'''$ in lemma (\ref{lem255}) it follows that if $\tau\leq \delta\leq \mathscr{M}'$
and by taking $\kappa$ and $A_0$ large enough    we can deduce that
\begin{equation}\label{ubound}
\begin{split}
[|S|^2\chi^{\tau}\sum_{p,q,l} g'^{p\bar{p}}g'^{q\bar{q}}|\phi_{p\bar{q}\bar{l}} |^2] (p_0 )&\leq \frac{-AC'' +\sqrt{(AC''  )^2 -4C' C'''  }}{2C'}\\
\quad &\leq  \tilde{C}_{\ref{ubound}} 
\end{split}
\end{equation}

where $\tilde{C}_{\ref{ubound}}$ depends only on $g_{reg}$ and $G$. We note that this term  is dominantly  generated by  $C_{\ref{c'''10}}\big ( 1+ \frac{\delta ^2}{(\mathscr{M}')^2}\big )$ in $C'''$.  We can now utilize (\ref{rel234}) and (\ref{rel235}) to coclude that:

\begin{equation}\label{ubound2}
\begin{split}
|S^4 \chi^{\tau} \sum_{p,q,l}& g'^{p\bar{p}}g'^{q\bar{q}}g'^{l\bar{l}} |\phi_{p\bar{q}\bar{l}} |^2]  \leq\\
\quad & |S|^2\chi^{\tau}\sum_{\substack{p,q\\ l\neq z}} [g'^{l\bar{l}} |S|^2] g'^{p\bar{p}}g'^{q\bar{q}} |\phi_{p\bar{q}\bar{l}} |^2] \\
\quad & +|S|^2\chi^{\tau}\sum_{p,q} [g'^{z\bar{z}} |S|^2] g'^{p\bar{p}}g'^{q\bar{q}} |\phi_{p\bar{q}\bar{z}} |^2]\\
\quad & \leq  |S|^2\chi^{\tau}[\sum_{\substack{p,q\\ l\neq z}}g'^{p\bar{p}}g'^{q\bar{q}} |\phi_{p\bar{q}\bar{l}} |^2] [\frac{|S|^2}{\mathscr{M}'}]\\
\quad & +|S|^2\chi^{\tau}\sum_{p,q} g'^{p\bar{p}}g'^{q\bar{q}} |\phi_{p\bar{q}\bar{z}} |^2][\mathscr{M}' _1 ]\\
\quad & = |S|^2\chi^{\tau}\sum_{p,q}  g'^{p\bar{p}}g'^{q\bar{q}} |\phi_{p\bar{q}\bar{z}} |^2] [(n-1 )\frac{|S|^2}{\mathscr{M}'}+\mathscr{M}' _1]
\end{split}
\end{equation}

   So from the relations (\ref{ubound}) and (\ref{ubound2})  since  $|S|\leq  \delta \leq \mathscr{M}'$ we deduce that





\begin{equation}\label{oet}
[|S|^4 \chi^{\tau}\sum_{p,q,l}g'^{p\bar{p}}g'^{q\bar{q}}g'^{l\bar{l}}|\phi_{p\bar{q}\bar{l}} |^2)] (p_0)\leq  C_{\ref{oet}} 
\end{equation}

for some constant $C_{\ref{oet}}$ which depends on $G$, $g_{reg}$.   
 Finaly from (\ref{oet}) we get 
\begin{lemma}   
\begin{equation}\label{constc}
(\eta_{\tau}\Psi)(p_0 )\leq  \mathscr{C}
\end{equation}
for some constant $\mathscr{C}$ which only depends on  $G$, $g_{reg}$. 
\end{lemma}

Since $\tau<1$  over $U_p$ we have
\[
\begin{split}
\log (\eta\Psi+\epsilon)+A\Delta_{reg} (\chi^{\tau}\phi)& \leq \log (\eta_{\tau}\Psi+\epsilon) (p_0)+A\Delta_{reg} (\chi^{\tau}\phi)(p_0)\\
\quad & \leq  \log ( \mathscr{C}+\epsilon)+A\Delta_{reg} (\chi^{\tau}\phi)(p_0 )
\end{split}
\]

where $A$ is defined by (\ref{AA}), thus we obtain
\[
\log \frac{\eta\Psi+\epsilon}{\mathscr{C}+\epsilon}\leq A[\Delta_{reg} (\chi^{\tau}\phi)(p_0 )-\Delta_{reg} (\chi^{\tau}\phi) ]
\]

\[
\begin{split}
\Delta_{reg} (\chi^{\tau} \phi)&=   \chi^{\tau} (\Delta_{reg}\phi )+(\Delta_{reg} \chi^{\tau} )\phi+ \sum_{\alpha} g_{reg} ^{\alpha\bar{\alpha}} (\chi^{\tau})_{\alpha} \phi_{\bar{\alpha}}\\
\end{split}
\]

So we have
\begin{equation}\label{b'b'}
|\Delta_{reg} (\chi^{\tau} \phi)|\leq B'
\end{equation}

where

\[
B' =  (C_{\ref{secondorder1}}+n)+ \|\phi\|_{\infty} \|\{\Delta_{reg} \chi^{\tau}\|_{\infty}+ \|\nabla_{reg} \phi\|_{\infty} \|g_{reg} ^{\alpha\bar{\alpha}} (\chi^{\tau})_{\alpha}\|_{\infty} 
\]

and by proposition (\ref{secondorder}) $B'$ it can also be seen that $B'$  depends on $G$ and 
$g_{reg}$ and $\delta$.

According tothe definition of $A$ and $C''_{\ref{epsi11}}$ respectively by (\ref{AA}) and (\ref{c''220}) Hence we get to the inequality 
\begin{equation}\label{kappa4}
\log \frac{\eta\Psi+\epsilon}{\mathscr{C} +\epsilon}\leq A_1
\end{equation}

where

\begin{equation}\label{AA1}
A_1=2B'A_0\frac{C_{\ref{epsi11}}}{(\mathscr{M}' )^2}(\frac{1}{\kappa^2} +\frac{1}{\kappa \sqrt{a} })
\end{equation}

and we have set 

\begin{equation}\label{a}
a:=\|\eta_{1-\tau} \Psi\|_{\infty},\hspace{1cm} \text{ and } \hspace{0.5cm} b=\|\eta\Psi\|_{\infty}
\end{equation}

then we get to 

 \[
\log \frac{b+\kappa a}{ \mathscr{C}+\kappa a}\leq A_1
\]
or
\[
\log \frac{\frac{b}{a}+\kappa}{\frac{\mathscr{C}}{a}+\kappa}\leq A_1
\]

From which we deduce 

\[
\kappa+\frac{\mathscr{C}}{a}\geq (\frac{b}{a}+\kappa) e^{-A_1}
\]

\[
a\leq \frac{ \mathscr{C}}{(\frac{b}{a}+\kappa)e^{-A_1}-\kappa}
\]

Now we note that    $\lim_{\kappa\rightarrow \infty}A_1 = 0$
 and so $\lim_{\kappa\rightarrow \infty }\frac{\mathscr{C}}{(\frac{b}{a}+\kappa)e^{-A_1}-\kappa} =\frac{a\mathscr{C}}{b} $. Hence from the above inequality we conclude that
\[
b\leq \mathscr{C}
\] 
\begin{theo}\label{thirdorder}  The following inequality holds on $X$
\[
|S|^4 \Psi\leq \mathscr{C} 
\]
  
 where $\mathscr{C}$ depends on $G$ and $g_{reg}$.
\end{theo}
 

\subsection{Further Third Order  Estimates}
Consider an  open neighborhood $U_p$ of a point $p\in D$     and let  $(\tilde{w}_1,...,\tilde{w}_{n-1},\tilde{z})$  be a holomorphic coordinates system on $U_p$ with
$U_p\cap D=\{\tilde{z}=0\}$.  Let  $V_1,...,V_n$ be the $g'$-orthogonal $g_{reg}$-normal moving frame on $U_p$ as constructed in the appendix (\ref{app7}).

  We also fix the 
foliation generated by $V_1,...,V_{n-1}$.

\begin{defi}
The holomorphic foliation generated by $\{V_1,...,V_{n-1}\}$ is denoted by $\mathcal{F}_D$.
\end{defi}

We set 

\[
\nabla_{reg} V_i ^* =\sum \Gamma^{j} _{ki} V_j ^* \otimes V_k ^*
\]

where $\Gamma^{j} _{ki}$ for $i,j,k=1,...,n$ denotes the Christoffel symbols of the Levi-Civita connection associated to the regular metric $g_{reg}$ on the cotangent bundle of $X$. 

Consider a tensor $T\in \Gamma (U_p, T^* (U_p)\otimes \bar{T}^* (U_p) )$ 
\[
T=\sum_{ij} T_{i\bar{j}} V_i ^* \otimes \bar{V}^* _j
\]
If 
\[
\nabla_{reg} T=\sum T_{i\bar{ j}k}V^*_i \otimes  \bar{V}^* _j \otimes V_{k} ^*
\]

 then we have

\[
T_{i\bar{j}k} =V_{k}.T_{i\bar{j}}+\sum_a \Gamma_{ka} ^{i} T_{a\bar{j}}
\]

Hence

\begin{equation}\label{cdt3}
V_{k}.T_{i\bar{j}}=T_{i\bar{j}k}-\sum_a \Gamma_{ka} ^{i} T_{a\bar{j}}
\end{equation}

Similarly for derivatives with respect to conjugate vectors

\[
\bar{V}_{k}.T_{i\bar{j}}=T_{i\bar{j}\bar{k}}-\sum_b \Gamma_{\bar{k}\bar{b}} ^{\bar{j}}  T_{i\bar{b}}
\]

In particular if $g'=g_{reg}+\partial\bar{\partial} \phi$ then we obtain
\begin{equation}\label{hols1}
\begin{split}
\bar{V}_{\beta}. g'_{p\bar{q}} &= g'_{p\bar{q}\bar{\beta}}-\sum_b \Gamma_{\bar{\beta}\bar{b}} ^{\bar{q}}g'_{p\bar{b}}\\
\quad & = \phi_{p\bar{q}\bar{\beta}}-\sum_b \Gamma_{\bar{\beta}\bar{b}} ^{\bar{q}}g'_{p\bar{b}}
\end{split}
\end{equation}

The second line holds  due to the fact that $\nabla_{reg} g_{reg}=0$.
Likewise for 3-tensor 
\[
T=\sum_{ijk} T_{i\bar{j}k}V_i ^* \otimes \bar{V}^* _j \otimes V^* _k
\]

and its covariant derivative
\[
\nabla_{reg} T=\sum_{ijkl}T_{i\bar{j}kl}V_i ^* \otimes \bar{V}^* _j \otimes V^* _k\otimes V^* _l
\]

we have the following  identities 
\begin{equation}\label{cdt4}
T_{i\bar{j}kl}=V_l. T_{i\bar{j}k} -\sum_a T_{a\bar{j}k}\Gamma_{la} ^i-\sum_c T_{i\bar{j}c}\Gamma_{lc} ^k
\end{equation} 

\begin{equation}\label{cdt5}
T_{i\bar{j}k\bar{l}}=\bar{V}_l. T_{i\bar{j}k} -\sum_b T_{i\bar{b}k}\Gamma_{\bar{l}\bar{b}} ^{\bar{j}}
\end{equation}


Associated with the foliation  $\mathcal{F}$  we consider the operator
\[
\nabla_{D}: T^*X\otimes T^* X\rightarrow T^* \mathcal{F}\otimes T^* X\otimes T^* X
\]

 where $\nabla_D$ denotes the covariant derivative with respect to the Levi-Civita connection restricted to the  distribution generating  the holomorphic foliation  $\mathcal{F}$.

In this subsection we work with    

\begin{equation}\label{sayek}
\Psi_1:=\| \nabla_D \circ \bar{\partial} \circ  \partial\phi \|_{g'} ^2 =\sum_{\substack{1\leq i,r,j,s\leq n \\ 1\leq k,t\leq n-1} } g'^{i\bar{r}}g'^{\bar{j}s}g'^{k\bar{t}}\phi_{i\bar{j}k}\phi_{\bar{r}s\bar{t}}
\end{equation}

We observe  that $\Psi_1$ depends only on the foliation $\mathcal{F}$ generated by $\{V_1,...,V_{n-1}\}$ and not on the vectors $V_1,...,V_{n-1}$ themselves.
\\

The following theorem is the fundamental approximation we will prove in this section.

\begin{theo}[Main Theorem]\label{3rdorder3}
There exists a constant $\mathscr{C}_1$ such that
\[
|\Psi_1| \leq \mathscr{C}_1
\]
in a small neighborhood of $p $  and     $\mathscr{C}_1$ depends only on $g_{reg}$, $G$.
\end{theo}

 
\subsection{Proof of theorem  \ref{3rdorder3}}

In order to do computation and make approximations on $U_p$ we take  a point $x$  and in a neighborhood of this point we consider an other holomorphic 
coordinates system $(z_1,...,z_n)$ as described for (\ref{eqlits}). 
We need the following definition,

\begin{defi}Let $A,B:U_p \rightarrow \mathbb{R}$ be two real valued functions defined on $U_p$.

i) We say that 
\[
A\sim B
\]
if 
\[
|A- B|\leq  \frac{\mathcal{C}_1\sqrt{\Psi_1} +\mathcal{C}'_1}{|S|^2} 
\]
\\
and,
 \\

ii) 
\[
A\simeq B
\]
\\
if 
\\
\[
|A-B|\leq \frac{\mathcal{C}_4\Psi_1 +\mathcal{C}'_4 \sqrt{\Psi_1}+\mathcal{C}''_4}{|S|^4}  + \frac{\mathcal{C}_5\sqrt{\Psi_1}+ \mathcal{C}_6}{|S|^2}
\]

where $ \mathcal{C}_1, \mathcal{C}'_1, \mathcal{C}_2, \mathcal{C}_4, \mathcal{C}'_4, \mathcal{C}''_4,  \mathcal{C}_5 $ and $\mathcal{C}_6$ are constants which depend only on $g_{reg}$, $G$,  $\delta$ and the lower bound of $g'|_{U_p\cap D}$.
\end{defi}




\subsection{Estimation on $\Delta_{g'}\Psi_1$}\label{esti62}
We carry out the computation of $\Delta_{g'}$ without the hypothesis that $\Gamma(q) =0$ according to the correction introduced in lemma  (\ref{lem30}) in the appendix (\ref{app4}). 
Since  the moving frame $(V_1,...,V_n)$ is holomorphic we  have
 
\[
\Delta_{g'} \Psi_1 :=\sum g'^{\alpha\bar{\beta}}  V_{\alpha}.\bar{V}_{\beta} .\Psi_1 -\sum \big ( \nabla_{reg, V_{\alpha}} \bar{V}_{\beta} \big ).\Psi_1= \sum g'^{\alpha\bar{\beta}}  V_{\alpha}.\bar{V}_{\beta} .\Psi_1
\]



On the other hand $\bar{V}_{\alpha}. (g'^{-1} g')=0$ from which we can deduce that $\bar{V}_{\alpha} .g'^{-1} = -g'^{-1} \bar{V}_{\alpha}. g' g'^{-1}$, and

\begin{equation}\label{komaki}
\begin{split}
\bar{V}_{\alpha}.g'^{i\bar{j}}=&\sum _{p,q}\big [ -g'^{i\bar{p}}(\bar{V}_{\alpha}.g'_{p\bar{q}})g'^{q\bar{j}}  \big ]\\
\quad & =  -\sum g'^{i\bar{p}}( \phi_{p\bar{q}\bar{\alpha}})g'^{q\bar{j}}+g'^{i\bar{p}}(\sum_b  \Gamma_{\bar{\alpha}\bar{b}} ^{\bar{q}}g'_{p\bar{b}})  g'^{q\bar{j}}\\
\quad & =  -\sum g'^{i\bar{p}}\big [ \phi_{p\bar{q}\bar{\alpha}}-\sum_b  \Gamma_{\bar{\alpha}\bar{b}} ^{\bar{q}} g'_{p\bar{b}}\big ]   g'^{q\bar{j}}
\end{split}
\end{equation}

\begin{equation}\label{mg1}
\begin{split}
V_{\alpha}.g'^{i\bar{j}}&=   -\sum g'^{i\bar{p}}( \phi_{p\bar{q}\alpha})g'^{q\bar{j}}+g'^{i\bar{p}}(\sum_b  \Gamma^{p} _{\alpha c}g'_{c\bar{q}})  g'^{q\bar{j}}\\
\quad & = -\sum g'^{i\bar{p}}\big [ \phi_{p\bar{q}\alpha}-\sum_c  \Gamma^{p} _{\alpha c}g'_{c\bar{q}}\big ]  g'^{q\bar{j}}
\end{split}
\end{equation}

Also using  (\ref{cdt4}) and (\ref{cdt5}) we obtain 


\begin{equation}\label{ford}
\begin{split}
\sum g'^{\alpha\bar{\beta}} V_\alpha. \bar{V}_{\beta}.\Psi_1 =\sum g'^{\alpha\bar{\beta}} V_\alpha.& \bigg [-g'^{i\bar{p}}g'^{q\bar{r}}g'^{\bar{j}s}g'^{k\bar{t}}\big [ \phi_{p\bar{q}\bar{\beta}}-\sum_b  \Gamma_{\bar{\beta}\bar{b}} ^{\bar{q}} g'_{p\bar{b}}\big ]\phi_{i\bar{j}k}\phi_{\bar{r}s\bar{t}}\\
\quad & -g'^{i\bar{r}}g'^{\bar{j}p}g'^{\bar{q}s}g'^{k\bar{t}}\big [ \phi_{p\bar{q}\bar{\beta}}-\sum_b  \Gamma_{\bar{\beta}\bar{b}} ^{\bar{q}} g'_{p\bar{b}}\big ]\phi_{i\bar{j}k}\phi_{\bar{r}s\bar{t}}\\
\quad & -g'^{i\bar{r}}g'^{\bar{j}s}g'^{k\bar{p}}g'^{q\bar{t}}\big [ \phi_{p\bar{q}\bar{\beta}}-\sum_b  \Gamma_{\bar{\beta}\bar{b}} ^{\bar{q}} g'_{p\bar{b}}\big ]\phi_{i\bar{j}k}\phi_{\bar{r}s\bar{t}}\\
\quad & +g'^{i\bar{r}}g'^{\bar{j}s}g'^{k\bar{t}}
\big [ \phi_{i\bar{j}k\bar{\beta}} -\sum \Gamma ^{\bar{j}} _{\bar{\beta}\bar{c}} \phi_{i\bar{c}k}\big ]\phi_{\bar{r}s\bar{t}} \\
\quad & +g'^{i\bar{r}}g'^{\bar{j}s}g'^{k\bar{t}}
\big [ \phi_{\bar{r}s\bar{t}\bar{\beta}}-\sum \Gamma^{\bar{r}} _{\bar{\beta}\bar{c}} \phi_{\bar{c}s\bar{t}}  
-\sum \Gamma^{\bar{t}} _{\bar{\beta}\bar{c}} \phi_{\bar{r}s\bar{c}} \big ] \phi_{i\bar{j}k}
\bigg ]
\end{split}
\end{equation}


Sample terms of the above expansion    are computed in the relations (\ref{ap2a0}) and (\ref{ap2b0}) in  the appendix (\ref{app2}).  By assuming    orthogonality assumption at $q$   the  relations  (\ref{ap2a}) and (\ref{ap2b}) in the appendix (\ref{app2}) can be derived.


According to appendix (\ref{app2})  and (\ref{app3}), $\Delta_{g'}\Psi$ has the following form
\[
\begin{split}
\Delta_{g'} \Psi_1=\sum_{\theta (ijk) \leq 2}& g'^{\alpha\bar{\alpha}} g'^{i\bar{i}}g'^{j\bar{j}}g'^{k\bar{k}}( |\phi_{i\bar{j}k\alpha}|^2 +|\phi_{\bar{i}j\bar{k}\alpha}|^2)\\
\quad & +\sum_{\theta (ijk) \leq 2}  B_{i\bar{j}k\alpha}(g'^{\alpha\bar{\alpha}} g'^{i\bar{i}}g'^{j\bar{j}}g'^{k\bar{k}})^{1/2}\phi_{i\bar{j}k\alpha}+\sum_{\theta (ijk) \leq 2}
(g'^{\alpha\bar{\alpha}} g'^{i\bar{i}}g'^{j\bar{j}}g'^{k\bar{k}})^{1/2} B_{\bar{i}j\bar{k}\alpha}\phi_{\bar{i}j\bar{k}\alpha}+\sum C_{ijk\alpha}
\end{split}
\]
where $\theta$ is defined by

\begin{equation}\label{thet}
\theta (ijk)= \delta_{in}+\delta_{jn}+\delta_{kn}
\end{equation}

The fact that  the terms $|\phi_{i\bar{j}k\alpha}|^2 +|\phi_{\bar{i}j\bar{k}\alpha}|^2$ only have $\theta(ijk)<3$ is straightforward. But at first the terms $\phi_{n\bar{n}nn}$ and $\phi_{\bar{n}n\bar{n}n}$ also appear and we use lemma (\ref{lem30})  in appendix (\ref{app3}) to replace them by  the terms satisfying $\theta(ijk)<3$. 
Our claim is that
 
\begin{equation}\label{bsim}
|B _{i\bar{j}k\alpha}|\sim 0 \hspace{1cm} \text{and} \hspace{1cm}| B_{\bar{i}j\bar{k}\alpha} |\sim 0 \hspace{1cm} \text{and} \hspace{1cm} C _{ijk\alpha}\simeq 0
\end{equation}
to see this we need a case by case consideration applying the computation carried out  in the appendices \ref{app2} and \ref{app3}.
Since according to our assumption  the correction discussed in   lemma (\ref{lem29}) in the appendix (\ref{app4}) is   applied   we can set $\Gamma=0$. So  for instance on the right hand side of the equation (\ref{ap2a}) in the appendix (\ref{app3}) the first four lines  all belong to $C_{ijk\alpha}$ and  can be upper estimated by $\Psi_1\Psi$ which is itself dominated by  a term of the form $\frac{C\Psi_1}{|S|^4}$ where by theorem (\ref{thirdorder}) $C$  is a constant that only depends on $g_{reg}$  and $G$, $\delta$ and the lower bound of  $g'|_{U_p\cap D}$. 

In the last   line of  (\ref{ap2a})   we have the following term  

\begin{equation}\label{g'1}
\begin{split}
g'^{\alpha\bar{\alpha}}g'^{i\bar{i}}g'^{q\bar{q}}g'^{j\bar{j}}g'^{k\bar{k}}&|\phi_{i\bar{q}\bar{\alpha}}\phi_{i\bar{j}k}\phi_{\bar{q}w\bar{k}}| \leq \frac{(g'^{\alpha\bar{\alpha}})^{1/2} (g'^{j\bar{j}})^{1/2}}{(g'^{w\bar{w}})^{1/2}}\sqrt{\Psi}\Psi_1\leq \frac{C_{\ref{g'1}}}{|S|^4}\Psi_1
\end{split}
\end{equation}

for some constant $C_{\ref{g'1}}$ that can be approximated in terms of $g_{reg}$, $G$ and the lower bound of $g'|_{U_p\cap D}$.\\
 Similarly the coefficient of $(g'^{\alpha\bar{\alpha}}g'^{q\bar{q}}g'^{\bar{j}j}g'^{k\bar{k}})^{1/2}\phi_{\bar{q}j\bar{k}\alpha}$ in this term is  upper-estimated by 

\begin{equation}\label{cps12}
C_{\ref{cps12}}\Psi^{1/2}\Psi_1 ^{1/2}\sim 0 
\end{equation}

 where $C_{\ref{cps12}}$ is a constant that can be approximated.

The forth order derivative  $\phi_{i\bar{q}\bar{\alpha}\alpha}$ in the right hand side of (\ref{ap2a}) 
\[
 g'^{\alpha\bar{\alpha}}g'^{i\bar{i}}g'^{q\bar{q}}g'^{\bar{j}j}g'^{k\bar{k}}\big [ \phi_{i\bar{q}\bar{\alpha}\alpha} -\sum \Gamma^i _{\alpha u}\phi_{u\bar{q} \bar{\alpha}}-\sum_b  V_{\alpha}.(\Gamma_{\bar{\alpha}\bar{i}} ^{\bar{q}} g'_{i\bar{i}}) \big ]\phi_{i\bar{j}k}\phi_{\bar{q}j\bar{k}}
\]
can  be equal to 
 $\phi_{n\bar{n}\bar{n}n}$ if $i=q=\alpha=n$. In this case since $k$ is different from $n$  and  for the coefficient of this term we have
\[
|(g'^{i\bar{i}}g'^{\bar{j}j}g'^{k\bar{k}})^{1/2}\phi_{i\bar{j}k}\times (  g'^{q\bar{q}}g'^{\bar{j}j}g'^{k\bar{k}})^{1/2}\phi_{\bar{q}j\bar{k}}|\leq \Psi_1
\]

 . Now  using lemma (\ref{lem30}) in the appendix (\ref{app3}) we can replace
 $\phi_{n\bar{n}\bar{n}n}$ by the right hand side of (\ref{4th}). By theorem (\ref{thirdorder}) we also know that $\Psi_1\leq \frac{\sqrt{\Psi_1}}{|S|^2}$   therefore the relations (\ref{bsim}) remains true after substitution.

About the  term containing 5-th  order derivative in (\ref{ap2b}) in the appendix (\ref{app2})
\begin{equation}\label{5ord}
 g'^{\alpha\bar{\alpha}}g'^{i\bar{i}}g'^{\bar{j}j}g'^{k\bar{k}}\phi_{i\bar{j}k\bar{\alpha}\alpha}\phi_{\bar{i}j\bar{k}}
\end{equation}
we first apply the relation (\ref{5tho}) to replace $\phi_{i\bar{j}k\bar{\alpha}\alpha} $ by
$\phi_{\alpha\bar{\alpha} i\bar{j}k} $. Then from relation (\ref{bib0}) in the appendix (\ref{app3} and lemma (\ref{lem29}) in the appendix (\ref{app4}) we can see that the terms containing only third order derivatives of $\phi$ in $(\ref{5ord})$  is upper estimated by $\Psi\sqrt{\Psi_1}$.  It is also required that we use  lemma (\ref{lem30}) in appendix (\ref{app3}) to replace the terms of the form $\phi_{n\bar{n}\bar{n}n}$  by  the terms of the form 
$\phi_{abcd}$ such that at least one of the letters $a,b,c$ or $d$ is different from $n$ and $\bar{n}$. In conclusion over the open set $U_p$ we have proved the following proposition:


\begin{prop}\label{cor6}
\[
\begin{split}
\Delta_{g'} \Psi_1 =\sum_{\theta (ijk) \leq 2}& g'^{\alpha\bar{\alpha}} g'^{i\bar{i}}g'^{j\bar{j}}g'^{k\bar{k}}( |\phi_{i\bar{j}k\alpha}|^2 +|\phi_{\bar{i}j\bar{k}\alpha}|^2)\\
\quad & +\sum_{\theta (ijk) \leq 2}  B_{i\bar{j}k\alpha}(g'^{\alpha\bar{\alpha}} g'^{i\bar{i}}g'^{j\bar{j}}g'^{k\bar{k}})^{1/2}\phi_{i\bar{j}k\alpha}+
\sum_{\theta (ijk) \leq 2}
(g'^{\alpha\bar{\alpha}} g'^{i\bar{i}}g'^{j\bar{j}}g'^{k\bar{k}})^{1/2} B_{\bar{i}j\bar{k}\bar{\alpha}}\phi_{\bar{i}j\bar{k}\alpha}+\sum C_{ijk\alpha}
\end{split}
\]

where
\[
|B_{i\bar{j}k\alpha}|\sim 0 \hspace{1cm} \text{and} \hspace{1cm}| B_{\bar{i}j\bar{k}\alpha}|\sim 0 \hspace{1cm} \text{and} \hspace{1cm} C _{ijk\alpha}\simeq 0
\]
\end{prop}
\vspace{1cm}

We now consider the cut-off function $\chi:X\rightarrow [0,1]$   defined by (\ref{chitild}) with the same properties described below that relation.
 Let 
\begin{equation}\label{eps}
\epsilon_1 := \kappa_1\sup_{U_p} \Psi_1 
\end{equation}
where $\kappa_1$ is the  a parameter that will be determined .
Then

 \begin{equation}\label{laplog}
\Delta_{g'} (\log(\chi \Psi_1+\epsilon_1))= \frac{(\Delta_{g'}\chi) \Psi_1 +\chi (\Delta _{g'} \Psi_1 )}{\chi \Psi_1 +\epsilon_1}+\sum g'^{\alpha\bar{\alpha}}\frac{2\Re(\chi_{\alpha} \Psi_{1,\bar{\alpha}})}{(\chi\Psi_1+\epsilon_1)}
-\sum  g'^{\alpha\bar{\alpha}}\frac{ |\chi \Psi_{1,\bar{\alpha}}+\chi_{\bar{\alpha}} \Psi_1 |^2 }{(\chi\Psi_1+\epsilon_1)^2} 
\end{equation}

 so by Cauchy- Schwarz

\begin{equation}\label{laplap2}
\begin{split}
\Delta_{g'} (\log(\chi \Psi_1+\epsilon_1)) \geq & \frac{\chi}{\chi\Psi_1 + \epsilon_1}[ \Delta_{g'} \Psi_1 ] -\sum\frac{2|\chi_{\alpha}|g'^{\alpha\bar{\alpha}}}{\chi\Psi_1+\epsilon_1}| \Psi_{1,\bar{\alpha}}|\\
\quad & -\sum \frac{2(\chi ) ^2 g'^{\alpha\bar{\alpha}}}{(\chi \Psi_1 +\epsilon_1 )^2} | \Psi_{1,\bar{\alpha}}|^2\\
\quad & -\frac{|\Delta_{g'} (\chi)| }{\chi\Psi_1 +\epsilon_1 }\Psi_1 -\sum \frac{2|\chi_{\bar{\alpha}}|^2}{(\chi\Psi_1 +\epsilon_1)^2}|\Psi_1 |^2
\end{split}
\end{equation}

We utilize orthogonality assumption at $q$ to the relation   (\ref{ford}) and  we set 

\begin{equation}\label{laplap22}
\Psi_{1,\bar{\alpha}}=\bar{V}_{\alpha}.\Psi_1 = \mathcal{T}_1+\mathcal{T}_2 
\end{equation}

where $\mathcal{T}_1$ and $\mathcal{T}_2$ are definde by    
\begin{equation}\label{tcal1}
\begin{split}
\mathcal{T}_1:=&  \sum_{\theta(ijk)\leq 2} \bigg [-g'^{i\bar{i}}g'^{q\bar{q}}g'^{\bar{j}j}g'^{k\bar{k}}\big [ \phi_{i\bar{q}\bar{\alpha}}-  \Gamma_{\bar{\alpha}\bar{i}} ^{\bar{q}} g'_{i\bar{i}}\big ]\phi_{i\bar{j}k}\phi_{\bar{q}j\bar{k}}\\
\quad & -g'^{i\bar{i}}g'^{\bar{j}j}g'^{\bar{q}q}g'^{k\bar{k}}\big [ \phi_{j\bar{q}\bar{\alpha}}-  \Gamma_{\bar{\alpha}\bar{j}} ^{\bar{q}} g'_{j\bar{j}}\big ]\phi_{i\bar{j}k}\phi_{\bar{i}q\bar{k}}\\
\quad & -g'^{i\bar{i}}g'^{\bar{j}j}g'^{k\bar{k}}g'^{q\bar{q}}\big [ \phi_{k\bar{q}\bar{\alpha}}- \Gamma_{\bar{\alpha}\bar{k}} ^{\bar{q}} g'_{k\bar{k}}\big ]\phi_{i\bar{j}k}\phi_{\bar{i}j\bar{q}}\bigg ]\\
\end{split}
\end{equation}

and

\begin{equation}\label{tcal2}
\begin{split}
\mathcal{T}_2&:=\sum_{\theta (ijk)\leq 2  } g'^{i\bar{i}}g'^{\bar{j}j}g'^{k\bar{k}}
\big [ \phi_{i\bar{j}k\bar{\alpha}} -\sum \Gamma ^{\bar{j}} _{\bar{\alpha}\bar{c}} \phi_{i\bar{c}k}\big ]\phi_{\bar{i}j\bar{k}} \\
\quad & +g'^{i\bar{i}}g'^{\bar{j}j}g'^{k\bar{k}}
\big [ \phi_{\bar{i}j\bar{k}\bar{\alpha}}-\sum \Gamma^{\bar{i}} _{\bar{\alpha}\bar{c}} \phi_{\bar{c}j\bar{k}}  
-\sum \Gamma^{\bar{k}} _{\bar{\alpha}\bar{c}} \phi_{\bar{i}j\bar{c}} \big ] \phi_{i\bar{j}k}
\end{split}
\end{equation}
  As before by using   lemma (\ref{lem29}) we can  assume $\Gamma (q)=0.$ Hence

\begin{equation}\label{tcal12}
\begin{split}
\mathcal{T}_1=&  \sum_{\theta(ijk)<3} \bigg [-g'^{i\bar{i}}g'^{q\bar{q}}g'^{\bar{j}j}g'^{k\bar{k}}  \phi_{i\bar{q}\bar{\alpha}} \phi_{i\bar{j}k}\phi_{\bar{q}j\bar{k}}
-g'^{i\bar{i}}g'^{\bar{j}j}g'^{\bar{q}q}g'^{k\bar{k}} \phi_{j\bar{q}\bar{\alpha}} \phi_{i\bar{j}k}\phi_{\bar{i}q\bar{k}} -g'^{i\bar{i}}g'^{\bar{j}j}g'^{k\bar{k}}g'^{q\bar{q}}  \phi_{k\bar{q}\bar{\alpha}} \phi_{i\bar{j}k}\phi_{\bar{i}j\bar{q}}\bigg ]\\
\end{split}
\end{equation}

and

\begin{equation}\label{tcal22}
\mathcal{T}_2=\sum_{\theta (ijk)\leq 2  }\bigg [ g'^{i\bar{i}}g'^{\bar{j}j}g'^{k\bar{k}}
 \phi_{i\bar{j}k\bar{\alpha}}\phi_{\bar{i}j\bar{k}}  +g'^{i\bar{i}}g'^{\bar{j}j}g'^{k\bar{k}}
 \phi_{\bar{i}j\bar{k}\bar{\alpha}} \phi_{i\bar{j}k}
\bigg ]
\end{equation}

 From (\ref{laplap22}) we deduce that

\begin{equation}\label{lapsi1}
\sum_{\alpha}\frac{|\chi_{\alpha}|g'^{\alpha\bar{\alpha}}}{\chi\Psi_1+\epsilon_1}| \Psi_{1,\bar{\alpha}}|\leq  \sum\frac{|\chi_{\alpha}|g'^{\alpha\bar{\alpha}}}{\chi\Psi_1+\epsilon_1} (|\mathcal{T}_1|+|\mathcal{T}_2| )
\end{equation}
 
and from the inequality

\[
|\mathcal{T}_1+\mathcal{T}_2 |^2\leq 2( \mathcal{T}_1 ^2 +\mathcal{T}_2 ^2  )
\]


we also have

\begin{equation}\label{lapsi2}
\sum_{\alpha} \frac{2\chi ^2 g'^{\alpha\bar{\alpha}}}{(\chi \Psi_1 +\epsilon_1)^2} | \Psi_{1,\bar{\alpha}}|^2\leq \sum_{\alpha} \frac{4\chi ^2 g'^{\alpha\bar{\alpha}}}{(\chi \Psi_1 +\epsilon_1)^2}(\mathcal{T}_1 ^2 +\mathcal{T}_2 ^2 )
\end{equation}

Therefore by (\ref{laplap2}),  (\ref{lapsi1}),  (\ref{lapsi2}), and  (\ref{laplap22})  we obtain

\begin{equation}\label{laplap222}
\begin{split}
\Delta_{g'} (\log(\chi \Psi_1+\epsilon_1)) \geq & \frac{|\chi|}{\chi\Psi_1 + \epsilon_1}[ \Delta_{g'} \Psi_1 ]\\
\quad &  -\sum\frac{2|\chi_{\alpha}|g'^{\alpha\bar{\alpha}}}{\chi\Psi_1+\epsilon_1}| \mathcal{T}_1+\mathcal{T}_2 |\\
\quad & -\sum \frac{4\chi ^2g'^{\alpha\bar{\alpha}}}{(\chi \Psi_1 +\epsilon_1)^2} (\mathcal{T}_1^2+\mathcal{T}_2^2 )\\
\quad & -\frac{|\Delta_{g'} (\chi)| }{\chi\Psi_1 +\epsilon_1}\Psi_1 -\sum \frac{2|\chi_{\bar{\alpha}}|^2}{(\chi\Psi_1 +\epsilon_1 )^2}|\Psi_1 |^2
\end{split}
\end{equation}

\begin{lemma}\label{lemsi}

\begin{equation}\label{k2s41}
\sum_{\alpha} \frac{ |\chi^2 |g'^{\alpha\bar{\alpha}}}{(\chi \Psi_1 +\epsilon_1)^2}\mathcal{T}_1 ^2\leq \frac{\chi^{2\tau}C_{\ref{k2s41}}}{\kappa _1 ^2 |S|^4}
\end{equation}
\begin{equation}\label{k2s42}
\sum_{\alpha}\frac{|\chi_{\alpha}|g'^{\alpha\bar{\alpha}}}{\chi\Psi_1+\epsilon_1} |\mathcal{T}_1|\leq   \frac{C_{\ref{k2s42}}\chi^{\tau}}{\kappa_1 \delta |S|^2}
\end{equation}
 
 where $C_{\ref{k2s41}}$ and $C_{\ref{k2s42}}$ depend only on $G$ and $g_{reg}$. We also have 

\begin{equation}\label{lap2t20}
\begin{split}
\sum\frac{|\chi_{\alpha}|g'^{\alpha\bar{\alpha}}}{\chi\Psi_1+\epsilon_1}|\mathcal{T}_2|\leq  \sum\frac{C_{\ref{lap2t20}}|\chi_{\alpha}|(g'^{\alpha\bar{\alpha}})^{1/2}(\Psi_1 )^{1/2}}{\chi\Psi_1+\epsilon_1} &(\big [g'^{\alpha\bar{\alpha}}g'^{i\bar{i}}g'^{j\bar{j}}g'^{k\bar{k}}|\phi_{i\bar{j} k\bar{\alpha}}|^2 \big ] ^{1/2}\\
\quad & +[g'^{\alpha\bar{\alpha}}g'^{i\bar{i}}g'^{j\bar{j}}g'^{k\bar{k}}|\phi_{\bar{i}j \bar{k}\bar{\alpha}}|^2 \big ] ^{1/2} )
\end{split}
\end{equation}

\begin{equation}\label{lap2t220}
\sum \frac{\chi ^2 g'^{\alpha\bar{\alpha}}}{(\chi\Psi_1 +\epsilon_1)^2} \mathcal{T}_2 ^2 \leq   \sum \frac{\chi ^2 C_{\ref{lap2t220}}\Psi_1}{(\chi \Psi_1 +\epsilon_1)^2}g'^{\alpha\bar{\alpha}}g'^{i\bar{i}}g'^{j\bar{j}}g'^{k\bar{k}}( |\phi_{i\bar{j} k\bar{\alpha}}|^2+ |\phi_{\bar{i}j\bar{k}\bar{\alpha}}|^2)
\end{equation}

where the constants $C_{\ref{lap2t20}}$  and $C_{\ref{lap2t220}}$ depend on the number of the terms..
\end{lemma}

\begin{proof}
By applying Cauchy-Shcwarz  inequality to the definition of $\mathcal{T}_1 $ (\ref{tcal12}) on can deduce 
\begin{equation}\label{lap2t1}
\sum_{\alpha} \frac{ \chi ^2 g'^{\alpha\bar{\alpha}}}{(\chi \Psi_1 +\epsilon_1)^2} \mathcal{T}_1 ^2  \leq C_{\ref{lap2t1}}  \frac{  \chi^2 \Psi_1 ^2 \Psi }{(\chi \Psi_1 +\epsilon_1)^2}
\end{equation}

 where $C_{\ref{lap2t1}}$ depends on the number of terms in the left hand side.
\begin{equation}\label{sichel1}
\frac{ \chi^2 \Psi_1 ^2 \Psi }{(\chi \Psi_1 +\epsilon_1)^2}\leq \frac{1}{\kappa_1 ^2}\chi^{2} \Psi\leq \frac{\mathscr{C} \chi ^{2}}{\kappa_1 ^2|S|^4}
\end{equation}
where $\mathscr{C}$ is the constant introduced in theorem (\ref{thirdorder}). 
Similar argument shows that

\begin{equation}\label{mohsa}
\sum\frac{|\chi_{\alpha}|g'^{\alpha\bar{\alpha}}}{\chi\Psi_1+\epsilon_1} |\mathcal{T}_1|\leq \frac{C_{\ref{k2s42}}\chi}{\kappa_1 \delta |S|^2}
\end{equation}
where $C_{\ref{k2s42}}$ depends on $g_{reg}$, $G$.

By Cauchy-Schwarz inequality we get  
\begin{equation}\label{t22t}
\begin{split}
\mathcal{T}_2 ^2 \leq C_{\ref{lap2t220}} \sum g'^{i\bar{i}}g'^{j\bar{j}}g'^{k\bar{k}}|  &\phi_{i\bar{j} k\bar{\alpha}}|^2g'^{i\bar{i}}g'^{j\bar{j}}g'^{k\bar{k}}(|\phi_{\bar{i}j\bar{k}}|^2 )\\
\quad  &+C_{\ref{lap2t220}} \sum g'^{i\bar{i}}g'^{j\bar{j}}g'^{k\bar{k}} |\phi_{\bar{i}j\bar{k}\bar{\alpha}}|^2g'^{i\bar{i}}g'^{j\bar{j}}g'^{k\bar{k}} (|\phi_{i\bar{j} k}|^2)
\end{split}
\end{equation}
where  $C_{\ref{lap2t220}}$ depends on the number of terms in the sum of the right hand side. Hence 
\begin{equation}\label{22tt}
\sum \frac{4\chi ^2 g'^{\alpha\bar{\alpha}}}{(\chi\Psi_1 +\epsilon_1 )^2} \mathcal{T}_2 ^2 \leq   \sum \frac{4\chi ^2C_{\ref{lap2t220}}\Psi_1}{(\chi \Psi_1 +\epsilon_1)^2}g'^{\alpha\bar{\alpha}}g'^{i\bar{i}}g'^{j\bar{j}}g'^{k\bar{k}}( |\phi_{i\bar{j} k\bar{\alpha}}|^2+ |\phi_{\bar{i}j\bar{k}\bar{\alpha}}|^2)
\end{equation}
and again by Cauchy-Schwarz  
\begin{equation}\label{lap2t2}
\begin{split}
\sum\frac{|\chi_{\alpha}|g'^{\alpha\bar{\alpha}}}{\chi\Psi_1+\epsilon_1}|\mathcal{T}_2|\leq  \sum\frac{C_{\ref{lap2t20}}|\chi_{\alpha}|(g'^{\alpha\bar{\alpha}})^{1/2}(\Psi_1 )^{1/2}}{\chi\Psi_1+\epsilon_1} &(\big [g'^{\alpha\bar{\alpha}}g'^{i\bar{i}}g'^{j\bar{j}}g'^{k\bar{k}}|\phi_{i\bar{j} k\bar{\alpha}}|^2 \big ] ^{1/2}\\
\quad & +[g'^{\alpha\bar{\alpha}}g'^{i\bar{i}}g'^{j\bar{j}}g'^{k\bar{k}}|\phi_{\bar{i}j \bar{k}\bar{\alpha}}|^2 \big ] ^{1/2} )
\end{split}
\end{equation}


 \end{proof}

Consequently   using  proposition (\ref{cor6}), relation (\ref{laplap22}) lemma (\ref{lemsi})   we obtain

\begin{prop}\label{corya}
\[
\begin{split}
\Delta_{g'}(\log (\chi\Psi_1 + \epsilon)) \geq \sum_{\substack{\theta(ijk)\leq 2\\ \alpha}} \big [ ( F_{i\bar{j}k\alpha}x_{i\bar{j}k\alpha} ^2 +F_{\bar{i}j\bar{k}\alpha}x_{\bar{i}j\bar{k}\alpha} ^2)+ G_{i\bar{j}k\alpha} x_{i\bar{j}k\alpha}+ G_{\bar{i}j\bar{k}\alpha} x_{\bar{i}j\bar{k}\alpha}+H_{ijk\alpha}\big ]\\
\end{split}
\]

where 

\[
x_{i\bar{j}k\alpha} =\big [ g'^{\alpha\bar{\alpha}}g'^{i\bar{i}}g'^{j\bar{j}}g'^{k\bar{k}}|\phi_{i\bar{j}k\alpha}|^2 \big ]^{1/2}
\]

\[
x_{\bar{i}j\bar{k}\alpha} =\big [ g'^{\alpha\bar{\alpha}}g'^{i\bar{i}}g'^{j\bar{j}}g'^{k\bar{k}}|\phi_{\bar{i}j\bar{k}\alpha}|^2 \big ]^{1/2}
\]

\[
F_{\bar{i}j\bar{k}\alpha}=F_{i\bar{j}k\alpha}= \frac{|\chi |}{\chi\Psi_1 +\epsilon_1}-\frac{4\chi ^2C_{\ref{lap2t20}}\Psi_1}{(\chi\Psi_1 +\epsilon _1)^2}
\]
\[
\begin{split}
G_{i\bar{j}k\alpha}=& \frac{|\chi |}{\chi\Psi_1 +\epsilon_1 }B_{i\bar{j}k\alpha}-\frac{C_{\ref{lap2t220}}|\chi_{\alpha}| (g'^{\alpha\bar{\alpha}})^{1/2}(\Psi_1)^{1/2}}{\chi \Psi_1 +\epsilon_1}\\
\end{split}
\]

\[
\begin{split}
\sum H_{ijk\alpha}=&\frac{\chi}{\chi\Psi_1+\epsilon_1}\sum C_{ijk\alpha}-\sum \frac{2|\chi_{\alpha}|g'^{\alpha\bar{\alpha}} }{\chi\Psi_1+\epsilon_1}|\mathcal{T}_1| - \sum \frac{4\chi ^2g'^{\alpha\bar{\alpha}}}{(\chi \Psi_1 +\epsilon_1)^2} \mathcal{T}_1^2 \\
\quad & -\frac{|\Delta_{g'} (\chi)| }{\chi\Psi_1 +\epsilon_1}\Psi_1 -\sum \frac{2|\chi_{\bar{\alpha}}|^2}{(\chi\Psi_1 +\epsilon_1)^2}|\Psi_1 |^2
\end{split}
\]
 
\end{prop}

As an immediate corollary we can derive

\begin{cor}\label{altcor1}

\begin{equation}\label{altcor}
\Delta_{g'}(\log (\chi\Psi_1 + \epsilon_1)) \geq -\sum_{}\bigg ( \frac{G_{\bar{i}j\bar{k}\alpha}^2}{4F_{\bar{i}j\bar{k}\alpha}}+ \frac{G_{i\bar{j}k\alpha}^2}{4F_{i\bar{j}k\alpha}}\bigg )+\sum H_{ijk\alpha}
\end{equation}

\end{cor}

According to proposition (\ref{cor6}) relations (\ref{sichel1}) and (\ref{mohsa}),  and  theorem (\ref{thirdorder}) it can be seen that
\begin{equation}\label{sichel2}
\big | \sum H_{ijk\alpha} \big | \leq C_{\ref{sichel2}}  \frac{\chi}{\kappa_1 ^2|S|^4} 
\end{equation}

where $C_{\ref{sichel2}}$ depends on $g_{reg}$, $G$, $\delta$ and the lower bound of $g'|_{U_p \cap D}$.
This means in particular that $H_{ijk\alpha}\simeq 0$.
 Hence  by proposition (\ref{cor6})   $B_{i\bar{j}k\alpha}\sim 0$ and we get

\begin{equation}\label{giefom}
\begin{split}
|\frac{G_{\bar{i}j\bar{k}\alpha}^2}{F_{\bar{i}j\bar{k}\alpha}} |\leq   |\frac{ (\frac{\chi B_{i\bar{j}k\alpha}}{\chi\Psi_1 +\epsilon_1})^2}{\frac{(1-4C_{\ref{lap2t20}})\chi^2\Psi_1+\chi\epsilon_1}{(\chi\Psi_1+\epsilon_1)^2}}| + |\frac{ \frac{C_{\ref{lap2t220}}^2|\chi_{\alpha}|^2 g'^{\alpha\bar{\alpha}} \Psi_1}{(\chi\Psi_1+\epsilon_1 )^2} }{\frac{(1-4 C_{\ref{lap2t20}})\chi^2\Psi_1+\chi\epsilon_1}{(\chi\Psi_1+\epsilon_1)^2}}|  
\end{split}
\end{equation}
\begin{equation}\label{c'lp1}
 |\frac{ (\frac{\chi B_{i\bar{j}k\alpha}}{\chi\Psi_1 +\epsilon_1})^2}{\frac{(1-4C_{\ref{lap2t20}})\chi^2\Psi_1+\chi\epsilon_1}{(\chi\Psi_1+\epsilon_1)^2}}|=\frac{\chi (B_{i\bar{j}k\alpha})^2}{|(1-4 C_{\ref{lap2t20}})\chi\Psi_1+\epsilon_1 |}\leq \frac{\chi\Psi_1}{|(1-4C_{\ref{lap2t20}})\Psi_1+\epsilon_1|}\times\frac{C_{\ref{c'lp1}}}{|S|^4}\leq \frac{C_{\ref{c'lp1}} \chi}{\kappa_1 |S|^4}
\end{equation}

\begin{equation}
 |\frac{ \frac{C_{\ref{lap2t220}}^2|\chi_{\alpha}|^2 g'^{\alpha\bar{\alpha}} \Psi_1}{(\chi\Psi_1+\epsilon_1 )^2} }{\frac{(1-4C_{\ref{lap2t20}})\chi^2\Psi_1+\chi\epsilon_1}{(\chi\Psi_1+\epsilon_1)^2}}|  =\frac{C_{\ref{lap2t220}}^2\frac{|\chi_{\alpha}|^2}{\chi^2} g'^{\alpha\bar{\alpha}} \chi\Psi_1}{|(1-4C_{\ref{lap2t20}})\chi\Psi_1+\epsilon_1|} 
\end{equation}
 
So if $\alpha = z$  by appying lemma (\ref{m'm'1})in the appendix (\ref{app7})  we get 

\begin{equation}\label{c'lp2}
\frac{C_{\ref{lap2t220}}^2\frac{|\chi_{\alpha}|^2}{\chi^2} g'^{\alpha\bar{\alpha}} \chi\Psi_1}{|(1-4C_{\ref{lap2t20}})\chi\Psi_1+\epsilon_1|} \leq \frac{C_{\ref{c'lp2}} \chi}{\kappa_1 \delta ^2 |S|^2}
\end{equation}\label{c'lp3}
for some constant  $C_{\ref{c'lp2}}$  which only depends on $G$ and $g_{reg}$.
 Also if $\alpha \neq z$ we can see that 

\begin{equation}
\frac{C_{\ref{lap2t220}}^2\frac{|\chi_{\alpha}|^2}{\chi^2} g'^{\alpha\bar{\alpha}} \chi\Psi_1}{|(1-4C_{\ref{lap2t20}})\chi\Psi_1+\epsilon_1|} \leq \frac{C_{\ref{c'lp2}} \chi }{\kappa_1 \delta ^2 \mathscr{M}'^2}
\end{equation}

    From corollary (\ref{altcor1}) and the above estimations we can  deduce that

\begin{equation}\label{theo1010}
\Delta_{g'}  (\log (\chi\Psi_1 + \epsilon_1))\geq  -C_{\ref{sichel2}}  \frac{\chi}{\kappa_1 ^2|S|^4} - \frac{C_{\ref{c'lp1}} \chi}{\kappa_1 |S|^4}- \frac{C_{\ref{c'lp2}} \chi}{\kappa_1 \delta ^2 |S|^2}- \frac{C_{\ref{c'lp2}} \chi }{\kappa_1 \delta ^2 \mathscr{M}'^2}
\end{equation}

\begin{theo}\label{theo1010}   For $\epsilon_1$ defined by  the relation (\ref{eps}) we have
\begin{equation}\label{theo10}
\Delta_{g'}  (\log (\chi\Psi_1 + \epsilon_1))\geq C_{\ref{theo10}}\frac{\chi}{\kappa_1^2 |S|^4} 
\end{equation}
where $C_{\ref{theo10}}$ depends on $g_{reg}$ and $G$, $\delta$ and the lower bound of $g'|_{U_p\cap D}$.
\end{theo}

\begin{theo}\label{3rdorder2}There exists a constant $\mathscr{C}_1$ which depends on $G$ and $g_{reg}$ such that the following inequality holds on $X$
\[
 \Psi_1\leq \mathscr{C}_1
\]
 
\end{theo}

\begin{proof}

To prove the above theorem we apply the relation  (\ref{mainineq2}) and theorem (\ref{theo1010}) to derive

\begin{equation}\label{ineq3rd}
\begin{split}
\Delta_{g'} (\log(\chi\Psi_1 +\epsilon_1 ) &+B ( \log (\eta \Psi +\epsilon) +A \Delta_{reg}( \chi^{\tau}\phi ))) \geq \\
\quad   &B C' \chi^{\tau}\sum_{p,q,l} g'^{p\bar{p}}  g'^{p\bar{p}}  |\phi_{p\bar{q}\bar{l}} |^2\\
\quad &+BA\frac{C''  }{|S|}  (\chi^{\tau}\sum_{p,q,l}  g'^{p\bar{p}}  g'^{p\bar{p}} |\phi_{p\bar{q}\bar{l}} |^2 )^{1/2}\\
\quad &  C_{\ref{theo10}}  \frac{\chi}{\kappa_1^2 |S|^4} + B\frac{C'''}{|S|^2} \\ 
\quad & +B\bigg [\frac{C_{\ref{lem20nn}} \chi^{\tau} }{\kappa|S|^4}\bigg ]
\end{split}
\end{equation}
where $B>0$ is a positive constant and  $C'$ , $C''$ and $C'''$ are defined  by (\ref{c'c'}), (\ref{c''c''}) and (\ref{c'''}) in lemma (\ref{lem255}). 
 Since $C_{\ref{lem20nn}}>0$ if $B$ is large enough then  $B\frac{C_{\ref{lem20nn}} \chi^{\tau} }{\kappa|S|^4}$
can dominate $C_{\ref{theo10}}  \frac{\chi}{\kappa_1^2 |S|^4} $.  From this we can deduce that

\begin{lemma}
For $B$ large enough we have 
\begin{equation}\label{ineq3rd}
\begin{split}
\Delta_{g'} (\log(\chi\Psi_1 +\epsilon_1 ) &+B ( \log (\eta \Psi +\epsilon) +A \Delta_{reg}( \chi^{\tau}\phi ))) \geq \\
\quad   & BC'  \chi^{\tau}\sum_{p,q,l}  g'^{p\bar{p}}  g'^{p\bar{p}}|\phi_{p\bar{q}\bar{l}} |^2\\
\quad &+BA\frac{C'' }{|S|}  (\chi^{\tau}\sum_{p,q,l}  g'^{p\bar{p}}  g'^{p\bar{p}} |\phi_{p\bar{q}\bar{l}} |^2 )^{1/2}\\
 \quad & +B\frac{C'''}{|S|^2}
\end{split}
\end{equation}
where   $C'$ , $C''$ and $C'''$ are defined  by (\ref{c'c'}), (\ref{c''c''}) and (\ref{c'''}) in lemma (\ref{lem255}).

\end{lemma}

Likewise in previous section we consider the point $q_0\in X$ where the maximum of $ \log(\chi\Psi_1 +\epsilon_1 ) +B ( \log (\eta \Psi +\epsilon) +A \Delta_{reg}( \chi^{\tau}\phi ))$ occurs. Since the maximum can occur on $D$ we need to multiply both sides of \ref{mainineq2} by $|S|^2$ and repeat the argument as in the proposition (\ref{secondorder}). Then from the above lemma we find that


\begin{equation}\label{c'c''}
\begin{split}
\bigg [ C' _1 \bigg (\chi ^{\tau}\sum_{p,q,l} g'^{p\bar{p}}  g'^{p\bar{p}}& |\phi_{p\bar{q}\bar{l}} |^2\bigg ) + AC''_1 \bigg ( \chi^{\tau} \sum_{p,q,l}  g'^{p\bar{p}}  g'^{p\bar{p}} |\phi_{p\bar{q}\bar{l}} |^2 \bigg )^{1/2}\\
\quad &   + C''' _1 \bigg ] (q_0)\leq 0 
\end{split}
\end{equation}

from this inequality  and the definition of $C'$, $C''$ and $C'''$ in lemma (\ref{lem255}) it follows that if $\tau\leq \delta\leq \mathscr{M}'$
 then we get 
\begin{equation}\label{ubound1}
\begin{split}
[\chi^{\tau}\sum_{p,q,l} g'^{p\bar{p}}  g'^{p\bar{p}}|^2] (q_0 )&\leq \frac{-AC'' _1 +\sqrt{(AC''_1  )^2 -4C'  _1C''' _1  }}{2C' _1}\\
\quad &\leq  \tilde{C}_{\ref{ubound1}} 
\end{split}
\end{equation}

where $\tilde{C}_{\ref{ubound1}}$ depends only on $g_{reg}$ and $G$.    We have thus proved that

\begin{lemma}   \label{lemc1}
\begin{equation}\label{constc}
(\chi ^{\tau}\Psi_1)(q_0 )\leq  \mathscr{C}_1
\end{equation}
for some constant $\mathscr{C}_1$ which only depends on  $G$, $g_{reg}$. 
\end{lemma}


From the definition of $q_0$  we know that for all  $x\in U_p $ we have
\begin{equation}\label{kok}
\begin{split}
( \log(\chi\Psi_1 +\epsilon_1 ) +B \log (\eta \Psi +\epsilon) +& AB \Delta_{reg}( \chi^{\tau}\phi ))(x)\\
\quad &\leq ( \log(\chi\Psi_1 +\epsilon_1 ) +B  \log (\eta \Psi +\epsilon) +A B\Delta_{reg}( \chi^{\tau}\phi ))(q_0)
\end{split}
\end{equation}

We set  
\[
a_1:=\|\chi^{1/2} \Psi_1\|_{\infty} , \hspace{1cm} b_1 := \|\chi \Psi_1\|_{\infty} 
\]
  and we recall   the definition  (\ref{a}) for   $a:=\|\eta_{1/2} \Psi\|_{\infty}$. Therefore we  have

\[
\epsilon = \kappa \|\eta_{1/2} \Psi\|_{\infty}=\kappa a
\]

 and 

\[
\epsilon_1 = \kappa \|\chi^{1/2} \Psi_1\|_{\infty}= \kappa a_1
\]
 



\begin{equation}\label{nesbat}
\lim _{\kappa \rightarrow \infty}\log \frac{ (\eta \Psi +\epsilon)(x) }{(\eta \Psi +\epsilon) (q_0)}=0
\end{equation}

uniformly with respect to $x$. Hence from (\ref{kok}) and (\ref{b'b'}) we conclude that 
\[
\log \frac{b_1+\kappa a_1}{\mathscr{C}_1+\kappa a_1} \leq  A_2 
 \]

where   $A_2=  2 AC_{\ref{ineq3rd}} B' - \log \frac{ (\eta \Psi +\epsilon)(x) }{(\eta \Psi +\epsilon) (q_0)}$ and where $B'$ is defined by relation (\ref{b'b'}). In order to repeat the argument as int the proof of theorem  (\ref{thirdorder})
 we need to  to assure that $\lim_{\kappa\rightarrow \infty} A_2= 0 $.  This occurs due to (\ref{nesbat}) and the fact that  as we mentioned before $C_{\ref{ineq3rd}} $  can be determined in such a way that it does not  dependent on $\kappa$. So we can repeat the proof of theorem (\ref{thirdorder}) to conclude that 
\[
b_1\leq \mathscr{C}_1
\]

 where $\mathscr{C}_1$ is presented in lemma (\ref{lemc1}). This completes the proof of theorem (\ref{3rdorder2}).

\end{proof}

Using theorem (\ref{3rdorder2}) we now prove that 

\begin{prop}\label{prop8}
The lower bound  of $g'|_{D}$  can be determined in terms of $G$ and $g_{reg}$.
\end{prop}
\begin{proof}
 
We  define $v: D\rightarrow \mathbb{R}$ by

\[
v=\bigg [\frac{ \det g'|_D }{\det g_{reg}|_{D}} \bigg ]^{\alpha}  
\] 

for some  constant  $\alpha$  that will be determined.
We consider a local holomorphic coordinates system  $(w_1,...,w_{n-1})$ around some point $p\in D$ on $D$  
in such a way that $\{\frac{\partial}{\partial w_1}(p),...,\frac{\partial}{\partial w_{n-1}}(p)\}$ form an orthogonal basis at $p$ with respect to both  $g_{reg}$ and $g'$. 
 We can  assume that this basis is orthonormal with respect to $g_{reg}$ as well. We take a smooth  curve $w(t)= (w_1 (t),...,w_{n-1} (t))$  of normal velocity  
with respect to the euclidean norm in this coordinates  and satisfying $w(0)=p$. We  then set $u(t):= v\circ w (t)$. So at the point $p$  we have 
 
\[
\begin{split}
\frac{du}{dt}&=\sum_{j=1} ^{n-1} \alpha\frac{d}{dt}(\prod_{i}  g_{w_i \bar{w}_i} ) (\prod_{i\neq j}  g_{w_i \bar{w}_i})^{\alpha-1} \\
\quad & =\sum_{j=1} ^{n-1} \alpha \sum_j \frac{dg_{w_j\bar{w}_j}}{dt}\big (  \prod_{i\neq j}  g_{w_i \bar{w}_i}\big)(\prod_{i\neq j}  g_{w_i \bar{w}_i})^{\alpha-1}\\
\quad & =\sum_{j=1} ^{n-1}  \alpha(g_{w_j\bar{w}_j})^{-1}\frac{dg_{w_j\bar{w}_j}}{dt}(\prod_{1\leq i \leq n-1}  g_{w_i \bar{w}_i})^{\alpha}\\
\end{split}
\]

\begin{equation}
\begin{split}
\frac{du}{dt} u^{-\alpha}&= \sum_{j=1} ^{n-1}  \alpha(g_{w_j\bar{w}_j})^{-1}\frac{d(g_{w_j\bar{w}_j})}{dt} \\
\quad &=  \sum_{j=1} ^{n-1}  \alpha(g_{w_j\bar{w}_j})^{-1}(\sum_k g_{w_j\bar{w}_j w_k} \frac{d w_k}{dt} + \sum_k g_{w_j\bar{w}_j\bar{w}_k} \frac{d\bar{w}_k }{dt}) \\
\end{split}
\end{equation}

Since  $w(t)$  has  normal velocity,

\begin{equation}\label{u'ualpha}
\begin{split}
|\frac{du}{dt} u^{-\alpha}|&\leq \sum_{j=1} ^{n-1}  \alpha(g_{w_j\bar{w}_j})^{-1}(\sum_k  |g_{w_j\bar{w}_j w_k}|^2 +\sum_k  |g_{w_j\bar{w}_j\bar{w}_k}|^2)^{1/2} \big(\sum_k |\frac{d w_k}{dt} |^2  + \sum_k |\frac{d\bar{w}_k }{dt}|^2\big  )^{1/2}\\
\quad & \leq   \sum_{j=1} ^{n-1}  \alpha  (\sum_k (g_{w_j\bar{w}_j})^{-2}  |g_{w_j\bar{w}_j w_k}|^2 +\sum_k  (g_{w_j\bar{w}_j})^{-2}|g_{w_j\bar{w}_j\bar{w}_k}|^2)^{1/2}
\end{split}
\end{equation}

Thus from proposition (\ref{secondorder}) andtheorem (\ref{3rdorder2}) we obtain
\begin{equation}\label{u'ualpht}
\begin{split}
(g_{w_j\bar{w}_j})^{-2}  |g_{w_j\bar{w}_j w_k}|^2 =g_{w_k\bar{w}_k} \big ( g^{w_k \bar{w}_{k}}(g^{w_j\bar{w}_j})^{2}  |g_{w_j\bar{w}_j w_k}|^2 \big ) \leq & C_{\ref{secondorder1}}\mathscr{C}_1 \\
\end{split}
\end{equation}

This means that

\begin{equation}\label{ude}
|\frac{du}{dt} u^{-\alpha}| \leq C_{\ref{ude}}
\end{equation}

 for some constant $C_{\ref{ude}} $ which only depends on $G$ and $g_{reg}$.
 Now if we set $\alpha=-\frac{3}{2}$ we get

\[
u^{-\frac{3}{2}}\frac{du}{dt}\leq C_{\ref{ude}}
\]

  or

\[
-2 \frac{du^{-\frac{1}{2}}}{dt}\leq C_{\ref{ude}}
\]

Assuming  that $t<t_0 $ then  by integrating over the interval $[t, t_0]$  we find

\[
-2u^{-\frac{1}{2}} (t_0) +2u^{-\frac{1}{2}} (t) \leq  C_{\ref{ude}} (t-t_0)
\]

\[
u^{-\frac{1}{2}} (t)\leq  u^{-\frac{1}{2}} (t_0) +\frac{C_{\ref{ude}} }{2}(t-t_0)
\]

Finally we deduce that 
\begin{equation}\label{ulb}
u(t) \geq \sqrt{\frac{1}{u^{-\frac{1}{2}} (t_0) +\frac{C_{\ref{ude}} }{2}(t-t_0)}}
\end{equation}

Now we assume that   $t_0$ is the point where the maximum of $u$ occurs. Since $\int_D (\omega'|_{D} )^{n-1}=\int_D (\omega_{reg}|_D)^{n-1}$ we must have $u(t_0 )\geq 1$. 
Since $t$ is proportional  to the parameter of length with respect to $g_{reg}$ then from  the  inequality (\ref{ulb}) it follows  that 
\begin{equation}\label{ude2}
u(t) \geq \sqrt{\frac{1}{1+\frac{C_{\ref{ude2}} }{2}  diam D}}
\end{equation}

where $diam D$ denotes the diameter of $D$ with respect to $g_{reg}$ and $C_{\ref{ude2}}$ only depends on $G$ and $g_{reg}$. This completes the proof of the proposition \ref{prop8}.

\end{proof}



\section{Solution of the equation}\label{sec7}
Using the estimations in sections 5 and 6 as well as Schauder theory we are now able to prove our main theorem.

\begin{theo}
Assume that $X$ is a compact K\"ahler manifold  of complex dimension $n$ with the metric $\omega= \sum g_{i\bar{j}} dz^i \otimes d\bar{z}^{j}$. Let $D\subset X$ be a smooth divisor and $S$ be a holomorphic section of $L:=[D]$ vanishing along $D$. Let $G$ be $C^{k} (X)$ with $k\geq 3$ and $\int_X \exp \{G\} |S|^2  =Vol (X)$. 
Then there exists   a function $\phi$ in $C^{k+1, \alpha} (X ) $ for $0\leq \alpha < \frac{1}{2}$ such that 
$\omega':=\sum (g_{i\bar{j}} +\partial ^2 \phi / \partial z^i \partial \bar{z}^j  ) dz^i \otimes d\bar{z}^j $ defines a K\"ahler metric on $X\setminus D$ and $\omega'|_{D}$ is a nondegenerate  K\"ahler metric on $D$ and 
\begin{equation}\label{maineq}
{\omega '} ^n=|S|^2 \exp\{G\} \omega ^n 
\end{equation}
\end{theo}

\begin{proof} We consider the set
\[
\begin{split}
\mathcal{R}=\{t &\in [0,1] | \text{ the equation }  \det (g_{i\bar{j}} +\frac{\partial^2 \phi }{\partial z^i \partial \bar{z}^j}) \det (g_{i\bar{j}})^{-1}=\\
\quad & \frac{Vol (X)  |S|^2\exp \{tG\}}{ \int  |S|^2 \exp{tG}}\text{ has a solution in } C^{k+1,\alpha}  (X)  \}
\end{split}
\]

\[
\begin{split}
\Theta =\{\phi &\in C^{k+1,\alpha} (X)\cap C^1 (X) |  \omega+\partial\bar{\partial} \phi >0 \text{  outside } D  \text{ and }  \\
\quad &  \det (g_{i\bar{j}} +\frac{\partial^2 \phi }{\partial z^i \partial \bar{z}^j}) \det (g_{i\bar{j}})^{-1} |S|^{-2}\in C^{k-1,\alpha} (X) \}
\end{split}
\]
 
Let also
\[
B=\{f\in C^{k-1,\alpha} (X) |\int_X f |S|^2 e^G \omega^n =Vol(X) \}
\]

To see that $\Theta $ is a Banach space we consider a holomorphic coordinates system $(w_1,...,w_{n-1},z)$ over an open neighborhood $U_p$ of  a point $p\in D$ such that  
 and   $D=\{z=0\}$ and there exists a potential $\phi_0$ for $g_{i\bar{j}}|_{U_p}$. The fact that $\phi\in \Theta$ is equivalent to say that 
\begin{equation}\label{det1}
\det (\frac{\partial^2 (\phi+ \phi_0 )}{\partial z^{i} \partial \bar{z}^{j}}) =O(|S|^2)
\end{equation}

Since $\phi_0 + \phi$ is suposed to belong to $C^{k+1 , \alpha}  (X)$ and $k\geq3$ we have   $(\phi+ \phi_0) $ belongs at least to
$ C^4$.  We consider  the Taylor series expansion of $\det (\frac{\partial^2 (\phi+ \phi_0 )}{\partial z^{i} \partial \bar{z}^{j}}) $
upto secont order  in terms of $S$ (or $z$ in the holomorphic  coordinates system above) 
\[
\det (\frac{\partial^2 (\phi+ \phi_0 )}{\partial z^{i} \partial \bar{z}^{j}}) =A_0 + A_1 z+ \bar {A}_{1} \bar{z}+
A_2 z^2 + \bar{A}_2 \bar{z}^2 + A_{1,1}|z|^2 +R
\]
 where $\lim_{z\rightarrow 0} \frac{R}{|z|^2} =0$. Thus 
we see that the coefficients of 
$A_0 $, $A_1$ and $A_2$ are polynomials in terms of partial derivatives of  $\phi_0 +\phi$ upto  4th order  and if $\phi\in \Theta$
then due to (\ref{det1})  all the coefficients $A_0 ,A_1$ and $A_2$
vanish on $U_p$.

Thus for a sequence  $\phi_n\in  \Theta$ such that $\phi_n\rightarrow \phi$ in $C^{k+1}   (X)$   we consider the Taylor series expansions

\[
\det (\frac{\partial^2 (\phi_n+ \phi_0 )}{\partial z^{i} \partial \bar{z}^{j}}) =A_0 ^n+ A_1 ^n z+ \bar {A}_{1} ^n \bar{z}+
A_2 ^n z^2 + \bar{A}_2 ^n \bar{z}^2 + A_{1,1} ^n|z|^2 +R^n
\]

with $\lim_{z\rightarrow 0} \frac{R^n}{|z|^2} =0$ Since $A^n _{i} \rightarrow A_i$ in $C^0 (U_p)$ we have 
$A_0 = A_1 = A_2 =0$. In other words $\phi$ also  belongs  to $\Theta$.

Now  we consider the mapping $\mathscr{G}: \Theta\rightarrow B$:
\[
\mathscr{G}(\phi)=\det(g_{i\bar{j}}+\frac{\partial^2 \phi }{\partial z^i \partial \bar{z}^j}) \det (g_{i\bar{j}})^{-1}
\]

The tangent space of $B$ consists of   the space of the  functions $\{ f\in C^{k-1, \alpha}|\int_X |S|^2 f=0\}$.  In order to prove the  openness of $\mathcal{R}$ we have to show that the equation
\begin{equation}\label{dlap}
\det(g_{i\bar{j}}+\frac{\partial^2 \phi_1 }{\partial z^i \partial \bar{z}^j}) \det (g_{i\bar{j}})^{-1} \Delta_{\phi_1} \phi =|S|^2f
\end{equation}

has a solution in $C^{k+1 , \alpha}  (X) $. And we know that a solution to the equation 
$\Delta_{\phi_1} =u$ exists iff $\int udVol_{\phi_1}=0$.  
 In fact the metric  $g_{i\bar{j}}+\frac{\partial^2 \phi_1 }{\partial z^i \partial \bar{z}^j}$ is conic of the type studied by Cheeger and Dai in (\cite{ch}). We apply zero order Hodge theory developed in (\cite{ch}) to prove the existence of  a solution to $\Delta_{\phi_1}= u$.
 Then we utilize  Schauder theory.

To prove that $\mathcal{R}$ is closed we take a sequence $\{t_e \}$ in $\mathcal{R}$, such that there exists $\phi_e \in C^{k+1, \alpha} (X)$ such that
\[
\det (g_{i\bar{j} } +\frac{\partial^2 \phi_e}{\partial z^i \partial \bar{z}^{j}}) \det (g_{i\bar{j}})^{-1}= 
\frac{Vol (X)|S|^2\exp\{t_e G\}}{ \int_X |S|^2\exp \{t_e G\} }
\]

We can also assume that $\int_X \phi_e =0$.

Differentiating the above equation we get

\begin{equation}\label{pr}
\begin{split}
\det (g_{i\bar{j}}+\frac{\partial^2 \phi_e }{\partial z^i \partial \bar{z}^j})&\sum_{i,j} g_e ^{i\bar{j}}  \frac{\partial ^2}{\partial z^i \partial \bar{z}^j} (\frac{\partial \phi_e}{\partial z^s})\\
\quad & =\frac{1}{\int_X  \exp\{t_e G\}}Vol(X) \frac{\partial}{\partial z^s} [|S|^2\exp \{t_e G\} \det (g_{i\bar{j}})],
\end{split}
\end{equation}

The coefficients of the above matrix are the coefficients of the adjoint matrix of 
\[
[g_{e,i\bar{j}}]:=[g_{i\bar{j}}+\frac{\partial \phi_e}{\partial z^i\partial \bar{z}^j}]
\]

If  we use the notation $(z_1,...,z_{n-1},z_n)$ the left hand side of the  above equation  is expressed  like

\[
A_{z_n\bar{z}_n}\partial_{z_n} \partial_{\bar{z}_n}+\sum_{i\leq n-1} A_{z_{i}\bar{z}_j} \partial_{z_i}\partial_{\bar{z}_j} +\sum_{i\leq n-1} A_{z_n\bar{z}_i}\partial_{z_n}\partial_{\bar{z_i}} +\sum_{i\leq n-1} A_{z_i\bar{z}_n}\partial_{z_i}\partial_{\bar{z}_n} 
\]

We then  set 

\[
\rho_s = \frac{1}{\det (g_{i\bar{j}}+\frac{\partial^2 \phi_e }{\partial z^i \partial \bar{z}^j})}\times \frac{1}{\int_X  \exp\{t_e G\}}Vol(X) \frac{\partial}{\partial z^s} [|S|^2\exp \{t_e G\} \det (g_{i\bar{j}})]
\]

  Clearly $ \rho_s\in C^{\alpha}$ for $1\leq s\leq n-1$, and $z_n\rho_n\in C^{\alpha}$. 
Moreover    by theorems  (\ref{thirdorder}) and (\ref{3rdorder2}), lemma (\ref{m'm'1}) and  proposition (\ref{prop8})   it turns out that 
\[
 |\frac{\partial}{\partial z_i} \frac{g_{e,n\bar{n}}}{|z_n|} |,|\frac{\partial}{\partial z_n} \frac{g_{e,n\bar{n}}}{|z_n|} |, 
\]

and 

\[
|\frac{\partial}{\partial z_n} g_{e,n\bar{n}}|, |\frac{\partial}{\partial z_i} g_{e,n\bar{n}}|, 
 |\frac{\partial}{\partial z_i} \frac{g_{e,n\bar{k}}}{\sqrt{|z_n|}}| 
\]


  are   all uniformly bounded, for $1\leq i,k,l\leq n-1$. Thus 
\begin{equation}\label{azij}
A_{z_n\bar{z}_n},\frac{A_{z_{i}\bar{z}_j}}{|z_n|},\frac{A_{z_n\bar{z}_i}}{\sqrt{|z_n|}}\in C^{\alpha}
\end{equation}

with bounded  norms $[A_{z_n\bar{z}_n}]_{\alpha}$,
  $[\frac{A_{z_{i}\bar{z}_j}}{|z_n|}]_{\alpha}$ and   $[\frac{A_{z_n\bar{z}_i}}{\sqrt{|z_n|}}]_{\alpha}$
for $1\leq i,j\leq n-1$. This can be seen by taking the derivatives  of   $A_{z_n\bar{z}_n}$,  $\frac{A_{z_{i}\bar{z}_j}}{|z_n|}$ and   $\frac{A_{z_n\bar{z}_i}}{\sqrt{|z_n|}}$,    at a point $q$ near $D$ with $g_{e,z_n \bar{z}_i}(q)=0$ for $i=1,...,n-1$. 
 Thus we can apply Schauder theorem (\ref{schth}) for $\gamma_1= \gamma_2 =\frac{1}{2}$ to deduce that

\begin{equation}\label{fie}
\begin{split}
|&z_n|\partial_{z_i} \partial_{\bar{z}_j} \partial_{z_k} \phi_e]_{\alpha} , [\sqrt{|z_n|}\partial_{z_n} \partial_{\bar{z}_j} \partial_{z_k} \phi_e]_{\alpha}]_{\alpha}, [\partial_{z_n} \partial_{\bar{z_n}} \partial_{z_k} \phi_e]_{\alpha}
\\
\quad & 
[|z_n|\partial_{z_i} \partial_{\bar{z}_j} \partial_{z_n} \phi_e]_{\alpha} , [\sqrt{|z_n|}\partial_{z_n} \partial_{\bar{z}_j} \partial_{z_n} \phi_e]_{\alpha}]_{\alpha}, [\partial_{z_n} \partial_{\bar{z_n}} \partial_{z_n} \phi_e]_{\alpha}
\end{split}
\end{equation}

are all bounded for $i,j,k=1,...,n-1$.

  Now in the following equation
\[
\begin{split}
(A_{z_n\bar{z}_n}\partial_{z_n} \partial_{\bar{z}_n}+\sum_{i,j\leq n-1}  \partial_{z_i}\partial_{\bar{z}_j})  \frac{\partial \phi_e}{\partial z_s}=&\det (g_{i\bar{j}}+\frac{\partial^2 \phi_e }{\partial z^i \partial \bar{z}^j}) \rho_s +
 \sum_{i,j\leq n-1} (1-A_{z_{i}\bar{z}_j}) \partial_{z_i}\partial_{\bar{z}_j}  \frac{\partial \phi_e}{\partial z^s}\\
\quad & -\sum_{i\leq n-1} A_{z_n\bar{z}_i}\partial_{z_n}\partial_{\bar{z}_i} \frac{\partial \phi_e}{\partial z^s} -\sum_{i\leq n-1} A_{z_i\bar{z}_n}\partial_{z_i }\partial_{\bar{z}_n}  \frac{\partial \phi_e}{\partial z^s}
\end{split}
\]

the left hand side is a non-degenerate operator and the right hand side has bounded $C^{\alpha}$ norm by the above relations (\ref{fie}) and (\ref{azij}).
Now by ordinary Schauder  estimate we get upper bound for  ordinary third order derivatives of $\phi_e$. This argument can obviously be repeated inductively.

Finally we note that  if $\phi:=\lim_{i\rightarrow}  \phi_{e_i}$    then   according to proposition (\ref{prop8}) the restriction $(g_{reg}+\partial\bar{\partial}\phi)|_{D}$  has nonzero lower bound hence is nondegenerate.  Also according to the relation (\ref{sincase2})
in the appendix (\ref{app7}) the angle $\theta_{n}$  with respect to $g_{reg} $ between the orthogonal direction to $D$ and $D$  with respect to $\omega_{reg}+\partial\partial \phi_{e_i}$ has a lower bound which only depends on $G$ and $g_{reg}$. So the same lower bound  works for  $(g_{reg}+\partial\bar{\partial}\phi)|_{D}$.
  \end{proof}



 \appendix
\section{Appendix} 


\subsection{}\label{app2}
 
Assume that $1\leq k\leq n-1$ and $1\leq i,j\leq n$. According to (\ref{cdt4}), (\ref{cdt5}) and (\ref{komaki}) we have

 

\begin{equation}\label{valfsi1}
\begin{split}
  \bar{V}_{\alpha}.\Psi_1 =  & \sum \bigg [-g'^{i\bar{i}}g'^{q\bar{q}}g'^{\bar{j}j}g'^{k\bar{k}}\big [ \phi_{i\bar{q}\bar{\alpha}}-  \Gamma_{\bar{\alpha}\bar{i}} ^{\bar{q}} g'_{i\bar{i}}\big ]\phi_{i\bar{j}k}\phi_{\bar{q}j\bar{k}}\\
\quad & -g'^{i\bar{i}}g'^{\bar{j}j}g'^{\bar{q}q}g'^{k\bar{k}}\big [ \phi_{j\bar{q}\bar{\alpha}}-  \Gamma_{\bar{\alpha}\bar{j}} ^{\bar{q}} g'_{j\bar{j}}\big ]\phi_{i\bar{j}k}\phi_{\bar{i}q\bar{k}}\\
\quad & -g'^{i\bar{i}}g'^{\bar{j}j}g'^{k\bar{k}}g'^{q\bar{q}}\big [ \phi_{k\bar{q}\bar{\alpha}}- \Gamma_{\bar{\alpha}\bar{k}} ^{\bar{q}} g'_{k\bar{k}}\big ]\phi_{i\bar{j}k}\phi_{\bar{i}j\bar{q}}\\
\quad & +g'^{i\bar{i}}g'^{\bar{j}j}g'^{k\bar{k}}
\big [ \phi_{i\bar{j}k\bar{\alpha}} -\sum \Gamma ^{\bar{j}} _{\bar{\alpha}\bar{c}} \phi_{i\bar{c}k}\big ]\phi_{\bar{i}j\bar{k}} \\
\quad & +g'^{i\bar{i}}g'^{\bar{j}j}g'^{k\bar{k}}
\big [ \phi_{\bar{i}j\bar{k}\bar{\alpha}}-\sum \Gamma^{\bar{i}} _{\bar{\alpha}\bar{c}} \phi_{\bar{c}j\bar{k}}  
-\sum \Gamma^{\bar{k}} _{\bar{\alpha}\bar{c}} \phi_{\bar{i}j\bar{c}} \big ] \phi_{i\bar{j}k}
\bigg ]
\end{split}
\end{equation}

\begin{equation}\label{valfsi1}
\begin{split}
  \bar{V}_{\beta}.\Psi_1 =  & \sum \bigg [- g'^{i\bar{p}}g'^{q\bar{r}}g'^{\bar{j}s}g'^{k\bar{t}}\big [ \phi_{p\bar{q}\bar{\beta}}-\sum_b  \Gamma_{\bar{\beta}\bar{b}} ^{\bar{q}} g'_{p\bar{b}}\big ]\phi_{i\bar{j}k}\phi_{\bar{r}s\bar{t}} \\
\quad & -g'^{i\bar{r}}  g'^{\bar{j}p}g'^{\bar{q} s}g'^{k\bar{t}}\big [ \phi_{q\bar{p}\bar{\beta}}-\sum_b  \Gamma_{\bar{\beta}\bar{b}} ^{\bar{p}} g'_{q\bar{b}}\big ]\phi_{i\bar{j}k}\phi_{\bar{r}s\bar{t}}
\\
\quad  &-g'^{i\bar{r}}g'^{\bar{j}s} g'^{k\bar{p}}g'^{q\bar{t}} \big [ \phi_{p\bar{q}\bar{\beta}}-\sum_b  \Gamma_{\bar{\beta}\bar{b}} ^{\bar{q}} g'_{p\bar{b}}]\phi_{i\bar{j}k}\phi_{\bar{r}s\bar{t}}\\
\quad & +g'^{i\bar{r}}g'^{\bar{j}s}g'^{k\bar{t}}
\big [ \phi_{i\bar{j}k\bar{\beta}} -\sum \Gamma ^{\bar{j}} _{\bar{\beta}\bar{c}} \phi_{i\bar{c}k}\big ]\phi_{\bar{r}s\bar{t}} \\
\quad & +g'^{i\bar{r}}g'^{\bar{j}s}g'^{k\bar{t}}
\big [ \phi_{\bar{r}s\bar{t}\bar{\beta}}-\sum \Gamma^{\bar{r}} _{\bar{\beta}\bar{c}} \phi_{\bar{c}s\bar{t}}  
-\sum \Gamma^{\bar{t}} _{\bar{\beta}\bar{c}} \phi_{\bar{r}s\bar{c}} \big ] \phi_{i\bar{j}k}
\bigg ]
\end{split}
\end{equation}

 To compute $\Delta_{g'}\Psi_1=\sum g'^{\alpha\bar{\beta}}V_{\alpha}.V_{\bar{\beta}}.\Psi_1$ we do computation for  the first line and the last line  in (\ref{valfsi1}),

\begin{equation}\label{ap2a0}
\begin{split}
 \sum & g'^{\alpha\bar{\beta}} V_\alpha.\bigg [-g'^{i\bar{p}}g'^{q\bar{r}}g'^{\bar{j}s}g'^{k\bar{t}}\big [ \phi_{p\bar{q}\bar{\beta}}-\sum_b  \Gamma_{\bar{\beta}\bar{b}} ^{\bar{q}} g'_{p\bar{b}}\big ]\phi_{i\bar{j}k}\phi_{\bar{r}s\bar{t}}\bigg ]=\\
\quad= & - \sum g'^{\alpha\bar{\beta}}g'^{i\bar{a}}g'^{b\bar{p}}g'^{q\bar{r}}g'^{\bar{j}s}g'^{k\bar{t}}\big [ \phi_{a\bar{b}\alpha}-\sum_c  \Gamma_{\alpha c} ^{a} g'_{c\bar{b}}\big ]\big [ \phi_{p\bar{q}\bar{\beta}}-\sum_b  \Gamma_{\bar{\beta}\bar{b}} ^{\bar{q}} g'_{p\bar{b}}\big ]\phi_{i\bar{j}k}\phi_{\bar{r}s\bar{t}}\\
\quad & - \sum g'^{\alpha\bar{\beta}}g'^{i\bar{p}}g'^{q\bar{a}}g'^{b\bar{r}}g'^{\bar{j}s}g'^{k\bar{t}}\big [ \phi_{a\bar{b}\alpha}-\sum_c  \Gamma_{\alpha c} ^{a} g'_{c\bar{b}}\big ]\big [ \phi_{p\bar{q}\bar{\beta}}-\sum_b  \Gamma_{\bar{\beta}\bar{b}} ^{\bar{q}} g'_{p\bar{b}}\big ]\phi_{i\bar{j}k}\phi_{\bar{r}s\bar{t}}\\
\quad & - \sum g'^{\alpha\bar{\beta}}g'^{i\bar{p}}g'^{q\bar{r}}g'^{\bar{j}b}g'^{\bar{a}s}g'^{k\bar{t}}\big [ \phi_{a\bar{b}\alpha}-\sum_c  \Gamma_{\alpha c} ^{a} g'_{c\bar{b}}\big ]\big [ \phi_{p\bar{q}\bar{\beta}}-\sum_b  \Gamma_{\bar{\beta}\bar{b}} ^{\bar{q}} g'_{p\bar{b}}\big ]\phi_{i\bar{j}k}\phi_{\bar{r}s\bar{t}}\\
\quad & - \sum g'^{\alpha\bar{\beta}}g'^{i\bar{p}}g'^{q\bar{r}}g'^{\bar{j}s}g'^{k\bar{a}}g'^{b\bar{t}}\big [ \phi_{a\bar{b}\alpha}-\sum_c  \Gamma_{\alpha c} ^{a} g'_{c\bar{b}}\big ]\big [ \phi_{p\bar{q}\bar{\beta}}-\sum_b  \Gamma_{\bar{\beta}\bar{b}} ^{\bar{q}} g'_{p\bar{b}}\big ]\phi_{i\bar{j}k}\phi_{\bar{r}s\bar{t}}\\
\quad &- \sum g'^{\alpha\bar{\beta}}g'^{i\bar{p}}g'^{q\bar{r}}g'^{\bar{j}s}g'^{k\bar{t}}\big [ \phi_{p\bar{q}\bar{\beta}\alpha} -\sum \Gamma^p _{\alpha u}\phi_{u\bar{q} \bar{\beta}}-\sum_b  V_{\alpha}.(\Gamma_{\bar{\beta}\bar{b}} ^{\bar{q}} g'_{p\bar{b}}) \big ]\phi_{i\bar{j}k}\phi_{\bar{r}s\bar{t}}\\
\quad & -\sum g'^{\alpha\bar{\beta}} g'^{i\bar{p}}g'^{q\bar{r}}g'^{\bar{j}s}g'^{k\bar{t}}\big [ \phi_{p\bar{q}\bar{\beta}}-\sum_b  \Gamma_{\bar{\beta}\bar{b}} ^{\bar{q}} g'_{p\bar{b}}\big ]\big [ \phi_{i\bar{j}k\alpha} -\sum \Gamma^{i} _{\alpha u} \phi_{u\bar{j}k}-\sum \Gamma^k _{\alpha v} \phi_{i\bar{j}v} \big ] \phi_{\bar{r}s\bar{t}}\\
\quad & -\sum  g'^{\alpha\bar{\beta}} g'^{i\bar{p}}g'^{q\bar{r}}g'^{\bar{j}s}g'^{k\bar{t}}\big [ \phi_{p\bar{q}\bar{\beta}}-\sum_b  \Gamma_{\bar{\beta}\bar{b}} ^{\bar{q}} g'_{p\bar{b}}\big ]\phi_{i\bar{j}k}\big [ \phi_{\bar{r}s\bar{t}\alpha} -\sum 
\Gamma^s _{\alpha w} \phi_{\bar{r}w\bar{t}}\big ]
\end{split}
\end{equation}

\begin{equation}\label{ap2b0}
\begin{split}
\sum & g'^{\alpha\bar{\beta}}V_{\alpha}.\bigg [ g'^{i\bar{r}}g'^{\bar{j}s}g'^{k\bar{t}}
\big [ \phi_{i\bar{j}k\bar{\beta}} -\sum \Gamma ^{\bar{j}} _{\bar{\beta}\bar{c}} \phi_{i\bar{c}k}\big ]\phi_{\bar{r}s\bar{t}} \bigg ]=\\
\quad & =\sum  g'^{\alpha\bar{\beta}} g'^{i\bar{a}}g'^{b\bar{p}}g'^{\bar{j}s}g'^{k\bar{t}}\big [ \phi_{a\bar{b}\alpha}-\sum_c  \Gamma_{\alpha c} ^{a} g'_{c\bar{b}}\big ]
\big [ \phi_{i\bar{j}k\bar{\beta}} -\sum \Gamma ^{\bar{j}} _{\bar{\beta}\bar{c}} \phi_{i\bar{c}k}\big ]\phi_{\bar{r}s\bar{t}}\\
\quad & +\sum  g'^{\alpha\bar{\beta}}g'^{i\bar{r}}g'^{\bar{j}b}g'^{\bar{a}s}g'^{k\bar{t}}\big [ \phi_{a\bar{b}\alpha}-\sum_c  \Gamma_{\alpha c} ^{a} g'_{c\bar{b}}\big ]
\big [ \phi_{i\bar{j}k\bar{\beta}} -\sum \Gamma ^{\bar{j}} _{\bar{\beta}\bar{c}} \phi_{i\bar{c}k}\big ]\phi_{\bar{r}s\bar{t}}\\
\quad &   +\sum  g'^{\alpha\bar{\beta}}g'^{i\bar{r}}g'^{\bar{j}s}g'^{k\bar{a}}g'^{b\bar{t}}\big [ \phi_{a\bar{b}\alpha}-\sum_c  \Gamma_{\alpha c} ^{a} g'_{c\bar{b}}\big ]
\big [ \phi_{i\bar{j}k\bar{\beta}} -\sum \Gamma ^{\bar{j}} _{\bar{\beta}\bar{c}} \phi_{i\bar{c}k}\big ]\phi_{\bar{r}s\bar{t}}\\
\quad & + \sum  g'^{\alpha\bar{\beta}}g'^{i\bar{r}}g'^{\bar{j}s}g'^{k\bar{t}}
\big [ \phi_{i\bar{j}k\bar{\beta}\alpha}-\sum \Gamma^{i} _{\alpha u} \phi _{u\bar{j}k\bar{\beta}}-\sum \Gamma ^k _{\alpha v}\phi_{i\bar{j}v\bar{\beta}} -\sum (V_{\alpha}.\Gamma ^{\bar{j}} _{\bar{\beta}\bar{c}}) \phi_{i\bar{c}k}-\\
\quad & \hspace{4cm }  -\sum \Gamma ^{\bar{j}} _{\bar{\beta}\bar{c}} [\phi_{i\bar{c}k\alpha} -\sum \Gamma^i _{\alpha u} \phi_{u\bar{c}k}-\sum \Gamma ^k _{\alpha v} \phi_{i\bar{c} v}] \big ]\times \phi_{\bar{r}s\bar{t}} \\
\quad &+\sum  g'^{\alpha\bar{\beta}}g'^{i\bar{r}}g'^{\bar{j}s}g'^{k\bar{t}}
\big [ \phi_{i\bar{j}k\bar{\beta}} -\sum \Gamma ^{\bar{j}} _{\bar{\beta}\bar{c}} \phi_{i\bar{c}k}\big ][\phi_{\bar{r}s\bar{t}\alpha}- \sum \Gamma^s _{\alpha w}\phi_{\bar{r} w\bar{t}}]
\end{split}
\end{equation}

\begin{equation}\label{ap2a}
\begin{split}
 \sum & g'^{\alpha\bar{\alpha}} V_\alpha.\bigg [-g'^{i\bar{i}}g'^{q\bar{q}}g'^{\bar{j}j}g'^{k\bar{k}}\big [ \phi_{i\bar{q}\bar{\alpha}}-  \Gamma_{\bar{\alpha}\bar{i}} ^{\bar{q}} g'_{i\bar{i}}\big ]\phi_{i\bar{j}k}\phi_{\bar{q}j\bar{k}}\bigg ]=\\
\quad= & - \sum g'^{\alpha\bar{\alpha}}g'^{i\bar{i}}g'^{p\bar{p}}g'^{q\bar{q}}g'^{\bar{j}j}g'^{k\bar{k}}\big [ \phi_{i\bar{p}\alpha}- \Gamma_{\alpha p} ^{i} g'_{p\bar{p}}\big ]\big [ \phi_{p\bar{q}\bar{\alpha}}-  \Gamma_{\bar{\alpha}\bar{p}} ^{\bar{q}} g'_{p\bar{p}}\big ]\phi_{i\bar{j}k}\phi_{\bar{q}j\bar{k}}\\
\quad & - \sum g'^{\alpha\bar{\alpha}}g'^{i\bar{i}}g'^{q\bar{q}}g'^{b\bar{b}}g'^{\bar{j}j}g'^{k\bar{k}}\big [ \phi_{q\bar{b}\alpha}-  \Gamma_{\alpha b} ^{q} g'_{b\bar{b}}\big ]\big [ \phi_{i\bar{q}\bar{\alpha}}- \Gamma_{\bar{\alpha}\bar{i}} ^{\bar{q}} g'_{i\bar{i}}\big ]\phi_{i\bar{j}k}\phi_{\bar{b}j\bar{k}}\\
\quad & - \sum g'^{\alpha\bar{\alpha}}g'^{i\bar{i}}g'^{q\bar{q}}g'^{\bar{j}j}g'^{\bar{s}s}g'^{k\bar{k}}\big [ \phi_{s\bar{j}\alpha}- \Gamma_{\alpha j} ^{s} g'_{j\bar{j}}\big ]\big [ \phi_{i\bar{q}\bar{\alpha}}-  \Gamma_{\bar{\alpha}\bar{i}} ^{\bar{q}} g'_{i\bar{i}}\big ]\phi_{i\bar{j}k}\phi_{\bar{q}s\bar{k}}\\
\quad & - \sum g'^{\alpha\bar{\alpha}}g'^{i\bar{i}}g'^{q\bar{q}}g'^{\bar{j}j}g'^{k\bar{k}}g'^{t\bar{t}}\big [ \phi_{k\bar{t}\alpha}-  \Gamma_{\alpha t} ^{k} g'_{c\bar{b}}\big ]\big [ \phi_{i\bar{q}\bar{\alpha}}-  \Gamma_{\bar{\alpha}\bar{i}} ^{\bar{q}} g'_{i\bar{i}}\big ]\phi_{i\bar{j}k}\phi_{\bar{q}j\bar{t}}\\
\quad &- \sum g'^{\alpha\bar{\alpha}}g'^{i\bar{i}}g'^{q\bar{q}}g'^{\bar{j}j}g'^{k\bar{k}}\big [ \phi_{i\bar{q}\bar{\alpha}\alpha} -\sum \Gamma^i _{\alpha u}\phi_{u\bar{q} \bar{\alpha}}-\sum_b  V_{\alpha}.(\Gamma_{\bar{\alpha}\bar{i}} ^{\bar{q}} g'_{i\bar{i}}) \big ]\phi_{i\bar{j}k}\phi_{\bar{q}j\bar{k}}\\
\quad & -\sum g'^{\alpha\bar{\alpha}} g'^{i\bar{i}}g'^{q\bar{q}}g'^{\bar{j}j}g'^{k\bar{k}}\big [ \phi_{i\bar{q}\bar{\alpha}}-  \Gamma_{\bar{\alpha}\bar{i}} ^{\bar{q}} g'_{i\bar{i}}\big ]\big [ \phi_{i\bar{j}k\alpha} -\sum \Gamma^{i} _{\alpha u} \phi_{u\bar{j}k}-\sum \Gamma^k _{\alpha v} \phi_{i\bar{j}v} \big ] \phi_{\bar{q}j\bar{k}}\\
\quad & -\sum  g'^{\alpha\bar{\alpha}} g'^{i\bar{i}}g'^{q\bar{q}}g'^{\bar{j}j}g'^{k\bar{k}}\big [ \phi_{i\bar{q}\bar{\alpha}}-  \Gamma_{\bar{\alpha}\bar{i}} ^{\bar{q}} g'_{i\bar{i}}\big ]\phi_{i\bar{j}k}\big [ \phi_{\bar{q}j\bar{k}\alpha} -\sum \Gamma^j _{\alpha w} \phi_{\bar{q}w\bar{k}}\big ]
\end{split}
\end{equation}

\begin{equation}\label{ap2b}
\begin{split}
\sum & g'^{\alpha\bar{\beta}}V_{\alpha}.\bigg [ g'^{i\bar{r}}g'^{\bar{j}s}g'^{k\bar{t}}
\big [ \phi_{i\bar{j}k\bar{\beta}} -\sum \Gamma ^{\bar{j}} _{\bar{\beta}\bar{c}} \phi_{i\bar{c}k}\big ]\phi_{\bar{r}s\bar{t}} \bigg ]=\\
\quad & =\sum  g'^{\alpha\bar{\alpha}} g'^{i\bar{i}}g'^{r\bar{r}}g'^{\bar{j}j}g'^{k\bar{k}}\big [ \phi_{i\bar{r}\alpha}-  \Gamma_{\alpha r} ^{i} g'_{r\bar{r}}\big ]
\big [ \phi_{i\bar{j}k\bar{\alpha}} -\sum \Gamma ^{\bar{j}} _{\bar{\alpha}\bar{c}} \phi_{i\bar{c}k}\big ]\phi_{\bar{r}j\bar{k}}\\
\quad & +\sum  g'^{\alpha\bar{\alpha}}g'^{i\bar{i}}g'^{\bar{j}j}g'^{\bar{s}s}g'^{k\bar{k}}\big [ \phi_{s\bar{j}\alpha}-  \Gamma_{\alpha j} ^{s} g'_{j\bar{j}}\big ]
\big [ \phi_{i\bar{j}k\bar{\alpha}} -\sum \Gamma ^{\bar{j}} _{\bar{\alpha}\bar{c}} \phi_{i\bar{c}k}\big ]\phi_{\bar{i}s\bar{k}}\\
\quad &   +\sum  g'^{\alpha\bar{\alpha}}g'^{i\bar{i}}g'^{\bar{j}j}g'^{k\bar{k}}g'^{t\bar{t}}\big [ \phi_{k\bar{t}\alpha}-  \Gamma_{\alpha t} ^{k} g'_{t\bar{t}}\big ]
\big [ \phi_{i\bar{j}k\bar{\alpha}} -\sum \Gamma ^{\bar{j}} _{\bar{\alpha}\bar{c}} \phi_{i\bar{c}k}\big ]\phi_{\bar{i}j\bar{t}}\\
\quad & + \sum  g'^{\alpha\bar{\alpha}}g'^{i\bar{i}}g'^{\bar{j}j}g'^{k\bar{k}}
\big [ \phi_{i\bar{j}k\bar{\alpha}\alpha}-\sum \Gamma^{i} _{\alpha u} \phi _{u\bar{j}k\bar{\alpha}}-\sum \Gamma ^k _{\alpha v}\phi_{i\bar{j}v\bar{\alpha}} -\sum (V_{\alpha}.\Gamma ^{\bar{j}} _{\bar{\alpha}\bar{c}}) \phi_{i\bar{c}k}-\\
\quad & \hspace{4cm }  -\sum \Gamma ^{\bar{j}} _{\bar{\alpha}\bar{c}} [\phi_{i\bar{c}k\alpha} -\sum \Gamma^i _{\alpha u} \phi_{u\bar{c}k}-\sum \Gamma ^k _{\alpha v} \phi_{i\bar{c} v}] \big ]\times \phi_{\bar{i}j\bar{k}} \\
\quad &+\sum  g'^{\alpha\bar{\alpha}}g'^{i\bar{i}}g'^{\bar{j}j}g'^{k\bar{k}}
\big [ \phi_{i\bar{j}k\bar{\alpha}} -\sum \Gamma ^{\bar{j}} _{\bar{\alpha}\bar{c}} \phi_{i\bar{c}k}\big ][\phi_{\bar{i}j\bar{k}\alpha}- \sum \Gamma^j _{\alpha w}\phi_{\bar{i} w\bar{k}}]
\end{split}
\end{equation}

 

\subsection{}\label{app3}
 We first apply  relations (\ref{cdt4}) and (\ref{cdt3}) to obtain

\begin{equation}\label{bib1}
\begin{split}
\bar{V}_{j}.V_{i}.\bar{V}_{\beta}.V_{\alpha}. \phi&= \phi_{\alpha\bar{\beta}i\bar{j}}- \sum \Gamma^{\bar{\beta}} _{\bar{j}\bar{a}}\phi_{\alpha\bar{a}i}-\sum (\bar{V}_j.\Gamma^{\alpha} _{ib} )\phi_{b\bar{\beta}}\\
\quad & -\sum\Gamma^{\alpha} _{ia} (\phi_{a\bar{\beta}\bar{j}}
 -\sum \Gamma^{\bar{\beta}}_{\bar{j}\bar{c}} \phi_{a\bar{c}})
\end{split}
\end{equation}

 Taking one more derivative   and multiplying by $g'^{\alpha\bar{\alpha}} g'^{i\bar{i}}g'^{j\bar{j}}g'^{k\bar{k}}\phi_{\bar{i}j\bar{k}}$  leads to

\[
\begin{split}
g'^{\alpha\bar{\alpha}} g'^{i\bar{i}}g'^{j\bar{j}}g'^{k\bar{k}}\phi_{\alpha\bar{\alpha}i\bar{j}k} \phi_{\bar{i}j\bar{k}}&=
g'^{\alpha\bar{\alpha}} \{g'^{i\bar{i}}g'^{j\bar{j}}g'^{k\bar{k}}\}^{1/2}\bigg [V_k. \bar{V}_{j}.V_{i}.\bar{V}_{\alpha}.V_{\alpha}. \phi\\
\quad &  +\sum (V_k.\Gamma^{\bar{\alpha}} _{\bar{j}\bar{a}}\phi_{\alpha\bar{a}i}+\Gamma^{\bar{\alpha}} _{\bar{j}\bar{a}}V_k.\phi_{\alpha\bar{a}i})+\sum [(V_k.\bar{V}_j.\Gamma^\alpha _{ib} )\phi_{b\bar{\alpha}}+(\bar{V}_j.\Gamma^\alpha _{ib} )V_k.\phi_{b\bar{\alpha}}]\\
\quad & +\sum V_k. \Gamma^\alpha _{ia} (\phi_{a\bar{\alpha}\bar{j}}
 -\sum \Gamma^{\bar{\alpha}}_{\bar{j}\bar{c}} \phi_{a\bar{c}})
+\sum\Gamma^\alpha _{ia} (V_k.\phi_{a\bar{\alpha}\bar{j}})\\
\quad & -\sum\Gamma^\alpha _{ia} (\sum (V_k.\Gamma^{\bar{\alpha}}_{\bar{j}\bar{c}}) \phi_{a\bar{c}})-\sum\Gamma^\alpha _{ia} \sum \Gamma^{\bar{\alpha}}_{\bar{j}\bar{c}}( V_k.\phi_{a\bar{c}})\\
\quad & +\sum \Gamma^{\alpha} _{ka}  \phi_{a\bar{\alpha}i\bar{j}}+\sum \Gamma^{i} _{kb}  \phi_{\alpha\bar{\alpha}b\bar{j}}
\bigg ]\times \{ g'^{i\bar{i}}g'^{j\bar{j}}g'^{k\bar{k}}\}^{1/2}  \phi_{\bar{i}j\bar{k}}
\end{split}
\]

In order to compute the term $V_k. \bar{V}_{j}.V_{i}.\bar{V}_{\alpha}.V_{\alpha}. \phi$ we apply the relation (2.5) in (\cite{y})
\begin{equation}\label{bib0}
\begin{split}
\sum g'^{\alpha\bar{\alpha}}\bar{V}_{j}.V_{i}.\bar{V}_{\alpha}.V_{\alpha}. \phi &= V_i .\bar{V}_{j}. F +\sum g'^{\alpha\bar{\alpha}}g'^{\alpha\bar{\alpha}}(\bar{V_j}. \bar{V}_{\alpha}.V_{\alpha}.\phi)(V_i.\bar{V}_{\alpha}. V_{\alpha}.\phi)\\
\quad & +\sum g'^{\alpha\bar{\alpha}}(\bar{V}_j. g_{reg, \alpha\bar{\alpha}}) V_i.\bar{V}_{\alpha}.V_{\alpha}.\phi+\sum 
g'^{\alpha\bar{\alpha}}(\bar{V}_i. g_{reg, \alpha\bar{\alpha}}) \bar{V}_j.\bar{V}_\alpha.V_\alpha.\phi
\end{split}
\end{equation}

As well as the following relation which was proved in (\ref{mg1}) 
\[
V_k. g'^{\alpha\bar{\alpha}}=-(g'^{\alpha\bar{\alpha}})^2 \phi_{\alpha\bar{\alpha}k}-(g'^{\alpha\bar{\alpha}})^2\sum \Gamma^\alpha _{ka}\phi_{a\bar{\alpha}}
\]

\begin{equation}\label{bib0}
\begin{split}
\sum g'^{\alpha\bar{\alpha}}V_k.\bar{V}_{j}.V_{i}.\bar{V}_{\alpha}.V_{\alpha}. \phi &=\sum ((g'^{\alpha\bar{\alpha}})^2 \phi_{\alpha\bar{\alpha}k}+(g'^{\alpha\bar{\alpha}})^2\sum \Gamma^\alpha _{ka}\phi_{a\bar{\alpha}}) \bar{V}_{j}.V_{i}.\bar{V}_{\alpha}.V_{\alpha}. \phi \\
\quad & 
+V_k. V_i .\bar{V}_{j}. F +\sum g'^{\alpha\bar{\alpha}}g'^{\alpha\bar{\alpha}}(\bar{V_j}. \bar{V}_{\alpha}.V_{\alpha}.\phi)(V_i.\bar{V}_{\alpha}. V_{\alpha}.\phi)\\
\quad & +\sum g'^{\alpha\bar{\alpha}}g'^{\alpha\bar{\alpha}}(V_k.\bar{V_j}. \bar{V}_{\alpha}.V_{\alpha}.\phi)(V_i.\bar{V}_{\alpha}. V_{\alpha}.\phi)+\sum g'^{\alpha\bar{\alpha}}g'^{\alpha\bar{\alpha}}(V_k.\bar{V_j}. \bar{V}_{\alpha}.V_{\alpha}.\phi)(V_i.\bar{V}_{\alpha}. V_{\alpha}.\phi)\\
\quad & -2\sum  [(g'^{\alpha\bar{\alpha}})^2 \phi_{\alpha\bar{\alpha}k}+(g'^{\alpha\bar{\alpha}})^2\sum \Gamma^\alpha _{ka}\phi_{a\bar{\alpha}}]g'^{\alpha\bar{\alpha}}(\bar{V_j}. \bar{V}_{\alpha}.V_{\alpha}.\phi)(V_i.\bar{V}_{\alpha}. V_{\alpha}.\phi)\\
\quad & +\sum g'^{\alpha\bar{\alpha}}(V_k.\bar{V}_j. g_{reg, {\alpha}\bar{\alpha}}) V_i.\bar{V}_{\alpha}.V_{\alpha}.\phi+\sum g'^{\alpha\bar{\alpha}}(V_k.\bar{V}_i. g_{reg, \alpha\bar{\alpha}}) \bar{V}_j.\bar{V}_{\alpha}.V_{\alpha}.\phi\\
\quad & +\sum g'^{\alpha\bar{\alpha}}( \bar{V}_j. g_{reg, \alpha\bar{\alpha}}) V_k.V_i.\bar{V}_\alpha.V_\alpha.\phi+\sum g'^{\alpha\bar{\alpha}}(\bar{V}_i. g_{reg, i\bar{i}}) V_k.\bar{V}_j.\bar{V}_\alpha .V_\alpha.\phi\\
\quad & -\sum [(g'^{\alpha\bar{\alpha}})^2 \phi_{\alpha\bar{\alpha}k}+(g'^{\alpha\bar{\alpha}})^2\sum \Gamma^\alpha _{ka}\phi_{a\bar{\alpha}}]g'^{\alpha\bar{\alpha}}(\bar{V}_j. g_{reg, \alpha\bar{\alpha}}) V_i.\bar{V}_\alpha.V_\alpha.\phi \\
\quad & -\sum [(g'^{\alpha\bar{\alpha}})^2 \phi_{\alpha\bar{\alpha}k}+(g'^{\alpha\bar{\alpha}})^2\sum \Gamma^\alpha _{ka}\phi_{a\bar{\alpha}}]g'^{\alpha\bar{\alpha}}(\bar{V}_i. g_{reg, \alpha\bar{\alpha}}) \bar{V}_j.\bar{V}_\alpha.V_\alpha.\phi \\
\end{split}
\end{equation}
\newpage

Assuming that at the point $q$ we have $\Gamma_{ab} ^c (q) =0$ then by 

\begin{lemma}\label{lem30}
\begin{equation}
\begin{split}
 (g'^{l\bar{l}} )^{1/2} (g'^{k\bar{k}})^{1/2}g'^{n\bar{n}}\phi_{n\bar{n} k\bar{l}}&= -\sum_{i\neq n} (g'^{l\bar{l}} )^{1/2} (g'^{k\bar{k}})^{1/2} g'^{i\bar{i}} \phi_{i\bar{i} k\bar{l}}+ (g'^{l\bar{l}} )^{1/2} (g'^{k\bar{k}})^{1/2} V_k .\bar{V}_{l}. F \\
\quad &+\sum_{i,j} (g'^{l\bar{l}} )^{1/2} (g'^{k\bar{k}})^{1/2} g'^{j\bar{j}}g'^{i\bar{i}}( \phi_{j\bar{i}\bar{l}})( \phi_{i\bar{j}k })\\  
\end{split}
\end{equation}
where $R_2$ satisfies
\begin{equation}\label{ar2}
|R_2|\leq \frac{C}{|S|^4}
\end{equation}

In particular we have (2.5) in (\cite{y}) we can get

\begin{equation}\label{4th}
\begin{split}
(g'^{n\bar{n}})^2\phi_{n\bar{n} n\bar{n}} &= -\sum_{i\neq n} g'^{i\bar{i}} g'^{n\bar{n}}\phi_{i\bar{i} n\bar{n}}+ g'^{n\bar{n}}V_n .\bar{V}_{n}. F +g'^{n\bar{n}}\sum_{i,j} g'^{j\bar{j}}g'^{i\bar{i}}( \phi_{j\bar{i}\bar{n}})( \phi_{i\bar{j}n }) \\
\end{split}
\end{equation}

Similarly we can write
\begin{equation}\label{4th2}
\begin{split}
(g'^{n\bar{n}})^2\phi_{n\bar{n} nn} &= -\sum_{i\neq n} g'^{i\bar{i}}g'^{n\bar{n}} \phi_{i\bar{i} nn}+ g'^{n\bar{n}}V_n .V_n. F +g'^{n\bar{n}}\sum_{i,j} g'^{j\bar{j}}g'^{i\bar{i}}( \phi_{j\bar{i}n})( \phi_{i\bar{j}n }) \\
\end{split}
\end{equation}

\end{lemma}
 
\subsection{}\label{app4}

\begin{lemma}\label{lem29}
Let $U$ be a neighborhood of a point $q\in X$ and assume that $(z_1,..,z_n)$ is a coordinate system on $U$
such that  $q=(0,...,0)$ in this coordinates system. Assume that $D_i$ for $i=1,...,2n$ denotes each of the $2n$
real coordinate vectors  $\frac{\partial}{\partial x_i}$ or $\frac{\partial}{\partial y_i}$ where $z_i=x_i+i y_i$.  Then there exists a K\"ahler correction $\omega =\omega_{reg}+\partial\bar{\partial}\phi_0$ to the initial metric $\omega_{reg}$ such that
 \[
(D_i g)  (q)= (D_i D_j g) (q)=0, \hspace{1cm} \text{ for } i,j=1,...,2n 
\]
where $g=[g_{i\bar{j}}]$ represents the matrix associated to the K\"ahler metric $\omega=\frac{i}{2}\sum_{i,j} g_{i\bar{j}}dz^i\wedge d\bar{z} ^j$.
In particular we can  assume that the coordinates system $(z_1,...,z_n)$ is such that the hyper surfaces $\{z_n=constant\}$ coincide with the leaves of the foliation $\mathcal{F}_D$ and we  have   
\[
 \Gamma_{ij} ^k (p)=0 \hspace{1cm} \text{and}\hspace{1cm} R^{j} _{ik\bar{l}}(p)=0 \hspace{1cm} i,j,k,l=1,...,n 
\]
where $ \Gamma_{ij} ^k$ and $R^{j} _{ik\bar{l}}$ denote respectively the connection and the curvature tensors associated to $\omega$  
with respect to the frame
$V_i=\frac{\partial}{\partial z_i}$, $i=1,...,n$.
\end{lemma}
\begin{proof}
We assume that $\omega_{reg} =\partial \bar{\partial } \phi_{reg}$ for a potential  $\phi_{reg}:U\rightarrow \mathbb{R}$.
We consider $\phi_1:U\rightarrow \mathbb{R}$ such that 
\begin{equation}\label{omcor}
D_i \phi_1 (q)=D_iD_j\phi_1 (q)= D_i \phi_1 (q)=0, 
\end{equation}
and such that 

\[
D_i D_jD_k \phi_1 (q) = D_i D_jD_k \phi_{reg} (q)\hspace{1cm}  D_i D_jD_k D_l\phi_1 (q) = D_i D_jD_kD_l \phi_{reg} (q)
\]
 We then consider 
a cut-off function $\chi:X\rightarrow \mathbb{R}$ such that
\[
Supp \, \chi \subset \{\sqrt{\sum |z_i|^2}\leq \delta \}\subset U
\]
for some $\delta >0$
 and $\chi^{-1} (\{1\})$ contains a neighborhood of $p$.  For  $\lambda \in \mathbb{R}$ we set $\chi_{\lambda}(z_1,...,z_n) = \chi (\lambda z_1,....,\lambda z_n)$.
We claim that $\phi_\lambda:=-\chi_{\lambda}\phi_1$ for $\lambda$ large enough  is the desired potential which creates the expected correction in the  lemma.
In fact since $\chi_{\lambda}=1$ in a neighborhood of $q$ the third order and forth order derivatives of $\phi_{reg}+\phi_{\lambda}$
all vanish at $q$. Thus one  needs to prove that for $\lambda$ large enough $\omega_{reg}+\partial\bar{\partial}\phi_{\lambda}$
is a K\"ahler metric. To show this we first note that the support of $\chi_{\lambda}$ lies in the disc $\sqrt{\sum |z_i|^2}\leq \frac{\delta}{\lambda}$ and due to (\ref{omcor}) on such a disc we have
\[
|\phi_1|\leq C\frac{1}{\lambda ^3},\hspace{1cm} |D_i\phi_1|\leq C\frac{1}{\lambda^2} ,\hspace{1cm}  |D_iD_j\phi_1|\leq C\frac{1}{\lambda}
\]
Also sacaling leads to 
\[
|D_i D_j\chi_{\lambda}|\leq C' \lambda^2, \hspace{1cm} |D_i\chi_{\lambda} |\leq C' \lambda , \hspace{1cm} |\chi_{\lambda}|\leq C
\]
where $i,j=1,...,2n$ and $C$ and $C'$ are constants which depend  on $\phi_0$ and $\chi$ respectively.
 Hence $D_i D_j (\chi_{\lambda} \phi_1)$ scales  by $\frac{1}{\lambda}$  hence for $\lambda$ large enough $\partial \bar{\partial}\phi_{\lambda}$ can be arbitrarily small and the lemma is proved.
\end{proof}

\subsection{}\label{app5}

\textbf{Canonical  Coordinates Around A point $p\in D$.} In order to simplify computation we need to fix an appropriate holomorphic coordinates $(w_1,...,w_{n-1},z)$ near the point $p\in D$. As before we assume that the divisor $D$ in this coordinates is given by $z=0$ and  
$p=(0,0,...,0)$. 


Let also $\Phi'$ be a local potential for $\omega'$ in this holomorphic coordinates system 
 and let

 \[
\Phi'= \sum_{k,l\geq 0} B_{k,\bar{l}} z^k \bar{z}^l + \sum_{k,l\geq 0} \bar{B}_{k,\bar{l}} \bar{z}^k z^l  
\]

be the Taylor series expansion of $\Phi'$ around the origin in this local coordinates system. Here $B_{k,\bar{l}} (w,\bar{w})$ for $k,l\geq 0$ are assumed to be functions of $(w,\bar{w})$ where $w=(w_1,...,w_{n-1})$.
 
\begin{lemma}\label{rela}
Given a positive integer $m$, there exists  a holomorphic  change of coordinates like $z\rightarrow z$, $w_k\rightarrow w_k+\sum_{i=1} ^m  b_{ki} z^i$ for $k=1,...,n-1$ such that

\[
 (B_{0,\bar{0}})_{w\bar{w}w}(0)=(B_{0,\bar{0}})_{w\bar{w}\bar{w}}(0)=0
\]

\[
(B_{i,\bar{0}})_{\bar{w}_k}(0)=(B_{0,\bar{i}})_{w_k}(0)=0 \hspace{1cm} \text{ for } i=1,...,m, k=1,...,n-1
\]

\[
B_{1,\bar{1}} (0)=(B_{1,\bar{1}})_{w_k} (0), \hspace{1cm}\text{ for } k=1,...,n-1
\]

\[
(B_{2,\bar{1}})_{w_k}(0)=(B_{1,\bar{2}})_{\bar{w}_k}(0)=0\hspace{1cm} \text{ for }  k=1,...,n-1
\]

Consequently in this coordinates the $z$-axis $w_1=...=w_{k-1}=0$  is orthogonal to the hyperplane $z=0$ at $p=(0,...,0)$
and the restriction of the metric $g'=\partial \bar{\partial}\Phi'$  to the $z$-axis in this coordinate is of the order
\[
g'_{w_i\bar{z}}=O(|z|^3), \hspace{1cm} \text{ for } i=1,...,n-1
\]
 \end{lemma}

\begin{proof} 

Let $\Phi_D:=\Phi'|_{D}$ and let 
\[
\Phi_{D}(w,\bar{w}):=B_{0,\bar{0}}(w,\bar{w})=\sum_{\substack{0\leq i+j\leq m \\1\leq k,l\leq n-1}} s^{k,l} _{i,j}(w_k)^i (\bar{w}_l)^j +O(|w|^{m+1})
\]

be the Taylor series expansion of $\Phi_D$
where $s^{k,l} _{i,j}$ for $0\leq i+j\leq m $ and $1\leq k,l\leq n-1$ are complex numbers.
First note that the coefficient of $\bar{w}_kz$ in $\Phi'$ after change of variable can be generated by the terms  $B_{1,\bar{0}}z +\sum_{i,j,l}  s^{k,l}_{1,1} w_k\bar{w}_{l}$. Converting  coordinates  as described in   lemma (\ref{rela}) leads to
\[
\sum_{i,j,l}s^{l,k}_{1,1}  \big (\bar{w}_k+\sum_{i=1} ^n  \bar{b}_{k,i} \bar{z}^i \big )  \big (\sum_{l}  ( w_l+\sum_{i=1} ^n  b_{l,i} z^i ) \big )+ B_{1,\bar{0}}(w_1+\sum_{i=1} ^n  b_{1,i} z^i,...,w_{n-1}+\sum_{i=1} ^n  b_{n-1,i} z^i  )z
\]
and thus the   coefficient of the term $\bar{w}_k z$ is given by
\[
\big [\sum_{l} s^{l,k}_{1,1}  b_{l,1}+ (B_{1,\bar{0}})_{\bar{w}_k}(0) \big ] \bar{w}_kz
\]
The matrix $[s^{l,k} _{1,1}]_{1\leq l,k,\leq n-1}$ is nothing but the  matrix of the restriction of $\omega'$ to $D$ at the origin and thus  is invertible.  Hence the  system of equations $\sum_{l} s^{l,k}_{1,1}  b_{l,1}+ (B_{1,\bar{0}})_{\bar{w}_k}(0) =0$ is  solvable for $b_{l,1}$,for $l=1,...,n-1$.


Now one can inductively determine the coefficients $b_{l,i} $ for $l=1,...,n-1$ and $2\leq i\leq m$ in such a way that all the terms $(B_{i,\bar{0}})_{\bar{w}_k} (0)$ which are the coefficients of $\bar{w}_k z^i$ in the Taylor series expansion of $\Phi '$ around $0$ disapear.
In fact after doing the change of coordinates described in  lemma (\ref{rela}) the term $\bar{w}_k z^l$ for $1\leq k \leq n-1$ and $2\leq l\leq m$ can be generated by any of the terms like $B_{i,0}(w,\bar{w})z^i$ for $0\leq i \leq l$.
The term $\bar{w}_k z^l$  after change of coordinates can occur in    the following sum
\[
\sum_{i=0} ^l [(\sum_{\substack{1\leq i_1,...,i_a\leq n-1\\ 1\leq a\leq l}} \frac{\partial^{a+1}B_{i,\bar{0}}}{\partial\bar{w}_k\partial w_{i_1}...\partial w_{i_a}}(0)\bar{w}_k\prod_{1\leq j\leq a} (w_{i_j}+\sum_{s=1} ^m b_{i_js}z^s)]z^i
\]
thus it is given by

\begin{equation}\label{ap51}
 \sum_{i=0} ^l \sum_{\substack{1\leq i_1,...,i_a\leq n-1\\ 1\leq a\leq l}} \frac{\partial^{a+1}B_{i,\bar{0}}}{\partial\bar{w}_k\partial w_{i_1}...\partial w_{i_a}}(0)(\sum_{\substack{s_1+...+s_a=l-i\\ 1\leq s_1,...,s_{a}\leq m\\1\leq i_1,....,i_a\leq n-1\\ 1\leq a \leq l}} b_{i_1,s_1}...b_{i_a,s_a})] =0
\end{equation}
 From the above relation we can compute all the sums $\sum_{i=1} ^{n-1} \frac{\partial ^2 B_{0,\bar{0}}}{\partial \bar{w}_k \partial w_i} (0)b_{i,l}$ for $k=1,...,n-1$ in terms of $b_{i,j}$ for $i=1,...,n-1$, $j<l$. Thus using induction  and invertibility of the matrix $[\frac{\partial ^2 B_{0,\bar{0}}}{\partial \bar{w}_k \partial w_i} (0)]_{1\leq k,i\leq n-1}$ we can find  $b_{i,j}$ in such a way that  the equation (\ref{ap51}) is satisfied.

Now consider  the equation
\begin{equation}\label{ap52}
\det [g'_{i\bar{j}}]=O(|z|^2)
\end{equation}

according to the above discussion we can assume that  we have $(B_{1,\bar{0}})_{w_i} (0) = (B_{0,\bar{1}})_{\bar{w}_i} (0) =0$.  The zero order term in the left hand side of  (\ref{ap52})  with respect to $z$-expansion consists of
\[
\det [ (B_{0,\bar{0}})_{w_i \bar{w}_j} (0) ]_{1\leq i,j \leq n-1}\times B_{1,\bar{1}}(0)
\] 

by the right hand side of the relation (\ref{ap52} )   this term vanishes. Since $[ (B_{0,\bar{0}})_{w_i \bar{w}_j} (0) ]_{1\leq i,j \leq n-1}$ represents the matrix of the restriction of the metric $g'$ on $D$ at $p$ it is invertible and  we obtain
\begin{equation}\label{ap53}
 B_{1,\bar{1}}(0)=0
\end{equation}

By taking the derivative of (\ref{ap52}) with respect to $w_k$ for $k=1,...,n-1$ and using the coordinates system on $D$ in such a way that 
\begin{equation}\label{ap2cn}
\frac{\partial g_{w_i\bar{w}_j} }{\partial  w_k } (0) =\frac{\partial g_{w_i\bar{w}_j} }{\partial  \bar{w}_k } (0) =0\hspace{1cm} \text{ for all } 1\leq i,j,k\leq n-1, 
\end{equation}
we also  obtain

\begin{equation}\label{ap54}
(B_{1,1})_{w_i}(0)=(B_{1,1})_{\bar{w}_i}(0)=0
\end{equation}


Setting the coefficient of $z$ in the left hand side of (\ref{ap52}) equal to zero and applying (\ref{ap53}) and (\ref{ap54}) one obtains
\[
2B_{2,1}(0)\det [(B_{0,0})_{w_i\bar{w}_j} (0)]=0
\]

hence 
\begin{equation}
B_{2,1}(0) =0
\end{equation}

Similarly   taking the derivative of both sides of (\ref{ap52}) with respect to $w_k$ and using (\ref{ap2cn}), (\ref{ap53}) and (\ref{ap54}) yields 
\[
2(B_{2,1}(0))_{w_k}\det [(B_{0,0})_{w_i\bar{w}_j} (0)]=0
\]

which is equivalent to

\[
(B_{2,1}(0))_{w_k}=0
\]
This completes the proof.

\end{proof}

\subsection{}\label{app6}

\begin{equation}\label{secordloc}
\begin{split}
\Delta_{g'} (\chi\exp \{-C\phi\} (m+\Delta_{reg} \phi))=& \exp \{-C\phi\}(\sum g'^{i\bar{i}} (C^2 \chi\phi_{i}\phi_{\bar{i}}  -C\phi_{i}\chi_{\bar{i}}-C\chi_i \phi_{\bar{i}}) )(m+\Delta_{reg} \phi)\\
\quad& +\exp \{-C\phi\}\sum g'^{i\bar{i}} (\chi_i -C\phi_{i}\chi) (\Delta_{reg} \phi)_{\bar{i}}\\
\quad  & +\exp \{-C\phi\}\sum g'^{i\bar{i}}  (\Delta_{reg} \phi)_{i}(\chi_{\bar{i}} -C\phi_{\bar{i}}\chi)\\
\quad & +\exp \{-C\phi\} (-C\chi\Delta _{g'} \phi+\Delta_{g'} \chi )  (m+\Delta_{reg} \phi)\\
\quad & +\chi \exp \{-C\phi\}\Delta_{g'} (\Delta_{reg} \phi)
\end{split}
\end{equation}

 By Cauchy-Schwarz inequality  we have

\begin{equation}\label{secordloc2}
\begin{split}
\sum g'^{i\bar{i}} [(-\chi_i +C\phi_{i}\chi)& (\Delta_{reg} \phi)_{\bar{i}}
 +  (\Delta_{reg} \phi)_{i}(-\chi_{\bar{i}} +C\phi_{\bar{i}}\chi)]\leq\\
\quad & \frac{1}{\chi} \sum g'^{i\bar{i}}(\chi_{\bar{i}} -C\phi_{\bar{i}}\chi )(\chi_i -C\phi_{i}\chi)  (m+\Delta_{reg} \phi )+\\
\quad &\hspace{3cm} \chi\sum   g'^{i\bar{i}}      (\Delta_{reg} \phi)_{i}  (\Delta_{reg} \phi)_{\bar{i}} (m+\Delta_{reg} \phi )^{-1}\\
\quad & =\sum   g'^{i\bar{i}}  (C^2 \chi\phi_{i}\phi_{\bar{i}}  -C\phi_{i}\chi_{\bar{i}}-C\chi_i \phi_{\bar{i}}) )(m+\Delta_{reg} \phi )\\
\quad & +\chi\sum   g'^{i\bar{i}}      (\Delta_{reg} \phi)_{i}  (\Delta_{reg} \phi)_{\bar{i}} (m+\Delta_{reg} \phi )^{-1}+\sum  g'^{i\bar{i}}\frac{\chi_{i}\chi_{\bar{i}}}{\chi}(m+\Delta_{reg} \phi )
\end{split}
\end{equation}
From (\ref{secordloc}) and (\ref{secordloc2})    
\begin{equation}
\begin{split}
\Delta_{g'} (\chi\exp \{-C\phi\} &(m+\Delta_{reg} \phi))\geq \\ 
\quad & - \chi\exp \{-C\phi\}\sum   g'^{i\bar{i}}      (\Delta_{reg} \phi)_{i}  (\Delta_{reg} \phi)_{\bar{i}} (m+\Delta_{reg} \phi )^{-1}-\exp\{-C\phi\}\sum  g'^{i\bar{i}}\frac{\chi_{i}\chi_{\bar{i}}}{\chi}(m+\Delta_{reg} \phi )\\
\quad & +\exp \{-C\phi\} (-C\chi\Delta_{g'} \phi+\Delta_{g'} \chi )  (m+\Delta_{reg} \phi)\\
\quad & +\chi \exp \{-C\phi\}\Delta_{g'} (\Delta_{reg} \phi)\\
\quad & =  -\chi \exp \{-C\phi\}\sum   g'^{i\bar{i}}      (\Delta_{reg} \phi)_{i}  (\Delta_{reg} \phi)_{\bar{i}} (m+\Delta_{reg} \phi )^{-1}-C\chi \exp \{-C\phi\} (\Delta_{g'} \phi)  (m+\Delta_{reg} \phi)\\
\quad &  +\exp \{-C\phi\} (\Delta_{g'} \chi -\frac{|\nabla_{g'} \chi|^2 _{g'}}{\chi} )(m+\Delta_{reg} \phi)+\chi \exp \{-C\phi\}\Delta_{g'} (\Delta_{reg} \phi)
\end{split}
\end{equation}

According to relation (2.15)   reference (\cite{y})  by taking  a coordinates such that $g_{i\bar{j}}=\delta_{ij}$ and $\phi_{i\bar{j}}=\phi_{i\bar{i}}\delta_{ij}$ we know that
\[
  \sum   (1+\phi_{i\bar{i}})^{-1}      (\Delta_{reg} \phi)_{i}  (\Delta_{reg} \phi)_{\bar{i}} (m+\Delta_{reg} \phi )^{-1}\leq \sum (1+\phi_{k\bar{k}})^{-1} (1+\phi_{i\bar{i}})^{-1}\phi_{k\bar{i}\bar{j}}\phi_{i\bar{k}j}
\]
therefore using relation (2.14) in (\cite{y}),
\[
\begin{split}
  -\chi (m+\Delta_{reg} \phi )^{-1}\sum   g^{i\bar{i}}   &   (\Delta_{reg} \phi)_{i}  (\Delta_{reg} \phi)_{\bar{i}}+\chi\Delta_{g'} (\Delta_{reg} \phi)\geq \\
\quad & \geq  \chi\Delta_{reg} G+ \chi(\inf _{i\neq l} R_{i\bar{i}l\bar{l}})[\sum _{i,l} \frac{1+\phi_{i\bar{i}}}{1+\phi_{l\bar{l}}}-m^2]
\end{split}
\]

This leads to 

\begin{equation}\label{f}
\begin{split}
\Delta_{g'} (\chi \exp \{-C\phi\} (m+\Delta_{g_{reg}} \phi))\geq & \exp\{-C\phi\} \bigg ( \chi  \Delta_{reg} (G )- C\chi m^2 |S|^2 \inf _{i\neq l} R_{i\bar{i} l\bar{l}}\bigg ) \\
\quad &\exp \{-C\phi\} (-Cm\chi +\Delta_{g'} \chi-\frac{| \nabla_{g'} \chi|_{g'}^2}{\chi}) (m+\Delta_{reg} \phi ) \\
&\quad +\exp \{-C\phi\}   \chi (C+\inf_{i\neq l} R_{i\bar{i}l\bar{l}})\exp \{\frac{-G}{m-1}\}  (m+\Delta_{reg} \phi) ^{1+1/(m-1)}\\
\end{split}
\end{equation}

If we set $\chi = e^u$ then we have
\[
\chi_{i\bar{i}}-\frac{\chi_{i}\chi_{\bar{i}}}{\chi}= u_{i\bar{i}}e^u
\]
So 
\[
\Delta_{g'} \chi -\frac{|\nabla_{g'} \chi|^2 _{g'}}{\chi}=(\Delta_{g'} u) e^u 
\]
\newpage

\subsection{}\label{app7}

Let $(w_1,...,w_{n-1},z)$ be a holomorphic coordinate system on an open neighborhood $U_p$ of a point $p\in D$.  We construct  two smooth  $(1,0)$- moving frames  $(V_1,...,V_{n-1})$ and $(e_1,...,e_{n})$ over $U_p$ by an  inductive process explained below. We take $e_n$ in such a way that  $e_n$ is $g_{reg}$-perpendicular to $Span \{\frac{\partial }{\partial w_1},...,\frac{\partial }{\partial w_{n-1}}\}$
\[
e_n\perp_{g_{reg}}Span \{\frac{\partial }{\partial w_1},...,\frac{\partial }{\partial w_{n-1}}\},  \text{ everywhere }  \text{ and } \hspace{1cm} \|e_n\|_{g_{reg}}=1
\]
 
the vector $V_n$ is also constructed in such a way that 

\[
V_n\perp_{g'} Span \{\frac{\partial }{\partial w_1},...,\frac{\partial }{\partial w_{n-1}}\} \text{ everywhere and }\hspace{1cm} \|V_n\|_{g_{reg}}=1
\]

Assuming $e_n,..., e_{n-k+1}$ and $V_n,...,V_{n-k}$ are given. 
 We consider the subspace 

\[
W_{k}:=\bigg ( Span\{V_n,..., V_{n-k
}\} \bigg )^{\perp_{g'}}
\]

consisting  of the subspace  orthogonal to $Span\{V_n,..., V_{n-k}\} $ with respect to $g'$. 
Then we choose $\tilde{V}_{n-k-1}$ to be the  orthogonal projection of $V_{n-k} $ on $W_{k}$ with respect to $g_{reg}$ and we set $V_{n-k-1}:= \frac{\tilde{V}_{n-k-1}}{\|V_{n-k-1}\|_{g_{reg}}}$.  The vector  $e_{n-k}\in Span \{V_{n-k},V_{n-k-1}\}$ is selected  such that 

\[
\|e_{n-k}\|_{g_{reg}}=1 \hspace{1cm} \text{ and }\hspace{1cm} e_{n-k}\perp_{g_{reg}} V_{n-k-1}
\]
 Then  $\tilde{V}_{n-k-1}$ is selected  to be the  orthogonal projection of $V_{n-k} $ on $W_{k}$ with respect to $g_{reg}$ and we set $V_{n-k-1}:= \frac{\tilde{V}_{n-k-1}}{\|V_{n-k-1}\|_{g_{reg}}}$. 
  Hence 

\[
V_{n} = \cos \theta_n V_{n-1} + \sin\theta_n e_n,
\]
\[
V_{n-1} = \cos \theta_{n-1} V_{n-2} + \sin\theta_{n-1} e_{n-1} 
\]
\centerline{...}
\[
V_{n-k+1} = \cos \theta_{n-k+1} V_{n-k} + \sin\theta_{n-k+1} e_{n-k+1} 
\]
\centerline{...}

for some angles $\theta_{n},...,\theta_{2}$.

\begin{lemma}\label{m'm'1}
Let $\mathscr{M}'$ be a lower bound for $g'|_{Span\{\frac{\partial}{\partial w_1},...,\frac{\partial}{\partial w_{n-1}}\}}$ so that $\|V_{i}\|_{g'}\geq \mathscr{M}'$ for $i=1,...,n$.
\begin{equation}
e^{-G}(\mathscr{M}')^{2n-2} \leq \frac{|S|^2}{\|V_n\|_{g'}^2}\leq \mathscr{M'}_1
\end{equation}

for some constant $\mathscr{M'}_1$ only depends on  $G$ and $g_{reg}$.
\end{lemma}

\begin{proof}
$\langle V_{n-k}, V_{n-k+1}\rangle_{g'} =0 $ so 

\[
  \langle V_{n-k} ,e_{n-k+1}\rangle_{g'}= -\cot \theta_{n-k+1} \|V_{n-k}\|^2 _{g'} 
\] 

From $\|V_{n-k+1}-V_{n-k} \cos \theta_{n-k+1}\|_{g'} ^2 =\|e_{n-k+1}\|_{g'} ^2$ we can deduce that  

\[
\|V_{n-k+1}\|^2 _{g'}=-\cos ^2 \theta_{n-k+1} \|V_{n-k}\| ^2 _{g'} +\sin ^2 \theta_{n-k+1}  \|e_{n-k+1}\|_{g'} ^2 
\]

In particular

\begin{equation}\label{tg2}
\tan^2 \theta  _{n-k+1} \|e_{n-k+1}\|_{g'} ^2 \geq \|V_{n-k}\|^2 _{g'} 
\end{equation}
 
equivalently we have
\begin{equation}\label{tannk}
\tan ^2  \theta_{n-k+1} \geq \frac{(\mathscr{M}')^2}{ C_{\ref{secondorder1}}}
\end{equation}

If $\frac{\pi}{4}\leq \theta_{n-k+1}\leq\frac{\pi}{2}$ then 

\begin{equation}\label{sincase1}
\sin^2 \theta_{n-k+1} \geq \frac{1}{2},
\end{equation}

 and if $0< \theta_{n-k}\leq \frac{\pi}{4}$
 then since $\frac{1}{\cos \theta_{n-k}}\leq \frac{1}{\sqrt{2}}$ from \ref{tannk} it  can be  deduced that
\begin{equation}\label{sincase2}
\sin^2 \theta_{n-k+1}\geq \frac{2(\mathscr{M}')^2}{ C_{\ref{secondorder1}}}
\end{equation}

From (\ref{sincase1}) and (\ref{sincase2}) we conclude that

\begin{equation}
\frac{1}{\sin ^2 \theta_{n-k+1}}\leq \max\{ 2,\frac{ C_{\ref{secondorder1}}}{2(\mathscr{M}')^2} \}
\end{equation}

therefore if we set 
\begin{equation}\label{scrm''}
\mathscr{M}'':=  \max\{ 2,\frac{ C_{\ref{secondorder1}}}{2(\mathscr{M}')^2} \}
\end{equation}

 Then  the Monge Amp\`ere equation implies
 \begin{equation}\label{v1vn}
\|V_n\|^2 _{g'}...\|V_1\|^2 _{g'}= e^G |S|^2 \prod_{k=2} ^n |\sin \theta_k|^2
\end{equation}
  
\begin{equation}\label{tkmin}
\frac{|S|^2}{\|V_n\|^{2} _{g'}}=e^{-G}\frac{\prod_{k=1} ^{n-1} \|V_k\| ^2 _{g'}}{\prod_{k=2} ^n |\sin \theta_k|^2}\leq \frac{C_{\ref{tkmin}}}{(\mathscr{M}'' )^{n-2}}
\end{equation}

where according to  proposiiton (\ref{secondorder}) $C_{\ref{tkmin}}$ only depends on  $G$ and $g_{reg}$.
Since $\mathscr{M}'' >2$ we have 
\begin{equation}\label{ttkmin}
\frac{|S|^2}{\|V_n\|^{2} _{g'}}\leq \mathscr{M'}_1
\end{equation}
where $\mathscr{M'}_1$ depends only on $G$ and $g_{reg}$.
 
\begin{equation}\label{tkmin2}
\frac{|S|^2}{\|V_n\|^{2} _{g'}}=e^{-G}\frac{\prod_{k=1} ^{n-1} \|V_k\| ^2 _{g'}}{\prod_{k=2} ^n |\sin \theta_k|^2} \geq e^{-G}(\mathscr{M}')^{2n-2}
\end{equation}

\end{proof}

\begin{rem}\label{lowerb}
We remark  here that if  $\|V_i\|_{g'}\geq \mathscr{M}'$, for some constant $ \mathscr{M}'$ and   for $1\leq i \leq n-1$  then $\prod_{1\leq i \leq n-1 } \|V_i\| ^2 _{g'}\geq ( \mathscr{M}')^{2n-2}$. Conversely if  $\prod_{1\leq i \leq n-1 } \|V_i\| ^2 _{g'}\geq  \mathscr{M}$ for some constant $ \mathscr{M}$ then we will have 
$\|V_i \|_{g'}\geq \frac{\sqrt { \mathscr{M}}}{C_{\ref{secondorder1}}^{n-2}}$, where $C_{\ref{secondorder1}}$ is an upper bound for $\|V_{i}\|_{g'}$ which according to   proposition (\ref{secondorder}) only depends on $G$ and $g_{reg}$. If we  set
\begin{equation}\label{ndet}
det_{g_{reg}} (g'_D):=\prod_{1\leq i \leq n-1 } \|V_i\| ^2 _{g'}
\end{equation}

so a lower bound for $g'|_{D}$ on  $U_p\cap D$ can be  determined in terms of $G$, $g_{reg} $ and  a lower bound $\mathscr{M}$ for  $det_{g_{reg}} (g'|_D)$  and vice versa.
 
 \end{rem}

\end{document}